\documentclass[10pt,reqno]{elsarticle}

\usepackage{textcomp,times,amsmath,amsthm,amssymb,amsfonts,epsfig,color,graphicx,enumitem,accents,booktabs,accents}
\usepackage{rotating,multirow,color,url,mathrsfs,mathrsfs,bm}
\usepackage{arydshln}
\usepackage{bbm}
\usepackage{esint}
\usepackage{gensymb}
\usepackage{epstopdf}
\usepackage{comment}
\usepackage{leftidx}
\usepackage{mathtools}
\usepackage{subfig}
\usepackage{float}
\usepackage{tcolorbox}
\usepackage{tabularx}
\usepackage[toc,title,page]{appendix}
\newtheorem{theorem}{Theorem}[section]
\newtheorem{corollary}[theorem]{Corollary}
\newtheorem{lemma}[theorem]{Lemma}
\newtheorem{proposition}[theorem]{Proposition}
\newtheorem{conjecture}[theorem]{Conjecture}

\theoremstyle{definition}

\newtheorem{remark}[theorem]{Remark}
\newtheorem{example}[theorem]{Example}

\numberwithin{equation}{section}
\numberwithin{figure}{section}
\numberwithin{table}{section}

\DeclareSymbolFont{matha}{OML}{txmi}{m}{it}
\DeclareMathSymbol{\varv}{\mathord}{matha}{118}

\newcommand{\e}{{\mathrm{e}}}

\newcommand{\R}{\mathbb{R}}

\newcommand{\Rd}{\R^2}

\newcommand{\Comp}{{\mathcal{C}}}
\newcommand{\jp}{{j_+}}
\newcommand{\rhop}{{\rho_+}}

\newcommand{\ind}{\mathbf{\chi}}
\newcommand{\ov}{\overline}

\newcommand{\Z}{{Z}}

\newcommand{\Prob}{{(\mathcal{P})}}

\newcommand{\symtens}{{\,\odot\,}}
\newcommand{\weakstar}{{\overset{\ast}{\rightharpoonup}}}
\newcommand{\mbf}[1]{{\mathbf{#1}}}

\newcommand{\opnorm}[1]{{\left\vert\kern-0.25ex\left\vert\kern-0.25ex\left\vert #1 
		\right\vert\kern-0.25ex\right\vert\kern-0.25ex\right\vert}}

\def\O{\Omega}
\newcommand{\Ob}{{\overline{\Omega}}}
\newcommand{\bO}{{\partial\Omega}}

\newcommand{\argu}{{\,\cdot\,}}

\newcommand{\Mes}{{\mathcal{M}}}

\newcommand{\Sdd}{{\mathcal{S}^{2 \times 2}}}
\newcommand{\Sddp}{{\mathcal{S}^{2 \times 2}_+}}
\newcommand{\Sddm}{{\mathcal{S}^{2 \times 2}_-}}

\newcommand{\DIV}{{\mathrm{Div}}}

\newcommand{\dive}{{\mathrm{div}}}

\newcommand{\sig}{{\sigma}}

\newcommand{\TAU}{{\varsigma}}

\newcommand{\tr}{{\mathrm{Tr}}}

\newcommand{\Cmin}{{\mathcal{C}_{\mathrm{min}}}}
\newcommand{\Totc}{{V_0}}
\newcommand{\ro}{{\rho}}
\newcommand{\dro}{{\rho^0}}

\newcommand{\relProb}{{(\overline{\mathcal{P}})}}
\newcommand{\dProb}{{(\mathcal{P}^*)}}

\newcommand{\pairing}[1]{{\left \langle #1 \right \rangle}}
\newcommand{\norm}[1]{\Arrowvert #1 \Arrowvert}
\newcommand{\abs}[1]{{\left \lvert #1 \right \rvert}}

\newcommand{\eps}{\varepsilon}
\newcommand{\Rb}{\overline{\mathbb{R}}}

\newcommand{\D}{\mathcal{D}}

\newcommand{\V}{{\mathcal{V}}}
\newcommand{\Ha}{\mathcal{H}}
\newcommand{\IM}{\mathrm{Im}}

\newcommand{\mres}{\mathbin{\vrule height 1.6ex depth 0pt width
		0.13ex\vrule height 0.13ex depth 0pt width 1.3ex}}

\newcommand{\DO}{{\D(\Omega;\R^2 \times \R)}}
\newcommand{\vro}{{\varrho}}
\newcommand{\jsq}{{\tfrac{1}{2}\langle\sig^{-1} q,q \rangle}}
\newcommand{\q}{{q}}

\newcommand{\Cm}{{\mathrm{C}}}
\newcommand{\te}{{\theta}}
\newcommand{\Kro}{{\ov{\mathcal{K}}_\rho}}
\newcommand{\Krop}{{\mathcal{K}_\rho}}

\DeclareRobustCommand{\ub}[1]{\underaccent{\bar}{#1}}
\newcommand{\Oh}{{\ub{\Omega}}}
\newcommand{\Obh}{{\ub{\Ob}}}

\makeatletter
\def\ps@pprintTitle{%
	\let\@oddhead\@empty
	\let\@evenhead\@empty
	\def\@oddfoot{}%
	\let\@evenfoot\@oddfoot}
\makeatother

\begin{document}
	
	\begin{frontmatter}
	
	\title{Optimal design of plane elastic membranes using the convexified F\"{o}ppl's model}

	
	\author{Karol Bo{\l}botowski}
	
	\address{Department of Structural Mechanics and Computer Aided Engineering, Faculty of Civil Engineering\\
	Warsaw University of Technology, 16 Armii Ludowej Street, 00-637 Warsaw
	}

	\ead{karol.bolbotowski@pw.edu.pl}
	
	\date{\today}

	\begin{abstract}
		This work puts forth a new optimal design formulation for planar elastic membranes. The goal is to minimize the membrane's compliance through choosing the material distribution described by a positive Radon measure.	
		The deformation of the membrane itself is governed by the convexified F\"{o}ppl's model.
		The uniqueness of this model lies in the convexity of its variational formulation despite the inherent nonlinearity of the strain-displacement relation.
		It makes it possible to rewrite the optimization problem as a pair of mutually dual convex variational problems. In the primal problem a linear functional is maximized with respect to displacement functions while enforcing that point-wisely the strain lies in an unbounded closed convex set. The dual problem consists in finding equilibrated stresses that are to minimize a convex integral functional of linear growth defined on the space of Radon measures. The pair of problems is analysed: existence and regularity results are provided, together with the system of optimality criteria. To demonstrate the computational potential of the pair, a finite element scheme is developed around it. Upon reformulation to a conic-quadratic \& semi-definite programming problem, the method is employed to produce numerical simulations for several load case scenarios.

	\end{abstract}
	
	\begin{keyword}
		Plane membrane \sep optimal design \sep minimum compliance \sep convexified F\"{o}ppl's model \sep convex duality \sep finite element method \sep conic-quadratic programming \sep semi-definite programming\\
		\vspace{0.2cm}
		\MSC[2020] 74P05, 49K20, 65K99, 49J45, 49J20
	
	\end{keyword}
	
	\end{frontmatter}


\section{Introduction}
\label{sec:introdction}

\subsection{Optimizing membranes versus classical optimal design formulations in conductivity and elasticity }
\label{ssec:linear_model}

The out-of-plane equilibrium of a planar membrane occupying a domain $\Omega \subset \R^2$ and clamped on its boundary $\bO$ can be written down as follows:
\begin{equation}
		\label{eq:memb_model}
	- \dive(\sigma \nabla w) = f \quad \text{in} \quad  \Omega.
\end{equation}
The unknown function $w:\Omega \to \R$ represents the membrane's out-of-plane displacement (normal displacement or deflection, for short). It must satisfy the Dirichlet boundary condition $w=0$ on $\bO$. The applied load $f: \Omega \to \R$ acts perpendicularly to the plane.  The in-plane stress in the membrane is modeled by a function  $\sigma: \Omega \to \Sddp$, valued in symmetric positive semi-definite matrices. The positivity condition enforces that the membrane is capable of withstanding tensile stresses only, thus excluding the compressive ones. In the simplest variant, one assumes that the stress does not depend on the load $f$, and that it is homogeneous and hydrostatic, i.e. $\sigma = s\, \mathrm{I}_2$ where $s>0$ is a positive constant, $\mathrm{I}_2$ is the $2 \times 2$ identity matrix. Such a setting is justified if on the boundary $\bO$ the membrane is subject to a pre-stress exerted by a normal traction of intensity $s$ that is sufficiently large, see e.g. \cite[Section 5.10]{ciarlet2000b}. By taking such $\sigma$ we recast the most basic membrane model to be found as the leading example of a linear boundary value problem in some of the classical books on mathematical physics \cite{courant1924}. Indeed, in this case the equilibrium equation \eqref{eq:memb_model} reduces to the Poisson equation  $- s\, \triangle w = f$. 

In the sequel of this subsection $\Omega$ will not represent the membrane itself, but rather the design domain in which the material is to be optimally distributed. When addressing the issue of optimal design of membranes one cannot ignore the fact that equation \eqref{eq:memb_model} governs yet another problem in physics: the one of planar stationary heat flow where, in turn, $w$ is temperature, $f$ is the heat source/sink, and in place of $\sigma$ one has the conductivity tensor $A:\Omega \to \Sddp$ that characterizes the 2D conductor. Meanwhile, optimization of conductors is a classical topic. Various settings are possible; in majority of them we essentially look for the field $A$ that satisfies some constraint on the amount of the material and that minimizes the adequately defined energy. If we assume that $A = a \mathbbm{1}_\omega \mathrm{I}_2$, where $a>0$ is a fixed positive constant and $ \mathbbm{1}_\omega$ is the characteristic function of a subdomain $\omega \subset \Omega$ satisfying the constraint on the volume $\abs{\omega} \leq V_0 < \abs{\Omega}$, then we find ourselves in the framework of \textit{shape optimization} \cite{sokolowski1992,plotnikov2023,bouchitte2014}. In general, optimal shapes do not exists, and it is when the \textit{relaxation by homogenization method} \cite{kohn1986optimal,allaire2002} intervenes. A different approach consists in finding an optimal \textit{mass} or material distribution \cite{bouchitte2001,bouchitte2007} being a positive measure $\mu \in \Mes_+(\Ob)$ satisfying $\mu(\Ob) \leq V_0$, which generates a matrix valued measure $\mathcal{A} = (a \mathrm{I}_2) \mu \in \Mes(\Ob;\Sddp)$. \textit{The free material design} \cite{bouchitte2001,lewinski2021} is a variant of the latter formulation, when an arbitrary anisotropic measure $\mathcal{A} \in \Mes(\Ob;\Sddp)$ is admissible provided that $\int_\Ob \tr \, \mathcal{A} \leq V_0$ . Most of the work cited above concerns not only the optimal design in conductivity, but also 2D or 3D linear elasticity. Then, one usually minimizes the \textit{compliance}, defined as the work of the load on the induced displacements; cf. \cite{bolbotowski2022b} for the free material design problem in this context.

Regardless of the aforementioned mathematical reminiscence between the statics of a membrane and the stationary heat flow, the far reaching advances in the optimal design of conductors quoted above cannot be transferred to optimizing membranes. We must mind that, unlike the conductivity tensor $A:\O \to \Sddp$, the stress tensor $\sigma:\O \to \Sddp$ is not a material property that we can control freely point by point as we see fit. Since $\sigma$ is an in-plane stress, it must meet the in-plane equilibrium equation:
\begin{equation}
	\label{eq:in-plane_eq}
	- \DIV\, \sigma = 0 \quad \text{in} \quad  \Omega,
\end{equation}
where the zero right hand side reflects the lack of an in-plane loading. In the paper \cite{bolbotowski2022a}, co-written by the present author, a formulation of optimal design of membranes is proposed by adding \eqref{eq:in-plane_eq} as a constraint in the free material design formulation. Formally, one seeks a stress tensor $\sigma: \O \to \Sddp$ satisfying $\int_\O \tr\,\sigma \,dx \leq V_0$ and $-\DIV \, \sigma = 0$ in $\Omega$; the target is to minimize the membrane's compliance. Assuming that for a given $\sigma$ the equation \eqref{eq:memb_model} admits a solution $w_\sigma$, the compliance can be defined as $\int_\O w_\sigma f \, dx$.  To the best of the present author's knowledge, such an optimization setting is entirely new in the literature, and so are the optimal designs obtained in \cite{bolbotowski2022a}. In fact, it is the author's belief  that the topic of optimizing membranes as mechanical systems is rather left out overall.

It is only fair to point out a drawback of the paper \cite{bolbotowski2022a}: the premise for the optimal design formulation is heuristic from the mechanical point of view. Above all, it is not \textit{a priori} justified to control the stress $\sigma$ in the whole domain $\Ob$ as if it was a design variable -- physically, we cannot impose the stress anywhere but on the boundary $\bO$. In the interior $\O$ the stress $\sigma$ should be induced via a constitutive law of elasticity, as a response to strain. The structure's cost being an integral of the trace of $\sigma$ is just as debatable. Both of these ideas were inspired by the famous Michell problem \cite{bouchitte2008,lewinski2019a}; the latter formulation, however, received many justifications over the years,  \cite{allaire1993,bouchitte2020,babadjian2023,bolbotowski2022b} to name a few. One of the main goals of the present contribution is to pose a design formulation that is equivalent to \cite{bolbotowski2022a} and, at the same time, more sound mechanically. For that purpose we will employ a more advanced membrane model that does not assume $\sigma$ in \eqref{eq:memb_model} to be pre-fixed but rather that $\sigma$ depends on the membrane's deformation and, effectively, on the load $f$ itself.

\subsection{The convexified F\"{o}ppl's membrane model: the primal and dual formulations}
\label{ssec:foppl_intro}

Discussed above, the very simplistic membrane model revolving around the out-of-plane equilibrium equation \eqref{eq:memb_model} suffers from several defects in the view of physics.  Firstly, the material does not play any role. Thus, such a membrane should not be considered an \textit{elastic} body. Indeed, the model lacks a constitutive law of elasticity that, typically, establishes the relation between the stresses and strains, see \cite{ciarlet2000a}. In fact, as mentioned before, in \eqref{eq:memb_model} the in-plane stress $\sigma$ is predefined. Meanwhile, the strain itself is not defined at all. Consequently, such a model is not predestined for a rational optimization formulation where a material distribution is looked for, in one setting or the other.

A model that responds to the foregoing issues was put forth by A. F\"{o}ppl \cite{foppl1907} at the dawn of the 19th century. Therein, the in-plane strain $\xi$, being a function valued in symmetric matrices $\Sdd$, depends not only on the deflection $w: \Omega \to \R $ but also on the in-plane (tangential) displacement function $u: \Omega \to \Rd$. Assuming smoothness of the two functions, the strain was proposed by F\"{o}ppl as follows:
\begin{equation}
	\label{eq:foppl_strain}
	\xi = \tfrac{1}{2} \nabla w \otimes \nabla w + e(u) \ : \ \O \  \to \ \Sdd,
\end{equation}
where $e(u) = \frac{1}{2} \big( \nabla u + (\nabla u)^\top \big)$ is the symmetric part of the gradient of $u$. We can see that the tangential displacement $u$ enters the strain formula to first order, whilst the normal one $w$ to the second. The formula \eqref{eq:foppl_strain} is  better known from the von K\'{a}rm\'{a}n's plate theory \cite{vonkarman1910,ciarlet1980} that emerged few years later; in the latter theory a bending term involving the Hessian $\nabla^2 w$ is also present. 

The quadratic term with respect to $\nabla w$ in \eqref{eq:foppl_strain} renders the F\"{o}ppl's membrane theory geometrically nonlinear as opposed to e.g.: linearized 2D and 3D elasticity, Kirchhoff's plate theory, or, virtually, to the model described in the previous subsection. The nonlinearity, however, is not "full" since the quadratic term $\frac{1}{2}(\nabla u)^\top \nabla u$ is omitted. One can say that the F\"{o}ppl's model assumes small in-plane deformation and moderate deflections. The additional term is accounted for in the \textit{fully nonlinear membrane models}, which may be found in e.g. \cite{fox1993,ledret1995}. The F\"{o}ppl's model can be thus deemed as \textit{partially geometrically nonlinear}.

Assuming a uniform thickness $b_0$ of a membrane that is clamped on the boundary $\bO$, let us now set the F\"{o}ppl's problem by writing down its variational formulation:
\begin{equation}
	\label{eq:foppl_model}\tag*{$(\mathrm{F\ddot{o}})$}
	\inf \left\{  \int_\O j\bigl(\tfrac{1}{2} \, \nabla w \otimes \nabla w + e(u) \bigr) \, b_0 \, dx  - \int_\O w f \,  dx  \ : \  (u,w) \in \D\bigl(\Omega;\R^2 \times \R\bigr) \right\}
\end{equation}
where $\D(\O;\R^d)$ stands for the space of $\R^d$-valued smooth functions with compact support in $\O$.
Above $j:\Sdd \to \R_+$ is a quadratic energy potential, more accurately $j(\xi) = \frac{1}{2} \pairing{\mathscr{H}\xi,\xi}$ for a symmetric positive semi-definite linear operator $\mathscr{H}$ on $\Sdd$, which is known as the Hooke tensor. Of course, the smoothness conditions for $(u,w)$ must be relaxed: the natural functional space for a solution would be $W^{1,2}(\O;\Rd) \times W^{1,4}(\O;\R)$. However, the first term of the minimized functional, i.e. the stored elastic energy, is not weakly lower semi-continuous in this space. In particular, the problem \ref{eq:foppl_model} is not convex. All in all, the formulation \ref{eq:foppl_model} is not guaranteed to have a solution in general.

The sequences on which the weak lower semi-continuity fails have a clear and appealing physical interpretation, cf. \cite{belgacem2002,conti2006}. Let us take affine functions  $(u,w) : \O \to \Rd \times \R$ that generate a contraction of the membrane, that is a constant strain $\xi = \tfrac{1}{2}  \nabla w \otimes \nabla w + e(u)$ which is non-zero and negative semi-definite. Then, one can point to a sequence of smooth functions $(u_h,w_h)$ such that $(u_h,w_h) \rightharpoonup (u,w)$ weakly in $W^{1,2}(\O;\Rd) \times W^{1,4}(\O;\R)$, and yet $\xi_h = \tfrac{1}{2}  \nabla w_h \otimes \nabla w_h + e(u_h) \to 0$ strongly in $L^2(\O;\Sdd)$. Accordingly, we find that $\int_\O j(\xi_h) \, b_0 \, dx $ converges to zero, whilst  $ \int_\O j (\xi)\, b_0 \, dx $ is strictly positive, ultimately disapproving the weak lower semi-continuity. A loose conclusion can be drawn: contractions of the membrane should not contribute to the stored elastic energy. Mechanically, contracting a membrane made of thin film or canvas does not require compressive stresses; instead, it can be realized through fine-scale oscillation of the displacements that produce a near-isometric deformation of the membrane. Such oscillation -- or, to put it differently, \textit{wrinkling} of the membrane -- is not captured by the functions $(u,w)$ themselves but is recovered through the sequence $(u_h-u,w_h-w)$ provided it is energetically profitable. The wrinkling phenomenon in thin films was the subject of many other works, see e.g. \cite{bella2014,tobasco2021}.

The foregoing remarks point to a need for relaxation of the formulation \ref{eq:foppl_model}. In general, the fine-scale oscillation sequences can be constructed so that  the strain $\xi:\Omega \to \Sdd$ defined in \eqref{eq:foppl_strain} is energetically corrected by a positive semi-definite function $\zeta : \Omega \to \Sddp$. Effectively, the stored elastic energy that recognizes the membrane's capability to wrinkle must use the \textit{relaxed energy potential}:
\begin{equation}
	\label{eq:j_plus_0}
	j_+(\xi) :=  \min\limits_{\zeta \in \Sddp} j(\xi +\zeta) \qquad \forall\,\xi \in \Sdd.
\end{equation}
The relaxed potential $j_+$ enjoys a crucial property that the potential $j$ does not: \begin{equation}
	\label{eq:g_0}
	\text{the mapping} \quad  \Sdd \times \Rd \ni (\xi,\te) \ \ \mapsto \ \ g(\xi,\te) :=  j_+\bigl(\tfrac{1}{2} \, \theta \otimes \theta + \xi \bigr) \quad \text{is jointly convex.}
\end{equation}
Accordingly, the following relaxation of the F\"{o}ppl's formulation \ref{eq:foppl_model} is its convexification:
\begin{equation}
	\label{eq:conv_foppl_model}\tag*{$(\mathrm{Co F\ddot{o}})$}
	\inf \left\{  \int_\O j_+\bigl(\tfrac{1}{2} \, \nabla w \otimes \nabla w + e(u) \bigr) \, b_0 \, dx  - \int_\O w f \,  dx  \ : \  (u,w) \in \D\bigl(\Omega;\R^2 \times \R\bigr) \right\}
\end{equation}
Of course, above we delivered merely a sketch of the reasoning that leads to proposing \textit{the convexified F\"{o}ppl's model} (or the \textit{relaxed  F\"{o}ppl's model}). Nonetheless, its rigorous justification can be found in \cite{conti2006} where the authors showed that \ref{eq:conv_foppl_model} emerges as a $\Gamma$-limit of a sequence of nonlinear 3D elasticity problems posed in cylinders $\Oh_\eps = \O \times (-\eps/2,\eps/2)$ for the thickness $\eps$ approaching zero. To arrive exactly at the problem \ref{eq:conv_foppl_model} certain assumptions have to be made. 
Firstly, the 3D bodies must be fully clamped on their sides $\bO \times (-\eps/2,\eps/2)$. Secondly, we must decide on how the load scales with the thickness by choosing a factor $\eps^\alpha$, and then we must suitably rescale the total potential energy by a factor of $1/\eps^\beta$. In \cite{conti2006} it is showed that the problem \ref{eq:conv_foppl_model} is obtained through $\Gamma$-convergence whenever $\alpha \in (0,3)$ and $\beta = 4/3\, \alpha$. The reader is also encouraged to confer the paper \cite{friesecke2006} which shows that driving the asymptotic 3D-to-2D process by the two exponents $\alpha,\beta$ leads to a whole hierarchy of different 2D theories, including the plate models of von K\'{a}rm\'{a}n and Kirchhoff.

Naturally, the smoothness conditions in \ref{eq:conv_foppl_model} must be relaxed as well. The minimized functional is now weakly lower semi-continuous in $W^{1,2}(\O;\Rd) \times W^{1,4}(\O;\R)$, but it is not coercive in this space. In fact, for the minimizing sequence $(u_h,w_h)$ the relaxed potential $j_+$ controls only the positive part of the strain $\tfrac{1}{2} \nabla w_h \otimes \nabla w_h + e(u_h)$. Therefore, the compactness result must be looked for in a larger functional space. It is here where the boundary condition intervenes: since $u_h=0$ on the boundary, the integral $\int_\O e(u_h) dx$ is a zero matrix by Gauss's divergence theorem. Effectively, this yields uniform bounds on the $L^1$ norm of $e(u_h)$ and on the $L^2$ norm of the gradient $\nabla w_h$. Ultimately, a relaxed solution $(u,w)$ of \ref{eq:conv_foppl_model} is guaranteed to exist in the space $BD(\O;\Rd) \times W^{1,2}_0(\O;\R)$, where $BD$ stands for the space of functions of bounded deformation \cite{temam1985}. Extra constraints are imposed on $u$. For instance, the singular part of the distributional symmetrized gradient, being a Radon measure $\eps(u) \in \Mes(\O;\Sdd)$, must be negative semi-definite. The mechanical interpretation is that jump-type discontinuities of $u$ reflect membrane's contraction by wrinkling concentration or folding: when one part of the membrane goes over the other part. The reader is referred to \cite{conti2006} for more details.

\bigskip

We conclude this subsection by addressing an aspect of the convexified F\"{o}ppl's model that was skipped in \cite{conti2006}: the one concerning stresses in the membrane.  Contrarily, the work \cite{ledret1995} on the fully nonlinear membrane model discusses this matter to large extent. However, the formulation \ref{eq:conv_foppl_model} is exceptional, for it is convex. This unlocks the use of the convex duality methods \cite{ekeland1999}. In the present paper we will show that they lead to the following dual formulation of the convexified F\"{o}ppl's membrane problem:
\begin{equation}
	\label{eq:dual_foppl}\tag*{$(\mathrm{CoF\ddot{o}}^*)$}
	\min \biggl\{ \int_\O j^*(\sig) \, b_0 \, dx + \int_\O \jsq\, b_0\, dx : \!
	\begin{array}{rr}
		\sig \in L^2(\O;\Sddp),\!  & -\DIV \,\sig = 0 \ \ \text{ in } \O\\
		q \in L^{4/3}(\O;\Rd),\!  & -\dive \, q= f/b_0 \ \ \text{ in } \O
	\end{array}\!\!
	\biggr\}
\end{equation}
where $j^* :\Sdd \to \R_+$ is the convex conjugate of $j$, cf. \cite{rockafellar1970convex}.
The functions $\sigma$ and $q$ enjoy the interpretations of in-plane and out-of-plane stress components, whilst the two constraints are membrane's equilibrium equations in the respective directions. It should be noted that the potential  $j^*$ enters the  formulation instead of the relaxed one $j_+^*$.
However, the relaxation is manifested by the condition on  positive semi-definiteness of $\sigma$. A dual face of the relaxation $j_+$ will be revealed: it is to guarantee that the membrane is subject to tensile stresses only.

\subsection{A survey of the possible optimal design formulations}

To the best knowledge of the author the  relaxed F\"{o}ppl's model \cite{conti2006}  has not yet been employed in optimal design of membranes. Meanwhile, the model enjoys features that makes the optimization worthy of exploring. Despite its geometrical nonlinearity, reflected in the strain-displacement relation \eqref{eq:foppl_strain}, its variational formulation \ref{eq:conv_foppl_model} is convex. As a consequence, the dual stress-based formulation \ref{eq:dual_foppl} is at hand. This, in turn, unlocks various convex compliance minimization methods that so far have been applied to bodies and structures modelled by geometrically linear theories. Below we will discuss several possibilities of optimal design formulations using the convexified F\"{o}ppl's model.

In the sequel, for the out-of-plane load of the membrane we shall take a signed Radon measure, namely $f \in \Mes(\Ob;\R)$. To facilitate optimization we will now let go of the assumption that the membrane has a constant thickness $b_0$. In fact, to encompass a more general class of feasible designs, a membrane will be identified with a positive Radon measure $\mu \in \Mes_+(\Ob)$  representing its \textit{material distribution}. For each such measure we define
\begin{equation}
	\label{eq:comp_foppl}\tag*{$(\mathrm{Co F\ddot{o}}_\mu)$}
	\Comp(\mu) := -\inf \left\{\int_\Ob j_+\bigl(\tfrac{1}{2} \, \nabla w \otimes \nabla w + e(u) \bigr) \, d\mu - \int_\Ob w\, df   \ : \  (u,w) \in \D\bigl(\Omega;\R^2 \times \R\bigr) \right\}
\end{equation}
Assuming that the initial potential $j$ is quadratic and that the above problem admits a continuous solution $(\check{u},\check{w})$, the following equality holds true: $\Comp(\mu) = \frac{3}{4}  \int_\Ob \check{w}\, df$, and this number is non-negative. Therefore, $\Comp(\mu)$  variationally defines  the \textit{compliance} of a membrane  whose material distribution is $\mu$.
A family of optimal membrane problems in the form of compliance minimization with a volume constraint may be thus put forward:
\begin{equation}
	\label{eq:OM0_family}\tag*{$(\mathrm{OM}_{\Mes_{\mathrm{adm}},V})$}
	\Cmin(\Mes_{\mathrm{adm}},V) =  \inf \Big \{ \Comp(\mu) \ : \ \mu \in \Mes_{\mathrm{adm}}, \ \mu(\Ob) \leq V \Big \} 
\end{equation}
Above $V>0$ is an upper bound on the material's volume, and $\Mes_{\mathrm{adm}} \subset \Mes_+(\Ob)$ is a class of admissible material distributions. The set $\Mes_{\mathrm{adm}}$ essentially decides what optimization framework we are in, and below we discuss the three natural choices.

For a constant thickness parameter $b_0>0$ we start by considering
\begin{equation*}
	\#1:\qquad  \Mes_{\mathrm{adm}} = \Mes_\mathrm{shape} := \Big\{\mu_\omega(dx) =  b_0  \mathbbm{1}_\omega(x) \mathcal{L}^2(dx) \, : \, \omega \subset \O  \Big\}.
\end{equation*}
Effectively, we are in the shape optimization approach, being a well rooted topic in the context of  2D/3D linearized elasticity, see e.g. \cite{sokolowski1992,plotnikov2023}. Taking experience from this vast literature, we should not expect that the problem $(\mathrm{OM}_{\Mes_{\mathrm{shape}},V})$ admits a solution in general. Since the homogenization method is not yet developed for the convexified F\"{o}ppl's model, we could instead perform the shape sensitivity analysis for the functional $\omega \mapsto \Comp(\mu_\omega)$. This potentially paves a way to a numerical scheme by the \textit{level-set method}  \cite{allaire2004}. As we can find from the latter paper, computing the shape derivative within a geometrically nonlinear theory poses a significantly greater challenge than in case of, e.g., compliance minimization of a linearly elastic body. Meanwhile, the convexity of the relaxed F\"{o}ppl's model opens up the possibility of adopting an alternative method of shape sensitivity analysis that was put forth in \cite{bouchitte2014}. The key lies in rewriting the stored energy in \ref{eq:comp_foppl} by means of the jointly convex integrand $g\big( e(u),\nabla w\big)$ defined in \eqref{eq:g_0}. This embeds the shape optimization problem of membranes in the general framework of \cite{bouchitte2014}, at least at the formal level. As a result, we can recast an explicit formula for the shape derivative. It is analogous to the one known for linear models, except that the quadratic energy is now replaced by  $ j_+\bigl(\tfrac{1}{2} \, \nabla w \otimes \nabla w + e(u) \bigr)$. At this moment, however, this is merely conjecture since the function $g$ does not satisfy the growth conditions set in \cite{bouchitte2014}.

The second standard choice is 
\begin{equation*}
		\#2: \qquad \Mes_{\mathrm{adm}} = \Mes_{\mathrm{thi}} := \Big\{\mu_b(dx) =  b(x)\mathcal{L}^2(dx) \, : \, b \in L^\infty(\Omega;[0,b_0]) \Big\}.
\end{equation*}
In this case the design variable is the function $b$ that represents membrane's thickness. It is worth noting that $b$ may vanish on e.g. open subsets of $\Omega$, so effectively the shape of the membrane is parallelly optimized as well. Thickness optimization is a well established topic in the context of \textit{sheets} in plane stress, see \cite{petersson1999,golay2001,czarnecki2013}. For Kirchhoff plates it is a  more delicate subject since $b$ enters the elastic energy cubicly. As a result, the minimized functional is not weakly-* lower semi-continuous in $L^\infty$, see \cite{kohn1984,kohn1986thin} where this issue is remedied by working with suitably averaged plate stiffnesses that model oscillatory stiffeners. No similar difficulties occur for plane stress sheets \cite{petersson1999,golay2001,czarnecki2013} where $b$ enters the energy linearly, and an optimal $\check{b} \in L^\infty(\Omega;[0,b_0])$ does exist. An analogous existence result holds true for the optimal membrane problem $(\mathrm{OM}_{\Mes_{\mathrm{thi}},V})$.

Solution of the thickness optimization problem $(\mathrm{OM}_{\Mes_{\mathrm{thi}},V})$ depends on the ratio $V / (b_0 \abs{\O})$.
It is natural to study the behaviour of the membrane's optimization problem when this ratio becomes small, e.g. by considering the volume bounds $V = \eps\, V_0$ when $\eps$ approaches zero. This way we enter the very popular topic of optimal design under the conditions of \textit{vanishing material volume} (or \textit{vanishing mass}). In case of  the shape optimization problem for linearly elastic plane stress sheet, sending $\eps$ to zero leads to the Michell problem, see \cite{allaire1993,bouchitte2007dimred,bouchitte2020,babadjian2023}. Alternatively, in \cite{bouchitte2007dimred} a similar asymptotic analysis was also carried out for the \textit{fictitious material} optimization formulation for 3D linear elasticity, and in this setting the vanishing mass limit yielded an optimization problem where a material distribution being any positive Radon measure is sought, see \cite[Proposition 3.1]{bouchitte2007dimred}. In this fashion, the design formulation in the framework of \textit{mass optimization} \cite{bouchitte2001,bouchitte2007} can be justified. Incidentally, for a sheet in plane stress the fictitious material approach coincides with the thickness optimization problem.

By the above token -- the result \cite[Proposition 3.1]{bouchitte2007dimred} in particular -- one may conjecture a connection that justifies the third variant of the problem \ref{eq:OM0_family}:
\begin{equation*}
	\#3: \qquad \Mes_{\mathrm{adm}} = \Mes_+(\Ob),
\end{equation*}
i.e. when every material distribution $\mu \in  \Mes_+(\Ob)$ that satisfies the bound on the total mass is feasible. Readily, we consider a sequence of thickness optimization problems $(\mathrm{OM}_{\Mes_{\mathrm{thi}},\eps V_0})$ for $\eps$ going to zero, whilst by $\check{b}_\eps \in L^\infty(\Omega;[0,b_0])$ we denote the corresponding solutions. Then, one can argue that the sequence of scaled optimal compliance $\eps^{1/3} \, \Cmin\big(\Mes_{\mathrm{thi}},\eps V_0 \big)$ converges to $\Cmin\big( \Mes_+(\Ob),V_0 \big)$, whereas (up to extracting a subsequence) the sequence of scaled absolutely continuous measures $\frac{1}{\eps}\, \check{b}_\eps\, \mathcal{L}^2$ weakly-* converges to a measure $\mu \in \Mes_+(\Ob)$ that solves \ref{eq:OM0_family} in the framework $\#3$. Naturally, the measure $\mu$ may no longer be absolutely continuous. It may, for instance, charge lower dimensional subsets of $\Ob$. The interpretation follows: if very little material is at disposal, then it is in our interest to locate it with a high precision where it matters the most. If these subsets are e.g. curves, then we gain the information that 1D \textit{strings} are more efficient than a diffused thin film.

\bigskip

In this work we focus entirely on the framework $\#3$ of the optimal membrane problem \ref{eq:OM0_family}, which essentially renders this contribution a continuation of the stream of research on optimal mass distribution \cite{bouchitte2001,bouchitte2007}. While the seminal paper \cite{bouchitte2001} concerned 2D or 3D elasticity, in \cite{bouchitte2007} we find a  general formulation assuming a geometrically linear model, which is then  particularized for plates. The novelty of the relaxed F\"{o}ppl's model, to which we will extend this optimization method, lies in its geometric nonlinearity \eqref{eq:foppl_strain}. The common feature is the models' convexity, which  facilitates the adoption of the mass optimization methods as proposed in this paper.

Of course, the justification of the mass optimization formulation $\#3$ using the relaxed F\"{o}ppl's  membrane model is, for the most part, heuristic. A debate starts already at the level of the generalization \ref{eq:comp_foppl} of the convexified F\"{o}ppl's model to membranes of variable thickness $b=b(x)$, let alone when an arbitrary material distribution $\mu \in \Mes_+(\Ob)$ is considered. In fact, the rigorous 3D-to-2D analysis, that lead the authors of \cite{conti2006} to the model, assumed not only a constant thickness of the membrane, but also that the membrane is clamped on its whole boundary. According to \cite{conti2006}, the latter condition is essential for providing stiffness to the membrane. This condition ceases to be true already in the shape optimization formulation: $\partial \omega \backslash \bO$ becomes the free part of the boundary. However, optimization is to guarantee an increase, not a loss of the stiffness. Therefore, heuristically speaking, we can hope that the  occurrence of the free boundary in an optimal membrane does not compromise the validity of the relaxed F\"{o}ppl's model. As far as membranes of varying thickness are concerned, the present author is not aware of such generalization of the 3D-to-2D justification \cite{conti2006}. Such extensions, however, were successful in the case of other 2D models (see e.g. \cite{kohn1984} in the case of Kirchhoff plates), so one can hope for an analogous result for the relaxed F\"{o}ppl's model. A more involved approach would consist in justifying the problem $\#3$  by performing a 3D-to-2D asymptotics for the optimal design problem directly, see e.g. \cite{bouchitte2007dimred}.

Ultimately, these delicate issues fall out of the scope of this work, and here the variational formulation \ref{eq:comp_foppl} will be \textit{a priori} assumed for membranes of an arbitrary material distribution $\mu \in \Mes_+(\Ob)$. Then, we directly attack the problem \ref{eq:OM0_family} in the framework of material distribution optimization $\#3$, without further reference to the thickness optimization problem $\#2$.  

\subsection{A synopsis of the main results}

As announced at the end of the previous subsection in this paper we address the optimal design problem below:
\begin{equation}
	\label{eq:OM0}\tag*{$(\mathrm{OM})$}
	\Cmin=  \min \Big \{ \Comp(\mu) \ : \ \mu \in \Mes_+(\Ob), \ \mu(\Ob) \leq V_0 \Big \} 
\end{equation}
where $\Omega \subset \Rd$ is a bounded domain, $V_0$ is a positive constant, and $\Comp: \Mes_+(\Ob) \to [0,\infty]$ is the compliance functional defined variationally in \ref{eq:comp_foppl}. The relaxed potential $j_+$ entering the definition of $\Comp$ is computed via the formula \eqref{eq:j_plus_0} for a convex, positively $p$-homogeneous potential $j$ that satisfies certain growth conditions to be named later, while $p \in (1,\infty)$ is a fixed exponent. Showing weak-* lower semi-continuity of $\Comp$ will be straightforward, and, thus, existence of solution to \ref{eq:OM0} will follow by the Direct Method of the Calculus of Variations. 

The main results of this work rest on the observation that the problem \ref{eq:OM0} can be reduced to a pair of mutually dual convex problems that falls within the class of problems encompassed by the work \cite{bouchitte2007}, namely:
\begin{equation}
	\label{eq:P}\tag*{$(\mathcal{P}_{\Lambda,\Cm,\mathscr{F}})$}
	\sup\biggl\{ \pairing{\mathscr{F},v} \ : \ v \in \D(\O;\R^m), \ \ (\Lambda v)(x)  \in \mathrm{C} \quad \forall\, x \in \Ob \biggr\}
\end{equation}
\begin{equation}
	\label{eq:dP}\tag*{$(\mathcal{P}^*_{\Lambda,\Cm,\mathscr{F}})$}
	\inf \biggl\{ \int_\Ob \chi_\mathrm{C}^*(\tau) \ : \ \tau \in \Mes\bigl( \Ob;V \bigr), \ \ \Lambda^*\tau = \mathscr{F} \ \text{ in } \D'(\O;\R^m)\biggr\}
\end{equation}
In the general setting $V$ is any finite dimensional space, $\Lambda: \D(\O;\R^m) \to C(\Ob;V)$ is a linear continuous operator, $\Cm$ is a closed convex subset of $V$, whilst $\mathscr{F} \in \D'(\O;\R^m)$ is a vector of $m$ distributions. The function $\chi_\Cm^*:V \to [0,\infty]$ is the \textit{support function} of $\Cm$, whilst the  operator $\Lambda^*: \Mes(\Ob;V) \to \D'(\O;\R^m)$ is the adjoint of $\Lambda$. In \ref{eq:dP} the objective is to be understood in the sense of the integral representation of a convex functional on the space of measures, cf. \cite{goffman1964}. When $\Cm$ contains a neighbourhood of the origin of $V$, a standard duality results yields the equality  $\sup \text{\ref{eq:P}} = \inf \text{\ref{eq:dP}}$, whilst the minimum is attained whenever this number is other than $+\infty$.

If $m=2$ and $\Lambda = e(\argu)$, then the pair \ref{eq:P}, \ref{eq:dP} encodes the optimal material distribution of a plane-stress sheet \cite{bouchitte2001,golay2001}. When $m=1$ and $\Lambda$ is the Hessian operator, this pair relates to optimal design of Kirchhoff plates \cite{bouchitte2007}. In both these cases $\Cm \subset \Sdd$ is a sub-level set of the energy potential $j$. In this paper we concentrate on the following framework:
\begin{subequations}
\label{eq:setting}
	\begin{align}
		v &= (u,w):\O  \to \R^m = \Rd \times \R, \qquad \pairing{\mathscr{F},v}   =  \int_ \Ob w\, df ,  \qquad V = \Sdd \times \Rd, \\ &
		\qquad \Lambda v = \big(e(u),\nabla w \big), \qquad
		\mathrm{C} = \Big\{ (\xi,\theta) \in \Sdd \times \Rd \ : \ j_+\big( \tfrac{1}{2} \theta \otimes \theta + \xi  \big) \leq 1/p \Big\}.
	\end{align}
\end{subequations}
Convexity of the set $\Cm$ follows by \eqref{eq:g_0}.
The mathematical analysis of the pair of problems \ref{eq:P}, \ref{eq:dP} in the setting  \eqref{eq:setting} was precisely the topic of the recent work \cite{bolbotowski2022a} co-written by the present author, yet only for a particular choice of the energy potential: $j(\xi) = j^\mathrm{M}(\xi):= \frac{1}{p} \big( \max_i \abs{\lambda_i(\xi)} \big)^p$. The index "$\mathrm{M}$" is to invoke the Michell problem in which this energy plays the central role. In this scenario, the set $\Cm = \Cm^\mathrm{M}$ enjoys special properties that led the authors of \cite{bolbotowski2022a} to a new link with the Monge-Kantorovich optimal transport problem \cite{villani2003,santambrogio2015}. Essentially, the paper \cite{bolbotowski2022a} departs from the dual problem $(\mathcal{P}^*_{\Lambda,\Cm^\mathrm{M}\!,\mathscr{F}})$ whose very motivation was optimizing membranes. As discussed before, however, the mechanical justification of that formulation was superficial. One of the novelties of this paper consists in showing that the pair \ref{eq:P}, \ref{eq:dP} factually underlies the optimal design problem \ref{eq:OM0} which rests on the mechanically-sound relaxed F\"{o}ppl's model. In a natural way, this promotes working with a larger class of energy potentials $j$ rather than just $j^\mathrm{M}$.

The connection between the design formulation \ref{eq:OM0} and the pair \ref{eq:P}, \ref{eq:dP} particularized in \eqref{eq:setting} is multifold. Below we state its most important aspect, the one focusing on recasting a solution of the original problem upon solutions of the foregoing pair. Let us note that by $\rho^0:\Sdd \to  \R_+$ we understand the uniquely defined convex positively one-homogeneous continuous function satisfying the equality $j^*(\sigma) = \frac{1}{p'} \big( \dro(\sigma)\big)^{p'}$ where $p' = p/(p-1)$. The result below is a preview of Theorem \ref{thm:constructing_lambda_OEM}:
\begin{theorem}
	\label{thm:main}
	Let $\hat\tau = (\hat\TAU, \hat\vartheta) \in \Mes(\Ob;\Sddp \times \Rd)$ be any solution of the problem \ref{eq:dP} in the setting \eqref{eq:setting}, and by $Z = \min \text{\ref{eq:dP}}$ denote the value function. Then, the material distribution $\check{\mu} = \frac{2V_0}{Z} \dro(\hat{\TAU})$ solves the optimal membrane problem \ref{eq:OM0}.
\end{theorem}
\noindent In the statement above $\rho^0(\hat{\TAU})$ can be understood as the variation measure  of $\hat{\TAU}$ with respect to $\rho^0$, see \cite{ambrosio2000}. One of the keys for proving this theorem is first deriving the dual definition of membrane's compliance expressed through the variational formulation $(\mathrm{CoF\ddot{o}}^*_\mu)$. It will be the generalization of \ref{eq:dual_foppl} towards any material distribution $\mu$, where now the stresses $\sigma: \Omega \to \Sddp$ and $q: \Omega \to \Rd$ will be $\mu$-measurable functions. This is achieved in Proposition \ref{prop:dual_comp_mem}.

After the material distribution $\check{\mu}$ is computed by means of Theorem \ref{thm:main}, further information on such optimal membrane follow from the solutions of the pair  of mutually dual problems.  For a suitable constant $\alpha$, the Radon-Nikodym derivatives $\check\sigma =  \alpha\, \frac{d \hat{\TAU}}{d \check{\mu}}$ and $\check{q} = \frac{d \hat{\vartheta}}{d \check{\mu}}$ enjoy the interpretation of the stresses in the optimal membrane. Next, up to rescaling and subsequence extraction, every maximizing sequence $({u}_h,w_h)$ for \ref{eq:P} converges in a suitable sense to functions that are the in-plane and out-of-plane displacements of the optimal membrane. 

One should put an emphasis on the two significant differences with respect to the stream of works \cite{bouchitte2001,bouchitte2007,bolbotowski2022b,lewinski2021}. The first difference concerns the way that the optimal design formulation and the pair of problems of problems \ref{eq:P}, \ref{eq:dP} connect. 
This discrepancy virtually stems from the nonlinearity of the strain operator \eqref{eq:foppl_strain}. We shall thus find that the passage from \ref{eq:OM0} to this pair will not be a direct reapplication of the established methods, and certain modifications will be essential.
The second novelty with respect to the aforementioned papers concerns the pair itself and is more subtle. In each of those works the closed convex set $\Cm$ was bounded. The one specified in \eqref{eq:setting} is not since $j_+$ vanishes on the cone of negative semi-definite matrices. The unboundedness of $\Cm$  affects the way that relaxation of the smoothness conditions in \ref{eq:P} is handled and, effectively, influences the regularity of the relaxed solutions. In the papers \cite{bouchitte2001,bouchitte2007,bolbotowski2022b,lewinski2021} the compactness of $\Cm$ yielded a uniform $L^\infty$ bound on $\Lambda v$. In case of conductor design, for instance, such bound furnished the $W^{1,\infty}$ regularity of the relaxed solutions $v$. Here, for the optimal membrane problem, the unboundedness of $\Cm$ causes a significant loss of the regularity comparing to those works. Nonetheless, from Proposition \ref{prop:compactness} in this work we will learn that the set of functions $(u,w)$ which are admissible in \ref{eq:P} is precompact in $L^1(\Omega;\Rd) \times C(\Ob;\R)$. Furthermore, each accumulation point $(\hat{u},\hat{w})$ of this set satisfies
\begin{equation}
	\label{eq:functional_space}
	(\hat{u},\hat{w}) \ \ \in \ \ (BV\cap L^\infty)(\Omega;\Rd) \   \times \ (C^{0,\frac1{2}}\cap W^{1,2})(\Omega;\R),
\end{equation}
and so does any relaxed solution of the problem \ref{eq:P}. Proving this statement is a point of the present paper where the reader is redirected to the work \cite{bolbotowski2022a}. Therein, the unique structure of the set $\Cm=\Cm^\mathrm{M}$ allows rewriting the constraint in $(\mathcal{P}_{\Lambda,\Cm^\mathrm{M}\!,\mathscr{F}})$ so that it paves a way to \eqref{eq:functional_space}. This way passes through the theory of maximal monotone maps \cite{alberti1999}.
In order to extend this result to sets $\Cm$  generated by functions $j$ other than $j^\mathrm{M}$, it is then necessary to show a type of equivalence between $\Cm$ and $\Cm^\mathrm{M}$. Thanks to the growth conditions imposed on $j$, this is achieved in Section \ref{sec:relaxation_and_optcond}. It is then essential to characterize the set of functions $(u,w)$ that are admissible in the primal problem after the relaxation. This characterization will be stated in Theorem \ref{prop:char_Kro}. With such a characterization at hand, in Proposition \ref{prop:OPM_optimality_conditions} we will then pass to the necessary and sufficient conditions for optimality in the problems \ref{eq:P}, \ref{eq:dP}.

\bigskip

The optimality criteria conclude the first, theoretical part of the work spanning Sections \ref{sec:formulation} and \ref{sec:pair}. In the second part we put forth a numerical method of solving the optimal membrane problem \ref{eq:OM0} by a finite element approximation. The primary goal is to demonstrate the computational potential of the pair of mutually dual convex variational problems \ref{eq:P}, \ref{eq:dP}.  Building the numerical scheme around this pair is an alternative to the more standard approach of attacking the optimal design problem (here \ref{eq:OM0}) directly, see e.g. \cite{allaire2002,haslinger2010,petersson1999}. In the latter spirit, each iteration of the algorithm entails solving the state problem (here \ref{eq:comp_foppl}) for a current material distribution. The state problem is usually tackled via a finite element method which, for the sake of well-posedness, requires keeping a soft material in the whole domain $\O$. As a result, true voids are infeasible. In contrast, treating the pair \ref{eq:P}, \ref{eq:dP} is free of such burden since the material distribution $\mu$ is absent from the formulation.  Instead, one focuses on solving numerically  the two abstract convex variational problems. Then, upon their approximate solutions $(u_h,w_h)$ and $(\TAU_h,\vartheta_h)$, we recast the approximately optimal material distribution $\mu_h$ by applying the formula in Theorem \ref{thm:main}.

Working with the pair of the type \ref{eq:P}, \ref{eq:dP} at the level of the numerical method  may not be classical, but it is not new either when more standard settings $\Lambda,\Cm,\mathscr{F}$ are considered: see e.g.  \cite{golay2001} on optimizing sheet's thickness or \cite{barrett2007mixed} on the Monge-Kantorovich optimal transport problem. Up to an affine transformation, in both contributions  $\Cm $ is the unit ball with respect to the Euclidean norm. The authors' strategy is to first regularize
the pair of problems: for a large parameter $\wp>1$ the pointwise constraint in \ref{eq:P} is replaced by the condition $(\int_\O \abs{A v}^\wp\,dx)^{1/\wp} \leq 1$. The dual problem is affected accordingly, and its unique solution $\tau_\wp$ is now an $L^{\wp/(\wp-1)}$ function. Such a regularized pair of problems is equivalent to a non-linear system of equations, including the $\wp$-Laplacian equation, whose unique solutions converge to a solution of the original pair when $\wp \to \infty$, see \cite{bouchitte2003pLap}. Ultimately, the finite element methods in \cite{golay2001,barrett2007mixed} are aimed at solving this system of equations for a suitably large $\wp$.

In this paper the finite element method is developed directly for the pair of problems \ref{eq:P}, \ref{eq:dP} in the setting \eqref{eq:setting}, without prior regularization. The idea is to perform the discretization so that the finite dimensional counterpart of the pair is a convex programming problem that is in the reach of the modern optimization solvers, cf.  \cite{ben2001}. The finite element part of the numerical scheme will be, in fact, very basic: for a triangulation of $\O$ the displacement functions $v = (u,w)$ will be approximated by continuous, element-wise affine functions $v_h = (u_h,w_h)$, whilst the stresses $\tau= (\TAU,\vartheta)$ by absolutely continuous measures of densities $(\sigma_h,q_h) \in L^\infty(\O;\Sdd\times\Rd)$, constant in each element. Employing such low order finite elements creates an advantage that can be well explained in the realm of the primal formulation. Since the functions $\Lambda v_h = \big(e(u_h),\nabla w_h\big)$ are element-wise constant, the point-wise constraint  $(\Lambda v_h)(x) \in \Cm$ is rigorously reduced  to an element-wise constraint, without the need for regularization. In Theorem \ref{thm:convergence_P_P*} we will state the convergence result for the finite element method thus developed: upon selecting a subsequence, $(u_h,w_h)$ will converge in the norm topology of $L^1(\O;\Rd) \times C(\Ob;\R)$ and $(\sigma_h,q_h)$ weakly-* in the space of measures to solutions of \ref{eq:P} and, respectively, \ref{eq:dP}.

Similar numerical strategy was used by the present author in his Ph.D. thesis \cite{bolbotowski2021} but for the pair of problems corresponding to the free material design. Nonetheless, in terms of practical computations each setting $\Lambda,\Cm,\mathscr{F}$ calls for an individual treatment  when rewriting the discretized pair  \ref{eq:P}, \ref{eq:dP} to a format that is canonical in convex programming. Accordingly, in the work \cite{bolbotowski2021} on the free material design, it proved suitable to use a \textit{conic-quadratic} programming reformulation \cite{andersen2003,ben2001}. Here, in the case of the optimal membrane design, the more involved form of the set $\Cm$  in \eqref{eq:setting} requires more work. Relying on the Schur complement result, in Lemma \ref{lem:quadruple_instead_of_pair} we will show that a pair $(\xi,\theta) \in \Sdd \times  \Rd$ belongs to $\Cm$ if and only if
\begin{equation}
	\label{eq:quadruple_instead_of_pair_0}
	\text{there exist} \quad  \epsilon \in \Sdd, \ \ \zeta \in \mathcal{S}^{3 \times 3}_+ \quad \text{such that:} \qquad
	\begin{cases}
		i) \ \
		\begin{bmatrix}
			\,\xi  &  \tfrac{1}{\sqrt{2}}\,\theta\,\\
			\tfrac{1}{\sqrt{2}}\,\theta^\top\, & 0\,
		\end{bmatrix} + \zeta = \begin{bmatrix}
			\,\epsilon  &  0_2\,\\
			\,0_2^\top\, & 1\,
		\end{bmatrix},\\
		ii) \ \ j(\epsilon) \leq 1/p,
	\end{cases}
\end{equation} 
where above $\theta$ is treated as a column vector, and so is the zero vector $0_2 \in\R^2$. The benefit of such characterization is that both $\xi$ and $\theta$ enter linearly. It is at the expense of the two extra variables $\epsilon$ and $ \zeta$ that are subject to the constraints: condition ii) and positive semi-definiteness, respectively. Although, if the potential $j$ is quadratic, then ii) can be rewritten as a conic-quadratic constraint, whilst the condition $\zeta \in \mathcal{S}^{3 \times 3}_+$ can be implemented directly. In the end, we find that the discrete algebraic variant of the pair  \ref{eq:P}, \ref{eq:dP} can be written as a \textit{conic-quadratic} \& \textit{semi-definite} programming problem. Its form is compatible with the MOSEK software \cite{mosek2019} developed for tackling large scale convex optimization problems. Hence, notwithstanding the quite crude approximation by the first order finite elements, the numerical method herein developed is capable of producing high-resolution optimal membrane designs on a laptop computer.

The finite element method was implemented in MATLAB, for which MOSEK serves as a third-party toolbox. In Section \ref{sec:simulations} a number of simulations will be performed under the assumption that the material is linearly isotropic, namely that $j(\xi) = \frac{1}{2} \pairing{\mathscr{H}\xi,\xi}$ for an isotropic Hooke operator $\mathscr{H}$. Such operator is defined by two constants, Young modulus $E>0$ and Poisson ratio $\nu \in (-1,1)$. For several load scenarios $f$ we will display the finite element approximations of optimal membranes for various $\nu$. It is the author's belief that the designs emerging are new, and they do not resemble any of the shapes that are available  in the form of optimal sheets, plates, or 3D bodies in the literature. In fact, the most intriguing membrane shapes will be obtained for negative Poisson ratios $\nu$, i.e. for \textit{auxetic materials} \cite{caddock1989}. In reality such materials are realized through certain porous microstructures. Physically speaking, employing such a material in a membrane and, moreover, describing it by the relaxed F\"{o}ppl's model may be considered questionable. In particular, we will be interested in optimal membranes for $\nu$ reaching $-1$ which, in the plane-stress elasticity, is attainable via the \textit{chiral honeycomb} microstructures \cite{prall1997}. Ultimately, we shall not dwell on the physical justification of using such materials for membranes in this work.

\subsection{An outline of the paper}

In Section \ref{sec:formulation} we state the formulation of the optimal membrane problem \ref{eq:OM0} from scratch. We start by examining the relaxed potential $j_+$ and the constitutive law it induces. In Proposition \ref{prop:convexity_g} we give an independent prove of the convexity of the relaxed F\"{o}ppl's model. Then, in Proposition \ref{prop:dual_comp_mem}, we establish the dual formulation \ref{eq:dual_comp_memb}. We conclude the section with the existence result for \ref{eq:OM0}.

Section \ref{sec:pair} expounds the connection between \ref{eq:OM0} and the pair of problems \ref{eq:P}, \ref{eq:dP} in the setting \eqref{eq:setting}, which in the sequel will be shortly denoted by \ref{eq:PM}, \ref{eq:dPM}. After proving that the duality gap vanishes in Proposition \ref{prop:supP=infP*}, in Theorem \ref{thm:constructing_lambda_OEM} we show how to recast the optimal material distribution  based on the solution of \ref{eq:dPM}. Further, we deal with the issue of relaxing the primal problem \ref{eq:PM}; in particular, Theorem \ref{prop:char_Kro} provides the characterization of the set of admissible displacement functions in the relaxed problem \ref{eq:relPM}. The necessary and sufficient optimality conditions for the solutions of the pair \ref{eq:relPM}, \ref{eq:dPM} are stated in Proposition \ref{prop:OPM_optimality_conditions}.

Section \ref{sec:FEM} is where we build the finite element method. After a quite standard discretization procedure, in Lemma \ref{lem:quadruple_instead_of_pair} the characterization \eqref{eq:quadruple_instead_of_pair_0} of the set $\Cm$ is proved. We are led to the algebraic pair  of conic-quadratic \& semi-definite programming problems \ref{eq:PMhalg}, \ref{eq:dPMhalg}, whose duality is showed in Theorem \ref{thm:duality_for_FEM} along with the optimality conditions.  Convergence of the finite element method is then proved in Theorem \ref{thm:convergence_P_P*}.

The numerical simulations are demonstrated in Section \ref{sec:simulations}. While restricting to a quadratic domain $\Omega$, for four different loading conditions we perform a case study with respect to the Poisson ratio. Computational details, including mesh resolution, CPU time, and the obtained value functions, are listed in the tables.

The last section  is devoted to the additional insight into the optimal membrane problem in the special case of the Michell energy potential $j = j^\mathrm{M}$. In this work, such framework is shown to be equivalent to working with the linearly isotropic material for $\nu = -1$. In Section \ref{ssec:FMD} we explain how the Michell-energy setting can be additionally recast through the free material design formulation \cite{bolbotowski2022b} provided that a particular set of admissible Hooke operators is chosen. Sections \ref{ssec:two_point} and \ref{ssec:maximal_MK} revisit the results of \cite{bolbotowski2022a} in order to recapitulate on the features of the pair \ref{eq:PM}, \ref{eq:dPM} that are exclusive to this special choice of the energy potential. This includes the link to finding a metric that maximizes the Monge-Kanotrovich transport cost and an alternative numerical scheme by approximation through string systems. Finally, in Section \ref{ssec:vault} we show how a 2D optimal membrane design can can be exploited in solving a long standing problem in engineering optimization \cite{rozvany1979,rozvany1982,lewinski2019a}: erecting an optimal 3D \textit{vault}, a structure in pure compression that concentrates on a single surface. The present author has recently shown in  \cite{bolbotowski2022c} that such a construction can be rigorously made once $j=j^\mathrm{M}$ precisely.

\subsection{Notation}

The basic symbols used in the paper are listed  in the index of notation. Below we explain in more details how vectors, matrices, and some of the differential operators are to be understood.

Elements $A \in \R^{m \times d}$ will be treated as matrices whose component in $i$-th  row and $j$-th column is $A_{ij}$. Elements $\theta\in \R^m$ will be understood either as  column vectors $[\theta_1 \ \ldots \ \theta_m]^\top$ or as sequences of numbers $(\theta_{1}, \ldots , \theta_{m})$, depending on the context. The tensor product  $\theta \otimes \eta$ of two vectors $\theta,\eta \in \R^m$ is the quadratic matrix $A \in \R^{m\times m}$ of the components $A_{ij} = \theta_i  \eta_j$. In Section \ref{sec:FEM} we will adapt a different notation for the vectors and matrices whose components $i$ and $j$ refer to $i$-th finite element and $j$-th node, respectively. Such vectors will be denoted in bold, e.g.  $\mbf{q} \in \R^n$, $\mbf{w} \in \R^m$, $\mbf{D} \in \R^{n \times m}$, and the components themselves will be written as $\mbf{q}(i)$, $\mbf{w}(j)$, $\mbf{D}(i,j)$, etc.

For a smooth scalar function  $w = w(x) = w(x_1,\ldots,x_d):\R^d \to \R$ by the gradient $\nabla w$ we will mean the vector function $\theta:\R^d \to \R^d$ whose components are, point-wisely, $\theta_i = \frac{\partial w}{\partial x_i}$. For $m>1$ and a vector-valued function $u: \R^d \to \R^m$, the matrix-valued  function $\xi = \nabla u : \R^d \to \R^{m \times d}$ is defined quite differently, namely $\xi_{ij} = \frac{\partial u_i}{\partial x_j}$. Note that we must not choose $m=1$ in the latter definition as it would give the transpose of the gradient defined for a scalar function. For a function $u:\R^d \to \R^d$ by $e(u)$ we understand the symmetric part of $\nabla u$, namely $\frac{1}{2} \big( \nabla u + (\nabla u)^\top \big)$.

For a bounded domain $\Omega \subset \R^d$ take Radon measures $\vartheta \in \Mes(\Ob;\R^d)$, $\TAU \in \Mes(\Ob;\R^{m \times d})$, $f \in \Mes(\Ob;\R)$, and $\mathcal{F} \in \Mes(\Ob;\R^m)$. We agree on the following definitions of the distributional divergences of vector- and, respectively, matrix-valued measures:
\begin{alignat}{2}
	\label{eq:div}
	-\dive \, \vartheta\ = f \quad &\text{in } \ \O \qquad \Leftrightarrow \qquad  \int_\O \pairing{\nabla \varphi,\vartheta} = \int_\O \varphi \, df \qquad &&\forall\, \varphi \in \D(\O;\R), \\
	\label{eq:DIV}
	-\DIV \,\TAU = \mathcal{F}  \quad &\text{in } \ \O \qquad \Leftrightarrow \qquad  \int_\O \pairing{\nabla \phi,\TAU} = \int_\O \pairing{\phi,\mathcal{F} } \qquad &&\forall\, \phi \in \D(\O;\R^m).
\end{alignat}
If $m=d$ and $\TAU \in \Mes(\Ob; \mathcal{S}^{d\times d})$, i.e. $\TAU$ is valued in symmetric matrices, then above $\nabla \phi$ can be replaced by $e(\phi)$. If one assumes that $\vartheta,\TAU, f,\mathcal{F} $ are absolutely continuous with smooth densities $q,\sigma,  \tilde{f}, \tilde{\mathcal{F} }$, then \eqref{eq:div} and \eqref{eq:DIV} are equivalent to the equations  $-\sum_i \frac{\partial q_i}{\partial x_i} = \tilde{f}$ and, respectively, $-\sum_j \frac{\partial \sigma_{ij}}{\partial x_j} = \tilde{\mathcal{F}}_i$, both holding in $\Omega$ only. Note that, due to the implied compactness of the supports of $\varphi,\phi$, no boundary conditions on $\bO$ follow from \eqref{eq:div} or \eqref{eq:DIV}.

In most of the text the above notation will be adopted for $d =m=2$. In particular, $\Omega$ will be a domain in  $\Rd$. In several places we will take a detour to three dimensional elasticity, and then we will use the under-bar symbol $\ub{\argu}$ to distinguish 3D objects from the 2D ones, for example $\Oh \subset \R^3, \ub\sigma: \Oh \mapsto \mathcal{S}^{3 \times 3}$, etc.

\begin{table}[p]
	\captionsetup{font=Large}\color{black}
	\caption*{Index of notation}
	\begin{tabularx}{\textwidth}{p{0.2\textwidth}X}
		
		\multicolumn{2}{l}{{\underline{The symbols used in the paper ($V$ is a finite dimensional vector space):}}}   \\		
		$\R_+$, $\R_-$ & the set of non-negative and non-positive real numbers\\
		$\Rb$ & the extended real line $[-\infty,\infty]$\\
		$\R^d$, $S^{d-1}$ & $d$-dimensional Euclidean space and its unit sphere\\
		$\R^{m \times d}$ & the space of  $m \times d$ matrices\\
		$\mathcal{S}^{d \times d}$ & the space of symmetric $d \times d$ matrices\\
		$\mathcal{S}^{d \times d}_+$, 	$\mathcal{S}^{d \times d}_-$ & the cones of positive and negative semi-definite symmetric $d \times d$ matrices\\
		$x \otimes y$ & the tensor product of vectors $x,y \in \R^d$, a $d \times d$ matrix\\
		$x \symtens y$ & the symmetric part of $x \otimes y$, an element of $\mathcal{S}^{d \times d}$\\
		$\mathrm{I}_d, 0_d$ & the identity matrix and the zero column vector, elements of $\mathcal{S}^{d \times d}$ and $\R^d$, respectively\\
		$\pairing{v_1,v_2}_V$ & or $\pairing{v_1,v_2}$, for short, is the scalar product of $v_1,v_2 \in V$\\
		$\abs{v}$ & the Euclidean norm $(\pairing{v,v})^{1/2}$ of a vector $v \in V$\\
		$\lambda_i(\sig)$ & the $i$-th eigenvalue of  a matrix $\sig \in \mathcal{S}^{d \times d}$\\
		$\tr \,\sig$ & the trace of $\sig$ being the sum of the respective eigenvalues\\
		$g^*:V \mapsto \Rb $ & the convex conjugate of a function $g: V \mapsto \Rb$\\
		$\partial g : V \mapsto 2^V $ & the subdifferential of a function $g: V \mapsto \Rb$\\
		$\chi_A:V \to \Rb$ & the indicator function of a set $A \subset V$ \\
		$\chi^*_A:V \to \Rb$ & the support function of a set $A \subset V$ \\
		$\mathrm{co}(A)$ & the convex hull of a set $A \subset V$\\
				
		$\O$ & a bounded domain in $\R^d$  ($d=2$ in most of the text)\\
		$C(\Ob;V)$ & $V$-valued continuous functions over  a compact set $\Ob$ \\
		$C^{0,\alpha}(\Ob;V)$ & $V$-valued $\alpha$-H\"{o}lder continuous functions over $\Ob$ for $\alpha \in [0,1]$\\
		$\norm{\argu}_{L^\infty(\Omega)}$ & the (essential) supremum norm on $\O$\\
		$e(u)$ & the symmetric part of the gradient of a function $u \in C^1(\O;\Rd)$, i.e. $\frac{1}{2} \big( \nabla u + (\nabla u)^\top \big)$ \\
		$\D(\O;\R^m)$ & $m$ copies of the space of smooth functions with compact support in $\O$\\
		$\D'(\O;\R^m)$ & $m$ copies of the space of distributions on $\O$\\
		$\Mes_+(\Ob)$ & positive Radon measures $\mu$ on the compact set $\Ob$\\
		$\mathcal{L}^d,\Ha^k, \delta_{x_0}$ & the Lebesgue, the $k$-dimensional Hausdorff, and the Dirac delta measures on $\R^d$ \\
		$\Mes(\Ob;V)$ & $V$-valued Radon measures on the compact set $\Ob$\\
		$\mathrm{sp} \,f$ & the support of a measure $f \in \Mes(\Ob;V)$\\
		$f \mres A$ & the restriction of a measure $f \in \Mes(\Ob;V)$ to a subset $A \subset \Ob$\\
		$L^q_\mu(\Ob;V)$ &$V$-valued $\mu$-measurable functions that are $\mu$-integrable with an exponent $q \in [1,\infty]$\\
		$L^q(\O;V)$ & the space $L^q_\mu(\Ob;V)$  for $\mu = \mathcal{L}^d$\\
		$W^{k,q}({\O};V)$ & the Sobolev space of the order $k \in \mathbbm{N} \cup\{\infty\}$\\
		$BV({U};V)$ &  $V$-valued functions of bounded variation on ${U} = \O$ or $U = \R^d$

		 \\	\end{tabularx}
\end{table}	

\section{Formulation of the optimal design problem}
\label{sec:formulation}

\subsection{The relaxed energy potential and the induced constitutive law}
\label{ssec:const_law}

Let us agree that $\xi \in \Sdd$ and $\sigma \in \Sdd$ will stand for, respectively, the in-plane strain and the in-plane stress in the membrane. We depart from an elastic energy potential $j:\Sdd \to \R_+$ that will be assumed to satisfy the conditions:
\begin{enumerate}[label={(H\arabic*)}]
	\item \label{as:convex} $j$ is real valued, non-negative and convex;
	\item \label{as:p-hom} $j$ is positively $p$-homogeneous for a fixed exponent $p \in (1,\infty)$;
	\item \label{as:coercivity} there exist positive constants $C_1, C_2$ such that
	\begin{equation*}
		C_1 \abs{\xi}^p \leq j(\xi) \leq C_2 \abs{\xi}^p \qquad \forall\,\xi \in \Sdd.
	\end{equation*} 
\end{enumerate}
It is worth noting that existence of $C_2$ is already guaranteed by assumptions \ref{as:convex} and \ref{as:p-hom}. It is also well established that $j$ must be continuous, see \cite{rockafellar1970convex}. By $j^*:\Sdd \to \Rb=[-\infty,\infty]$ we will denote the  convex conjugate of $j$, cf. Appendix \ref{ap:convex}.  By virtue of \cite[Corollary 15.3.1]{rockafellar1970convex}  the function $j^*$ is convex positively $p'$-homogeneous where $p' = \frac{p}{p-1}$. On top of that, $j^*$ is also continuous and satisfies the bounds $\widetilde{C}_2 \abs{\sig}^{p'} \leq j^*(\sig) \leq \widetilde{C}_1 \abs{\sig}^{p'}$ for suitably chosen  $\widetilde{C}_1, \widetilde{C}_2>0$.

An elastic material that is modelled by the energy potential $j$ is capable of withstanding both tension and compression, i.e. for any stress $\sigma \in \Sdd$ the constitutive relation $\sigma \in \partial j(\xi)$ holds true for at least one strain $\xi \in \Sdd$. In what follows, from a similar perspective we will analyse the relaxed energy potential 
\begin{equation}
	\label{eq:j_plus}
	j_+(\xi) =  \min\limits_{\zeta \in \Sddp} j(\xi +\zeta)
\end{equation}
(note that, thanks to the lower bound in \ref{as:coercivity}, we have coerciveness of $j$, and thus attainment of the minimum follows). Earlier in Section \ref{ssec:foppl_intro}, the energy $j_+$ was introduced as a relaxation of $j$, crucial from the point of view of the well-posedness of the F\"{o}ppl's membrane problem. In the process, the matrix $\zeta$ was interpreted as a strain correction aimed at lowering the stored energy and realized by means of fine-scale wrinkling. Apart from the mathematical properties of the relaxed potential, the following result is to show that $j_+$ is a modification of the initial energy potential $j$ which guarantees that the stress is purely tensile. Namely,  the constitutive law	${\sig} \in \partial j_+(\xi)$ implies that $\sigma \in \Sddp$. Indeed, we will show that -- on the dual side -- the modification $j_+$ is aimed at penalizing non-tensile stresses $\sigma \notin \Sddp$ exactly.

\begin{proposition}
	The function $j_+:\Sdd \to \R_+$ satisfies conditions \ref{as:convex}, \ref{as:p-hom}; namely, it is a real-valued non-negative convex continuous positively $p$-homogeneous function that satisfies the upper bound $j_+ \leq C_2 \abs{\argu}^p$. Its convex conjugate reads 
	\begin{equation}
		\label{eq:jplus_star}
		j_+^*(\sig) = j^*(\sig)+\chi_{\Sddp}(\sig) =\begin{cases}
			j^*(\sig) & \text{if} \quad \sig \in \Sddp,\\
			\infty & \text{if} \quad \sig \notin \Sddp,
		\end{cases} 
	\end{equation}
	where $\chi_\Sddp$ is the indicator function of the closed convex cone $\Sddp$.
	Moreover, for every $\xi \in \Sdd$ the following characterization of the induced constitutive law holds true:
	\begin{equation}
		\label{eq:char_const_law}
		\hat{\sig} \in \partial j_+(\xi)  \qquad \Leftrightarrow \qquad \hat{\sig} \in \Sddp \quad \text{and there exists} \quad  \hat\zeta \in \Sddp \ \ \text{such that} \ \ 
		\begin{cases}
			\hat\sig \in \partial j(\xi + \hat\zeta),\\
			\langle\hat\zeta ,\hat\sig \rangle = 0.
		\end{cases}
	\end{equation}
\end{proposition}
\begin{proof}
		Starting from the fact that $j = (j^*)^*$ we find the following chain of equalities:
		 \begin{align}
		\nonumber
		j_+(\xi) &= \min\limits_{\zeta \in \Sddp} j(\xi +\zeta) =     \min\limits_{\zeta \in \Sddp} \max\limits_{\sig \in \Sdd} \Big\{\pairing{\xi+\zeta,\sig} - j^*(\sig) \Big\}
		= \max\limits_{\sig \in \Sdd} \inf\limits_{\zeta \in \Sddp} \Big\{\pairing{\xi+\zeta,\sig} - j^*(\sig) \Big\} \\
		\label{eq:chain_1}
		&=  \max\limits_{\sig \in \Sddp} \Big\{\pairing{\xi,\sig} - j^*(\sig) \Big\} =  \max\limits_{\sig \in \Sdd} \Big\{\pairing{\xi,\sig} - \big(j^*(\sig) + \chi_{\Sddp}(\sig)\big) \Big\}
		\end{align}
		Above we used the min-max theorem in the variant  \cite[ChapterVI, Proposition 2.2]{ekeland1999} for the functional $\Sddp \times \Sdd \ni (\zeta,\sigma) \mapsto \pairing{\xi+\zeta,\sig} - j^*(\sig)$, whereas to meet the theorem's prerequisite we exploit the lower bound $\widetilde{C}_2 \abs{\sig}^{p'} \leq j^*(\sig)$ while taking into the account that $p'>1$. By the same theorem a saddle point $(\hat{\zeta},\hat{\sigma})$ for this functional exists. To pass to the second line in \eqref{eq:chain_1} we acknowledged that $\inf_{\zeta \in \Sddp} \pairing{\zeta,\sigma} = - \infty$ once $\sigma \notin \Sddp$, while the infimum is zero otherwise. From the chain above we thus see that $j_+ = \big(j^*(\argu)+\chi_\Sddp(\argu)\big)^*$, namely $j_+$ is a convex conjugate of a convex lower semi-continuous, positively $p'$-homogeneous extended-real function.  As the upper bound on $j_+$ is clear by its definition,  by \cite[Corollary 15.3.1]{rockafellar1970convex} we prove the first assertion.
		
		We now pass to examine when does the constitutive law $\hat{\sigma} \in \partial j_+(\xi)$ hold true for a given $\xi$. By virtue of the characterization \eqref{eq:subdiff} (see Appendix \ref{ap:convex}) this law is equivalent to $\langle \xi,\hat{\sig} \rangle = j_+(\xi)+j_+^*(\hat\sig)$. Considering \eqref{eq:jplus_star} and the chain above, $\hat{\sigma} \in \partial j_+(\xi)$ holds if and only if $\hat{\sigma}$ solves any of the maximization problems in the second line of \eqref{eq:chain_1}. Accordingly, from the elementary properties of a saddle point (together with its existence showed above) the sufficient and necessary condition for $\hat{\sigma} \in \partial j_+(\xi)$ is that there exists $\hat{\zeta} \in \Sddp$ such that the pair $(\hat{\zeta},\hat{\sigma})$ is such a point, namely:
		\begin{equation*}
			\max\limits_{\sig \in \Sdd} \Big\{\langle \xi +\hat{\zeta},\sig \rangle - j^*(\sig) \Big\} = \langle{\xi +\hat{\zeta},\hat{\sig}} \rangle - j^*(\hat{\sig}) = \min_{\zeta \in \Sddp} \Big\{ \langle{\xi +\zeta,\hat{\sig}} \rangle - j^*(\hat{\sig}) \Big\}.
		\end{equation*}
		The left equation is equivalent to $j(\xi+\hat{\zeta}) = \langle{\xi +\hat{\zeta},\hat{\sig}} \rangle - j^*(\hat{\sig})$, or to $\hat{\sig} \in \partial j (\xi + \hat\zeta)$ thanks to \eqref{eq:subdiff}. From the right one, we find that $\hat{\sig} \in \Sddp$ and, moreover, $\langle\hat\zeta,\hat\sigma\rangle = 0$. This shows the characterization \eqref{eq:char_const_law} and thus completes the proof.
\end{proof}

\begin{remark}
	While the upper bound $j_+ \leq C_2 \abs{\argu}^p$ is true, the relaxed potential surely cannot be estimated from below by the Euclidean norm of $\xi$. In fact, the following equivalence holds true, and it is straightforward to show thanks to the lower bound for $j$ that is assumed in \ref{as:coercivity}:
	\begin{equation}
		\label{eq:zero_jp}
		j_+(\xi) = 0 \qquad \Leftrightarrow \qquad \xi \in \Sddm.
	\end{equation}
\end{remark}

\begin{remark}[\textbf{Relation with the masonry material model}]
	In the work \cite{giaquinta1985} a mathematical formulation of a problem of statics of 2D masonry structures is put forward. The masonry material is treated as \textit{perfectly brittle}, i.e. it cannot withstand tensile stresses at all, quite opposite to membranes. The model in \cite{giaquinta1985} assumes equations of linearized 2D elasticity apart from the constitutive law: the authors propose the stress-strain relation analogous to \eqref{eq:char_const_law}, except that the cones $\Sddp$ are replaced by $\Sddm$. The correction $\hat\zeta$, reflecting wrinkling in the case of membranes, for the masonry material represents the material's susceptibility to cracking. Although the mathematical forms of the  $\sigma\leftrightarrow\xi$ relations are virtually identical, one should emphasize the discrepancy in terms of their motivation and mechanical aspects. In case of the membrane model the modification $j_+$ of the potential embraces the distinctive wrinkling-type deformations and is, in fact, inherent to the model in which the strain is assumed to equal $\xi = \tfrac{1}{2} \nabla w \otimes \nabla w + e(u)$. In the case of the masonry model, the analogous modification $j_-$ is a choice, aimed at modelling a specific material.
\end{remark}

The paper \cite{giaquinta1985} limits itself to the case of an isotropic linearly elastic material, i.e. to the case when the relation $\sigma \in \partial j (\xi  + \hat{\zeta})$ is linear and reduces to $\sigma = \mathscr{H} \xi$ for an isotropic Hooke operator $\mathscr{H}$. In this scenario, the authors of \cite{giaquinta1985} derived the explicit formula for the potential $j_-$, and they specified the  relation $\sigma \in \partial j_- (\xi)$. The following example essentially quotes their results, with obvious changes of signs to match the membrane setting.

\begin{example}[\textbf{Departing from the linear isotropic constitutive law}]
	\label{ex:iso}
	Let us examine the modified energy potential $j_+$ coming from a linear constitutive law for a 2D isotropic material; namely, 
	\begin{equation*}
		j(\xi) = \tfrac{1}{2} \pairing{\mathscr{H}\xi,\xi}, \qquad j^*(\sig) = \tfrac{1}{2} \pairing{\mathscr{H}^{-1} \sig,\sig}
	\end{equation*}
	for a linear operator $\mathscr{H}:\Sdd \to \Sdd$ as below:
	\begin{equation*}
	\mathscr{H}=  \frac{E}{1-\nu}\Bigl( \tfrac{1}{2}\, \mathrm{I}_2 \otimes \mathrm{I}_2 \Bigr) + \frac{E}{1+\nu} \, \Bigl( \mathrm{Id}_2- \tfrac{1}{2}\, \mathrm{I}_2 \otimes \mathrm{I}_2 \Bigr), \qquad \mathscr{H}^{-1} =  \frac{1-\nu}{E}\Bigl( \tfrac{1}{2}\, \mathrm{I}_2 \otimes \mathrm{I}_2 \Bigr) + \frac{1+\nu}{E} \, \Bigl( \mathrm{Id}_2- \tfrac{1}{2}\, \mathrm{I}_2 \otimes \mathrm{I}_2 \Bigr).
	\end{equation*}
	Above $\mathrm{I}_2$ is the $2 \times 2$ identity matrix, whilst $\mathrm{Id}_2:\Sdd \to \Sdd$ is the identity operator. In order that \ref{as:convex}-\ref{as:coercivity} are met (with $p=2$), the Young modulus $E$ must be greater than zero, and the Poisson ratio $\nu$ must  range in $(-1,1)$.
	
	For a chosen strain $\xi \in \Sdd$ consider its spectral decomposition $\xi = \sum_{i\in\{\mathrm{I},\mathrm{II}\}} \lambda_i(\xi) \,\e_i \otimes  \e_i $ where $\e_\mathrm{I},\e_\mathrm{II} \in S^1$ . Assuming that $ \lambda_\mathrm{I}(\xi)  \geq \lambda_\mathrm{II}(\xi)$ we find the following formula for the relaxed energy potential:
	\begin{equation}
	\label{eq:iso_j_+}
	j_+(\xi) = \left\{
	\begin{array}{ccc}
	\frac{1}{2} \pairing{\mathscr{H}\xi,\xi} \quad & \text{if}\quad   &  \lambda_\mathrm{I}(\xi) +\lambda_\mathrm{II}(\xi) \geq   \frac{1-\nu}{1+\nu} \, \big(\lambda_\mathrm{I}(\xi) - \lambda_\mathrm{II}(\xi) \big) ,\\
	\frac{1}{2} E\, \bigl(\lambda_\mathrm{I}(\xi) \bigr)^2 \quad  &  \text{if} \quad   & \frac{1-\nu}{1+\nu} \, \big(\lambda_\mathrm{I}(\xi) - \lambda_\mathrm{II}(\xi) \big) >  \ \lambda_\mathrm{I}(\xi) + \lambda_\mathrm{II}(\xi) > \lambda_\mathrm{II}(\xi) - \lambda_\mathrm{I}(\xi),  \\
	0  \quad & \text{if} \quad   & \xi \in \mathcal{S}^{2\times 2}_-.
	\end{array}
	\right.
	\end{equation}
	The level sets of $j$ and $j_+$ are displayed in Fig. \ref{fig:contours}(a) and \ref{fig:contours}(b) for a positive and negative $\nu$, respectively. The vector of the eigenvalues of $\sigma \in j_+(\xi)$ is added as well. An interpretation of $j_+$ can be put forth: it is the largest convex l.s.c. minorant of $j$ which guarantees that $\sig \in \Sddp$ whenever $\sigma \in j_+(\xi)$.
	
	\begin{figure}[h]
		\centering
		\subfloat[]{\includegraphics*[trim={0cm 0cm -0cm -0cm},clip,width=0.3\textwidth]{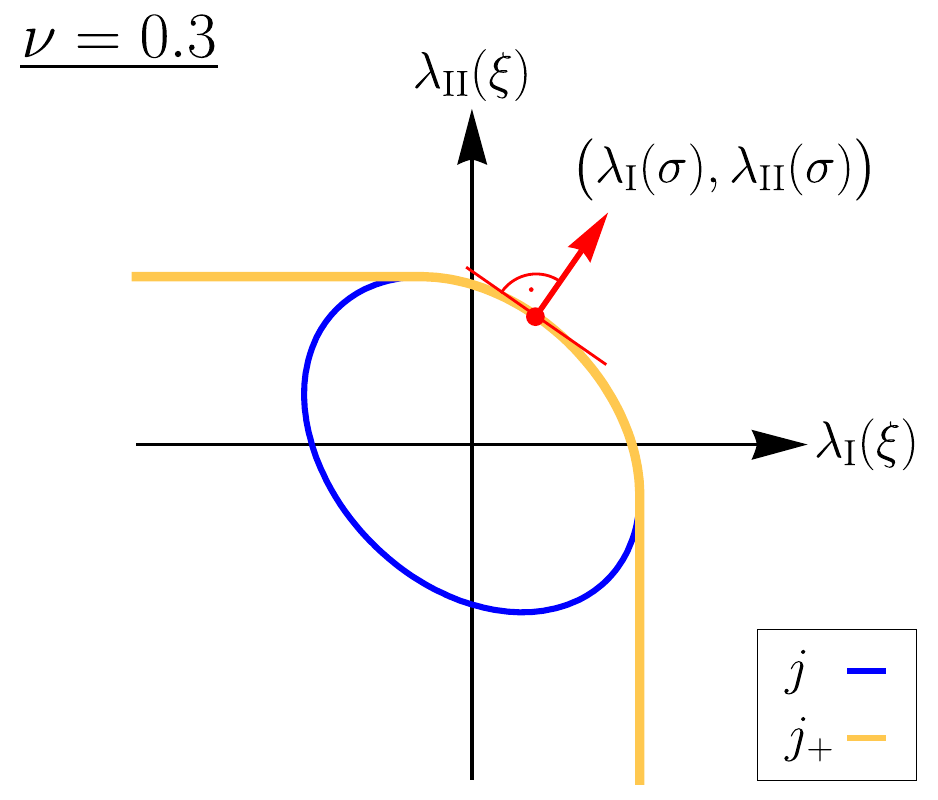}}\hspace{0.7cm}
		\subfloat[]{\includegraphics*[trim={0cm 0cm -0cm -0cm},clip,width=0.3\textwidth]{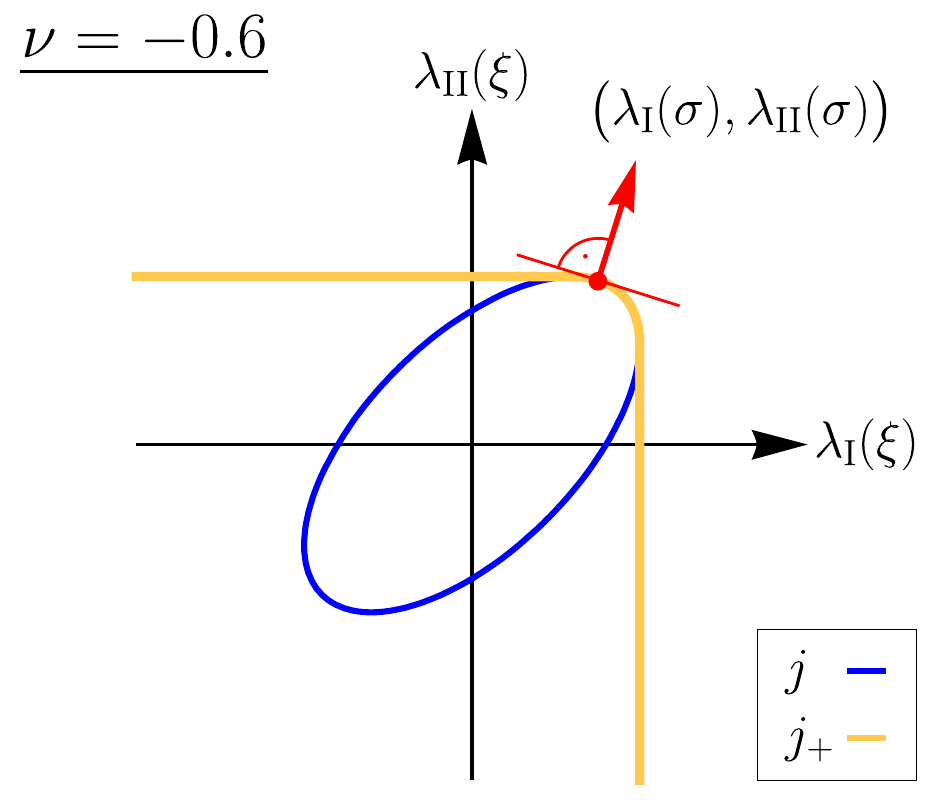}}\hspace{0.7cm}
		\subfloat[]{\includegraphics*[trim={0cm 0cm -0cm -0cm},clip,width=0.3\textwidth]{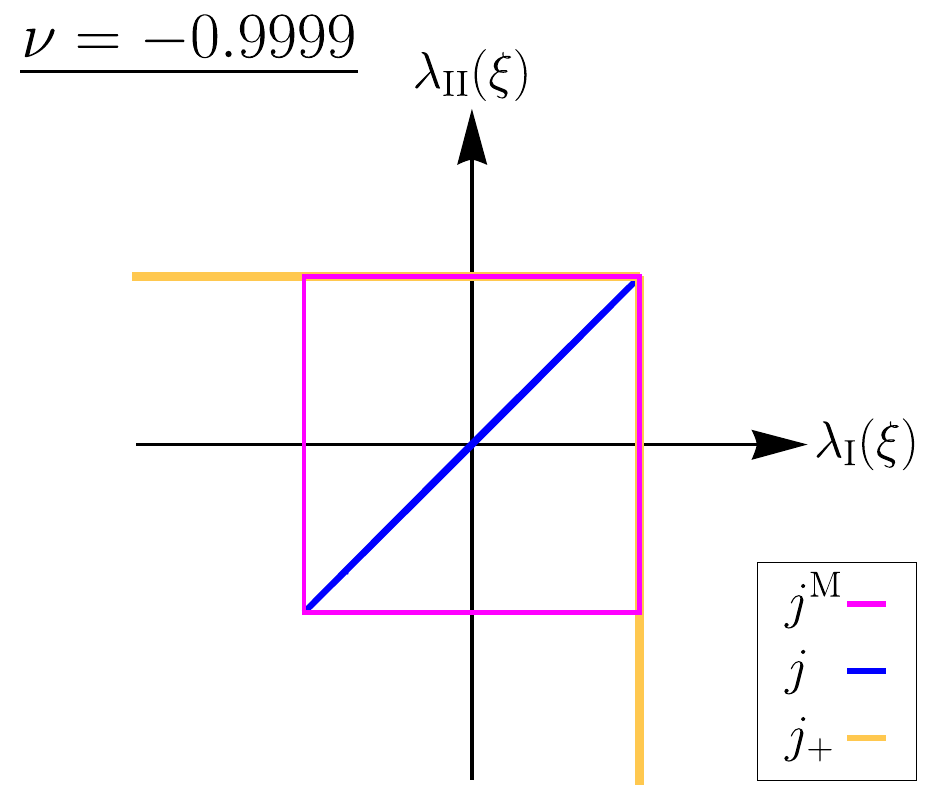}}
		\caption{The level sets of the potentials $j$ and $j_+$ drawn against the eigenvalues of $\xi$ in the case of a linearly elastic isotropic material for different Poisson ratios: (a) $\nu>0$; (b) $\nu<0$; (c) $\nu$ close to $-1$, along with the Michell energy potential $j^\mathrm{M}$.}
		\label{fig:contours}       
	\end{figure}

	Let us now take a closer look at the relation $\sig \in \partial j_+(\epsilon)$ for $\sigma \in \Sddp$ of different ranks: two, one, and zero. First of all, the eigenvectors of $\sigma$ and $\xi$ must coincide, hence	$\sigma = \sum_{i\in\{\mathrm{I},\mathrm{II}\}} \lambda_i(\sig) \,\e_i \otimes  \e_i $ where $\lambda_\mathrm{I}(\sig),\lambda_\mathrm{II}(\sigma) \geq 0$.
	\begin{enumerate}[wide=0pt, leftmargin=15pt, labelwidth=15pt, align=left]
		\item[\underline{Case 1)}] $\lambda_\mathrm{I}(\sig)>0$, $\lambda_\mathrm{II}(\sig)>0$, then
		\begin{equation*}
			\sig \in \partial j_+(\xi)  \qquad \Leftrightarrow \qquad \xi = \mathscr{H}^{-1} \sigma.
		\end{equation*}
		In this case the solution of the minimization problem \eqref{eq:j_plus} is $\hat{\zeta} = 0$. This means that no wrinkling can occur.
		\item[\underline{Case 2)}] $\lambda_\mathrm{I}(\sig)>0$, $\lambda_\mathrm{II}(\sig)=0$, then
		\begin{equation*}
			\sig \in \partial j_+(\xi)  \qquad \Leftrightarrow \qquad \lambda_\mathrm{I}(\xi) = \frac{\lambda_\mathrm{I}(\sig)}{E}, \quad \lambda_\mathrm{II}(\xi)\leq - \nu\, \frac{\lambda_\mathrm{I}(\sig)}{E}.
		\end{equation*}
		In this case the unique solution of \eqref{eq:j_plus} is $\hat{\zeta} = \big(-\nu\frac{\lambda_\mathrm{I}(\sig)}{E} -\lambda_\mathrm{II}(\xi)\big) \, \e_\mathrm{II} \otimes \e_\mathrm{II}$. Hence, we find that the membrane wrinkles along $\e_\mathrm{II}$, which is responsible for the contribution $ \lambda_\mathrm{II}(\xi) +\nu\frac{\lambda_\mathrm{I}(\sig)}{E} \leq 0$ to the membrane's contraction  in this direction.
		\item[\underline{Case 3)}] $\lambda_\mathrm{I}(\sig) =0$, $\lambda_\mathrm{II}(\sig)=0$, then
		\begin{equation*}
		\sig \in \partial j_+(\xi)  \qquad \Leftrightarrow \qquad  \lambda_\mathrm{I}(\xi)  \leq 0, \quad  \lambda_\mathrm{II}(\xi) \leq 0.
		\end{equation*}
		In this case, with no stress at all, the solution of \eqref{eq:j_plus} is  $\hat{\zeta} = - \xi$. This means that the membrane is free to wrinkle arbitrarily (in fact the directions $\e_\mathrm{I}, \e_\mathrm{II}$ are arbitrary, it is enough that $\xi \in \Sddm$).
	\end{enumerate}
\end{example}

\begin{remark}[\textbf{Michell-type material as a limit case when the Poisson ratio approaches $-1$}]
	\label{rem:Michell}
	We will now investigate the behaviour of the potentials $j=j^\nu$ and $j=j_+^\nu$ from Example \ref{ex:iso} when $\nu$ tends to $-1$. For $\nu=-1$ exactly the isotropic material becomes singular. More precisely, since $\mathscr{H}$ is not defined for $\nu=-1$, for a strain 
	$\xi = \sum_{i\in\{\mathrm{I},\mathrm{II}\}} \lambda_i(\xi) \,\e_i \otimes  \e_i $ we must get by with the limit:
	\begin{equation*}
	\lim_{\nu \searrow -1} j^\nu(\xi) = 	\begin{cases}
	\frac{1}{2} E\, (\lambda_0)^2 \quad & \text{if}\quad     \lambda_\mathrm{I}(\xi)  = \lambda_\mathrm{II}(\xi) =: \lambda_0,\\
	\infty \quad  &  \text{otherwise.}
	\end{cases}
	\end{equation*}
	The foregoing singularity is illustrated in Fig. \ref{fig:contours}(c) where for $\nu=-0.9999$ the level set of $j^\nu$ "nearly touches" the origin. Similar occurrence is not observed for $j^\nu_+$. Indeed, one can check that 
	\begin{equation}
	\label{eq:lim_nu_-1}
	\lim_{\nu \searrow -1} j^\nu_+(\xi) =  \frac{1}{2}	E  \Big(\max \big\{\lambda_\mathrm{I}(\xi), \lambda_\mathrm{II}(\xi),0 \big\} \Big)^2.
	\end{equation}
	This discrepancy can be explained in terms of stress energies $(j^\nu)^*$ and $(j^\nu_+)^*$. It is easy to check that $\lim_{\nu \searrow -1} (j^\nu)^*(\sig) = \tfrac{1}{2E} (\tr\,\sig)^2$ for any $\sigma$. This degeneracy of the stress energy reflects the well known property of materials with $\nu=-1$: they are perfectly rigid against a \textit{pure shear stress} $\sigma$, i.e. when $\tr\,\sig = 0$ (see \cite{prall1997} on 2D materials of chiral microstructure). Any non-zero stress $\sig$ of zero trace, however, is not an element of $\Sddp$, or, in other words, pure shear stress is not permitted in a membrane. Accordingly, formula \eqref{eq:jplus_star} guarantees that the limit $\lim_{\nu \searrow -1} (j^\nu_+)^*(\sig)$ is non-degenerate, i.e. strictly positive for any $\sig \neq 0$: it equals $\tfrac{1}{2E} (\tr\,\sig)^2$ for $\sig \in \Sddp$, and $\infty$ for $\sig \notin\Sddp$.
	
	Finally, based on \eqref{eq:lim_nu_-1} we recognize that (see \cite[Example 5.4]{bolbotowski2022b} for $\kappa_+=1$ and $\kappa_-=0$)	
	\begin{equation}
	\label{eq:nu=-1_Michell}
	\lim_{\nu \searrow -1} j^\nu_+(\xi) =  j^\mathrm{M}_+(\xi), \qquad \text{where} \qquad j^\mathrm{M}(\xi) := \frac{1}{2}	E  \Big(\max \big\{ \abs{\lambda_\mathrm{I}(\xi)},\abs{\lambda_\mathrm{II}(\xi)} \big\} \Big)^2.
	\end{equation}
	Namely, the potential $j_+^\nu$ in the limit case $\nu=-1$ can be recast as the the potential $j_+^\mathrm{M}$ computed for the so called \textit{Michell potential} $j^\mathrm{M}$ that models the material behaviour in the famous Michell structures, cf. \cite{bouchitte2008,lewinski2019a}. To visualize the relation \eqref{eq:nu=-1_Michell} one can confer Fig. \ref{fig:contours}(c) where a level set of $j^\mathrm{M}$ is illustrated. This observation has consequences that will be expounded later on in Section \ref{sec:Michell}.
	
	One readily checks that the Michell potential $j^\mathrm{M}$ meets the assumptions \ref{as:convex}-\ref{as:coercivity}. Accordingly, from \eqref{eq:nu=-1_Michell} we deduce that it is meaningful to use a material modelled by $j^\nu_+$ with $\nu=-1$ in the sequel.
\end{remark}

\subsection{The convexified variant of the F\"{o}ppl's membrane model}
\label{ssec:foppl}

For $\Omega \subset \Rd$ let us take a bounded domain; in this section we shall require no regularity of the boundary. For an out-of-plane load we take any signed Radon measure $f \in \Mes(\Ob;\R)$.

The compliance of a membrane whose material distribution is $\mu \in \Mes_+(\Ob)$ will be defined as the minus total potential energy, which then can be rewritten as a supremum:
\begin{equation}
\label{eq:nlcomp}\tag*{$(\mathrm{Co F\ddot{o}}_\mu)$}
\	\Comp(\mu) = -\inf \left\{\int_\Ob j_+\bigl(\tfrac{1}{2} \, \nabla w \otimes \nabla w + e(u) \bigr) \, d\mu - \int_\Ob w\, df   \ : \  (u,w) \in \D\bigl(\Omega;\R^2 \times \R\bigr) \right\}
\end{equation}
In order to analyse the underlying minimization problem above we employ the function $g: \Sdd \times \Rd \to \R_+$:
\begin{equation*}
g(\xi,\te) =  j_+\bigl(\tfrac{1}{2} \, \theta \otimes \theta + \xi \bigr).
\end{equation*}
Before stating the properties of $g$ let us agree on a convention:
\begin{equation}
	\label{eq:jsq}
	\jsq := \begin{cases} \frac1{2} \langle\sig\hat{\te},\hat{\te} \rangle  & \text{if $\q = \sig \hat{\te}$ \ for \ $\hat{\te}\in \Rd$,   }\\
	\infty & \text{if \  $\q \notin  \IM(\sig),$ }
	\end{cases} \qquad \text{for any $\sigma \in \Sddp$, \ $q \in \Rd.$}
\end{equation}
Since $\sig$ is symmmetric, the number $\jsq$ is uniquely defined even in the case when $\sigma$ is singular, despite the ambiguity in the choice of $\hat{\theta}$ in such a scenario. Of course, in the case when $\sigma$ is invertible, the value $\jsq$ adheres to a common understanding.
\begin{proposition}
	\label{prop:convexity_g}
	The function $ \Sdd \times \Rd \ni (\xi,\te) \mapsto g(\xi,\te) \in \R_+$ is jointly convex and continuous. Its convex conjugate reads:
	\begin{align}
	\label{eq:g_star}
	g^*(\sig,\q) &= j_+^*(\sig) + \jsq \\
	\nonumber
	&=\begin{cases} j^*(\sig) + \frac1{2} \langle\sig\hat{\te},\hat{\te} \rangle  & \text{if $\sig\in \Sddp$ \ and \  $\q = \sig \hat{\te}$ \ for \ $\hat{\te}\in \Rd$,   }\\
	\nonumber
	\infty & \text{if $\sig\notin \Sddp$ \ or if \  $\q \notin  \IM(\sig),$ }
	\end{cases}
	\end{align}
	while the subdifferential of $g$ with respect to the pair $(\xi,\te)$ enjoys the characterization:
	\begin{equation}
	\label{eq:g_subdiff}
	(\sig,q) \in \partial g(\xi,\te) \qquad \Leftrightarrow \qquad \begin{cases}
	\sig \in \partial j_+\big(\tfrac{1}{2}\, \te \otimes \te + \xi \big),\\
	q = \sig \te.
	\end{cases}
	\end{equation}
\end{proposition}
\begin{proof}
	Since of $\xi \mapsto j(\xi) $ is convex continuous,  we have $j_+ = (j^*_+)^*$, and, as a result, by \eqref{eq:jplus_star} we find that
	\begin{alignat*}{1}
	g(\xi,\theta)=\jp\!\left(\tfrac{1}{2}\,  \te \otimes \te + \xi \right) = &\sup\limits_{\sig \in \Sdd} \left\{\pairing{\tfrac{1}{2}\, \te \otimes \te + \xi,\sig} - j^*(\sig) - \ind_{\Sddp}(\sig) \right\} \\
	= &\sup\limits_{\sig \in \Sddp} \left\{\pairing{\tfrac{1}{2}\, \te \otimes \te + \xi,\sig} - j^*(\sig)  \right\}.
	\end{alignat*}
	For any positive semi-definite matrix $\sig$ the function $ (\xi,\te) \mapsto \pairing{\tfrac{1}{2}\, \te \otimes \te + \xi,\sig}$ is convex continuous. Since the function $g$ is a point-wise supremum of a family of such functions, it is convex and lower semi-continuous. In addition it admits finite values on the whole space $\Sdd \times \Rd$, and, therefore, it is continuous as well,  cf. \cite{rockafellar1970convex}. We compute its convex conjugate:
	\begin{align*}
	g^*(\sig,\q) &= \sup_{\te\in \R^d} \sup_{\xi\in\Sdd} 
	\left\{\pairing{\xi,\sig} + \pairing{\te,\q} - \jp\!\left(\tfrac{1}{2}\,  \te \otimes \te + \xi \right)  \right\}\\
	&=\sup_{\te\in \R^d} \sup_{\tilde{\xi}\in\Sdd} 
	\left\{\pairing{\te,\q} + \Big\langle\tilde{\xi} -\tfrac{1}{2}\,  \te \otimes \te,\sig \Big\rangle - \jp\!\big(\tilde{\xi}\big)   \right\}\\
	&=\sup_{\te\in \R^d} 
	\Big\{\pairing{\te,\q} - \tfrac{1}{2} \pairing{\sig \te,\te}\Big\} +   \sup_{\tilde{\xi}\in\Sdd}\left\{ \langle\tilde{\xi},\sig \rangle - \jp \big(\tilde{\xi}\big) \right\}  = \jsq + j_+^*(\sig)
	\end{align*} 
	where to pass to the second line we made a change of variables: $\tilde{\xi} = \tfrac{1}{2}\,  \te \otimes \te + \xi$. The final formula for $g^* (\sig,q)$ in \eqref{eq:g_star} follows once we recall \eqref{eq:jplus_star}. From the above we infer that the extremality equality  $\pairing{\xi,\sig} + \pairing{\te,\q} =  g(\xi,\te) +	g^*(\sig,\q)$ holds true if and only if $\pairing{\te,q}  = \tfrac{1}{2} \pairing{\sig \te,\te} + \jsq$ and $\langle\tilde{\xi},\sig \rangle = \jp \big(\tilde{\xi}\big) + j_+^*\big({\sig}\big)$. Characterization \eqref{eq:g_subdiff} follows owing to \eqref{eq:subdiff}.
\end{proof}

Since $e$ and $\nabla$ are linear operators, an immediate corollary of the above statement is the following:
\begin{corollary}
	\label{cor:convexity_of_VK_energy}
	For any $\mu \in \Mes_+(\Ob)$ the functional $$\D\bigl(\Omega;\R^2 \times \R\bigr) \ni (u,w) \mapsto \int_\Ob j_+\bigl(\tfrac{1}{2} \, \nabla w \otimes \nabla w + e(u) \bigr) \, d\mu$$
	is convex.
\end{corollary}

\noindent As a consequence we recover the statement from \cite{conti2006} that the variational formulation \ref{eq:nlcomp} is a convex minimization problem. As a result, the convex duality tools may be employed to derive the stress-based formulation:
\begin{proposition}
	\label{prop:dual_comp_mem}
	For any material distribution $\mu \in \Mes_+(\Ob)$ there holds
	\begin{equation}
	\label{eq:dual_comp_memb}\tag*{$(\mathrm{CoF\ddot{o}}^*_\mu)$}
	\Comp(\mu) = \inf \biggl\{ \int_\Ob j^*(\sig) \, d\mu + \int_\Ob \jsq\,d\mu : \!
	\begin{array}{cc}
		\sig \in L^{p'}_\mu (\Ob;\Sddp),\!  & -\DIV (\sig\mu) = 0\\
	q \in L^{(2p)'}_{\mu} \! (\Ob;\Rd),\!  & -\dive(q\mu)=f
	\end{array} \text{ in } \O
	\biggr\}
	\end{equation}
	where $(2p)' = 2p/(2p-1)$. Moreover, the minimum is achieved whenever $\Comp(\mu) <\infty$.
\end{proposition}
\begin{proof}
	After noticing that \ref{eq:nlcomp} can be rewritten as a supremum, the assertion can be proved almost directly by employing Theorem \ref{thm:duality_classical}. We define:
	\begin{equation*}
	X = \Big\{ (u,w) \in \DO \Big\}, \quad  Y = L^p_\mu(\Ob;\Sdd) \times L^{2p}_\mu(\Ob;\Rd), \quad  \Lambda(u,w) = \bigl(e(u),\nabla w\bigr),
	\end{equation*}
	and $\Phi(u,w) = - \int_\Ob w \,df$, \  $\Psi(\xi,\theta) =  \int_\Ob g(\xi,\theta)\, d\mu$.
	Then, for $(q,\sig) \in Y^*=L^{p'}_\mu(\Ob;\Sdd) \times  L^{(2p)'}_\mu \! (\Ob;\Rd)$ we can check that $ \Lambda^*(\q,\sig) = \big(-\DIV(\sig \mu) ,-\dive(\q \mu)\big)$, and that  $\Phi^*\big(-\Lambda^*(\q,\sig) \big)$ is finite if and only if $-\DIV (\sig\mu) = 0$ and $-\dive(\q \mu)=f$ in $\O$, see the convention \eqref{eq:div}, \eqref{eq:DIV}.
	
	Due to the celebrated Rockafellar's result we find that  $\Psi^*(\q,\sig) =  \int_\Ob g^*(\q,\sig)\, d\mu$. Indeed, the function $g$ is continuous, and, therefore, it is a normal convex integrand in the understanding of \cite{rockafellar1968}. The equality now follows by \cite[Theorem 2]{rockafellar1968}. Now, the main assertion can be obtained by combining Theorem \ref{thm:duality_classical} and formula \eqref{eq:g_star}, noting that instead of writing the $\int_\Ob j^*_+(\sig)\,d\mu$ in the minimized functional with  $\sig \in L^{p'}_\mu(\Ob;\Sdd)$, we may write $\int_\Ob j^*(\sig)\,d\mu$ if we explicitly require that $\sig \in L^{p'}_\mu(\Ob;\Sddp)$.
\end{proof}

\begin{remark}
	Let us assume that $\Omega$ has a smooth boundary, and consider displacement functions $(u,w) \in C^1(\Ob;\Rd \times \R)$ satisfying $(u,w) = (0,0)$ on $\bO$, and stresses $(\sigma,q) \in L^p_\mu(\Ob;\Sdd) \times L^{2p}_\mu(\Ob;\Rd)$.
	
	Then, functions $(u,w)$ and $(\sigma,q)$ solve the problems \ref{eq:nlcomp} and, respectively, \ref{eq:dual_comp_memb} if and only if the following system holds true:
	\begin{equation}
	\label{eq:Foppl_system}
	\begin{cases}
	(i) &   -\DIV(\sig \mu) = 0, \quad -\dive(\q \mu) = f \quad \text{ in } \O,\\
	(ii) & \sig \in  \partial j_+\big(\frac{1}{2} \, \nabla w \otimes \nabla w + e(u) \big) \quad \mu\text{-a.e.},\\
	(iii) & \q = \sig \,\nabla w \quad \mu\text{-a.e.}
	\end{cases}
	\end{equation}
	Let us now sketch the argument behind this fact. From the proof of Proposition \ref{prop:dual_comp_mem}, we can deduce that optimality occurs if and only if $\pairing{e(u),\sigma}+\pairing{\nabla w,q} = g\big(e(u),\nabla w\big) + g^*(\sigma,q)$ $\ \mu$-a.e. in $\Ob$ or, equivalently, $(\sigma,q) \in \partial g \big(e(u),\nabla w\big)$. By the virtue of \eqref{eq:g_subdiff} this is in turn equivalent with conditions (ii),\,(iii) above. Meanwhile, (i) is simply the admissibility condition for \ref{eq:dual_comp_memb}.
	
		It should be stressed, however, that one should not expect solutions $(u,w)$ of \ref{eq:nlcomp} to be of class $C^1$, especially considering that $\mu$ can degenerate to zero in subdomains of $\O$. The question of the right functional spaces for the solution in case of an arbitrary measure $\mu$ lies outside of the scope of this paper. The reader is referred to \cite{bouchitte1997,bouchitte2003} for an insight on the Sobolev spaces with respect to measure $\mu$.
	\end{remark}

	\begin{remark}[\textbf{A heuristic connection to 3D non-linear elasticity}]
	In terms of mathematics the role of the functions $\sigma$ and $q$ is clear, for they naturally serve as the dual variables. Their mechanical interpretation is less obvious. In the introduction they were called the in-plane and, respectively, out-of-plane components of stress. To be more specific, in this remark we make an attempt to put $\sigma,q$,  and the displacement functions in the context of 3D non-linear elasticity \cite{ciarlet2000a}. With complete success this was done in \cite{ledret1995} for the fully nonlinear membrane model. Contrarily, in the realm of the "partially geometrically nonlinear" relaxed F\"{o}ppl's model, the equations of the 3D elasticity can be revived approximately only: by getting rid of the term $\nabla u$ in the strategic places.

	Let us restrict ourselves to a membrane of constant thickness, i.e. to $\mu = b_0 \mathcal{L}^2 \mres \O$. Alternatively, we describe the membrane it a 3D cylinder $\Oh = \O \times (-b_0/2,b_0/2)$. With $\mathrm{I}_d$ being the $d \times d$ identity matrix and with $0_d$ denoting the $d \times 1$ column vector of zeros, we introduce the following matrix/vector valued functions:
	\begin{equation*}
		\ub{u} = \begin{bmatrix}
			\,u\,\\
			\,w\,
		\end{bmatrix},\qquad
		\ub{F}_{\mathrm{ap}}  = \begin{bmatrix}
			\mathrm{I}_2  &  0_2\,\\
			\,(\nabla w)^\top\, & 1\,
		\end{bmatrix}, \qquad
		\ub{\mathscr{E}}_{\mathrm{ap}} =
		\begin{bmatrix}
			\,\frac{1}{2} \, \nabla w \otimes \nabla w + e(u)   &  \nabla w\,\\
			\,(\nabla w)^\top\, & 0\,
		\end{bmatrix}.
	\end{equation*}
	The left hand sides are to be understood as functions of the 3D coordinate $\ub{x} = (x,x_3) \in \Oh$. In fact, these functions do not depend on $x_3$; namely, the functions on the right hand sides depend on $x$ only. We recall that by the \textit{deformation gradient} we understand the $3 \times 3$ matrix function $\ub{F} = \ub\nabla \ub{u} + \mathrm{I}_3$. We can thus check that $\ub{F} - \ub{F}_{\mathrm{ap}}  = \nabla u$. Therefore, heuristically speaking,  $\ub{F}_{\mathrm{ap}}$ can be deemed an approximation of the deformation gradient under the assumption of smallness of $\nabla u$. By a similar token, the formula $\ub{\mathscr{E}} = \frac{1}{2} (\ub{F}^\top\! \ub{F} - \mathrm{I}_3)$ defines the \textit{Green-Saint Venant strain tensor}. We then check that  $\ub{\mathscr{E}}  - \ub{\mathscr{E}}_{\mathrm{ap}} = \frac{1}{2}(\nabla u)^\top \nabla u$. Hence, this time, we can think of  $\ub{\mathscr{E}}_{\mathrm{ap}}$ as of the approximate strain tensor which ignores the higher order terms with respect to $\nabla u$.
	
	Next, we assume that the 3D elastic body is \textit{hyper-elastic}, and we start off with the potential $\ub{j}: \mathcal{S}^{3 \times 3} \to \R_+$ modeling an \textit{isotropic Saint Venant-Kirchhoff material} (see \cite{ciarlet2000a} for details):
	\begin{equation*}
		\ub{j}(\ub{\xi}) =  \tfrac{1}{2} \pairing{\ub{\mathscr{H}}\ub{\xi},\ub{\xi}} \qquad \text{where} \qquad 	\ub{\mathscr{H}}=  \frac{E}{1-2\nu}\Bigl( \tfrac{1}{3}\, \mathrm{I}_3 \otimes \mathrm{I}_3 \Bigr) + \frac{E}{1+\nu} \, \Bigl( \mathrm{Id}_3- \tfrac{1}{3}\, \mathrm{I}_3 \otimes \mathrm{I}_3 \Bigr);
	\end{equation*}
	above $E>0$, $\nu \in (-1,1/2)$, and $\mathrm{Id}_3$ is the identity operator on $\mathcal{S}^{3 \times 3}$. Further, we introduce the following relaxation of the potential:
	\begin{equation*}
		\ub{j}_{\mathrm{2D},+}(\ub{\xi}) = \min_{\ub\zeta \in \mathcal{S}^{3 \times 3}_+} \min_{\ \ub\eta \in \R^3} \ \ub{j}\big(\ub{\xi} + \ub\zeta + \ub{\eta} \symtens \mathrm{e}_3 \big),
	\end{equation*}
	where $\e_3$ is the basis vector orthogonal to $\O \times \{0\}$, while $\symtens$ stands for the symmetrized product tensor. Minimization with respect to $\ub{\zeta}$ has the interpretation revolving around wrinkling, similarly as in definition \eqref{eq:j_plus} of $j_+$. Meanwhile, the second minimization expresses the thinness of the structure.
	
	Assume now that in \eqref{eq:Foppl_system} the potential $j_+$ is a relaxation \eqref{eq:j_plus} of  $j(\xi) = \min_{\ \ub\eta \in \R^3} \ \ub{j}({\xi} + \ub{\eta} \symtens \mathrm{e}_3 )$.
	Then, straightforward computations readily show that the system \eqref{eq:Foppl_system} is equivalent to the one below:
	\begin{equation}
		\label{eq:Foppl_system_3D}
		\begin{cases}
			(i') &  -\ub{\DIV} \, (\ub{P} \mathcal{L}^3) = \ub{\mathcal{F}},\\
			(ii') & \ub{S} \in \partial \ub{j}_{\mathrm{2D},+}(\ub{\mathscr{E}}_{\mathrm{ap}} ),\\
			(iii') & \ub{P} = \ub{F}_{\mathrm{ap}}  \ub{S}
		\end{cases}
		\qquad \quad \text{ in \ \ } \Oh = \O\times(-b_0/2,b_0/2),
	\end{equation}
	provided that we define:
	\begin{equation*}
		\ub{P} = 
		\begin{bmatrix}
			\,\sig  &  0_2\\
			\,q^\top\, & 0\,
		\end{bmatrix}, \qquad
		\ub{S}= 
		\begin{bmatrix}
			\,\sig  &  0_2\\
			\,0_2^\top & 0\,
		\end{bmatrix},\qquad
		\ub{\sigma}= 
		\begin{bmatrix}
			\,\sig  &  q\,\\
			\,q^\top\, & \langle \sigma^{-1}q,q  \rangle\,
		\end{bmatrix},\qquad
		\ub{\mathcal{F}}= 
		\begin{bmatrix}
			\,0_2\,\\
			\,f/b_0\,
		\end{bmatrix};\qquad
	\end{equation*}
	moreover, there holds		
	\begin{equation}
		\label{eq:Cauchy}
		\ub{\sigma} = \frac{1}{\mathrm{det} \, \ub{F}_{\mathrm{ap}}} \, \ub{P} \ub{F}^\top_{\mathrm{ap}} = \frac{1}{\mathrm{det} \, \ub{F}_{\mathrm{ap}}} \, \ub{F}_{\mathrm{ap}} \ub{S} \ub{F}^\top_{\mathrm{ap}}, \qquad \mathrm{Im}\, \ub{\sigma}  \subset  \mathrm{Im}  \begin{bmatrix}
			\mathrm{I}_2 \,\\
			\,(\nabla w)^\top\,
		\end{bmatrix}.
	\end{equation}
	Up to the fact that $F_\mathrm{ap}$ and $\mathscr{E}_\mathrm{ap}$ are \textit{approximate} deformation gradient and strain tensor and not the exact ones,  \eqref{eq:Foppl_system_3D} is the system of equations of nonlinear hyper-elasticity \cite{ciarlet2000a}. Indeed, the functions $\ub{P}$ and $\ub{S}$ play the roles of \textit{the first} and, respectively, \textit{the second Piola-Kirchhoff} stress tensors, while $\ub{\sigma}$ is \textit{the Cauchy stress tensor}. Note that each function involved is independent of the out-of-plane variable. What is more, from the inclusion in \eqref{eq:Cauchy} we conclude that $\ub{\sigma}$ has rank two. It represents a 2D stress acting in the plane that is inclined with respect to $\O \times \{0\}$ at the slope $\nabla w$.

	\end{remark}

\subsection{Formulation of the compliance minimization problem}

As announced in the introduction, by the \textit{optimal membrane problem} (OM) we shall understand finding the material distribution that minimizes the membrane's compliance:
\begin{equation}
\label{eq:OEM}\tag*{$(\mathrm{OM})$}
\Cmin =  \min \Big \{ \Comp(\mu) \ : \ \mu \in \Mes_+(\Ob), \ \  \mu(\Ob) \leq \Totc \Big\} 
\end{equation}
\begin{proposition}
	\label{prop:existence}
	For any bounded domain $\Omega \subset \Rd$, load $f \in \Mes(\Ob;\R)$, and potential $j$ satisfying \ref{as:convex}-\ref{as:coercivity} the optimal membrane problem \ref{eq:OEM} admits a solution.
\end{proposition}
\begin{proof}
	For any fixed pair $(u,w) \in \D(\O;\Rd\times \R)$ the mapping $\Ob\ni x \mapsto j_+\big(\frac{1}{2} \nabla w(x) \otimes \nabla w(x) + e(u)(x)\big)$ is continuous thanks to the continuity of $j_+$ itself. Accordingly, the mapping  $\Mes_+(\Ob) \ni \mu \mapsto \int_\Ob w\, df -   \int_\Ob j_+\bigl(\tfrac{1}{2} \, \nabla w \otimes \nabla w + e(u) \bigr) d\mu$ is affine and continuous with respect to weak-* topology on $\Mes(\Ob;\R)$. In turn, $\mu \mapsto \Comp(\mu)$ is convex and lower semi-continuous on $\Mes(\Ob;\R)$ as a point-wise supremum of the family of previous mappings. Moreover, the set of admissible measures $\mu$ in \ref{eq:OEM} is weakly-* compact. The existence now follows by the Direct Method of the Calculus of Variations.
\end{proof}

\begin{remark}
	The proof of Proposition \ref{prop:existence} should include showing that $\Cmin$ is finite. This fact will be proved in the next section -- one should combine the results: Lemma \ref{lem:problem_P_m}, Proposition \ref{prop:supP=infP*}, and the estimate \eqref{eq:finite_Z}.
\end{remark}

\section{The underlying pair of mutually dual variational problems}
\label{sec:pair}

\subsection{Formulation of the pair $(\mathcal{P})$, $(\mathcal{P}^*)$}
\label{ssec:rho}

Let us define a $\rho_+:\Sdd \to \R_+$ through the following relation:
\begin{equation}
\label{eq:recall_rhop}
\frac{1}{p} \bigl(\rho_+(\xi) \bigr)^{p} =	j_+(\xi).
\end{equation}
Owing to \ref{as:convex},\,\ref{as:p-hom}, one can show that $\rho_+$ is a positively 1-homogeneous convex continuous function. To that aim one can use \cite[Corollary 15.3.1]{rockafellar1970convex}, which also furnishes:
\begin{equation*}
	\frac{1}{p'} \bigl(\rho_+^0(\sig) \bigr)^{p'} = j_+^*(\sig) \qquad \text{where} \qquad \rho^0_+(\sig) := \sup\limits_{\xi \in \Sdd}\Big\{ \pairing{\xi,\sigma} \, : \, \rho_+(\xi) \leq 1 \Big\};
\end{equation*}
above $p'=p/(p-1)$ is the H\"{o}lder conjugate exponent to $p$.
Above $\rho_+^0:\Sdd \to \Rb$ is the polar function of $\rho_+$, a positively 1-homogeneous convex lower semi-continuous function. In contrast to $\rho_+$, the function $\rho_+^0$ is not continuous; in particular it takes the value $+\infty$ for $\sigma \notin \Sddp$.

In addition, we define the mutually polar functions $\rho,\dro:\Sdd \to\R_+$ that satisfy:
\begin{equation}
\label{eq:recall_rho}
\frac{1}{p} \bigl(\rho(\xi) \bigr)^{p} =	j(\xi), \qquad \frac{1}{p'} \bigl(\rho^0(\sig) \bigr)^{p'} = j^*(\sig),  \qquad \rho^0(\sig) = \sup\limits_{\xi \in \Sdd}\Big\{ \pairing{\xi,\sigma} \, : \, \rho(\xi) \leq 1 \Big\}.
\end{equation}
The following equalities hold true:
\begin{equation}
\label{eq:rhop_rho}
\rho_+(\xi) =  \min\limits_{\zeta \in \Sddp} \rho(\xi +\zeta), \qquad \rho_+^0(\sig) = \dro(\sig) + \chi_{\Sddp}(\sig).
\end{equation}
Let us now study the properties of the set:
\begin{equation}
\label{eq:Cm}
\Cm:=\Big\{ (\xi,\te)\in \Sdd \times \Rd\, :\, \rho_+\!\Big(\tfrac1{2}\, \te\otimes \te + \xi \Big) \leq 1   \Big\} = \Big\{(\xi,\te)\in \Sdd \times \Rd\, :\, g(\xi,\te) \leq 1/p \Big\}.
\end{equation}
The latter equality is thanks to: $g(\xi,\theta) = j_+\big(\frac{1}{2} \theta\otimes\theta + \xi\big) =\frac{1}{p} \bigl(\rho_+(\frac{1}{2} \theta\otimes\theta + \xi) \bigr)^{p}$. Since $\Cm$ is a sub-level set of $g$, and $g$ is convex continuous, we deduce that $\Cm$ is closed and convex. In addition, $g(0,0) = 0$ guarantees that $\Cm$ contains a neighbourhood of the origin. On top of that, for every negative semi-definite $\xi_0$ the unbounded set $\big\{(t \xi_0,0) \in \Sdd \times \Rd \, : \, t\geq 0  \big\}$ is contained in $\Cm$, thus rendering $\Cm$ unbounded itself.

We define $\vro:\Sdd \times \Rd \to \R_+$ as the \textit{gauge} of the set $\Cm$ (alternatively: the \textit{Minkowski functional} of $\Cm$): 
\begin{equation}
\label{eq:vro}
\vro(\xi,\te) := \inf_{t > 0} \left\{ t \, : \, \frac{1}{t}\,(\xi,\te) \in \Cm \right\}.
\end{equation}
being a positively 1-homogeneous convex continuous function. Its polar $\vro^0$ coincides with the support function $\chi_\Cm^*$ of the set $\Cm$:
\begin{equation}
\label{eq:vro_and_Cm}
\vro^0(\sig,\q)  = \sup_{(\xi,\te) \in \Sdd \times \Rd} \Big\{  \pairing{\xi,\sig}+\pairing{\te,\q} : (\xi,\te) \in \Cm \Big\} =  \ind_{\Cm}^*(\sig,q),
\end{equation}
where we acknowledged yet another characterization of $C$
\begin{equation}
\label{eq:Cm_vro}
\Cm = \Big\{(\xi,\te)\in \Sdd \times \Rd\, :\, \vro(\xi,\te) \leq 1\Big\}
\end{equation}
(see \cite[Corollary 15.1.2]{rockafellar1970convex} for a general relation between gauges and support functions).

\begin{remark}
	Let us observe that the two convex continuous mappings $(\xi,\theta) \mapsto \rho_+(\tfrac1{2}\, \te\otimes \te + \xi )$ and $(\xi,\theta) \mapsto \vro(\xi,\theta)$ share the sub-level set $\Cm$ for the value $1$. Nonetheless, the functions do not coincide. The argumentation is simple: the first mapping is not positively homogeneous. Notwithstanding this, we can write down the equivalence that will be useful in the sequel:
	\begin{equation}
		\label{eq:zero_vro}
		\vro(\xi,\theta) = 0 \qquad \Leftrightarrow \qquad \rhop\big(\tfrac{1}{2}\,\theta \otimes \theta + \xi \big) =0 \qquad \Leftrightarrow \qquad \tfrac{1}{2}\,\theta \otimes \theta + \xi \in \Sddm,
	\end{equation}
	where the last equivalence is due to \eqref{eq:zero_jp}.
\end{remark}

\begin{proposition}\label{prop:varrho_polar}
	The following statements hold true:
	\begin{enumerate}[label={(\roman*)}]
		\item
		The  set $\Cm$ is a closed unbounded convex subset of $\Sdd \times \Rd$ that contains a neighbourhood of the origin. Its support function reads
		\begin{align}\label{eq:vro_polar}
		\chi_\Cm^*(\sig,q) = \vro^0(\sig,\q) &= \rho_+^0(\sig) + \jsq \\
		&=\begin{cases} \dro(\sig) +  \frac1{2} \langle\sig\hat{\te},\hat{\te} \rangle  & \text{if $\sig\in \Sddp$ \ and \  $\q = \sig \hat{\te}$ \ for \ $\hat{\te}\in \Rd$,   }\\
		\nonumber
		\infty & \text{if $\sig\notin \Sddp$ \ or if \  $\q \notin  \IM(\sig).$ }
		\end{cases}
		\end{align}
		\item Assume $(\xi,\theta) \in \Sdd \times \Rd$ such that $\vro(\xi,\theta) \leq 1$ (equivalently $(\xi,\theta) \in \Cm$ or $\rho_+(\tfrac{1}{2} \,\theta \otimes \theta +\xi) \leq 1$). Then, for any $(\sig,q) \in \Sdd \times \Rd$ there holds the following characterization of the extremality relation:
		\begin{equation*}
		\pairing{\xi,\sig} + \pairing{\theta,q} = \vro^0(\sig,q) \qquad \Leftrightarrow \qquad
		\begin{cases}
		\sig \in \Sddp,\\
		\pairing{\frac{1}{2}\,\theta \otimes \theta + \xi,\sig} = \dro(\sig),\\
		q = \sig \theta.
		\end{cases}
		\end{equation*}
		\item The function $\vro^0$ is coercive on $\Sdd \times \Rd$ in the sense that there exists a constant $C_1>0$ such that the lower bound holds true:
		\begin{equation}
		\label{coercif}
		\vro^0(\sig,\q)\ \ge\ C_1 \bigl( |\sig| + \abs{\q} \bigr).
		\end{equation}
		\item For each $\theta \in \Rd$ the function $\vro(\argu,\theta):\Sdd \to \R_+$ is monotone in the following sense: 
		\begin{equation*}
			\xi_1 \preceq \xi_2 \qquad \Rightarrow \qquad \vro(\xi_1,\theta) \leq \vro(\xi_2,\theta),
		\end{equation*}
		where the inequality $\preceq$ is understood in the sense of the induced quadratic forms.
	\end{enumerate}
\end{proposition}
\begin{proof}
	The properties of $\Cm$ stated in (i) have been already established in the text above. The proof of the rest of the assertions is similar to the one of Lemma \ref{prop:convexity_g}. We compute
	\begin{align*}
	\vro^0(\sig,q) &= \sup_{\theta\in \R^d} \sup_{\xi\in\Sdd} 
	\Big\{\pairing{\xi,\sig}+\pairing{\theta,\q} : \rhop\!\left(\tfrac{1}{2}\,  \theta \otimes \theta + \xi \right) \leq 1 \Big\}\\
	\nonumber
	&=\sup_{\theta\in \R^d} \sup_{\tilde{\xi}\in\Sdd} 
	\left\{\pairing{\theta,\q} + \Big\langle\tilde{\xi} -\tfrac{1}{2}\,  \theta \otimes \theta,\sig \Big\rangle: \rhop\big(\tilde{\xi}\,\big) \leq 1  \right\}\\
	\nonumber
	&=\sup_{\theta\in \R^d} 
	\Big\{\pairing{\theta,\q} - \tfrac{1}{2} \pairing{\sig \theta,\theta}\Big\} +   \sup_{\tilde{\xi}\in\Sdd}\left\{ \langle\tilde{\xi},\sig \rangle : \rhop\big(\tilde{\xi}\,\big) \leq 1 \right\}  = \jsq + \rho_+^0(\sig),
	\end{align*} 
	where we performed the substitution $\tilde{\xi} = \tfrac{1}{2}\, \theta \otimes \theta + \xi$. Under the prerequisite of the claim (ii) we have $\rho_+(\tilde{\xi}) \leq 1$. Then, the assertion (ii) follows from the last line above, while in addition we acknowledge the fact that $\rho^0_+(\sig) < \infty$ if and only if $\sig \in \Sddp$, and $\rho^0_+(\sig) = \dro(\sig)$ in this case. Considering the finite dimension of $\Sdd \times \Rd$, the statement (iii) is a straightforward consequence of the fact that $\Cm$ contains a neighbourhood of the origin. The assertion (iv) is an easy consequence of the fact that $\vro = (\vro^0)^0$, whereas the condition $\vro^0(\sig,q) \leq 1$  implies that $\sigma \in \Sddp$.
\end{proof}

We are now in a position to formulate the pair of mutually dual convex variational problems:
\bigskip
\begin{align}
\label{eq:PM}\tag*{$(\mathcal{P})$}
Z := &\,\sup \left \{ \int_\Ob w \,df \, : \, (u,w) \in \DO,  \  \rho_+\Big(\tfrac{1}{2} \, \nabla w \otimes \nabla w + e(u)\Big) \leq 1  \ \ \text{in } \Ob  \right\}\\
\nonumber \\
\label{eq:dPM}\tag*{$(\mathcal{P}^*)$}
&\,\inf \biggl\{ \int_\Ob \vro^0(\TAU,\vartheta) \ : \ (\TAU,\vartheta) \in \Mes(\Ob;\Sddp \times \Rd), \  -\DIV \,\TAU = 0, \ \ -\dive \, \vartheta = f \ \   \text{in } \Omega \biggr \}
\end{align}
\bigskip

\noindent In the problem \ref{eq:dPM} the symbol $\int_\Ob \vro^0(\TAU,\vartheta)$ stands for a convex lower semi-continuous functional on the space of measures that should be understood in the sense of the Goffman-Serrin convention \cite{goffman1964}. Namely, for any positive measure $\mu \in \Mes_+(\Ob)$ such that $\TAU \ll \mu$ and $\vartheta \ll \mu$, by definition there holds:
\begin{equation*}
	\int_\Ob \vro^0(\TAU,\vartheta) = \int_{\Ob} \vro^0(\sig,q)\,d\mu = \int_{\Ob} \Big( \rho_+^0(\sig) +\jsq  \Big) d\mu \qquad \text{where} \quad \sigma = \frac{d \TAU}{d\mu}, \ \ q=\frac{d \vartheta}{d\mu}.
\end{equation*}
The value of the functional $\int_\Ob \vro^0(\TAU,\vartheta)$ does not depend on the particular choice of $\mu$ thanks to the positive 1-homogeneity of $\vro^0$. In addition, let us remark that finiteness of $\int_\Ob \vro^0(\TAU,\vartheta)$ entails the absolute continuity of $\vartheta$ with respect to $\TAU$. More precisely:
\begin{equation}
	\label{eq:vro_fin}
	\int_\Ob \vro^0(\TAU,\vartheta) < \infty \qquad \Rightarrow \qquad 
	\begin{cases}
		\TAU \in \Mes(\Ob;\Sddp) \quad \text{($\TAU$ is valued in positive semi-definite matrices)},\\
		\vartheta \ll \mu \qquad \text{for \quad $\mu:=\dro(\vartheta)$},\\
		q(x) \in \mathrm{Im}\,\sigma(x) \quad \text{for $\mu$-a.e. $x$}, \qquad \text{where} \quad \sigma = \tfrac{d \TAU}{d\mu}, \ \ q=\tfrac{d \vartheta}{d\mu}.
	\end{cases}
\end{equation}
The above implication may be showed in analogy with the proof of \cite[Lemma 3.10 (i)]{bolbotowski2022a}. As one of the consequences of \eqref{eq:vro_fin} we can equivalently require that $\TAU \in \Mes(\Ob;\Sdd)$ or $\TAU \in \Mes(\Ob;\Sddp)$ in \ref{eq:dPM}.

\begin{proposition}
	\label{prop:supP=infP*}
	There holds the zero-gap equality:
	\begin{equation*}
	Z=\sup \mathcal{P} = \min \mathcal{P}^* < \infty
	\end{equation*}
	whereas the minimum is achieved.
\end{proposition}
\begin{proof}
	The assertion follows by a direct application of Theorem \ref{thm:duality_classical}. Indeed, let us take:
	\begin{equation*}
	X = \D(\O;\Rd \times \R), \qquad Y = C(\Ob;\Sdd \times \Rd), \qquad \Lambda(u,w) = \big(e(u),\nabla w\big),
	\end{equation*}
	\begin{equation*}
	\Phi(u,w) = - \int_\Ob w\,df, \qquad \Psi(\xi,\theta)= \chi_{\mathscr{C}}(\xi,\theta),
	\end{equation*}
	where $\mathscr{C} = \big \{ (\xi,\theta) \in C(\Ob;\Sdd\times\Rd)\,  : \, \big(\xi(x),\theta(x)\big) \in \mathrm{C} \quad  \forall\,x \in \Ob \big\}$, while $\mathrm{C}$ is the closed convex set defined in \eqref{eq:Cm}. We readily recognize the maximization problem on the left hand side of \eqref{eq:sup=inf} as $\Prob$. For any $(\TAU,\vartheta) \in Y^* = \Mes(\Ob;\Sdd\times\Rd)$ it is straightforward to check that 
	\begin{equation*}
	\Phi^*(-\Lambda^*\TAU) = 
	\begin{cases}
	0 & \text{if}\quad \  -\DIV \,\TAU = 0, \ \ -\dive \, \vartheta = f \ \   \text{in } \Omega,  \\
	\infty & \text{otherwise,}
	\end{cases}\qquad
	\Psi^*(\TAU,\vartheta) = \chi^*_{\mathscr{C}}(\TAU,\vartheta) = \int_\Ob \chi^*_\mathrm{C}(\TAU,\vartheta),
	\end{equation*}
	whilst the integral representation of $\chi^*_{\mathscr{C}}$ is due to the  Rockafellar's result, cf. \cite[Theorem 6]{rockafellar1971}. Since $\chi^*_\mathrm{C} = \vro^0$ according to \eqref{eq:vro_and_Cm}, we recognize the minimization problem on the right hand side of \eqref{eq:sup=inf} in Theorem \ref{thm:duality_classical} as $\dProb$. Since the set $\mathrm{C}$ contains a neighbourhood of the origin in $\Sdd \times \Rd$, one can check that for $(u_0,w_0)=(0,0)$ clearly $\Phi(u_0,w_0) < \infty$, and $\Psi$ is continuous at $\Lambda (u_0,w_0) = (0,0)$ in the topology of uniform convergence. The proof will be complete if we show that $\sup \mathcal{P} < \infty$, which we postpone to Section \ref{sec:relaxation_and_optcond}, see \eqref{eq:finite_Z}.
\end{proof}

\subsection{The link between the optimal membrane problem and the pair \ref{eq:PM},\,\ref{eq:dPM}}

In this subsection we expound how does the pair \ref{eq:PM},\,\ref{eq:dPM} connect to the original optimal membrane problem \ref{eq:OEM}. Let us begin with the simple relation between the value functions $Z$ and $\Cmin$:
\begin{lemma}
	\label{lem:problem_P_m}
	The minimum value of compliance in \ref{eq:OEM} equals
	\begin{equation}
	\label{eq:Cmin_Z_m}
	\Cmin = \frac{2p-1}{2 p} \left( \frac{\Z^{2p}}{2 \Totc}\right)^{\frac{1}{2p-1}},
	\end{equation}
	where $\Z = \sup \mathcal{P}$.
\end{lemma}
\begin{proof}
	By plugging definition of compiance $\Cmin$ (rewritten as a supremum) into \ref{eq:OEM} we arrive at a min-max problem. Thanks to Corollary \ref{cor:convexity_of_VK_energy}, we have concavity of the mapping $(u,w) \mapsto \int_\Ob w\, df  - \int_\Ob j_+\bigl(\tfrac{1}{2} \, \nabla w \otimes \nabla w + e(u) \bigr) \,d\mu$. On the other hand, the map $ \Mes(\Ob;\R)\ni\mu \mapsto \int_\Ob w\, df  - \int_\Ob j_+\bigl(\tfrac{1}{2} \, \nabla w \otimes \nabla w + e(u) \bigr) \,d\mu$ is linear (hence convex) and weakly-* continuous (hence weakly-* lower semi-continuous). In addition the set $\big\{\mu \in \Mes_+(\Ob) \, : \, \int_\Ob d\mu \leq \Totc \big\}$ is weakly-* compact in $\Mes(\Ob;\R)$. We may thus apply  Ky Fan's theorem (see \cite[Theorem 2.10.2, p.\,144]{zalinescu2002}) in order to interchange the order of $\inf$ and $\sup$:
	\begin{align*}
	\Cmin = & \inf_{\substack{\mu \in \Mes_+(\Ob), \  \mu(\Ob) \leq \Totc}} \ \sup\limits_{(u,w)\in \DO}   \left\{  \int_\Ob w\, df  -  \int_\Ob j_+\bigl(\tfrac{1}{2} \, \nabla w \otimes \nabla w + e(u) \bigr)  d\mu \right\}  \\
	= &  \ \sup\limits_{(u,w)\in \DO}   \left\{  \int_\Ob w\, df  - \max_{\substack{\mu \in \Mes_+(\Ob), \  \mu(\Ob) \leq \Totc}}  \int_\Ob \frac{1}{p}\Big(\rho_+\bigl(\tfrac{1}{2} \, \nabla w \otimes \nabla w + e(u) \bigr)\Big)^p  d\mu \right\}  \\
	=& \, \sup\limits_{(u,w) \in \DO} \  \left\{ \int_\Ob w\, df- \frac{\Totc}{p}  \left(\left\Vert\ro_+\!\left( \tfrac{1}{2} \, \nabla w \otimes \nabla w + e(u)\right) \right\Vert_{\infty,\Omega} \right)^p \right\}.
	\end{align*}
	Above, to pass to the third line we explicitly solve the maximization problem in $\mu$: since $\rho_+$ is continuous, there exists $\bar{x} \in \Ob$ such that $\rho_+\!\left( \tfrac{1}{2} \, \nabla w \otimes \nabla w + e(u)\right)\!(\bar{x}) =\left\Vert\ro_+\!\left( \tfrac{1}{2} \, \nabla w \otimes \nabla w + e(u)\right) \right\Vert_{\infty,\Omega}$. Accordingly, we can choose $\bar{\mu} = \Totc \,\delta_{\bar{x}}$ as a solution.
	
	The next step requires a variation of the technique used in works \cite{bouchitte2001,bouchitte2007}. We recognize that each $(u,w) \in \D(\O;\Rd\times\R)$ can be written as $(t^2 u_1,t\,w_1)$ for $t\geq 0$ and $(u_1,w_1) \in \D(\O;\Rd\times\R)$ with  $\left\Vert\ro_+\!\left( \tfrac{1}{2} \, \nabla w_1 \otimes \nabla w_1 + e(u_1)\right) \right\Vert_{\infty,\Omega}\leq 1$ and $t^2= \left\Vert\ro_+\!\left( \tfrac{1}{2} \, \nabla w \otimes \nabla w + e(u)\right) \right\Vert_{\infty,\Omega}$. To prove it one must show that $t \neq0$ whenever $(u,w)\neq (0,0)$. This can be seen from the estimates (i) in Proposition \ref{prop:estimates_uw} in the next subsection. Indeed, if $\left\Vert\ro_+\!\left( \tfrac{1}{2} \, \nabla w \otimes \nabla w + e(u)\right) \right\Vert_{\infty,\Omega}$ equalled zero for some $(u,w) \neq 0$ (and, hence, $\norm{u}_{\infty,\O} + \norm{w}_{\infty,\O}>0$ ), then the same would hold true for $(\alpha\,u,\sqrt{\alpha}\,w)$ for any $\alpha>0$, which would contradict those estimates.
	
	Effectively, by the substitution put forth above, we obtain:
	\begin{align}
	\nonumber
	\Cmin & = \sup_{t \geq 0}\sup\limits_{\substack{(u_1,w_1)\in \DO}} \!  \left\{\! \biggl( \int_\Ob w_1\, df\biggr) t - \frac{\Totc}{p} \, t^{2p}   :\, \left\Vert\ro_+\!\left( \tfrac{1}{2} \, \nabla w_1 \otimes \nabla w_1 + e(u_1)\right) \right\Vert_{\infty,\Omega}\! \leq 1\right\}\\
	\label{eq:relying_on_F=0}
	& = \sup_{t \geq 0} \ \biggl\{ \Z \,t - \frac{\Totc}{p} \, t^{2p}\biggr\}= \frac{2p-1}{2 p} \left( \frac{\Z^{2p}}{2 \Totc}\right)^{\frac{1}{2p-1}}
	\end{align}
	where the maximum with respect to $t$ was found for $\bar{t} = \bigl(\frac{\Z}{2\Totc}\bigr)^{\frac{1}{2p -1}}$. 
\end{proof}

The next result will constitute the core of the optimization methodology proposed in the present work: we will show how to recast an optimal material distribution $\check{\mu}$ for \ref{eq:OEM} based on a pair $(\hat{\TAU},\hat{\vartheta})$ that solves \ref{eq:dPM}.

In the remainder of this subsection for every pair $ (\hat\TAU,\hat\vartheta) \in \Mes(\Ob;\Sddp \times \Rd)$ we shall denote by
$\hat{\mu} \in \Mes_+(\Ob)$, $\hat{\sig} \in L^\infty_{\hat{\mu}}(\Ob;\Sddp)$, $\hat{q} \in L^1_{\hat{\mu}}(\Ob;\Rd)$ the measure and the functions below:
\begin{equation}
	\label{eq:decomposition}
	\hat{\mu} := \dro(\hat{\TAU}), \qquad \hat{\sigma} := \frac{d \hat{\TAU}}{ d\hat{\mu}}, \qquad \hat{q} := \frac{d \hat{\vartheta}}{ d\hat{\mu}}.
\end{equation}
Note that we have $\dro(\hat{\sig}) = 1$ \ $\hat{\mu}$-a.e. The following \textit{equi-repartition of energy rule} holds true:
\begin{equation}
	\label{eq:equi-repartition}
	(\hat{\TAU},\hat{\vartheta}) \ \text{ solves } \ \dProb \qquad \Rightarrow \qquad \int_\Ob \dro(\hat{\sig})\, d\hat{\mu} = \int_\Ob \tfrac{1}{2} \pairing{\hat{\sig}^{-1}\hat{q},\hat{q}}\, d\hat{\mu} = \frac{\Z}{2},
\end{equation}
where we recall that $Z=\min \mathcal{P}^*$. Indeed, since the condition $-\DIV\,{\TAU} =0$ is homogeneous, we deduce that (by optimality of $\hat{\TAU} = \hat{\sigma} \hat{\mu}$) the function $(0,\infty) \ni t \mapsto \int_\Ob \dro(t\,\hat{\sig})\, d\hat{\mu} + \int_\Ob \tfrac{1}{2\,t} \pairing{\hat{\sig}^{-1}\hat{q},\hat{q}}\, d\hat{\mu}$ must attain its minimum for $t=1$. Then, the Euler-Lagrange equation furnishes equality between the two integrals exactly. 

\begin{theorem}
	\label{thm:constructing_lambda_OEM}
	For any pair $(\hat{\TAU},\hat{\vartheta})$ that solves $\dProb$ let us take the triple $(\hat{\mu},\hat{\sigma},\hat{q})$ as defined in \eqref{eq:decomposition} and put:
	\begin{equation}
	\label{eq:link_OEM_dProb}
	\check{\mu} = \frac{2\Totc}{\Z}\, \hat{\mu}, \qquad  \check{\sig} = \left(\frac{\Z}{2\Totc}\right)^{\frac{2p-2}{2p-1}} \hat{\sig}, \qquad \check{q} = \frac{\Z}{2\Totc} \, \hat{q}.
	\end{equation}
	Then:
	\begin{enumerate}[label={(\roman*)}]
		\item the material distribution $\check{\mu}$ solves the optimal membrane problem \ref{eq:OEM};
		\item the pair $(\check{\sigma},\check{q})$ solves the minimization problem \ref{eq:dual_comp_memb} for $\mu = \check{\mu}$.
	\end{enumerate}
\end{theorem}
\begin{proof}
	Admissibility of $\check{\mu}$ can be showed by exploiting the equi-repartition rule in \eqref{eq:equi-repartition}:
	\begin{equation*}
		\check{\mu}(\Ob) = \frac{2\Totc}{Z} \int_\Ob d\hat{\mu} =\frac{2\Totc}{\Z} \int_\Ob \dro(\hat{\TAU}) = \Totc. 
	\end{equation*}
	We shall now show that the pair $(\check{\sig},\check{q})$ is admissible for the problem \ref{eq:dual_comp_memb}:
	\begin{align}
	\label{eq:Div_sigcheck_sighat}
	-\DIV (\check{\sig} \check{\mu}) &= -\left(\frac{\Z}{2\Totc}\right)^{\frac{2p-2}{2p-1}} \frac{2\Totc}{\Z} \, \DIV (\hat{\sig} \hat{\mu}) = -\left( \frac{2\Totc}{\Z} \right)^{\frac{1}{2p-1}} \DIV (\hat{\sig} \hat{\mu}) =0,\\
	\label{eq:div_qcheck_qhat}
	-\dive\big( \check{q} \check{\mu}\big) &= -\frac{\Z}{2\Totc}\, \frac{2\Totc}{\Z} \,\dive\big(\hat{q}\hat{\mu} \bigl) = -\dive\big(\hat{q}\hat{\mu} \bigl) = f.
	\end{align}
	Readily, by employing Proposition \ref{prop:dual_comp_mem} we may estimate:
	\begin{align*}
	\Cmin \leq \Comp(\check{\mu}) \leq& \int_\Ob j^*(\check{\sig}) \, d\check{\mu} + \int_\Ob \tfrac{1}{2}\pairing{\check{\sig}^{-1}\check{q},\check{q}}\,d\check{\mu}\\
	=& \int_\Ob \frac1{p'} \bigl(\dro(\check{\sig})\bigr)^{p'} \, d\check{\mu} + \int_\Ob \tfrac{1}{2}\pairing{\check{\sig}^{-1}\check{q},\check{q}}\,d\check{\mu} \\
	=& \left(\frac{\Z}{2\Totc}\right)^{\frac{1}{2p-1}} \left(\frac1{p'} \int_\Ob \dro(\hat{\sigma})\, d \hat{\mu} + \int_\Ob \tfrac{1}{2}\pairing{\hat{\sig}^{-1}\hat{q},\hat{q}}\,d\hat{\mu}   \right)\\
	=& \left(\frac{\Z}{2\Totc}\right)^{\frac{1}{2p-1}} \left(\frac1{p'} \frac{\Z}{2} + \frac{\Z}{2}   \right)  = \frac{2p-1}{2 p} \left( \frac{\Z^{2p}}{2 \Totc}\right)^{\frac{1}{2p-1}} =\Cmin
	\end{align*}
	where the last equality is due to Lemma \ref{lem:problem_P_m}. To pass to the third line we acknowledged the equalities $\int_\Ob (\dro(\hat{\sigma}))^{p'} \, d\hat{\mu}=\int_\Ob \dro(\hat{\sigma}) \, d\hat{\mu} =\int_\Ob d\hat{\mu}$, while to pass to the forth we used \eqref{eq:equi-repartition}. Ultimately, the chain of inequalities above is a chain of equalities. Thus we infer that $\check{\mu}$ is optimal for \ref{eq:OEM} and, moreover, that $(\check{\sig},\check{q})$ solves \ref{eq:dual_comp_memb} for $\mu = \check{\mu}$.
\end{proof}

\begin{remark}
	The converse of Theorem \ref{thm:constructing_lambda_OEM} holds as well: whenever (i),\! (ii) is satisfied for the triple $(\check{\mu},\check{\sig},\check{q})$, the triple $(\hat{\mu},\hat{\sig},\hat{q})$ determined by the relations \eqref{eq:link_OEM_dProb} furnishes a solution $(\hat{\TAU},\hat{\vartheta}) = (\hat{\sig}\hat{\mu},\hat{q}\hat{\mu})$ of the problem \ref{eq:dPM}. The proof requires some extra arguments, and, since we will not be needing this converse result in the sequel, we will skip it.
\end{remark}

\subsection{Relaxation of the problem \ref{eq:PM} and the optimality conditions}
\label{sec:relaxation_and_optcond}

In this subsection we shall more closely study the problem \ref{eq:PM}. In particular. We define the set of admissible pairs of smooth functions $(u,w)$:
\begin{equation*}
	\Krop := \Big\{  (u,w) \in \DO \, :  \,  \rho_+\!\left(\tfrac{1}{2} \, \nabla w \otimes \nabla w + e(u)\right) \leq 1  \ \ \text{in } \Ob\, \Big\},
\end{equation*}
with a convention that for a function $\rho$ by $\rhop$ we understand the function given by \eqref{eq:rhop_rho}. Let us recall that $\Krop$ is convex by convexity of the set $\Cm$ defined in \eqref{eq:Cm}.

In \cite{bolbotowski2022a} this set was investigated for a particular choice of the function $\rho = \rho^\mathrm{M}$, which for $\xi \in \Sdd$ yields the spectral norm $\rho^\mathrm{M}(\xi) = \max \big\{ \abs{\lambda_1(\xi)}, \abs{\lambda_2(\xi)} \big\}$, where $\lambda_i(\argu)$ stands for the $i$-th eigenvalue of a symmetric matrix. In the case when $E=1,p=2$, the induced energy potential $j^\mathrm{M} = \frac{1}{p} \big(\rho^\mathrm{M}(\argu)\big)^p$ coincides with the Michell energy potential discussed in Remark \ref{rem:Michell}.
Owing to the assumption \ref{as:coercivity}, for every $\rho$ that comes from $j$ via \eqref{eq:recall_rho} there exist positive constants $C_1,C_2>0$ such that $C_1\,\rho^\mathrm{M} \leq \rho\leq C_2\,\rho^\mathrm{M}$, and, consequently, $\frac{1}{C_2}(\rho^\mathrm{M})^0 \leq \rho^0\leq \frac{1}{C_1}(\rho^\mathrm{M})^0$. Next, by the second equality in \eqref{eq:rhop_rho}, we find that $	\rhop(\xi) = \sup_{\sig \in \Sddp} \big\{ \pairing{\xi,\sig}:\rho^0(\sig) \leq 1 \big\}$, and a similar formula holds true for $\rho_+^\mathrm{M}$. This way we find that:
\begin{equation}
	\label{eq:est_rhoM}
	C_1\, \rho^\mathrm{M}_+(\xi) \leq \rho_+(\xi)\leq C_2\,\rho^\mathrm{M}_+(\xi) \qquad \forall\,\xi \in \Sdd, \qquad \text{where} \quad \rho_+^\mathrm{M}(\xi) = \max \big\{ \lambda_1(\xi), \lambda_2(\xi),0 \big\}.
\end{equation}

The reason we desired to relate to the Michell potential $\rho^\mathrm{M}$ is because the set $\mathcal{K}_{\rho^\mathrm{M}}$ admits a special characterization through a two-point inequality, which was exploited in \cite{bolbotowski2022a}. This topic will be explored more widely in Section \ref{sec:Michell} below. For now, we observe that, thanks to \eqref{eq:est_rhoM}, any pair $(u,w) \in \Krop$ satisfies $\big(C_1 u,\sqrt{C_1} w\big) \in \mathcal{K}_{\rho^\mathrm{M}}$, which through \cite[Lemma 5.3]{bolbotowski2022a} lead us to what follows:
\begin{equation}
	\label{eq:two-point_0}
	(u,w) \in \Krop \quad \Rightarrow \quad  \frac{1}{2} \big(w(x)-w(y) \big)^2 + \pairing{u(x)-u(y),x-y} \leq \frac{1}{C_1} \abs{x-y}^2 \qquad \forall\,(x,y) \in \Ob \times \Ob.
\end{equation}
Consequently, we can use \cite[Lemma 3.16]{bolbotowski2022a} to deduce the result below.
\begin{proposition}
	\label{prop:estimates_uw}
	Let $\Omega \subset \Rd$ be any bounded domain. Then, for any $(u,w) \in \Krop \subset \D(\O;\Rd\times\R)$ the following estimates hold true:
	\begin{enumerate}[label={(\roman*)}]
		\item $\norm{u}_{\infty,\O} \leq 1/C_1\, \mathrm{diam}(\O)$, \ \ $\norm{w}_{\infty,\O} \leq 1/\sqrt{2\,C_1}\, \mathrm{diam}(\O)$,
		\item $\abs{w(x) - w(y)} \leq \sqrt{2\,  \mathrm{diam}(\O)/C_1}\, \abs{x-y}^{1/2}$ for every $(x,y) \in \Ob \times \Ob$,
		\item $\int_\O \abs{\nabla u} dx + \int_\O \abs{\nabla w}^2 dx \leq  C_0/C_1 \big(\mathrm{diam}(\O)\big)^2$,
	\end{enumerate}
	where: $\mathrm{diam}(\O)$ is the diameter of $\O$, $C_1$ is any constant satisfying \eqref{eq:est_rhoM}, and $C_0$ is a constant independent of $\Omega$ (it depends on the dimension of the ambient space).
\end{proposition}
From the assertion (i) above we immediately obtain the estimate:
\begin{equation}
	\label{eq:finite_Z}
	Z \leq \tfrac{1}{C_1} \,  \mathrm{diam}(\O)\,\abs{f}(\Ob)  < \infty,
\end{equation}
where $\abs{f}$ is the variation measure of $f$. The finiteness of $Z$ is essential since, via Lemma \ref{lem:problem_P_m}, it guarantees finiteness of $\Cmin$, i.e. that there exists at least one design $\mu$ furnishing finite compliance.

Based on the estimates in Proposition \ref{prop:estimates_uw} we can readily state the compactness result for the problem \ref{eq:PM}. The result below stems from combining the Arzel\`{a}-Ascoli theorem, reflexiveness of $W^{1,2}(\O)$, and a standard compactness theorem for $BV$ functions \cite[Section 5.2, Theorem 4]{evans1992}. The last result necessitates some regularity of the boundary $\bO$.
\begin{proposition}
	\label{prop:compactness}
	Let $\Omega \subset \Rd$ be a bounded domain with Lipschitz continuous boundary. Then, the convex set $\Krop$ is precompact in the norm topology of the Cartesian product $L^1(\Omega;\Rd) \times C(\Ob;\R)$. Its closure, namely the compact set
	\begin{equation*}
		\Kro :=  \mathrm{clos}_{L^1(\Omega;\Rd) \times C(\Ob;\R)}  \big(\Krop\big),
	\end{equation*}
	satisfies the inclusion
	\begin{equation}
	\label{eq:regularity_u_w}
	\Kro \ \subset \ (BV\cap L^\infty)(\Omega;\Rd)  \  \times \  (C^{0,\frac1{2}}\cap W^{1,2})(\Omega;\R) .
	\end{equation}
\end{proposition}

We are now in a position to formulate the relaxed variant of the problem \ref{eq:PM} that for a Lipschitz-regular domain $\O$ admits a solution:
\begin{equation}
\label{eq:relPM}\tag*{$(\overline{\mathcal{P}})$}
Z = \max \left \{ \int_\Ob w \,df \, : \, (u,w) \in \Kro  \right\}
\end{equation}
Since the linear mapping $w \mapsto \int_\Ob w \,df$ is continuous on $C(\Ob;\R)$, the existence of solution is due to the Direct Method of the Calculus of Variations. Equality $Z = \max \overline{\mathcal{P}}$ follows by the continuity of the objective functional as well.

\begin{remark}
	Assume any solution $\check{\mu}$ of \ref{eq:OEM}. As it will unravel in Proposition \ref{cor:uw} below,  the restriction to $\check{\mu}$ of any (suitably scaled) solution $(\hat{u},\hat{w})$ of the problem \ref{eq:relPM} becomes the displacements in the optimally designed membrane. Hence, it is worthwhile to compare the regularity of functions $(u,w) \in \Kro$ stated in Proposition \ref{prop:compactness} to the one obtained in \cite{conti2006}, namely for the non-optimized membrane of constant thickness. As mentioned in Section \ref{ssec:foppl_intro}, such functions $(u,w)$ lie in the space  $BD(\O;\Rd) \times W^{1,2}(\O;\R)$, where $BD(\O;\Rd)$ stands for the space of functions of bounded deformations \cite{temam1985} for whom $BV(\O;\Rd)$ is a proper subspace. The increased, $BV$ regularity of $u$ guaranteed by Proposition \ref{prop:compactness} follows from the control of the $L^1$ norm of the gradient of $u$ in Proposition \ref{prop:estimates_uw} (iii) and not just of its symmetrical part like in \cite{conti2006}. In turn, this control rests upon the theory of maximal monotone maps \cite{alberti1999}: from the two-point condition \eqref{eq:two-point_0} it follows that the function $v(x) = x - C_1 u(x)$ is monotone. From the inequality \eqref{eq:two-point_0} one also obtains the H\"{o}lder estimate in assertion (ii) of Proposition \ref{prop:estimates_uw}, which ultimately leads to the $C^{0,1/2}$ regularity of $w$. The reader is referred to \cite[Lemma 3.16]{bolbotowski2022a} for more details.
	
	A similar increase of regularity owing to optimization has been well observed throughout the former works on optimal mass distribution \cite{bouchitte2001,bouchitte2007,bolbotowski2022b,lewinski2021}.  For instance, in the optimally designed heat conductors the temperature function is Lipschitz continuous, whilst it is merely an element of $W^{1,2}(\O;\R)$ for a homogeneous conductor.
\end{remark}

For the next step we wish to characterize the set $\Kro$ as a subset of $L^1(\Omega;\Rd) \times C(\Ob;\R)$. Such characterization will be essential from the perspective of the finite element approximation to be put forward in Section \ref{sec:FEM}. The forthcoming result will be a generalization of \cite[Proposition 5.11]{bolbotowski2022a}, which was established for the case $\rho = \rho^\mathrm{M}$ only. In \cite{bolbotowski2022a} the proof of one of the implication relied heavily on the structure of $\rho^\mathrm{M}$. We will therefore give an independent proof of this implication for an arbitrary function $\rho$. Although the proof of the converse implication is similar to the one in \cite{bolbotowski2022a}, we shall repeat it here for the reader's convenience.

In order to avoid technical difficulties below we shall strengthen the regularity requirements for $\O$. A bounded domain $\Omega$ will be called \textit{star-shaped with respect to a ball} if there exists a ball $B(x_0,\eps) \subset \Omega$ such that for any $x\in \Omega$ and $y \in B(x_0,\eps)$ the segment $[x,y]$ is contained in $\Omega$. Such domains satisfy the following contraction property: assuming that $x_0$ is the origin, for any $\delta\in (0,1)$  the contracted set $\Omega_\delta:= \frac1{1+\delta} \, \Omega$ is a positive distance away from $\bO$. Below we agree that for a function $u \in BV(\Rd;\Rd)$ its distributional symmetric derivative will be denoted by $\eps(u) = \frac{1}{2} \big(Du+(Du)^\top\big)\in \Mes(\Rd;\Sdd)$, while by writing  $\eps({u}) = e({u}) \,\mathcal{L}^2+ \eps_s({u})$ we shall understand its Lebesgue decomposition:   $\eps_s({u}) \in \Mes(\Rd;\Sdd)$ is the part of $\eps(u)$ that is singular with respect to Lebesgue measure, whilst $e(u) \in L^1(\Rd;\Sdd)$ is the Radon-Nikodym derivative of the absolutely continuous part with respect to $\mathcal{L}^2$.

\begin{theorem}
	\label{prop:char_Kro}
	Let $\Omega \subset \Rd$ be a bounded domain that is star-shaped with respect to a ball.
	
	Then, a pair of functions $(u,w) \in  L^1(\Omega;\Rd) \times C(\Ob;\R)$ is an element of the set $\Kro$ if and only if the following conditions are satisfied:
	\begin{enumerate}[label={(\roman*)}]
		\item $w \in W^{1,2}(\Omega)$, and $w =0$ on $\bO$.
		\item The extension $\ov{u} \in L^1(\Rd;\Rd)$ of $u$ by zero (i.e. $\ov{u}=u$ on $\Omega$, and $\ov{u} = 0$ on $\Rd \backslash \Ob$) belongs to $BV(\Rd;\Rd)$.
		\item Let $\nabla w$ be the weak gradient of $w$, and consider the decomposition $\eps(\ov{u}) = e(\ov{u}) \,\mathcal{L}^2+ \eps_s(\ov{u}) \in \Mes(\Rd;\Sdd)$; the condition then reads:
		\begin{alignat}{1}
		\label{eq:eps_s_is_negative}
		&\eps_s(\ov{u}) \in \Mes(\Rd;\Sddp) \text{, \  namely $\eps_s(\ov{u}) $ is a negative semi-definite valued Radon measure},\\
		\label{eq:quadratic_condition_a.e.}
		& \rho_+\Big(\tfrac{1}{2}\nabla w \otimes \nabla w + e(\ov{u}) \Big)  \leq 1 \qquad \text{a.e. in } \Rd.
		\end{alignat}
	\end{enumerate}	
\end{theorem}

\begin{remark}
	From the above characterization, for functions $u$ that are admissible in \ref{eq:relPM} one can conclude all the same properties that were expounded in the appendix of \cite{conti2006}. Let us shortly revisit the matter of the boundary conditions only. From the property (ii) above it follows that such a function $u$ may have non-zero values on the boundary $\bO$ in the sense of trace $\mathrm{tr}\,u$ of a function in $BV(\O;\Rd)$. However, because of the condition \eqref{eq:eps_s_is_negative}, we can infer that $\mathrm{tr}\,u(x) = \lambda(x)\, \nu_\O(x)$ for a.e. $x \in \bO$, where $\nu_\O$ is the outer normal to $\O$, and $\lambda \in L^\infty(\bO;\R_+)$ is a non-negative function. 
\end{remark}

\begin{proof}
	First we show that for any pair $(u,w) \in \Kro$ the conditions (i), (ii), (iii) are met. By definition there exists a sequence of smooth functions $(u_n,w_n) \in \Krop$ such that $u_n \to \ov{u}$ in $L^1(\Rd;\Rd)$, and $w_n \to w$ uniformly on $\Ob$; note that $(u_n,w_n) = (0,0)$ on $\Rd \backslash \O$. Accordingly, conditions (i), (ii) follows from \eqref{eq:regularity_u_w}. Thanks to \eqref{eq:regularity_u_w} we know that (up to extracting a subsequence) $\eps(u_n) = e(u_n) \,\mathcal{L}^2 \,\weakstar\, \eps(\ov{u})$ and $w_n \rightharpoonup w$ weakly in $W^{1,2}(\O)$; in particular $\nabla w_n \rightharpoonup \nabla w$ weakly in $L^2(\O;\Rd)$.  One can thus write that $\big(e(u_n)\mathcal{L}^2, \nabla w_n\mathcal{L}^2 \big)\, \weakstar \, \big( \eps(\ov{u}), \nabla w \,\mathcal{L}^2 \big)$ in $\Mes(\Rd;\Sdd \times \Rd)$. Using the Goffman-Serrin convention, for any open set $U \subset \Rd$ let us propose the functional $\Mes(\Rd;\Sdd \times \Rd) \ni (\epsilon,\lambda) \mapsto \int_U \vro(\epsilon,\lambda)$, where $\vro$ is the positively one-homogeneous convex continuous  function defined in \eqref{eq:vro}. Thanks to Reshetnyak theorem (see \cite[Theorem 2.38]{ambrosio2000}) this functional is weakly-* lower semi-continuous (note the importance of the openness of  $U$), and so
	\begin{equation*}
		\int_U \vro\big( \eps(\ov{u}), \nabla w \,\mathcal{L}^2 \big) \ \leq \ \liminf_{n \to \infty} 	\int_U \vro\big(e(u_n)\mathcal{L}^2, \nabla w_n\mathcal{L}^2 \big) =   \liminf_{n \to \infty} 	\int_U \vro\big(e(u_n), \nabla w_n \big) \,d\mathcal{L}^2 \leq\mathcal{L}^2(U),
	\end{equation*}
	where in the last inequality we acknowledge that $(u_n,w_n) \in \Krop$ while utilizing the two characterizations  \eqref{eq:Cm} and \eqref{eq:Cm_vro} of the set $\Cm$. Due to arbitrariness of $U$ we deduce the inequality between two positive Radon measures: $ \vro\big( \eps(\ov{u}), \nabla w \,\mathcal{L}^2 \big)  \leq \mathcal{L}^2$. Let now $E \subset \Rd$ be a Borel set of zero Lebesgue measure such that $\abs{\eps_s(\ov{u})} (E) = \abs{\eps_s(\ov{u})} (\R^2)$. By testing the inequality with Borel sets $B \subset E$ we infer that $\vro\big(\eps_s(\ov{u}), 0\big) (B) \leq 0$, hence the statement \eqref{eq:eps_s_is_negative} owing to \eqref{eq:zero_vro}. Next, we test the inequality with Borel subsets of $\Rd \backslash E$. Since $E$ is of zero Lebesgue measure, we conclude that $\vro\big(e(\ov{u}), \nabla w\big) \leq 1$ a.e., which is the condition \eqref{eq:quadratic_condition_a.e.}. This ends the proof of the first implication.
	
	Conversely, we take functions $\ov{u} \in BV(\Rd;\Rd)$ and $w \in C(\Rd;\R) \cap W^{1,2}(\Rd;\R)$ such that $\ov{u} = 0$, $w = 0$ a.e. in $\Rd \backslash \Ob$ and satisfying \eqref{eq:eps_s_is_negative}, \eqref{eq:quadratic_condition_a.e.}. Our goal is to construct a sequence $(u_n,w_n) \in \Krop \subset \D(\O;\Rd \times \R)$ such that $u_n \rightarrow \ov{u}$ in $L^1(\Rd;\Rd)$ and $w_n \rightarrow w$ uniformly in $\Rd$.
	Without loss of generality we may assume that $\Omega$ is star shaped with respect to a ball centred at the origin. For $\delta>0$ we define the scaled contractions of functions $\ov{u}, w$:
	\begin{equation*}
		u_\delta(x) = \frac{1}{1+\delta}\, \ov{u}\big( (1+\delta) x  \big), \qquad w_\delta(x) = \frac{1}{1+\delta}\, w\big( (1+\delta) x  \big).
	\end{equation*}
	It is clear that still $u_\delta \in BV(\Rd;\Rd)$ and $w_\delta \in C(\Rd;\R) \cap W^{1,2}(\Rd;\R)$, while $u_\delta \to \ov{u}$ in $L^1(\Rd;\Rd)$ and $w_\delta \to w$ uniformly on $\Rd$ when $\delta \to 0$. It is also easy to check that $u_\delta, w_\delta$ satisfy \eqref{eq:eps_s_is_negative} and \eqref{eq:quadratic_condition_a.e.} that is:
	\begin{equation*}
		\eps_s(u_\delta) \in \Mes(\Rd;\Sddm) \qquad \text{and} \qquad \vro\big(e(u_\delta),\nabla w_\delta \big) \leq 1\quad \text{ a.e. in } \Rd.
	\end{equation*}
	Finally $(u_\delta,w_\delta) = 0$ in $\Rd \setminus \Omega_{\delta}$ where the contracted domain $\Omega_\delta = \frac1{1+\delta} \, \Omega$ satisfies $ \epsilon_\delta  := D_H(\bO,\Omega_{\delta})>0$ for every $\delta>0$, whilst $D_H$ stands for the Hausdorff distance between sets.
	
	For $0<\epsilon<\epsilon_\delta$ let us take the standard radial-symmetric mollification kernel $\eta_\epsilon$ whose support is contained in $\{ \abs{x} \leq \epsilon\}$. The mollified functions $u_{\delta,\epsilon} = \eta_\epsilon * u_\delta$, \, $w_{\delta,\epsilon} =  \eta_\epsilon * w_\delta$ are then compactly supported in $\Omega$, i.e. $(u_{\delta,\epsilon},w_{\delta,\epsilon}) \in \DO$, and, owing to the classical results  \cite{evans1992},  $u_{\delta,\epsilon} \to u_\delta$ in $L^1(\Rd;\Rd)$ and $w_{\delta,\epsilon} \to w_\delta$ uniformly in $\Rd$ when $\epsilon \to 0$. Finally, in order to prove that $(u_{\delta,\epsilon},w_{\delta,\epsilon}) \in \Krop$ we must show that $\vro\big(e(u_{\delta,\epsilon}),\nabla w_{\delta,\epsilon} \big) \leq 1$ in $\Omega$. Since $w \in W^{1,2}(\Rd;\R)$  and  $u \in BV(\Rd;\Rd)$, it is well establishes that $\nabla w_{\delta,\epsilon} = \eta_\epsilon * \nabla w_\delta$ and $e(u_{\delta,\epsilon}) = \eta_\epsilon * \eps(u_{\delta}) = \eta_\epsilon*e(u_{\delta})+ \eta_\epsilon * \eps_s(u_{\delta})$.
	Since $\eps_s(u_\delta) \in \Mes(\Rd;\Sddm)$, it is straightforward that $\big(\eta_\epsilon * \eps_s(u_\delta)\big)(x) \in \Sddm$ for every $x \in \Rd$. The following chain of inequalities then holds for all $x \in \Rd$:
	\begin{align*}
		&\,\vro\bigl(e(u_{\delta,\epsilon})(x),\nabla w_{\delta,\epsilon}(x) \bigr) = \vro\biggl(  \bigl(\eta_\epsilon*e(u_{\delta})\bigr)(x) + \bigl(\eta_\epsilon*\eps_s(u_{\delta})\bigr)(x) \ , \ \bigl(\eta_\epsilon *\nabla w_{\delta}\bigr)(x) \biggr)\\
		& \quad\leq \vro\biggl(  \bigl(\eta_\epsilon*e(u_{\delta})\bigr)(x)  \ , \ \bigl(\eta_\epsilon *\nabla w_{\delta}\bigr)(x) \biggr)  = \vro\biggl( \int_{\Rd} \Big(e(u_{\delta})(y), \nabla w_{\delta}(y)\Big)\, \eta_\epsilon(x-y)\,\mathcal{L}^2(dy)  \biggr)\\
		& \quad\leq  \int_{\Rd} \vro\Big(e(u_{\delta})(y), \nabla w_{\delta}(y)\Big)\, \eta_\epsilon(x-y)\,\mathcal{L}^2(dy) \leq 1
	\end{align*}
	where:
	\begin{itemize}
		\item[-] to pass to the second line we used the fact that for $\te \in \Rd$, $\xi \in \Sdd$, and $\zeta\in \Sddm$ there holds $\vro(\xi + \zeta,\theta) \leq \vro(\xi, \theta)$, which is precisely the monotonicity property stated in the assertion (iv) of Proposition \ref{prop:varrho_polar};
		\item[-] to pass to the third line Jensen's inequality was employed for the convex function $\vro: \Sdd \times \Rd \to \R_+$.
	\end{itemize}
	We have thus showed that $(u_{\delta,\epsilon},w_{\delta,\epsilon}) \in \Krop$. Owing to the aforementioned convergences in $\delta$ and $\epsilon$ the sequence $(u_n,w_n) := (u_{\delta_n,\epsilon_n},w_{\delta_n,\epsilon_n})$ converging to $(u,w)$ in  $L^1(\Rd;\Rd) \times C_0(\Rd;\R)$ may be found by using the diagonalization argument. The proof of the converse implication, and thus of the whole theorem, is complete.
\end{proof}

Let us next address the matter of the optimality criteria for pairs $(u,w)$ and $(\TAU,\vartheta)$ to solve \ref{eq:relPM} and \ref{eq:dPM}, respectively. Once again, we will be generalizing the result from \cite{bolbotowski2022a} that was given for $\rho=\rho^\mathrm{M}$ only. Let us remark, however, that even in \cite{bolbotowski2022a} the conditions were not derived for an arbitrary pair $(u,w) \in \overline{\mathcal{K}}_{\rho^\mathrm{M}}$. The issue lies in the necessity of derivatives $e(u),\nabla w$ to be defined $\mu = \dro(\TAU)$-a.e. Accordingly, in the paper \cite{bolbotowski2022a} $u,w$ are assumed to be Lipschitz continuous, which unlocks the \textit{$\mu$-tangential calculus} first put forward in \cite{bouchitte1997} (see also \cite{bouchitte2003,bouchitte2007}), where one uses the $\mu$-tangential operators $e_\mu(u), \nabla_\mu w$. In contrast to \cite{bolbotowski2022a}, this level of generality is not essential in this paper, and, because employing the $\mu$-tangential calculus is very delicate, we will be satisfied by confining ourselves to the $C^1$ setting.
\begin{proposition}
	\label{prop:OPM_optimality_conditions}
	For a bounded domain $\Omega$ that is star-shaped with respect to a ball let us take a quintuple $(u,w,\mu,\sig,\q)$ such that: $(u,w)\in C^1(\Ob;\Rd \times \R)$, ${\mu} \in \Mes_+(\Ob)$, ${\sig} \in L^1_\mu(\Ob;\Sddp)$, $\q \in L^1_\mu(\Ob;\Rd)$; in addition assume that $(u,w) = (0,0)$ on $\bO$. 
	
	Then, the pairs $(u,w)$ and  $(\TAU,\vartheta) = (\sig\mu,\q \mu)$ solve problems $\relProb$ and $\dProb$, respectively, if and only if the following optimality conditions are met:
	\begin{equation}
	\label{eq:opt_cond_memb}
	\begin{cases}
	(i) &  -\DIV(\sig \mu) = 0, \quad -\dive(\q \mu) = f \quad \text{ in } \O,\\
	(ii) & \rho_+\big(\frac{1}{2} \, \nabla w \otimes \nabla w + e(u) \big) \leq 1  \quad \text{everywhere in } \Ob,\\
	(iii) & \pairing{\frac{1}{2} \, \nabla w \otimes \nabla w + e(u),\sig} = \dro(\sig) \quad \mu\text{-a.e.},\\
	(iv) & \q = \sig \,\nabla w \quad \mu\text{-a.e.,}\\
	(v) & (\sig\mu,\q \mu) \mres \bO = 0.
	\end{cases}
	\end{equation}
\end{proposition}
\begin{proof}
	Equations (i) are clearly the admissibility conditions for $(\sig\mu,\q \mu)$ in the problem $\dProb$.  Considering the characterization in Theorem  \ref{prop:char_Kro}, condition (ii) plays the analogous role for the smooth functions $(u,w)$ and the problem $\relProb$. For these reasons, both (i) and (ii) can be assumed to be true in the remainder of the proof.
	
	Owing to the duality result in Proposition \ref{prop:supP=infP*}, the pairs $(u,w)$ and $(\sig \mu,q\mu)$ are solutions of $\relProb$ and, respectively, $\dProb$ if and only if the global optimality condition is satisfied: $\int_\Ob w\,df = \int_\Ob \vro^0(\sig,q) \,d\mu$. Meanwhile, we can write down the following chain:
	\begin{equation*}
		\int_\Ob w\,df = \int_\O w\,df = \int_\Omega \Big(\pairing{e(u),\sigma} + \pairing{\nabla w,q} \Big)\,d\mu \leq   \int_\Omega \vro^0(\sig,q) \,d\mu \leq \int_\Ob \vro^0(\sig,q) \,d\mu,
	\end{equation*}
	where:
	\begin{itemize}
		\item[-] the first equality is due to the fact that $w =0 $ on $\bO$;
		\item[-] the second equality is the integration by parts formula which acknowledges (i) and a simple density argument;
		\item[-] to obtain the first inequality we recall that condition (ii) can be rewritten as $\vro\big(e(u),\nabla w\big) \leq 1$ in $\Ob$.
	\end{itemize}
	Due to the assertion (iii) of Proposition \ref{prop:varrho_polar}, the last inequality in the chain above is an equality if and only if (v) holds. The first inequality is an equality if and only if $\big(\pairing{e(u),\sigma} + \pairing{\nabla w,q} \big) = \vro^0(\sig,q)$ \ $\mu$-a.e. By virtue of assertion (ii) in Proposition \ref{prop:varrho_polar} the latter condition is equivalent to the pair of conditions (iii), (iv), which completes the proof.
\end{proof}

We recall that Theorem \ref{thm:constructing_lambda_OEM} -- apart from the optimal material distribution itself -- yields precise stress functions $(\check{\sigma},\check{q})$ that are induced in the optimally designed membrane $\check{\mu}$ due to the load $f$. Those informations are recast based on a solution $(\hat{\TAU},\hat{\vartheta})$ of the dual problem \ref{eq:dPM}. Let us now turn to the issue of recovering displacements $(\check{u},\check{w})$ in the optimal membrane based on solutions $(\check{u},\check{w})$ of the relaxed primal problem \ref{eq:relPM}. For each such optimal pair let us introduce the scaling:
\begin{equation}
	\label{eq:scaling_uw}
	\check{u} = \left(\frac{\Z}{2\Totc}\right)^{\frac{2}{2p-1}} \hat{u}, \qquad \check{w} = \left(\frac{\Z}{2\Totc}\right)^{\frac{1}{2p-1}} \hat{w}.
\end{equation}
In order to exploit the optimality conditions we assume that the above functions are of $C^1$ class:
\begin{corollary}
	\label{cor:uw}
	For a bounded domain $\Omega$ that is star-shaped with respect to a ball assume that functions $(\hat{u},\hat{w}) \in C^1(\Ob;\Rd\times\R)$ satisfying $(\hat{u},\hat{w}) = (0,0)$ solve the problem \ref{eq:relPM}. In addition, take a triple $(\check{\mu},\check{\sig},\check{q})$ in accordance with Theorem \ref{thm:constructing_lambda_OEM}.
	
	Then, the quadruple $({u},{w},{\sig},{q})=(\check{u},\check{w},\check{\sig},\check{q})$ solves the system \eqref{eq:Foppl_system} for $\mu = \check{\mu}$. Namely, $(\check{u},\check{w})$ and $(\check{\sigma},\check{q})$ are the displacements and, respectively, stresses in a F\"{o}ppl's membrane of an optimal material distribution $\check{\mu}$. 
\end{corollary}
\begin{proof}
	By virtue of Proposition \ref{prop:OPM_optimality_conditions} conditions \eqref{eq:opt_cond_memb}(i,iii,iv) are met for $({u},{w},\mu,{\sig},{q})=(\hat{u},\hat{w},\hat{\mu},\hat{\sig},\hat{q})$, where $\hat{\mu},\hat{\sig},\hat{q}$ are as in Theorem \ref{thm:constructing_lambda_OEM}. Using the scaling formulas \eqref{eq:link_OEM_dProb} and \eqref{eq:scaling_uw} we recover the conditions \eqref{eq:Foppl_system}(i,ii,iii). To verify \eqref{eq:Foppl_system}(ii) we use the following fact.
	If for a positive 1-homogeneous convex l.s.c function $h:\Sdd\to \Rb$ there holds $\pairing{\xi,\sig} = 1$ for a pair with $h(\xi)\leq 1$ and $h^0(\sig)=1$, then $(t_2 \sig) \in \partial f(t_1 \xi)$ for $f(\argu) = \frac{1}{p} \big(h(\argu)\big)^p$ whenever $t_1^p = t_2^{p'}$. In our case $h =\rhop$, $f = j_+$, and $t_1 = (Z/(2V_0))^{2/(2p-1)}$, $t_2 = (Z/(2V_0))^{(2p-2)/(2p-1)}$.
\end{proof}

\begin{remark}
	\label{rem:uw}
	In the case when solutions $(\hat{u},\hat{w}) \in \Kro$ of the problem \ref{eq:relPM} are non-smooth the functions $(\check{u},\check{w})$ in \eqref{eq:scaling_uw} can be still interpreted as the displacements in the optimal membrane $\check{\mu}$ provided we employ the variational formulation \ref{eq:nlcomp} for $\mu=\check{\mu}$. More accurately, one can prove that there always exists a minimizing sequence $(u_h,w_h) \in \D(\O;\Rd \times \R)$ for \ref{eq:nlcomp} such that $u_h \to \check{u}$ in $L^1(\O;\Rd)$ and $w_h \to \check{w}$ uniformly on $\Ob$. We skip the proof here.
\end{remark}

\section{The finite element approximation for the pair $\Prob$, $\dProb$ }
\label{sec:FEM}

\subsection{Formulation of the finite element problem}

In the whole section we shall assume that $\Omega \subset \Rd$ is a polygonal domain that is star-shaped with respect to a ball. For a mesh parameter $h \in  (0,\infty)$ we consider a triangulation of $\Omega$:
\begin{equation*}
\mathcal{T}^h = \left\{E^h_i \, : \, E^h_i \subset \O \ \text{ is an open simplex},\ E^h_{i_1} \cap E^h_{i_2} = \varnothing \ \text{for } i_1 \neq i_2, \ \Ob = \bigcup_{i=1}^n \overline{E_i^h}, \  \mathrm{diam}(E_i^h) \leq h \right\},
\end{equation*} 
where $n$ stands for the number of elements in $\mathcal{T}^h$.
Of course, the condition $\Ob = \bigcup_{i=1}^n \overline{E_i^h}$ is possible only thanks to the fact that $\Omega$ is a polygon. We define the finite dimensional linear space of continuous displacements that are element-wise affine:
\begin{align*}
\label{eq:Uh}
\mathcal{V}^h := \Big\{ (u,w) \in & \, C(\Ob;\Rd\times \R)\, : \  (u,w) = (0,0) \text{ on } \bO, \\
& \qquad \big(u(x),w(x)\big) = a_i + A_i x \quad \forall\,x\in E_i^h, \ a_i \in \R^3, \ A_i \in \R^{3 \times 2} \Big\}.
\end{align*}
The finite dimensional counterpart of the relaxed primal problem $\relProb$ may be readily proposed:
\begin{equation}
\label{eq:PMh}\tag*{$(\mathcal{P}_h)$}
Z_h := \max \left \{ \int_\Ob w \,df \, : \, (u,w) \in \mathcal{K}^h_\rho  \right\} \quad
\end{equation}
where
\begin{equation*}
\mathcal{K}^h_\rho :=\, \Kro \cap \mathcal{V}^h.
\end{equation*}
Naturally, since $\mathcal{V}^h$ is a finite dimensional subspace of $L^1(\O;\Rd) \times C(\Ob;\R)$, the convex set $\mathcal{K}^h_\rho$ is compact, and, therefore, solution of \ref{eq:PMh} always exists. By comparing the formulations \ref{eq:relPM} and \ref{eq:PMh} we immediately deduce that
\begin{equation}
\label{eq:Z_h_leq_Z}
Z_h \leq Z.
\end{equation}
By a duality argument, (that is a variant of the one in the proof of Proposition \ref{prop:supP=infP*} with  $X = \mathcal{V}^h$, $Y = L^\infty(\O;\Sdd\times\Rd)$), we arrive at the dual problem:
\begin{equation}
\label{eq:dPMh}\tag*{$(\mathcal{P}^*_h)$}
\inf \left \{ \int_\Omega \Big(\dro(\sig) + \jsq \Big) d\mathcal{L}^2 \, : \, (\sigma,q) \in \Sigma^h(f)  \right\} \quad
\end{equation}
where the \textit{set of approximately admissible stress functions} reads
\begin{equation*}
\Sigma^h(f) : = \left\{  (\sig,q) \in L^1(\Omega;\Sddp \times \Rd) \, : \, \int_\Omega\Big( \pairing{e(\phi) , \sig} + \pairing{\nabla \varphi,  q} \Big) \, d\mathcal{L}^2 = \int_\Ob \varphi\, df \quad  \forall \, (\phi,\varphi) \in \mathcal{V}^h \right\}.
\end{equation*}
The problem $(\mathcal{P}^*_h)$ above is \textit{a priori} infinite dimensional and using standard tools we merely obtain inequality $Z_h=\sup \mathcal{P}_h \leq \inf \mathcal{P}_h^*$. In due course, however, we will prove that the converse inequality holds as well while $(\mathcal{P}_h^*)$ will turn out to be equivalent to a finite dimensional problem, cf. Corollary \ref{cor:Ph_Ph*_finite_and_infinite} below.

\subsection{Algebraic reformulation to a conic programming problem}
\label{ssec:algebraic}
 We assume a standard Cartesian frame $\e_1 = (1,0,0), \, \e_2=(0,1,0),\, \e_3=(0,0,1)$ in the space $\R^3$. Let $X^h$ stand for the set of vertices (nodes) of the triangulation $\mathcal{T}^h$ \textit{without the vertices belonging to $\bO$}. With $m$ being the number of those vertices, let us assume that $X^h = \{ \varv_1, \ldots,\varv_m \}$ for $\varv_j \in \Rd$. In the sequel the elements of the space $\mathcal{V}^h$ will be represented through vectors of nodal displacements. Those vectors will be denoted in bold font: $\mbf{u}_1,\mbf{u}_2,\mbf{w} \in \R^m$, whilst by $\mbf{u}_1(j),\mbf{u}_2(j),\mbf{w}(j) \in \R$ we will understand their $j$-th components.
 
Let us introduce the interpolation operator $\iota^h$ being an isomorphism between $(\R^m)^3$ and $\mathcal{V}^h$
that, upon denoting $\iota^h(\mbf{u}_1,\mbf{u}_2,\mbf{w})=\big(\iota_\shortparallel^h(\mbf{u}_1,\mbf{u}_2),\iota_\perp^h(\mbf{w})\big)$, satisfies 
\begin{equation*}
\begin{cases}
	u = \iota_\shortparallel^h(\mbf{u}_1 ,\mbf{u}_2),\\
	w = \iota_\perp^h(\mbf{w})
\end{cases}
 \quad \Leftrightarrow \qquad  (u,w) \in \mathcal{V}^h, \quad \ \
\begin{cases}
	u(\varv_j) = \mbf{u}_1(j)\, \e_1 + \mbf{u}_2(j)\, \e_2,\\
	w(\varv_j) = \mbf{w}(j)
\end{cases}
 \ \  \forall\, j \in \{1,\ldots,m\}.
\end{equation*}
Next, we define two isomorphisms $\chi_2 :\R^3 \to \Sdd$ and $\chi_3:\R^6 \to \mathcal{S}^{3\times 3}$:
\begin{align}
\nonumber
\chi_2\bigl(a^{11}, a^{22}, a^{12} \bigr) :=&\ a^{11}\,\e_1 \otimes \e_1 + a^{22}\, \e_2 \otimes \e_2 + \sqrt{2} \, a^{12} \,\e_1 \symtens \e_2,\\
\label{eq:isometry}
\chi_3\bigl(a^{11}, a^{22}, a^{33},a^{12},a^{13},a^{23} \bigr) :=&\ a^{11}\,\e_1 \otimes \e_1 + a^{22}\, \e_2 \otimes \e_2 + a^{33}\, \e_3 \otimes \e_3\\
\nonumber
	&\  \qquad + \sqrt{2} \, a^{12} \,\e_1 \symtens \e_2 + \sqrt{2} \, a^{13} \,\e_1 \symtens \e_3 + \sqrt{2} \, a^{23} \,\e_2 \symtens \e_3,
\end{align}
where $\e_i \odot \e_j = \frac{1}{2}(\e_i\otimes \e_j + \e_j \otimes \e_i)$ is the symmetrized tensor product. Both maps are linear isometries, namely
\begin{equation}
\label{eq:isometry_chi}
\Big\langle\bigl(a^{11}, a^{22}, a^{12} \bigr),\bigl(b^{11}, b^{22}, b^{12} \bigr) \Big\rangle_{\R^3} = \Big\langle \chi_2\bigl(a^{11}, a^{22}, a^{12} \bigr),\chi_2\bigl(b^{11}, b^{22}, b^{12} \bigr)\Big\rangle_{\mathcal{S}^{2 \times 2}},
\end{equation}
and an analogous equality holds true for $\chi_3$.

Any pair $(u,w) \in \mathcal{V}^h$ is Lipschitz continuous; therefore, by the Rademacher theorem, $u$ and $w$ are a.e. differentiable. Moreover, $e(u)$ and $\nabla w$ are constant in each element $E_i^h$. As a consequence, there exist \textit{geometric matrices} $ \mbf{B}^{11}, \mbf{B}^{22}, \mbf{B}^{12}_1, \mbf{B}^{12}_2, \mbf{D}^{1},  \mbf{D}^{2} \in \R^{n\times m}$ that are uniquely defined by the following relation:
\begin{equation}
\label{eq:geo_mat_FMD_def}
\begin{cases}
	e\bigl(\iota_\shortparallel^h(\mbf{u}_1 ,\mbf{u}_2) \bigr) = \chi_2 \bigl( \bm{\epsilon}^{11}(i),\bm{\epsilon}^{22}(i),\bm{\epsilon}^{12}(i)\bigr),\\
	\nabla \big(\iota_\perp^h(\mbf{w})\big) = \bm{\theta}^1(i)\,\e_1 + \bm{\theta}^2(i)\,\e_2
\end{cases} 
\quad \text{a.e. in } E_i^h
\qquad
\Leftrightarrow
\qquad
\begin{cases}
\bm{\epsilon}^{11} = \mbf{B}^{11} \mbf{u}_1, \\
\bm{\epsilon}^{22} =  \mbf{B}^{22} \mbf{u}_2,\\
\bm{\epsilon}^{12} = \mbf{B}_1^{12} \mbf{u}_1 + \mbf{B}_2^{12} \mbf{u}_2,\\
\bm{\theta}^{1} = \mbf{D}^1 \mbf{w},\\
\bm{\theta}^{2} = \mbf{D}^2 \mbf{w}.
\end{cases}
\end{equation}
where $\bm{\epsilon}^{11},\bm{\epsilon}^{22},\bm{\epsilon}^{12},\bm{\theta}^{1},\bm{\theta}^{2} \in \R^n$, $n$ being the number of elements in $\mathcal{T}^h$.

To every pair of integrable functions $(\sigma,q) \in L^1(\O;\Sdd\times\Rd)$ we will assign vectors $\bm{\sigma}^{11},\bm{\sigma}^{22}, \bm{\sigma}^{12},  \mbf{q}^{1},  \mbf{q}^{2} \in \R^n$ as below:
\begin{equation}
\label{eq:mbsigma}
\begin{cases}
\bm{\sigma}^{11}(i) = \int_{E_i^h} \pairing{\sigma,\e_1 \otimes \e_1} \, d\mathcal{L}^2,\\
\bm{\sigma}^{22}(i) = \int_{E_i^h} \pairing{\sigma,\e_2 \otimes \e_2} \, d\mathcal{L}^2, \\
\bm{\sigma}^{12}(i) = \int_{E_i^h} \sqrt{2}\,\langle \sigma,\e_1 \symtens \e_2 \rangle \, d\mathcal{L}^2,\\
\mbf{q}^{1}(i) = \int_{E_i^h} \pairing{q,\e_1} \, d\mathcal{L}^2, \\
\mbf{q}^{2}(i) = \int_{E_i^h} \pairing{q,\e_2} \, d\mathcal{L}^2.
\end{cases}
\end{equation}
Finally, the load $f$ will be represented by $\mbf{f} \in \R^m$ being the unique vector that satisfies the relation:
\begin{equation}
\label{eq:nodal_load_vector}
\pairing{\mbf{w},\mbf{f}} = \int_{\Ob}\iota_\perp^h(\mbf{w}) \, df \qquad \qquad  \forall\, \mbf{w} \in \R^m.
\end{equation} 
By means of a simple argument that combines \eqref{eq:isometry_chi} and \eqref{eq:geo_mat_FMD_def}, we arrive at an algebraic representation of the set $\Sigma^h(f)$.
For any pair of functions $(\sigma,q) \in L^1(\Omega;\Sddp\times \Rd)$ there holds the equivalence:
	\begin{equation}
	\label{eq:algebraic_EqEq_FMD}
	(\sigma,q) \in \Sigma^h(f) \qquad \Leftrightarrow \qquad	
	\begin{cases}
	(\mbf{B}^{11})^\top  \bm{\sigma}^{11} +(\mbf{B}_1^{12})^\top  \bm{\sigma}^{12} = \mbf{0},  \\
	(\mbf{B}_2^{12})^\top  \bm{\sigma}^{12} + (\mbf{B}^{22})^\top  \bm{\sigma}^{22} = \mbf{0}, \\
	(\mbf{D}^{1})^\top  \mbf{q}^{1} + (\mbf{D}^{2})^\top  \mbf{q}^{2} = \mbf{f},
	\end{cases}
	\end{equation}
	where $\bm{\sigma}^{11},\bm{\sigma}^{22}, \bm{\sigma}^{12},  \mbf{q}^{1},  \mbf{q}^{2} \in \R^n$ are given by \eqref{eq:mbsigma}.

We shall now give the algebraic characterization of the set $\mathcal{K}^h_\rho =\, \Kro \, \cap \, \mathcal{V}^h$. First we recall that any pair $(u,w) \in \mathcal{V}^h$ is Lipschitz continuous, and $(u,w) = (0,0)$ on $\bO$. Therefore, conditions (i),\,(ii) in Theorem  \ref{prop:char_Kro} are automatically satisfied for any such $(u,w)$. In turn, for $(u,w)$ to be an element of $\Kro$ it is enough that the condition (iii) is met, which for Lipschitz continuous functions reduces to: $\rho_+\!\left(\tfrac{1}{2} \, \nabla w \otimes \nabla w + e(u)\right) \leq 1$ a.e. in $\O$. But the derivatives $e(u), \nabla w$ are element-wise constant, hence
\begin{align}
	\label{eq:Kroh_char}
	\mathcal{K}^h_\rho =
	\Bigg\{ \big(\iota_\shortparallel^h(\mbf{u}_1 ,\mbf{u}_2),\iota_\perp^h(\mbf{w}) \big) \, : & \, \  \mbf{u}_1,\, \mbf{u}_2, \, \mbf{w} \in \R^m, \\
	\nonumber	
	 & \quad\rho_+\!\left( \frac{1}{2} 
	\begin{bmatrix}
		\bm{\theta}^1(i)\\
		\bm{\theta}^2(i)
	\end{bmatrix}
	\otimes
	\begin{bmatrix}
	\bm{\theta}^1(i)\\
	\bm{\theta}^2(i)
	\end{bmatrix}
	+
	\begin{bmatrix}
		\bm{\epsilon}^{11}(i) & \tfrac{1}{\sqrt{2}}\bm{\epsilon}^{12}(i)\\
		\tfrac{1}{\sqrt{2}}\bm{\epsilon}^{12}(i) & \bm{\epsilon}^{22}(i)
	\end{bmatrix}
	 \right) \leq 1 \quad \forall\,i \Bigg\},
\end{align}
where $\bm{\epsilon}^{11},\bm{\epsilon}^{22},\bm{\epsilon}^{12},\bm{\theta}^{1},\bm{\theta}^{2} \in \R^n$ are ruled by the relations \eqref{eq:geo_mat_FMD_def}.
\bigskip

The equality above, together with \eqref{eq:nodal_load_vector}, already furnishes an algebraic formulation of the problem \ref{eq:PMh}. Our goal is now to rewrite it as a convex conic programming problem \cite{andersen2003}.
We begin by observing that the following set is a closed convex cone in $\R^4$:
\begin{equation*}
\mathrm{K}^4_\rho := \Big\{ (r,\epsilon^{11},\epsilon^{22},\epsilon^{12}) \in \R^4 \, : \, \rho\Big( \chi\bigl(\epsilon^{11}, \epsilon^{22}, \epsilon^{12} \bigr) \Big) \leq r \Big\}.
\end{equation*}
Then, the dual cone $(\mathrm{K}^4_\rho)^*$ can be expressed by means of the polar function $\dro$. Let us recall that, for any cone $\mathrm{K} \subset \R^k$, by its \textit{dual cone} we understand the closed convex cone $\mathrm{K}^* := \big\{ z^*\in \R^k \, : \, \pairing{z,z^*} \geq 0 \ \ \  \forall\,z \in \mathrm{K} \big\}$.
\begin{proposition}
	\label{prop:cone_rho_dual}
	The cone dual to $\mathrm{K}^4_\rho$ enjoys the following characterization:
	\begin{equation*}
	\bigl(r^0,-\sigma^{11},-\sigma^{22},-\sigma^{12}\bigr) \in \bigl(\mathrm{K}^4_\rho\bigr)^*  \qquad \Leftrightarrow \qquad \bigl(r^0,\sigma^{11},\sigma^{22},\sigma^{12} \bigr)\in \mathrm{K}^4_\dro
	\end{equation*}
	where 
	\begin{equation*}
	\mathrm{K}^4_\dro = \Big\{ (r^0,\sigma^{11},\sigma^{22},\sigma^{12}) \in \R^4 \, : \, \dro\Big( \chi\bigl(\sigma^{11}, \sigma^{22}, \sigma^{12} \bigr) \Big) \leq r^0 \Big\}.
	\end{equation*}
\end{proposition}
\begin{proof}
	Since $\chi$ is a linear isometry, cf. \eqref{eq:isometry_chi}, we observe that
	\begin{equation*}
	\pairing{\,(r,\epsilon^{11},\epsilon^{22},\epsilon^{12}) \, , \, (r^0,-\sigma^{11},-\sigma^{22},-\sigma^{12}) \, }_{\R^4} = r\,r^0 - \pairing{\xi,\sig}_{\mathcal{S}^{2 \times 2}}
	\end{equation*}
	where $\xi = \chi_2\bigl(\epsilon^{11}, \epsilon^{22}, \epsilon^{12} \bigr)$ and $\sig = \chi_2\bigl(\sigma^{11}, \sigma^{22}, \sigma^{12} \bigr)$. The problem thus reduces to determining necessary and sufficient condition for $r^0,\sig$ for which $\pairing{\xi,\sig}_{\mathcal{S}^{2 \times 2}} \leq r \, r^0$ for any $r,\xi$ satisfying $\rho(\xi) \leq r$. This condition precisely reads $\dro(\sig) \leq r^0$, which furnishes the claim.
\end{proof}
Next, we define a cone in  $\R^6$:
\begin{equation*}
	\mathrm{K}_+^6 := \Big\{ (\zeta^{11}, \zeta^{22},\zeta^{33},\zeta^{12},\zeta^{13},\zeta^{23}) \in \R^6 \ : \ \chi_3(\zeta^{11}, \zeta^{22},\zeta^{33},\zeta^{12},\zeta^{13},\zeta^{23}) \in \mathcal{S}^{3 \times 3}_+  \Big\}.
\end{equation*}
It is well established that $\mathcal{S}^{3 \times 3}_+$ is a self-dual convex closed cone as a subset of $\mathcal{S}^{3 \times 3}_+$. Since $\chi_3$ is a linear isometry, the same holds for $ \mathrm{K}_+^6$, namely
\begin{equation*}
	\big(\mathrm{K}_+^6\big)^* = \mathrm{K}_+^6.
\end{equation*}

Classically, a conic programming problem involves a linear objective function, linear equality constraints, and conic constraints. When reformulating \ref{eq:PMh} to a conic program one of the obstacles to overcome lies in the quadratic term with respect to $\theta$ in the constraint $\rho_+(\frac{1}{2}\theta\otimes\theta +\xi) \leq 1$ or in its algebraic counterpart in \eqref{eq:Kroh_char}. By employing some extra variables, the following result resolves this issue and, along the way, unlocks the use of the cones $\mathrm{K}^4_\rho$ and $\mathrm{K}_+^6$.
\begin{lemma}
	\label{lem:quadruple_instead_of_pair}
	For any matrix $\xi \in \Sdd$ and any vector $\theta \in \Rd$ the following conditions are equivalent:
	\begin{enumerate}[label={(\roman*)}]
		\item $\rho_+(\frac{1}{2} \, \theta \otimes \theta + \xi) \leq 1$,
		\item $\vro(\xi,\theta) \leq 1$,
		\item there exist matrices $\epsilon \in \Sdd$ and $\zeta \in \mathcal{S}^{3 \times 3}_+$ such that:
		\begin{equation}
		\label{eq:quadruple_instead_of_pair}
		\begin{bmatrix}
		\,\xi  &  \tfrac{1}{\sqrt{2}}\,\theta\,\\
		\tfrac{1}{\sqrt{2}}\,\theta^\top\, & 0\,
		\end{bmatrix} + \zeta = \begin{bmatrix}
		\,\epsilon  &  0_2\,\\
		\,0_2^\top\, & 1\,
		\end{bmatrix}  \qquad \text{and} \qquad	\rho(\epsilon) \leq 1,
		\end{equation}
		where  $0_2 \in\Rd$ is the $2 \times 1$ column vector of zeros.
	\end{enumerate}
\end{lemma}
\begin{remark}
	The two block matrices in \eqref{eq:quadruple_instead_of_pair} can be equivalently written  as, respectively, $\xi + \sqrt{2}\,\theta \symtens \e_3$ and $\epsilon + \e_3 \otimes\e_3$.
\end{remark}
\begin{proof}
	Equivalence between (ii) and (iii) follows from the construction of the gauge $\vro$, cf. Section \ref{ssec:rho}. By $\mathscr{C}$ let us denote the set of pairs $(\xi,\theta)$ satisfying (iii). It can be easily checked that $\mathscr{C}$ is closed convex and contains the origin. Therefore, the equivalence between (ii) and (iii) will be proved if we show that for every $\sig \in \Sddp$ and $q \in \Rd$
	\begin{equation}
	\label{eq:vro=sup}
	\vro^0(\sig,q) = \sup_{\xi,\theta}\Big\{ \pairing{\xi,\sig} + \pairing{\theta , q} \ : \ (\xi,\theta) \in \mathscr{C}  \Big\}.
	\end{equation}
	The right hand side can be rewritten as below ($\zeta_{33}$ stands for $\pairing{\zeta,\e_3 \otimes \e_3}$):
	\begin{align*}
	&\sup_{\epsilon, \zeta} \left\{\bigg\langle\left( \begin{bmatrix}
	\,\epsilon  &  0_2\,\\
	\,0_2^\top\, & 1\,
	\end{bmatrix} -\zeta \right) , \begin{bmatrix}
	\,\sig  &  \tfrac{1}{\sqrt{2}}\,q\,\\
	\tfrac{1}{\sqrt{2}}\,q^\top\, & 0\,
	\end{bmatrix} \bigg\rangle \ : \ \rho(\epsilon) \leq 1, \ \ \zeta \in \mathcal{S}^{3 \times 3}_+, \ \  \zeta_{33} = 1 \right\}\\
	= & \sup_\epsilon\Big\{ \pairing{\epsilon,\sig} \ : \ \rho(\epsilon) \leq 1  \Big\} + \sup_\zeta \left\{ - \bigg\langle\zeta , \begin{bmatrix}
	\,\sig  &  \tfrac{1}{\sqrt{2}}\,q\,\\
	\tfrac{1}{\sqrt{2}}\,q^\top\, & 0\,
	\end{bmatrix}\bigg\rangle  \ : \  \zeta \in \mathcal{S}^{3 \times 3}_+, \ \  \zeta_{33} = 1  \right\}.
	\end{align*}
	Since $\vro^0(\sigma,q) = \infty$ whenever $\sigma \notin \Sddp$ or $q \notin \mathrm{Im}\,\sigma$, let us show that the second supremum above equals $+\infty$ in this case. We choose $\zeta = (t\,\theta + \e_3) \otimes (t\,\theta + \e_3)$ for $t\geq 0$ and $\theta \in \Rd$ with $\pairing{\theta,q}\leq 0$ and $\pairing{\sigma,\theta\otimes\theta}  \leq 0$. Provided that $\sigma \notin \Sddp$ or $q \notin \mathrm{Im}\,\sigma$, the vector $\theta$ can be found such that at least of the two inequalities is strict. Then we send $t$ to $+\infty$ to show the claim.
	
	Once $\sigma \in \Sddp$ and $q \in \mathrm{Im}\,\sigma$ we have $\jsq <\infty$, and the right hand side of \eqref{eq:vro=sup} can be further rewritten:
	\begin{align*}
	 \dro(\sig) + \sup_\zeta \left\{ - \bigg\langle\zeta , \begin{bmatrix}
	\,\sig  &  \tfrac{1}{\sqrt{2}}\,q\,\\
	\tfrac{1}{\sqrt{2}}\,q^\top\, & \jsq\,
	\end{bmatrix}\bigg\rangle + \zeta_{33}\, \jsq \ : \  \zeta \in \mathcal{S}^{3 \times 3}_+, \ \  \zeta_{33} = 1  \right\}.
	\end{align*}
	By the Schur complement lemma (see Proposition \ref{prop:pos_semdef_block_mat} in the appendix) the block matrix belongs to $\mathcal{S}^{3\times 3}_+$; hence, the supremum above cannot be greater than $\jsq$. One can check that this value is attained for $\zeta = (-\theta/\sqrt{2} + \e_3) \otimes (-\theta/\sqrt{2} + \e_3)$ for any $\theta \in \Rd$ such that $q = \sigma \theta$ (in this case $\zeta$ yields the zero scalar product above). Ultimately, the right hand side of \eqref{eq:vro=sup} equals $\dro(\sig)+\jsq$ which is $\vro^0(\sig,q)$ by Proposition \ref{prop:varrho_polar}. The proof is complete.
\end{proof}

Finally, by collecting the results \eqref{eq:nodal_load_vector}, \eqref{eq:Kroh_char}, and Lemma \ref{lem:quadruple_instead_of_pair}, we are in a position to pose the algebraic, conic programming formulation of the finite element primal formulation \ref{eq:PMh}:
\begin{equation*}\tag*{$(\mbf{P}_h)$}
\label{eq:PMhalg}
\sup_{\substack{\mbf{u}_1,\, \mbf{u}_2, \, \mbf{w} \in \R^m \\
		\bm{\epsilon}^{11},\, \bm{\epsilon}^{22},\, \bm{\epsilon}^{12}  \in \R^n, \\
		\bm{\zeta}^{11},\, \bm{\zeta}^{22},\, \bm{\zeta}^{33}\in \R^n\\ \bm{\zeta}^{12},\,\bm{\zeta}^{13},\,\bm{\zeta}^{23}  \in \R^n,\\
		\mbf{r}\in \R^n}} 
\left\{ \pairing{\mbf{w}, \mbf{f}}  \, : 
\begin{array}{r}
\mbf{r} = \mbf{1}\\
\mbf{B}^{11} \mbf{u}_1  + \bm{\zeta}^{11} - \bm{\epsilon}^{11}= \mbf{0}\\
\mbf{B}^{22} \mbf{u}_2  + \bm{\zeta}^{22}- \bm{\epsilon}^{22} = \mbf{0}\\
\bm{\zeta}^{33}  \qquad \ \ = \mbf{1}\\
\mbf{B}_1^{12} \mbf{u}_1 + \mbf{B}_2^{12} \mbf{u}_2  + \bm{\zeta}^{12}- \bm{\epsilon}^{12} = \mbf{0}\\
\mbf{D}^{1} \mbf{w} + \bm{\zeta}^{13} \qquad \ \ = \mbf{0}\\
\mbf{D}^{2} \mbf{w} +  \bm{\zeta}^{23}  \qquad \ \ = \mbf{0}\\
\bigl(\mbf{r}(i), \bm{\epsilon}^{11}(i), \bm{\epsilon}^{22}(i), \bm{\epsilon}^{12}(i) \bigr) \in \mathrm{K}^4_\rho \quad \ \ \forall\,i\\
\bigl(\bm{\zeta}^{11}(i), \bm{\zeta}^{22}(i), \bm{\zeta}^{33}(i), \bm{\zeta}^{12}(i), \bm{\zeta}^{13}(i), \bm{\zeta}^{23}(i)\bigr) \in \mathrm{K}^6_+ \quad  \ \forall\,i
\end{array}	
\right\}
\end{equation*}
where $\mbf{1} \in \R^n$ stands for the vector whose all components are equal to $1$. Let us remark that the lack of the square root $\sqrt{2}$ (as opposed to \eqref{eq:quadruple_instead_of_pair}) is due to the presence of $\sqrt{2}$ in \eqref{eq:isometry}.

Below we shall show that the dual of $(\mbf{P}_h)$  reads
\begin{equation*}\tag*{$(\mbf{P}^*_h)$}
\label{eq:dPMhalg}
\inf_{\substack{
		\bm{\tau}^{11},\, \bm{\tau}^{22},\, \bm{\tau}^{33}\in \R^n,\\ \bm{\tau}^{12},\,\bm{\tau}^{13},\,\bm{\tau}^{23}  \in \R^n, \\ 
		\mbf{r}^0\in \R^n}} 
\left\{ \pairing{\mbf{1}, \mbf{r}^0} + \pairing{\mbf{1}, \bm{\tau}^{33}}  \, : \!
\def\arraystretch{1.3}
\begin{array}{r}
(\mbf{B}^{11})^\top  \bm{\tau}^{11} +(\mbf{B}_1^{12})^\top  \bm{\tau}^{12} = \mbf{0}  \\
(\mbf{B}_2^{12})^\top  \bm{\tau}^{12} + (\mbf{B}^{22})^\top  \bm{\tau}^{22} = \mbf{0} \\
(\mbf{D}^{1})^\top  \bm{\tau}^{13} + (\mbf{D}^{2})^\top  \bm{\tau}^{23} = \mbf{f} \\
\bigl(\mbf{r}^0(i), \bm{\tau}^{11}(i), \bm{\tau}^{22}(i), \bm{\tau}^{12}(i) \bigr) \in \mathrm{K}^4_\dro \quad \ \ \forall\,i\\
\bigl(\bm{\tau}^{11}(i), \bm{\tau}^{22}(i), \bm{\tau}^{33}(i), \bm{\tau}^{12}(i), \bm{\tau}^{13}(i), \bm{\tau}^{23}(i)\bigr) \in \mathrm{K}^6_+\ \, \quad  \ \forall\,i
\end{array}	
\right\} \quad 
\end{equation*}
The connection between problems \ref{eq:dPMhalg} and \ref{eq:dPMh} is less apparent and will be explained by means of the two results to come. The proof of the first one is moved to Appendix \ref{app:algebraic}.

\begin{theorem}
	\label{thm:duality_for_FEM}
	The following zero duality gap holds true:
	\begin{equation*}
	Z_h = \max \mbf{P}_h = \min \mbf{P}_h^* < \infty;
	\end{equation*}
	in particular, there exist solutions of both problems.
	
	Moreover, the variables $(\mbf{u}_1,\mbf{u}_2,\mbf{w})$ and $(\bm{\tau}^{11},\, \bm{\tau}^{22},\, \bm{\tau}^{33},\bm{\tau}^{12},\,\bm{\tau}^{13},\,\bm{\tau}^{23},\mbf{r}^0)$ that are admissible for problems  $(\mbf{P}_h)$ and $(\mbf{P}^*_h)$, respectively, are solutions of those problems if and only if for any $i \in \{1,\ldots,n\}$ there hold the optimality conditions:
	\begin{enumerate}[label={(\roman*)}]
		\item $\pairing{\tfrac{1}{2} \theta_i \otimes \theta_i + \xi_i,\sig_i} = \dro(\sig_i)$,
		\item $q_i = \sigma_i \theta_i$,
		\item $\mbf{r}^0(i) = \dro(\sig_i)$ \quad and \quad $\bm{\tau}^{33}(i) = \frac{1}{2}\pairing{\sigma_i^{-1} q_i,q_i}$,
	\end{enumerate}
	where $\xi_i,\sig_i \in \Sdd$ and $\theta_i,q_i \in \Rd$ are defined as follows:
	\begin{alignat*}{2}
		\xi_i &= \chi_2 \Bigl( \big(\mbf{B}^{11} \mbf{u}_1\big)\!(i),\big(\mbf{B}^{22} \mbf{u}_2\big)\!(i), \big(\mbf{B}^{12}_1 \mbf{u}_1\big)\!(i)+\big(\mbf{B}^{12}_2 \mbf{u}_2\big)\!(i) \Bigr), \qquad &&\theta_i = \big(\mbf{D}^{1} \mbf{w}\big)\!(i) \,\e_1 + \big(\mbf{D}^{2} \mbf{w}\big)\!(i) \,\e_2\\
		\sigma_i &= \chi_2\Big(\bm{\tau}^{11}(i), \bm{\tau}^{22}(i), \bm{\tau}^{12}(i) \Big), \qquad &&q_i =  \bm{\tau}^{13}(i) \,\e_1 + \bm{\tau}^{23}(i) \,\e_2.
	\end{alignat*}
\end{theorem}

Based on solutions of the algebraic conic programs \ref{eq:PMhalg},\,\ref{eq:dPMhalg} we can recast solutions of the problems \ref{eq:PMh},\,\ref{eq:dPMh} via explicit formulas:
\begin{corollary}
	\label{cor:Ph_Ph*_finite_and_infinite}
	Let $(\mbf{u}_1,\mbf{u}_2, \mbf{w})$ and  $(\bm{\tau}^{11},\, \bm{\tau}^{22},\, \bm{\tau}^{33},\bm{\tau}^{12},\,\bm{\tau}^{13},\,\bm{\tau}^{23},\mbf{r}^0)$ be solutions of problems $(\mbf{P}_h)$ and $(\mbf{P}^*_h)$, respectively. We choose the unique functions $(\hat{u}_h,\hat{w}_h) \in \mathcal{V}^h$ and $(\hat{\sig}_h,\hat{q}_h) \in L^1(\Omega;\Sdd \times \Rd)$ satisfying
	\begin{alignat*}{2}
	\hat{u}_h &=\iota_\shortparallel^h(\mbf{u}_1 ,\mbf{u}_2), \qquad  && \hat{w}_h =\iota_\perp^h(\mbf{w}),\\
	\hat{\sig}_h &= \frac{1}{\vert{E_i^h}\vert}\, \chi_2 \bigl( \bm{\tau}^{11}(i),\bm{\tau}^{22}(i),\bm{\tau}^{12}(i)\bigr), \qquad && \hat{q}_h = \frac{1}{\vert{E_i^h}\vert}\, \big(\bm{\tau}^{13}(i) \,\e_1 + \bm{\tau}^{23}(i) \,\e_2 \big) \quad \text{a.e. in } E_i^h.
	\end{alignat*}
	Then, the pairs of functions $(\hat{u}_h,\hat{w}_h)$ and $(\hat{\sig}_h,\hat{q}_h)$ solve problems $(\mathcal{P}_h)$ and $(\mathcal{P}^*_h)$, respectively, while $\min \mathcal{P}^*_h =Z_h$.
\end{corollary}
\begin{proof}
	The problem $(\mbf{P}_h)$ is equivalent to the already finite dimensional problem $(\mathcal{P}_h)$; therefore, optimality of $(\hat{u}_h,\hat{w}_h)$ is clear. It is easy to check that the element-wise constant function $(\hat{\sig}_h,\hat{q}_h)$ and vectors $\big(\bm{\sigma}^{11}, \bm{\sigma}^{22}, \bm{\sigma}^{12},\mbf{q}^1,\mbf{q}^2 \big) = \big(\bm{\tau}^{11}, \bm{\tau}^{22},\bm{\tau}^{12},\bm{\tau}^{13},\bm{\tau}^{23}\big)$ satisfy relations \eqref{eq:mbsigma}, and, thus, $(\hat{\sig}_h,\hat{q}_h) \in \Sigma^h(f)$ thanks to \eqref{eq:algebraic_EqEq_FMD}. By utilizing the optimality condition (iii) in Theorem \ref{thm:duality_for_FEM} we obtain
	\begin{align*}
	\int_\Omega \Big(\dro(\hat{\sig}_h) + \tfrac{1}{2} \,\pairing{\hat\sigma_h^{-1} \hat{q}_h,\hat{q}_h} \Big) d\mathcal{L}^2  &= \sum_{i =1}^n \frac{1}{\vert{E_i^h}\vert} \int_{E_i^h} \Big(\dro(\sig_i) + \tfrac{1}{2} \,\pairing{\sigma_i^{-1} q_i,q_i} \Big)\,d\mathcal{L}^2 \\
	&= \sum_{i =1}^n \Big(\mbf{r}^0(i) + \bm{\tau}^{33}(i) \Big) = \pairing{\mbf{1}, \mbf{r}^0} + \pairing{\mbf{1}, \bm{\tau}^{33}} = Z_h,
	\end{align*}
	rendering $(\hat\sig_h,\hat{q}_h)$ optimal for $(\mathcal{P}^*_h)$ since $Z_h = \max \mathcal{P}_h \leq \inf \mathcal{P}^*_h$.
\end{proof}

\bigskip

\subsection{Convergence of the finite element method}

In order to examine the convergence of the proposed finite element approximation we will assume that we work with regular triangulations $\mathcal{T}^h$ of a polygonal domain $\O$. Since the elements are triangular, such regularity condition may be referred to the triangles' angles, see \cite{ciarlet2002}. We shall say that a triangulation $\mathcal{T}^h$ is \textit{$\alpha_0$-regular} for $\alpha_0 \in (0,\pi/3]$ if all the angles in each triangular element $E_i^h$ are non-smaller than $\alpha_0$.

Let us introduce the interpolation operator $P_h:C_0(\O;\Rd\times \R)\to\mathcal{V}^h$ being the linear mapping satisfying
\begin{equation}
	\label{eq:projection}
	(u_h,w_h) = P_h (\tilde{u},\tilde{w}) \qquad \Leftrightarrow \qquad 	(u_h,w_h) \in \mathcal{V}^h, \quad 	(u_h,w_h)|_{X^h} = (\tilde{u},\tilde{w})|_{X^h}
\end{equation}
where $X^h$ is the set of those vertices of $\mathcal{T}^h$ which do not belong to $\bO$. For any $\alpha_0$-regular triangulation $\mathcal{T}^h$ functions $P_h (\tilde{u},\tilde{w})$ satisfy the standard interpolation estimates; for the $L^\infty$ norm those estimates are given in Lemma \ref{lem:interpolation} in the appendix.

\begin{theorem}
	\label{thm:convergence_P_P*}
	Assume a polygonal bounded domain $\Omega$ that is star-shaped with respect to a ball and a sequence of $\alpha_0$-regular triangulations $\mathcal{T}^h$ such that $h \to 0$. Next, let $(\hat{u}_h,\hat{w}_h) \in \mathcal{K}^h_\rho$ and $(\hat{\sig}_h,\hat{q}_h) \in L^\infty(\Omega;\Sdd \times \Rd)$ be functions in accordance with Theorem \ref{cor:Ph_Ph*_finite_and_infinite}. Let us put $(\hat{\TAU}_h,\hat{\vartheta}_h) := (\hat{\sigma}_h,\hat{q}_h) \, \mathcal{L}^2 \mres \Omega$.
	
	Then, there exist solutions $(\hat{u},\hat{w}) \in \Kro$ and $(\hat{\TAU},\hat{\vartheta}) \in \Mes(\Ob;\Sdd\times\Rd)$ of $\relProb$ and $\dProb$, respectively, and a subsequence $h(k)$ such that
	\begin{align}
	\label{eq:conv_u}
	\hat{u}_{h(k)} \to&\ \hat{u}   \ \ \text{ strongly in } L^1(\O;\Rd),\\
	\label{eq:conv_w}
	\hat{w}_{h(k)} \to&\ \hat{w}   \ \ \text{ uniformly in } \Ob,\\
	\label{eq:conv_sig}
	\hat{\TAU}_{h(k)}\ \weakstar & \ \hat{\TAU}  \ \ \text{ in } \Mes(\Ob;\Sdd),\\
	\label{eq:conv_q}
	\hat{\vartheta}_{h(k)}\ \weakstar & \ \hat{\vartheta}  \ \ \text{ in } \Mes(\Ob;\Rd).
	\end{align}
	Moreover, we have
	\begin{equation}
	\label{eq:Zh_to_Z}
	Z_h \to Z,
	\end{equation}
	where $Z_h = \max \mbf{P}_h = \min \mbf{P}_h^*$, and $Z = \max \ov{\mathcal{P}} = \min \mathcal{P}^*$.
\end{theorem}
\begin{proof}
	First, we shall prove the convergence of the displacement functions $(\hat{u}_h,\hat{w}_h)$. Since $Z_h = \int_\Ob \hat{w}_h \,df$, it is enough to show \eqref{eq:Zh_to_Z}. Then, since $\mathcal{K}^h_\rho \subset \Kro$, $(\hat{u}_h,\hat{w}_h)$ is a maximizing sequence for $\relProb$, and existence of a subsequence $h(k)$ such that \eqref{eq:conv_u}, \eqref{eq:conv_w} hold true follows by the compactness of $\Kro$ as a subset of $L^1(\Omega;\Rd) \times C(\Ob;\R)$.
	
	As observed in \eqref{eq:Z_h_leq_Z}, there holds $Z_h \leq Z$ for any $h$; therefore to prove that $Z_h \to Z$ it is enough that for any $\eps >0$ we point to $h$ and $(u_h,w_h) \in \mathcal{K}^h_\rho$ such that
	\begin{equation}
	\label{eq:Z_h>=Z}
	\int_\Ob w_h\, df \geq Z -\eps.
	\end{equation}
	Let us choose any solution $(\hat{u}_0,\hat{w}_0) \in \Kro$ of the problem $\relProb$. Of course, we have $Z = \int_\Ob \hat{w}_0 \,df$; thus, to prove \eqref{eq:Z_h>=Z} it is enough to construct $(u_h,w_h)\in \mathcal{K}^h_\rho$ such that
	\begin{equation*}
	\norm{w_h - \hat{w}_0}_{L^\infty(\O)} \leq \eps_1 := \frac{\eps}{\abs{f}(\Ob)}.
	\end{equation*}
	By the very definition of the set $\Kro$ we may choose smooth functions $(\tilde{u}_0,\tilde{w}_0) \in \Krop \subset \D(\Omega;\Rd\times\R)$ satisfying
	\begin{equation*}
	\norm{\tilde{w}_0 - \hat{w}_0}_{L^\infty(\O)} \leq \frac{\eps_1}{3}
	\end{equation*}
	Note that $\vro\big(e(\tilde{u}_0),\nabla \tilde{w}_0 \big) \leq 1$ everywhere in $\Ob$. Let us put $t = 1 -\frac{\eps_1}{3 \norm{\tilde{w}_0}_{L^\infty(\O)}}$, whilst we assume that $\eps$ is small enough so that $t \in (0,1)$. We define
	\begin{equation*}
		(	\tilde{u} ,	\tilde{w} )= t \,\big(\tilde{u}_0,\tilde{w}_0\big).
	\end{equation*}
	Next, for every triangulation $\mathcal{T}^h$ we choose  
	$(u_h,w_h) = P_h (\tilde{u},\tilde{w}) \in \V^h(\O)$. Then, Lemma \ref{lem:interpolation} guarantees that the inequalities $\norm{w_h-\tilde{w}}_{L^\infty(\O)} \leq C_2 h^2$ and $\norm{\vro\big(e(u_h)-e(\tilde{u}),\nabla w_h-\nabla \tilde{w}\big)}_{L^\infty(\O)} \leq C_3 h$ hold true, where $C_2,C_3$ are computed for $\alpha_0$ and $(\tilde{u},\tilde{w})$, and they do not depend on $h$. By using the triangle inequality we obtain:
	\begin{align*}
	\norm{\vro\big(e(u_h),\nabla w_h\big)}_{L^\infty(\O)} &\leq \norm{\vro\big(e(\tilde{u}),\nabla \tilde{w}\big)}_{L^\infty(\O)} + \norm{\vro\big(e(u_h)-e(\tilde{u}),\nabla w_h-\nabla \tilde{w}\big)}_{L^\infty(\O)}\\
	& \leq t + C_3 \,h = 1 + C_3\,h-\frac{\eps_1}{3 \norm{\tilde{w}_0}_{L^\infty(\O)}}.
	\end{align*}
	Accordingly, we have $(u_h,w_h) \in \mathcal{K}^h_\rho$ provided that $h \leq \frac{\eps_1}{3\, C_3 \norm{\tilde{w}_0}_{L^\infty(\O)}}$. Moreover, there holds
	\begin{align*}
	\norm{w_h - \hat{w}_0}_{L^\infty(\O)} &\leq \norm{w_h - \tilde{w}}_{L^\infty(\O)} + \norm{\tilde{w} - \tilde{w}_0}_{L^\infty(\O)} + \norm{\tilde{w}_0 - \hat{w}_0}_{L^\infty(\O)}\\
	& \leq C_2 \, h^2 + (1-t) \norm{\tilde{w}_0}_{L^\infty(\O)} + \tfrac{\eps_1}{3} = C_2 h^2 +\tfrac{2 \eps_1}{3}.
	\end{align*}
	Ultimately, we find that $(u_h,w_h)$ is the element of $\mathcal{K}^h_\rho$ that meets the condition \eqref{eq:Z_h>=Z} provided we choose a triangulation $\mathcal{T}^h$ for which
	\begin{equation*}
	h \leq \min \left\{ \frac{\eps_1}{3\, C_3 \norm{\tilde{w}_0}_{L^\infty(\O)}} \, , \, \sqrt{\frac{\eps_1}{3\,C_2}}  \right\},
	\end{equation*}
	thus concluding the first part of the proof, from which we obtain \eqref{eq:conv_u}, \eqref{eq:conv_w}, and \eqref{eq:Zh_to_Z}.
	
	For the second step we shall prove the convergences \eqref{eq:conv_sig},\,\eqref{eq:conv_q}. By Corollary \ref{cor:Ph_Ph*_finite_and_infinite} we have $Z_h = \int_\Ob \vro^0(\hat{\TAU}_h,\hat{\vartheta}_h)$. Therefore, since $Z_h \leq Z$ and by the fact that $\vro^0(\sig,q) \geq C \,\big(\abs{\sig}+\abs{q}\big)$ for some $C>0$ owing to statement (iii) in Proposition \ref{prop:varrho_polar}, the sequence $(\hat{\TAU}_h,\hat{\vartheta}_h)$ is bounded in $\Mes(\Ob;\Sdd\times\Rd)$. Hence, there exist measures $(\hat{\TAU},\hat{\vartheta})\in \Mes(\Ob;\Sdd\times\Rd)$ such that (up to extracting a subsequence) $(\hat{\TAU}_h,\hat{\vartheta}_h) \ \weakstar \ (\hat{\TAU},\hat{\vartheta})$. Consequently, by the lower semi-continuity of $\vro^0$ we have  
	\begin{equation}
	\label{eq:optimality_TAU_hat}
	\int_\Ob \vro^0(\hat{\TAU},\hat{\vartheta}) \leq \liminf_{h \to 0} \int_\Ob \vro^0(\hat{\TAU}_h,\hat{\vartheta}_h) = \liminf_{h \to 0} Z_h = Z.
	\end{equation}
	We must next show that the pair $(\hat{\TAU},\hat{\vartheta})$ is admissible for $\dProb$, i.e. that $-\DIV \, \hat{\TAU} = 0$ and $-\dive \, \hat{\vartheta} = f$ understood as in \eqref{eq:div}, \eqref{eq:DIV}. Let us take smooth test functions $(\phi,\varphi) \in \D(\Rd;\Rd\times \R)$, and then for every triangulation $\mathcal{T}^h$ we choose functions $(\phi_h,\varphi_h) = P_h(\phi,\varphi)\in \V^h$ that again by virtue of Lemma \ref{lem:interpolation} satisfy $\norm{\varphi_h-{\varphi}}_{L^\infty(\O)} \leq C_2 h^2$ and $\norm{\vro\big(e(\phi_h)-e(\phi),\nabla \varphi_h-\nabla \varphi\big)}_{L^\infty(\O)} \leq C_3 h$ where $C_2,C_3$ are fixed for the chosen pair $(\phi,\varphi)$. The following chain of inequalities may be written down (all the integrals are computed on $\O$):
	\begin{align*}
	&\abs{\int \Big(\pairing{e(\phi),\hat{\TAU}}+\langle\nabla\varphi,\hat{\vartheta}\rangle\Big) - \int \varphi\,df } \\
	& \quad= \abs{\lim_{h \to 0} \int \Big(\pairing{e(\phi),\hat{\TAU}_h}+\langle\nabla\varphi,\hat{\vartheta}_h\rangle\Big) - \int \varphi\,df }\\
	& \quad= \bigg\vert\lim_{h \to 0}\bigg\{ \left(\int \Big(\pairing{e(\phi_h),\hat{\sig}_h}+\pairing{\nabla \varphi_h,\hat{q}_h}\Big) d\mathcal{L}^2 -  \int \varphi_h\,df \right) \\
	& \qquad \qquad \qquad \qquad \qquad + \int\Big(\pairing{e(\phi)-e(\phi_h),\hat{\sig}_h}+ \pairing{\nabla \varphi-\nabla \varphi_h,\hat{q}_h} \Big) d\mathcal{L}^2 +  \int (\varphi_h - \varphi)\,df  \bigg\} \bigg\vert\\
	& \quad = \abs{\lim_{h \to 0}\left\{  \int\Big(\pairing{e(\phi)-e(\phi_h),\hat{\sig}_h}+ \pairing{\nabla \varphi-\nabla \varphi_h,\hat{q}_h} \Big) d\mathcal{L}^2 +  \int (\varphi_h - \varphi)\,df  \right\}}\\
	& \quad \leq \limsup_{h \to 0} \abs{\int\Big(\pairing{e(\phi)-e(\phi_h),\hat{\sig}_h}+ \pairing{\nabla \varphi-\nabla \varphi_h,\hat{q}_h} \Big) d\mathcal{L}^2 } + \limsup_{h \to 0} \abs{\int (\varphi_h - \varphi)\,df }\\
	& \quad \leq \limsup_{h \to 0}  \biggl\{\norm{\vro\big(e(\phi_h)-e(\phi),\nabla \varphi_h-\nabla \varphi\big)}_{L^\infty(\O)} \int \vro^0(\hat{\sig}_h,\hat{q}_h)\, d\mathcal{L}^2\biggr\}  + \limsup_{h \to 0} \biggl\{ \norm{\varphi_h-\varphi}_{L^\infty(\O)} \,\abs{f}(\Ob) \biggr\}\\
	& \quad \leq \limsup_{h \to 0}  \Big\{C_3 \, h\, Z_h\Big\}  + \limsup_{h \to 0} \Big\{ C_2\,h^2\, \abs{f}(\Ob) \Big\} = 0.
	\end{align*}
	Above, the third equality is due to the fact that $(\hat\sig_h,\hat{q}_h) \in \Sigma^h(f)$ by virtue of \eqref{eq:algebraic_EqEq_FMD}. We have thus obtained that the pair $(\hat{\TAU},\hat{\vartheta})$ is indeed admissible for $\dProb$. Optimality of $(\hat{\TAU},\hat{\vartheta})$ for $\dProb$ now follows by \eqref{eq:optimality_TAU_hat}. The proof is complete.
\end{proof}

Of course, \ref{eq:relPM},\,\ref{eq:dPM} merely constitute an auxiliary pair of problems from the perspective of the original optimal membrane problem \ref{eq:OEM}. More precisely, based on the solution $\hat{\TAU}$ for \ref{eq:dPM}, one recasts an optimal material distribution $\check{\mu}$ via the formula \eqref{eq:link_OEM_dProb}. The next result shows that such passage can be done on the level of a finite element approximation:
\begin{corollary}
	\label{cor:mu_h}
	Assuming the prerequisites of Theorem \ref{thm:convergence_P_P*} define
	\begin{equation*}
		\check{\mu}_h := \frac{2 V_0}{Z_h} \,\rho^0(\hat{\TAU}_h) \quad \in \quad \Mes_+(\Ob).
	\end{equation*}
	Then, there exists a solution $\hat{\mu}$ of \ref{eq:OEM} and a subsequence $h(k)$ such that
	\begin{equation*}
		\check{\mu}_{h(k)} \ \weakstar \ \check{\mu}.
	\end{equation*}
\end{corollary}
\begin{proof}
	Taking $h(k)$ and $\hat{\TAU}$ as in the assertion of Theorem \ref{thm:convergence_P_P*} we define $\check{\mu} =  \frac{2 V_0}{Z} \,\rho^0(\hat{\TAU})$. Then, $\check{\mu}$ solves \ref{eq:OEM} by Theorem \ref{thm:constructing_lambda_OEM}, and, in particular, $\check{\mu}(\Ob) = V_0$. An equi-repartition of energy rule similar to \eqref{eq:equi-repartition} can be inferred for an approximate solution $(\hat{\sigma}_h,\hat{\vartheta}_h)$ as well, from which we deduce that $\check{\mu}_h(\Ob)  = V_0$. Consequently, there exists $\tilde{\mu} \in \Mes_+(\Ob)$ such that (possibly upon extracting a subsequence) $\check{\mu}_{h(k)} \, \weakstar \, \tilde{\mu}$, and $\tilde{\mu}(\Ob)  = V_0$. By combining \eqref{eq:conv_sig} and \eqref{eq:Zh_to_Z} we can show that $\frac{2 V_0}{Z_{h(k)}} \,\hat{\TAU}_{h(k)} \, \weakstar \frac{2 V_0}{Z} \, \hat{\TAU}$. Owing to \cite[Proposition 1.62(b)]{ambrosio2000} there must hold $\tilde\mu \geq \check{\mu}$; however, since $\tilde{\mu}(\Ob)  = \check{\mu}(\Ob)  = V_0$, there must be $\tilde{\mu} = \check{\mu}$, thus completing the proof.
\end{proof}
\begin{remark} From Theorem \ref{thm:convergence_P_P*} we can see that the proposed finite element method furnishes approximates solutions $(\hat{\TAU}_h,\hat{\vartheta}_h)$ for the problem \ref{eq:dPM} that are absolutely continuous (with respect to Lebesgue measure). As a result, so is the material distribution $\check{\mu}_h$ in Corollary \ref{cor:mu_h} that approximates a solution of \ref{eq:OEM}. More accurately, we find that
	\begin{equation}
		\label{eq:bh}
		\check{\mu}_h = \check{b}_h \, \mathcal{L}^2 \mres \O, \qquad \text{where} \qquad  \check{b}_h = \frac{1}{\vert{E_i^h}\vert}\, \frac{2 V_0}{Z_h}\, \mbf{r}^0(i) \quad \text{a.e. in } E_i^h,
	\end{equation}
and $\mbf{r}^0$ is one of the vectors solving \ref{eq:dPMhalg}. The function $\check{b}_h$ has the interpretation of a subopitmal thickness distribution of the membrane.
One should note, however, that the optimal material distribution $\check{\mu}$, to which $\check{\mu}_h$ converges, may not be absolute continuous in general.
\end{remark}

\section{Numerical simulations}
\label{sec:simulations}

The finite element method put forth in the last section was implemented in MATLAB. To tackle the pair of algebraic conic programming problems \ref{eq:PMhalg},\,\ref{eq:dPMhalg} the MOSEK \cite{mosek2019} toolbox was exploited. This toolbox is aimed at solving large scale convex programs and it involves sparse matrix algebra. It mainly uses various interior point methods \cite{andersen2003}.

Naturally, the MOSEK toolbox is capable of solving convex conic programming problems for the conic constraints of specific structure only. Currently, it enables the use of: quadratic, exponential, or power cones. In addition, another module is devoted to semi-definite programming, i.e. to engaging variables that are positive semi-definite symmetric matrices. We shall employ such $3 \times 3$ matrix variables in order to realize the conic constraint $\mathrm{K}^6_+$. The second constraint, involving the cones $\mathrm{K}_\rho^4$, depends on the choice of the energy potential $j$. In this paper, we confine the implementation of the programs \ref{eq:PMhalg},\,\ref{eq:dPMhalg} to the case of $j$ being quadratic, cf. Example \ref{ex:iso}. Below we will show that the constraint $\mathrm{K}_\rho^4$ may be then slightly rewritten so that the classical quadratic cone is used. This way, we can reformulate the pair \ref{eq:PMhalg},\,\ref{eq:dPMhalg} to a conic-quadratic \& semi-definite programming problem that can be then tackled by the MOSEK toolbox directly.

\subsection{Implementation in MATLAB as a conic-quadratic \& semi-definite programming problem}
\label{ssec:implementation}

As announced above, in the sequel we assume that $j$ is quadratic as in Example \ref{ex:iso}. As a result, we obtain the positive 1-homogeneous functions as below:
\begin{equation*}
	\rho(\xi) = \big(\pairing{\mathscr{H} \xi,\xi}\big)^{\frac{1}{2}}, \qquad \rho^0(\sig) = \big(\pairing{\mathscr{H}^{-1} \sig,\sig}\big)^{\frac{1}{2}},
\end{equation*}
where for definition of $\mathscr{H},\,\mathscr{H}^{-1}$ we refer to Example \ref{ex:iso}. If $(\epsilon^{11},\epsilon^{22},\epsilon^{12}) = \chi_2^{-1}(\xi)$ and $(\sigma^{11},\sigma^{22},\sigma^{12}) = \chi_2^{-1}(\sig)$ then relations $\sig = \mathscr{H} \xi$ is transformed to $[\sigma^{11} \ \sigma^{22} \ \sigma^{12}]^\top  = H\, [\epsilon^{11} \ \epsilon^{22} \ \epsilon^{12} ]^\top$ for the matrix $H \in \mathcal{S}^{3\times 3}_+$ and its inverse as below:
\begin{equation*}
H = E
\begin{bmatrix}
\frac{1}{1-\nu^2} & \frac{\nu}{1-\nu^2}  & 0\\
\frac{\nu}{1-\nu^2}  & \frac{1}{1-\nu^2}  & 0\\
0 & 0 & \frac{1}{1+\nu} 
\end{bmatrix}, \qquad \qquad H^{-1} = \frac{1}{E}
\begin{bmatrix}
1 & - \nu & 0\\
-\nu & 1 & 0\\
0 & 0 & 1+\nu
\end{bmatrix}.
\end{equation*}
Accordingly, the cones $\mathrm{K}_\rho^4$,  $\mathrm{K}_\dro^4$ defined in Section \ref{ssec:algebraic} admit the following form: 
\begin{equation*}
	\mathrm{K}_\rho^4 = \Big\{ (r,\epsilon^{11},\epsilon^{22},\epsilon^{12}) \in \R^4 \, : \, \Big(\pairing{H [\epsilon^{11} \ \epsilon^{22} \ \epsilon^{12} ]^\top , \,  [\epsilon^{11} \ \epsilon^{22} \ \epsilon^{12} ]^\top} \Big)^{\frac{1}{2}} \leq r \Big\},
\end{equation*}
\begin{equation*}
\mathrm{K}_\dro^4 = \Big\{ (r^0,\sigma^{11},\sigma^{22},\sigma^{12}) \in \R^4 \, : \, \Big(\pairing{H^{-1} [\sigma^{11} \ \sigma^{22} \ \sigma^{12}]^\top , \, [\sigma^{11} \ \sigma^{22} \ \sigma^{12}]^\top} \Big)^{\frac{1}{2}} \leq r^0 \Big\}.
\end{equation*}
The above formulas pave a way to rewriting the conic constraints involving $\mathrm{K}_\rho^4$,  $\mathrm{K}_\dro^4$ as classical conic quadratic constraints. Let us recall that by the $d$-dimensional \textit{quadratic cone} we understand
\begin{equation}
\label{eq:quadratic_cone}
\mathrm{K}^d_{\mathrm{quad}} := \biggl\{ \mbf{z} \in \R^d \, : \,  \Big(\bigl(\mbf{z}(2)\bigr)^2 + \ldots + \bigl(\mbf{z}(d)\bigr)^2 \Big)^{\frac{1}{2}} \leq  \mbf{z}(1) \biggr\}.
\end{equation}
It is well established that the quadratic cone is self-dual, i.e. $(\mathrm{K}^d_{\mathrm{quad}})^* = \mathrm{K}^d_{\mathrm{quad}}$.  We make the Cholesky decomposition of the matrix $H$:
\begin{equation*}
	H= LL^\top, \qquad \text{where} \quad L=
	\sqrt{E} \begin{bmatrix}
	\frac{1}{\sqrt{1-\nu^2}} & 0  & 0\\
	\frac{\nu}{\sqrt{1-\nu^2}}  & 1  & 0\\
	0 & 0 & \frac{1}{\sqrt{1+\nu}} 
	\end{bmatrix}.
\end{equation*}
As a result, we have
\begin{equation*}
H^{-1}= M M^\top, \qquad \text{where} \quad M=(L^{-1})^\top=
\frac1{\sqrt{E}} \begin{bmatrix}
\sqrt{1-\nu^2} & -\nu  & 0\\
0  & 1  & 0\\
0 & 0 & \sqrt{1+\nu} 
\end{bmatrix}.
\end{equation*}
Readily, we can rewrite the conic constraints $\mathrm{K}_\rho^4$,  $\mathrm{K}_\dro^4$ via the conic quadratic constraint $\mathrm{K}^d_{\mathrm{quad}}$ and a linear transformation:
\begin{alignat*}{2}
	 (r,\epsilon^{11},\epsilon^{22},\epsilon^{12}) \in \mathrm{K}_\rho^4 \qquad &\Leftrightarrow \qquad \ov{L}^\top\![r \ \epsilon^{11} \ \epsilon^{22} \ \epsilon^{12}]^\top \in \mathrm{K}_\mathrm{quad}^4,  \qquad &&\text{where} \quad
	\ov{L} = \begin{bmatrix}
	1 & 0_3 \\
	0_3^\top & L
	\end{bmatrix},\\
	(r^0,\sigma^{11},\sigma^{22},\sigma^{12}) \in \mathrm{K}_\dro^4 \qquad &\Leftrightarrow \qquad \ov{M}^\top\! [r \ \sigma^{11} \ \sigma^{22} \ \sigma^{12}]^\top \in \mathrm{K}_\mathrm{quad}^4,  \qquad &&\text{where} \quad
	\ov{M} = \begin{bmatrix}
	1 & 0_3 \\
	0_3^\top & M
	\end{bmatrix}.
\end{alignat*}

Finally, the we can pose the problems \ref{eq:PMhalg},\,\ref{eq:dPMhalg} as a conic quadratic \& semi-definite programming problem, i.e. a convex program with linear objective, linear equality constraints, conic quadratic constraints on the variables $\ov{L}^\top\!\bigl(\mbf{r}(i), \bm{\epsilon}^{11}(i), \bm{\epsilon}^{22}(i), \bm{\epsilon}^{12}(i) \bigr)$, \ $\ov{M}^\top\!\bigl(\mbf{r}^0(i), \bm{\tau}^{11}(i), \bm{\tau}^{22}(i), \bm{\tau}^{12}(i) \bigr)$, and semi-definite variables $\zeta_i = \chi_3\bigl(\bm{\zeta}^{11}\!(i), \bm{\zeta}^{22}\!(i), \bm{\zeta}^{33}\!(i), \bm{\zeta}^{12}\!(i), \bm{\zeta}^{13}\!(i), \bm{\zeta}^{23}\!(i)\bigr)\! \in\! \mathcal{S}_+^{3\times 3}$ and $\tau_i = \chi_3\bigl(\bm{\tau}^{11}\!(i), \bm{\tau}^{22}\!(i), \bm{\tau}^{33}\!(i) , \bm{\tau}^{12}\!(i), \bm{\tau}^{13}\!(i), \bm{\tau}^{23}\!(i)\bigr)\! \in\! \mathcal{S}_+^{3\times 3}$.

\begin{remark}
	The MOSEK toolbox requires only one of the formulations \ref{eq:PMhalg},\,\ref{eq:dPMhalg} as an input, although it automatically returns solutions of both. The present author's experience is, however, that the solver is more stable once the dual \ref{eq:dPMhalg} is implemented, and this was indeed the choice for all the examples solved in this paper.
\end{remark}

\begin{remark}
	The matrices $L$ and $M$ are well defined for any Poisson ratio $\nu$ in the range $(-1,1)$. However, we have already learnt in Remark \ref{rem:Michell} that the optimal membrane problem remains meaningful for $\nu=-1$. There are several ways of incorporating this extreme case into our MATLAB code.
	
	First of all, since only \ref{eq:dPMhalg} is being programmed -- and thus only the matrix $M$ enters the code explicitly -- there is no issue of putting $\nu=-1$ directly. Then again, for such data the numerical experiments showed instabilities of the MOSEK solver, which most likely comes from the singularity of $M$.
	
	Another possibility is to take advantage of \eqref{eq:lim_nu_-1} and consider formulations \ref{eq:PMhalg},\,\ref{eq:dPMhalg} with the cones $\mathrm{K}^4_{\rho^\mathrm{M}}$, $\mathrm{K}^4_{(\rho^\mathrm{M})^0}$ where $\rho^\mathrm{M}$ is the spectral norm on $\Sdd$. Rewriting such constraints as conic quadratic constraints is possible, but a more involved linear transformation of variables is required.
	
	With the previous comment in mind, by utilizing a very particular form of $\rho_+^\mathrm{M}$ we could propose an alternative algebraic formulation \ref{eq:PMhalg},\,\ref{eq:dPMhalg} without conic constraints and with one semi-definite variable $\zeta_i$ per finite element. This is due to the following equivalences:
	\begin{equation*}
		\rho_+^\mathrm{M}\Big(\tfrac{1}{2} \theta \otimes \theta +\xi\Big) \leq 1 \quad \Leftrightarrow \quad \tfrac{1}{2} \theta \otimes \theta +\xi \preceq \mathrm{I}_2 \quad \Leftrightarrow \quad \exists\,\zeta \in \mathcal{S}_+^{3\times 3} \ \ \text{such that} \ \
		\begin{bmatrix}
		\,\xi  &  \tfrac{1}{\sqrt{2}}\,\theta\,\\
		\tfrac{1}{\sqrt{2}}\,\theta^\top\, & 0\,
		\end{bmatrix} + \zeta = \mathrm{I}_3,
	\end{equation*}
	where $\mathrm{I}_d$ is a $d\times d$ identity matrix, and the inequality $\preceq$ is in the sense of the induced quadratic forms. The last condition above may replace condition (iii) in Lemma \ref{lem:quadruple_instead_of_pair}. Such a semi-definite formulation was implemented by the present author.
	
	Finally, the simplest but also the least elegant way is to take $\nu$ very close to $-1$. Experiments have shown that for $\nu = -0.9999$ the algorithm was stable, and the results were practically the same as for $\nu=-1$ when using one of the foregoing approaches. In the examples below we shall choose this crude fashion to obtain the results for $\nu=-1$. This way we can stay consistent with the simulations for other $\nu$, in terms of computational efficiency, for instance.
\end{remark}

\subsection{Examples for a square design domain}

In all of the examples below we will consider a quadratic domain $\Omega =(0,a) \times (0,a)$ of the edge length $a$. We shall use a triangulation whose vertices form a regular rectangular mesh: $\big\{ (a \tfrac{i_1}{n}, a \tfrac{i_2}{n}) \, : \, i_1,i_2 \in \{0,1,\ldots,n\} \big\}$ for some natural $n$. Each square constituted by four adjacent vertices is divided into two triangular elements of one edge parallel to $\e_1$, one to $\e_2$, and one inclined at $ 45\degree$ angle. If the direction of this third edge is to be kept constant, e.g. pointing from south-east to north-west, then the triangulation is precisely: 
\begin{align}
	\label{eq:lower_upper}
	\mathcal{T}^h = \mathcal{T}^h_1 \cup \mathcal{T}^h_2 = &\, \bigg\{ \mathrm{int\,co}\big( \big\{ (a\tfrac{i_1}{n},a\tfrac{i_2}{n}), (a\tfrac{i_1+1}{n},a\tfrac{i_2}{n}),(a\tfrac{i_1}{n},a\tfrac{i_2+1}{n}) \big\} \big) \ : \ i_1, i_2 \in  \{0,1,\ldots,n-1\}  \bigg\}\\
	\nonumber
	& \cup   \bigg\{ \mathrm{int\,co}\big( \big\{ (a\tfrac{i_1+1}{n},a\tfrac{i_2+1}{n}), (a\tfrac{i_1+1}{n},a\tfrac{i_2}{n}),(a\tfrac{i_1}{n},a\tfrac{i_2+1}{n}) \big\} \big) \ : \ i_1, i_2 \in  \{0,1,\ldots,n-1\}  \bigg\},
\end{align}
where for $h$ we choose $\sqrt{2}\,a/n$; above $\mathrm{int\,co}$ stands for the interior of a convex hull. Such a triangulation consists of $2n^2$ triangular elements. 

In most of the examples tackled below the load $f$ will be symmetric, therefore we will take care of symmetry of the triangulation. Namely, depending on the quadrant of $\O$, the direction of the inclined edges changes from south-east -- north-west to south-west -- north-east. It is clear that each such triangulations are $\alpha_0$-regular for $\alpha_0= 45\degree$.

For each example we shall consider five different Poisson ratios $\nu$. In each case we will display the heat-map for the (sub)optimal thickness distribution $\check{b}_h$ computed in accordance with \eqref{eq:bh} (and modified in line with Remarks \ref{rem:checkerboard}, \ref{rem:trim}). The computations will be conducted for $n=800$, which gives a mesh counting $2n^2 = 1\,280\,000$ finite elements. The approximate value functions $Z_h$, as well as the computation times, are listed in Tab. \ref{tab:poisson}. In a separate Tab. \ref{tab:mesh} a comparison between meshes of three different sizes may be found (for the case $\nu=0$ only). The computations were performed on an Apple Macbook Pro M2 Pro laptop with 16GB RAM running Ventura 13.3.1 macOS; the version R2023a  of MATLAB was used (running via the  Rosetta 2 emulator) along with version 10.0 of the MOSEK toolbox.

\begin{table}[H]
	\scriptsize
	\centering
	\caption{Summary on numerical computations for the pair of problems \ref{eq:PMh},\,\ref{eq:dPMh}. Different Poisson ratios $\nu$.}
	\begin{tabular}{lcclcccc}
		\toprule
		Example & Mesh (vertices)  &  Mesh (element no.) & Poisson ratio    & CPU time & Objective value $\Z_h$\\
		\midrule[0.36mm]
		Ex. \ref{ex:3_point_loads}
		& $801\!\times\!801$ & $1\,280\,000$  & $ \ \nu= 0.5$ & $20$ min. $48$ sec. & $1.5479 \, P a/ \sqrt{E}$ \\
		(Fig. \ref{fig:3_point_loads}) &  &  & $ \ \nu= 0.0$    & $30$ min. $36$ sec. & $1.5765 \, P a / \sqrt{E}$\\
		&  &  & $ \ \nu=-0.6$    & $28$ min. $29$ sec. & $1.5680 \, P a/ \sqrt{E}$ \\
		&  &  & $ \ \nu=-0.9$    & $32$ min. $00$ sec. & $1.5343 \, P a/ \sqrt{E}$ \\
		&  &  & $ \ \nu=-1.0$    & $28$ min. $25$ sec. & $1.5148 \, P a/ \sqrt{E}$ \\
		\midrule[0.1mm]
		Ex. \ref{ex:multiple_point_loads}
		& $801\!\times\!801$ & $1\,280\,000$  & $ \ \nu= 0.5$ & $15$ min. $54$ sec. & $161.44 \, P a/ \sqrt{E}$ \\
		(Fig. \ref{fig:multiple_point_loads})&  &  & $ \ \nu= 0.3$    & $16$ min. $44$ sec. & $169.06 \, P a/ \sqrt{E}$\\
		&  &  & $ \ \nu=0.0$    & $16$ min. $39$ sec. & $175.85 \, P a/ \sqrt{E}$ \\
		&  &  & $ \ \nu=-0.6$    & $14$ min. $28$ sec. & $181.56 \, P a/ \sqrt{E}$ \\
		&  &  & $ \ \nu=-1.0$    & $22$ min. $51$ sec. & $179.48 \, P a/ \sqrt{E}$ \\
		\midrule[0.1mm]
		Ex. \ref{ex:diagonal_load}
		& $801\!\times\!801$ & $1\,280\,000$  & $ \ \nu= 0.0$ & $13$ min. $36$ sec. & $1.1091 \, t a^2/ \sqrt{E}$ \\
		(Fig. \ref{fig:diagonal_load})&  &  & $ \ \nu= -0.85$    & $14$ min. $22$ sec. & $1.1889 \, t a^2/ \sqrt{E}$\\
		&  &  & $ \ \nu=-0.92$    & $14$ min. $53$ sec. & $1.1850 \, t a^2/ \sqrt{E}$ \\
		&  &  & $ \ \nu=-0.95$    & $15$ min. $38$ sec. & $1.1807 \, t a^2/ \sqrt{E}$ \\
		&  &  & $ \ \nu=-1.0$    & $13$ min. $49$ sec. & $1.1713 \, t a^2/ \sqrt{E}$ \\
		\midrule[0.1mm]
		Ex. \ref{ex:pressure}
		& $801\!\times\!801$ & $1\,280\,000$  & $ \ \nu= 0.3$ & $12$ min. $44$ sec. & $0.26862 \, p a^3/ \sqrt{E}$ \\
		(Fig. \ref{fig:pressure})&  &  & $ \ \nu= 0.0$    & $13$ min. $00$ sec. & $0.28633 \, p a^3/ \sqrt{E}$\\
		&  &  & $ \ \nu=-0.6$    & $12$ min. $50$ sec. & $0.30630 \, p a^3/ \sqrt{E}$ \\
		&  &  & $ \ \nu=-0.95$    & $14$ min. $20$ sec. & $0.30899 \, p a^3/ \sqrt{E}$ \\
		&  &  & $ \ \nu=-1.0$    & $13$ min. $12$ sec. & $0.30823 \, p a^3/ \sqrt{E}$ \\
		\bottomrule
	\end{tabular}
	\label{tab:poisson}
\end{table}

\begin{table}[h]
	\scriptsize
	\centering
	\caption{Summary on numerical computations for the pair of problems \ref{eq:PMh},\,\ref{eq:dPMh}. Different finite element mesh sizes.}
	\begin{tabular}{lccccccc}
		\toprule
		Example & Mesh (vertices)  &  Mesh (element no.)  & Poisson ratio      & CPU time & Objective value $\Z_h$\\
		\midrule[0.36mm]
		Ex. \ref{ex:3_point_loads}
		& $101\!\times\!101$ & $20\,000$ & $ \ \nu= 0.0$  & $0$ min. $07$ sec. & $1.5676 \, P a/ \sqrt{E}$ \\
		(Fig. \ref{fig:3_point_loads_mesh})& $301\!\times\!301$ & $180\,000$ &  & $1$ min. $44$ sec. & $1.5738 \, P a/ \sqrt{E}$\\
		& $801\!\times\!801$ & $1\,280\,000$ &  & $30$ min. $36$ sec. & $1.5765 \, P a/ \sqrt{E}$ \\
		\midrule[0.1mm]
		Ex. \ref{ex:multiple_point_loads}
		& $101\!\times\!101$ & $20\,000$ & $ \ \nu= 0.0$  & $0$ min. $05$ sec. & $173.98 \, P a/ \sqrt{E}$ \\
		& $301\!\times\!301$ & $180\,000$ & & $1$ min. $01$ sec. & $175.33 \, P a/ \sqrt{E}$\\
		& $801\!\times\!801$ & $1\,280\,000$ &  & $16$ min. $39$ sec. & $175.85 \, P a/ \sqrt{E}$ \\
		\midrule[0.1mm]
		Ex. \ref{ex:diagonal_load}
		& $101\!\times\!101$ & $20\,000$ & $ \ \nu= 0.0$ & $0$ min. $04$ sec. & $1.1083 \, t a^2/ \sqrt{E}$ \\
		& $301\!\times\!301$ & $180\,000$ &  & $0$ min. $46$ sec. & $1.1089 \, t a^2/ \sqrt{E}$\\
		& $801\!\times\!801$ & $1\,280\,000$ &  & $13$ min. $36$ sec. & $1.1091 \, t a^2/ \sqrt{E}$ \\
		\midrule[0.1mm]
		Ex. \ref{ex:pressure}
		& $101\!\times\!101$ & $20\,000$ & $ \ \nu= 0.0$ & $0$ min. $04$ sec. & $0.28564 \, p a^3/ \sqrt{E}$ \\
		& $301\!\times\!301$ & $180\,000$ & & $0$ min. $48$ sec. & $0.28617 \, p a^3/ \sqrt{E}$\\
		& $801\!\times\!801$ & $1\,280\,000$ &  & $13$ min. $00$ sec. & $0.28633 \, p a^3/ \sqrt{E}$ \\
		\bottomrule
	\end{tabular}
	\label{tab:mesh}
\end{table}

\begin{figure}[h]
	\centering
	\subfloat[]{\includegraphics*[trim={0cm 0cm -0cm -0cm},clip,width=0.3\textwidth]{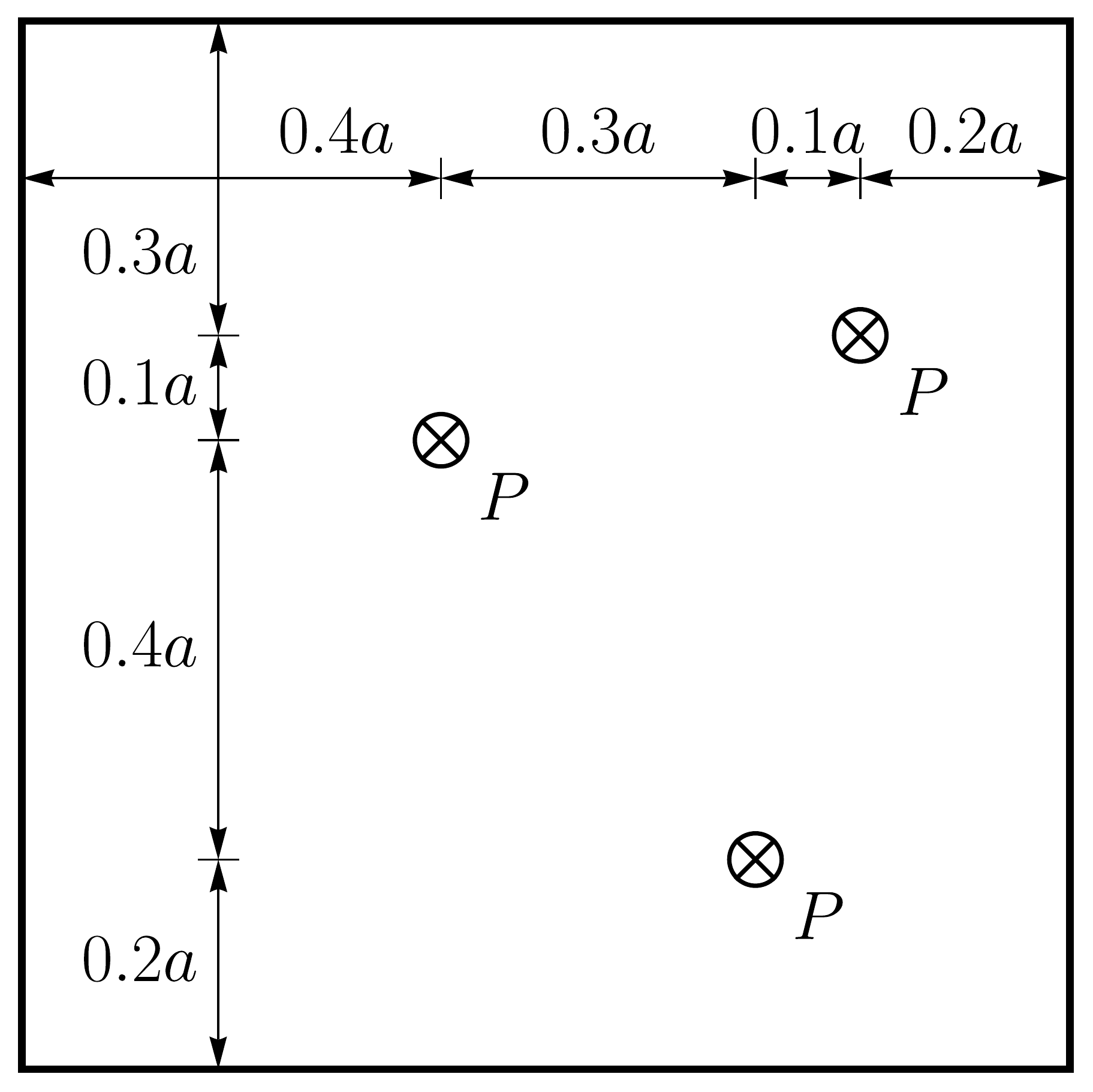}}\hspace{0.7cm}
	\subfloat[]{\includegraphics*[trim={0cm 0cm -0cm -0cm},clip,width=0.3\textwidth]{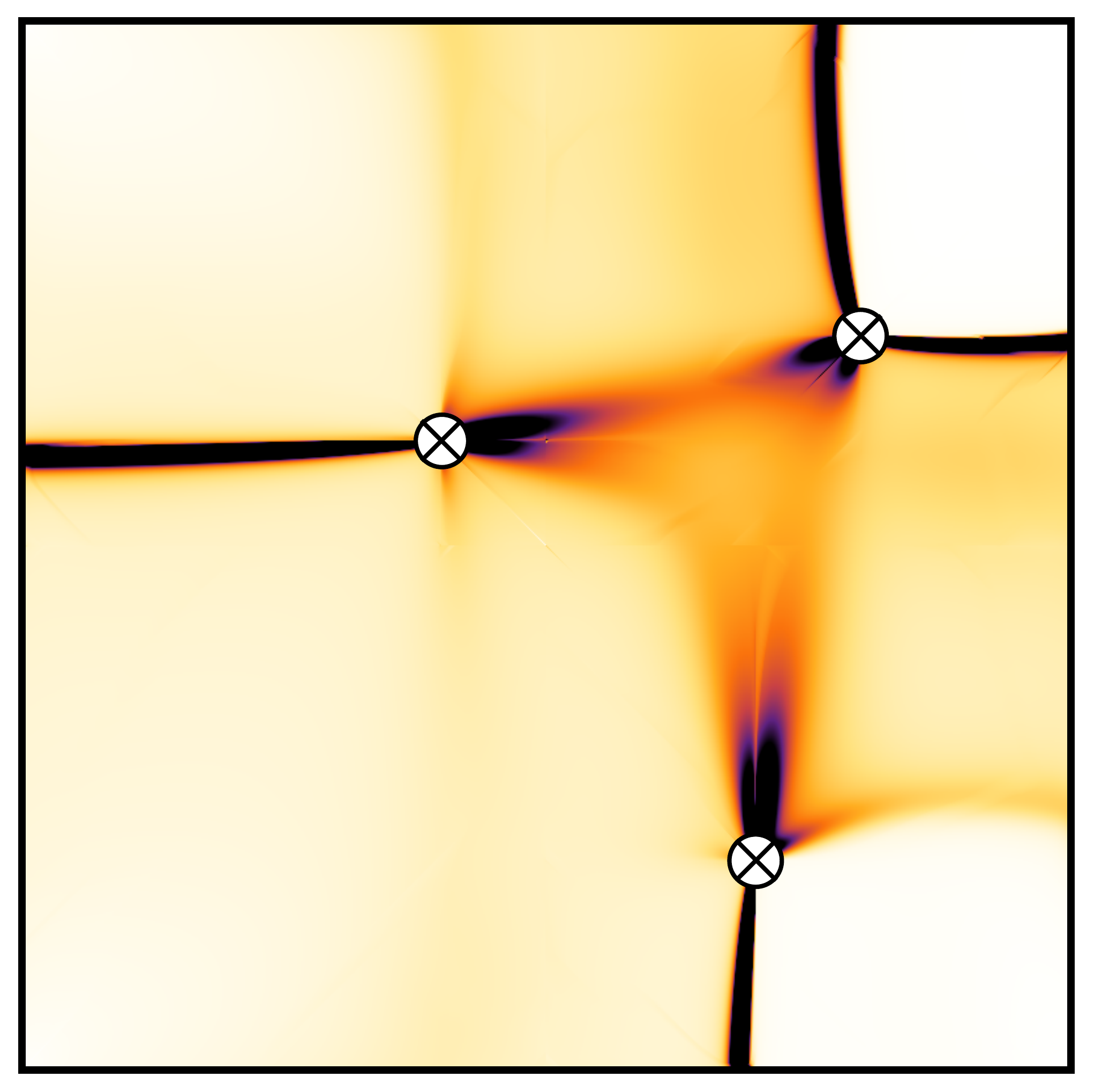}}\hspace{0.7cm}
	\subfloat[]{\includegraphics*[trim={0cm 0cm -0cm -0cm},clip,width=0.3\textwidth]{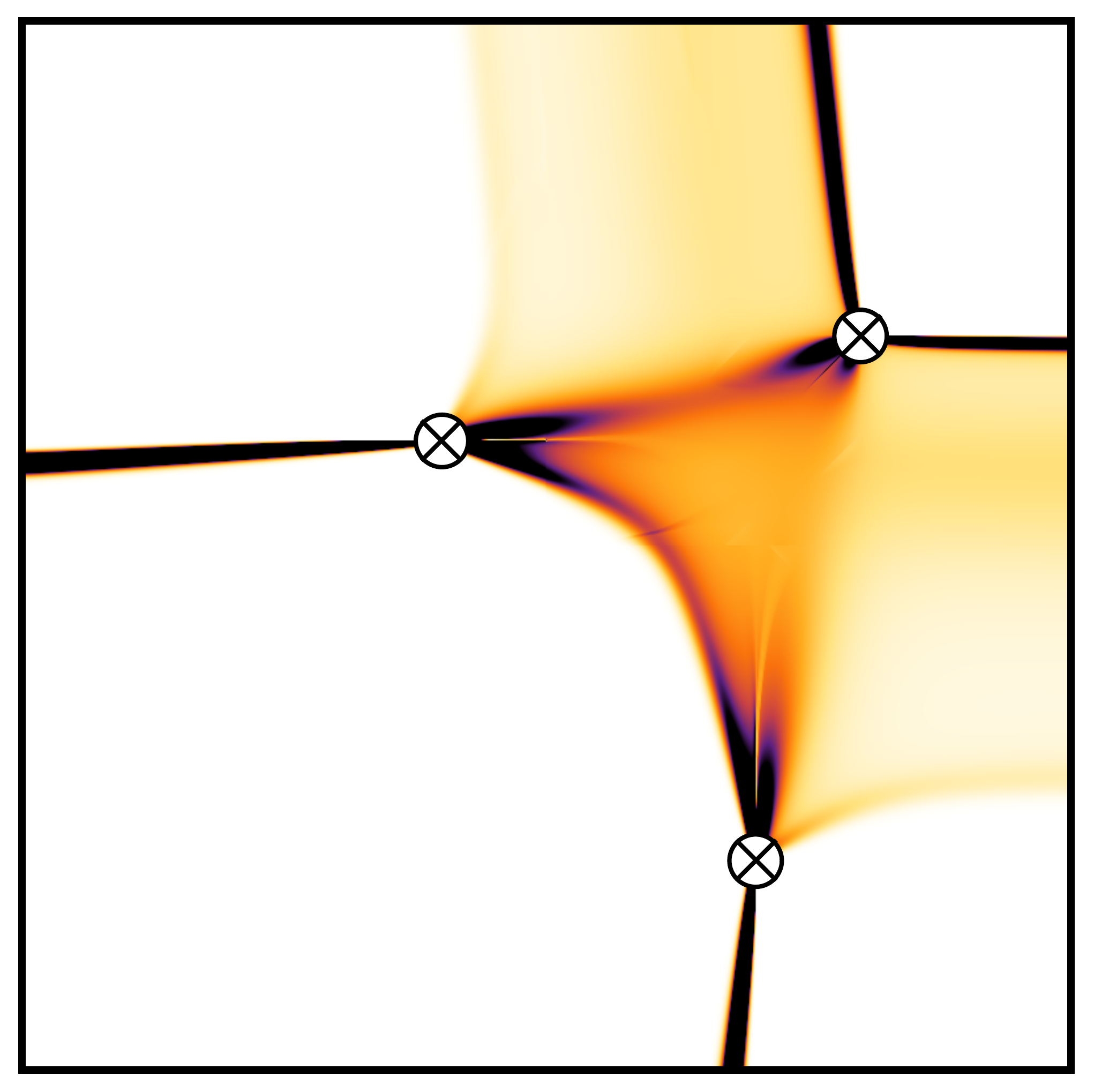}}\\
	\subfloat[]{\includegraphics*[trim={0cm 0cm -0cm -0cm},clip,width=0.3\textwidth]{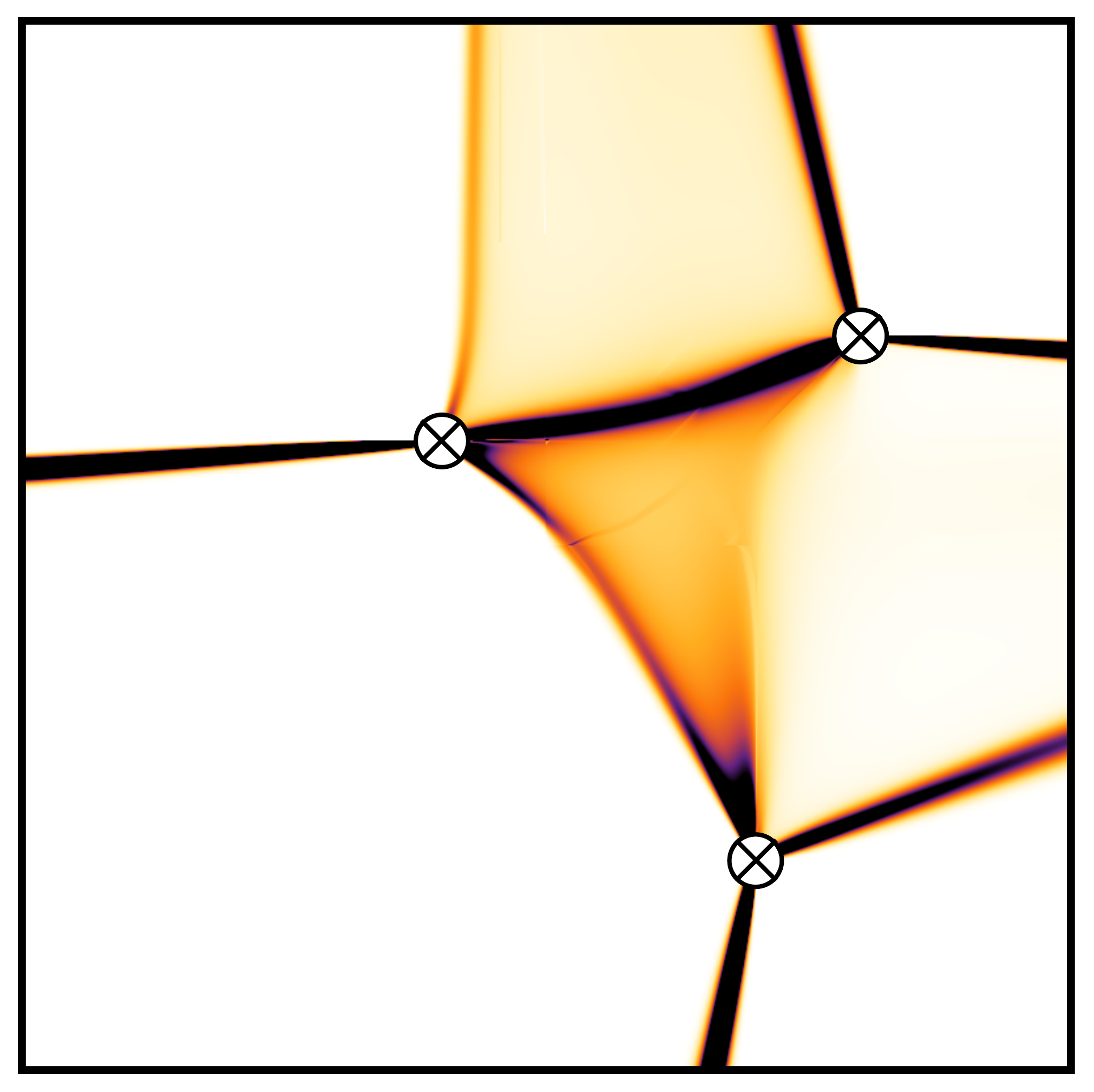}}\hspace{0.7cm}
	\subfloat[]{\includegraphics*[trim={0cm 0cm -0cm -0cm},clip,width=0.3\textwidth]{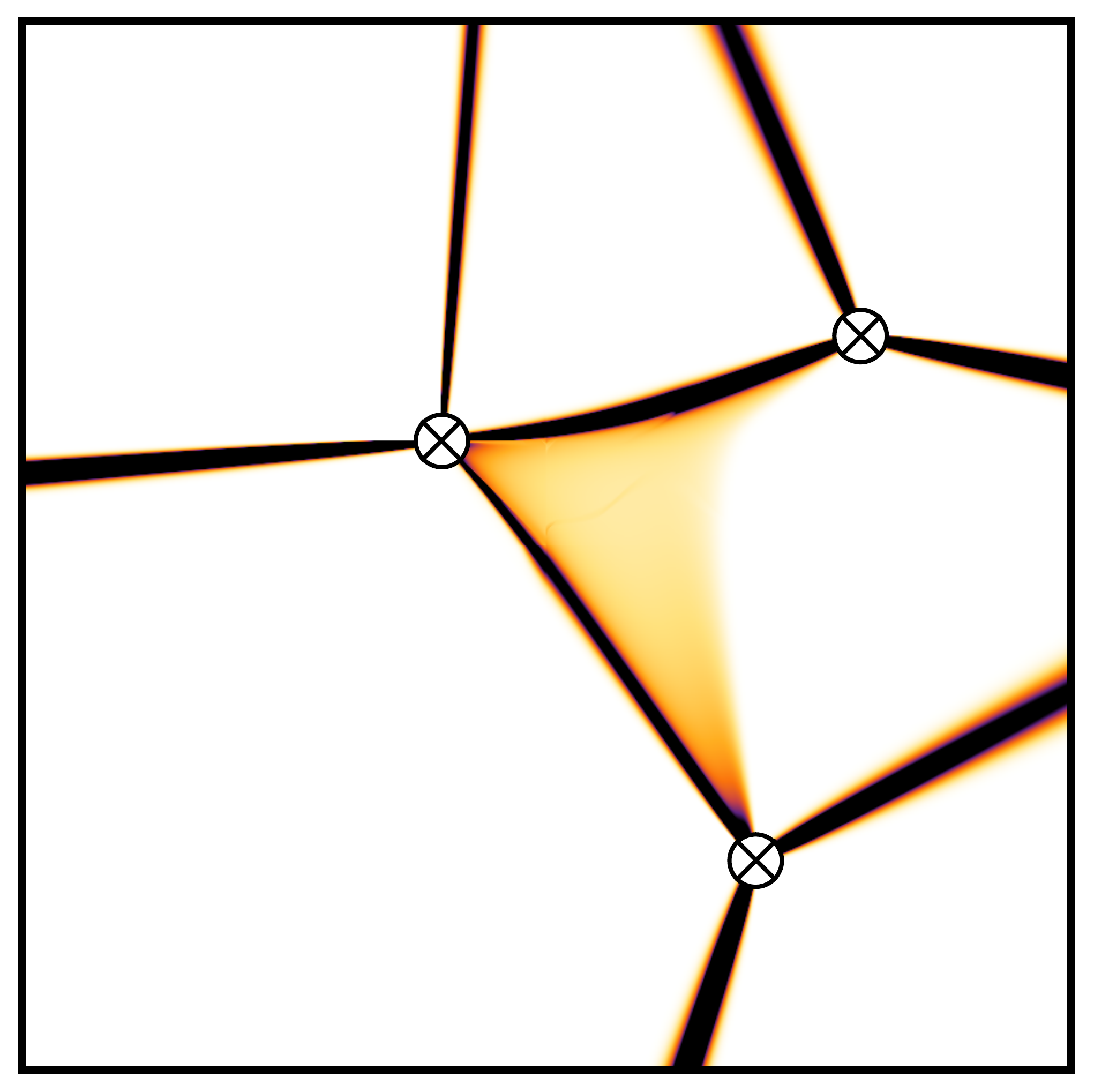}}\hspace{0.7cm}
	\subfloat[]{\includegraphics*[trim={0cm 0cm -0cm -0cm},clip,width=0.3\textwidth]{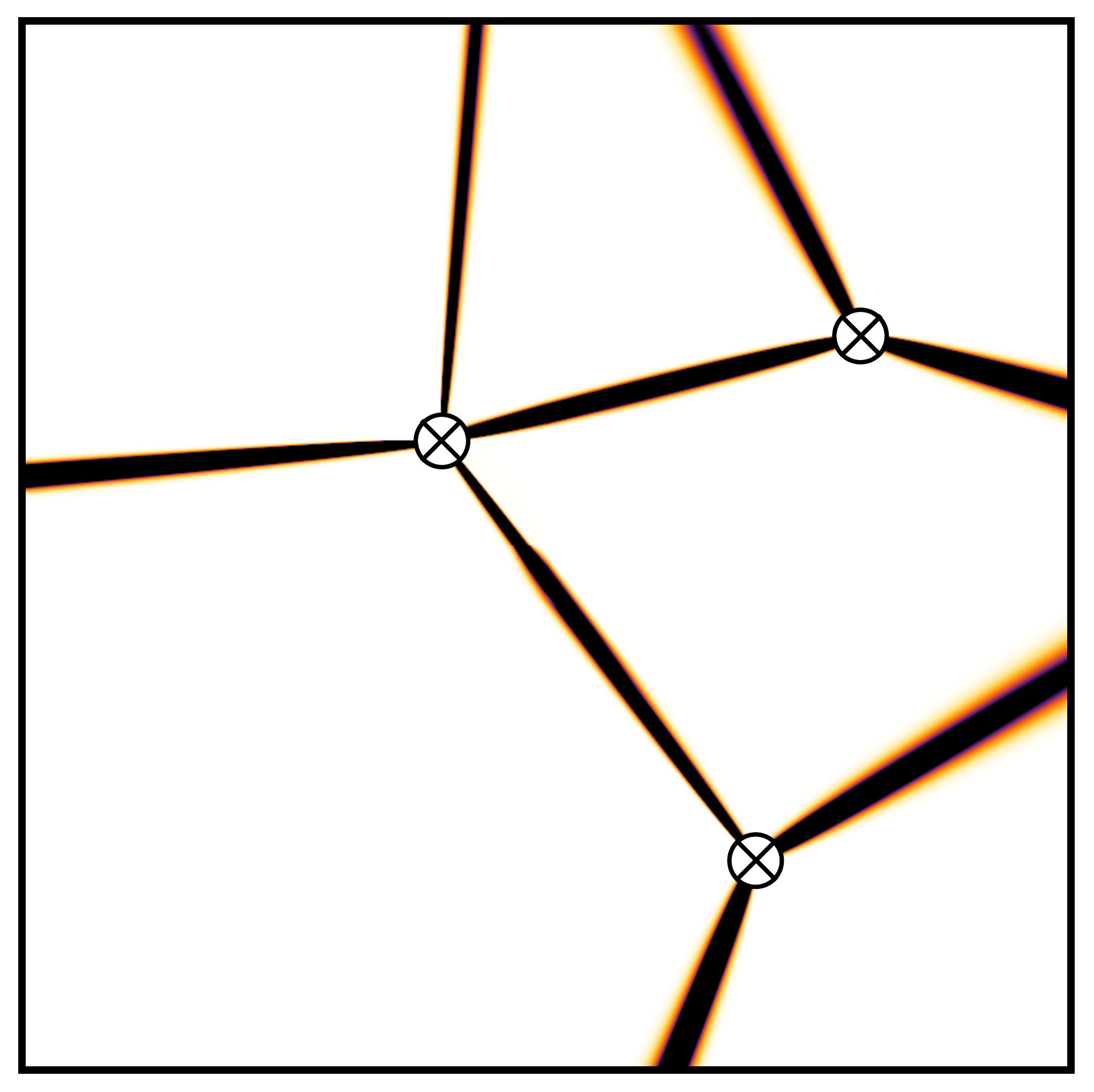}}
	\caption{(a) The optimal membrane problem for an asymmetric three-point load $f = \sum_{i = 1}^3 P\,\delta_{x_i}$. The trimmed averaged approximate optimal thickness $\bar{b}_h$ for Poisson ratios: (b) $\nu=0.5$, \ (c) $\nu=0.0$, \ (d) $\nu=-0.6$, \ (e) $\nu = -0.9$, \ (f) $\nu = -1.0$.}
	\label{fig:3_point_loads}       
\end{figure}

\begin{figure}[h]
	\centering
	\subfloat[]{\includegraphics*[trim={0cm 0cm -0cm -0cm},clip,width=0.3\textwidth]{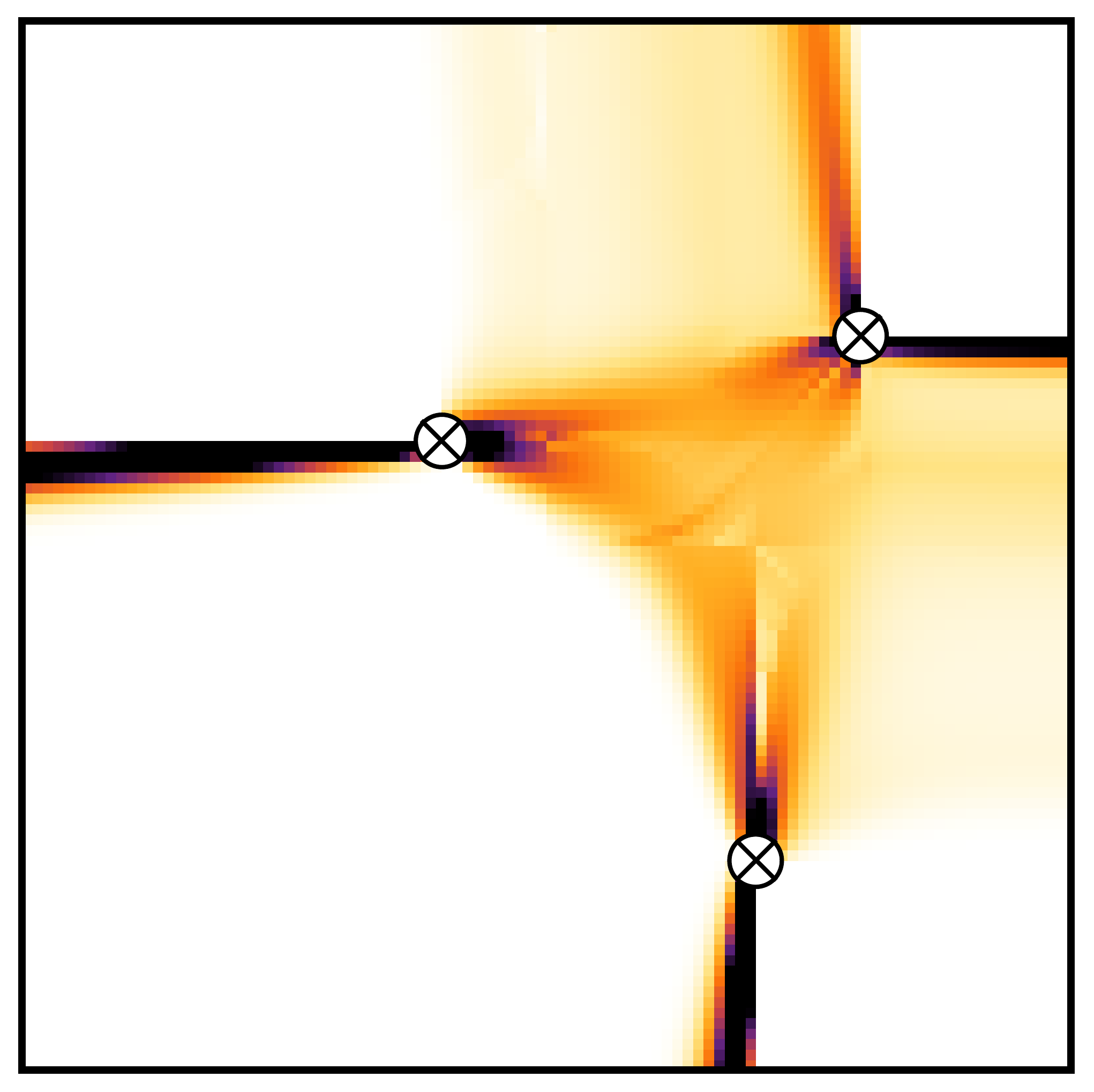}}\hspace{0.7cm}
	\subfloat[]{\includegraphics*[trim={0cm 0cm -0cm -0cm},clip,width=0.3\textwidth]{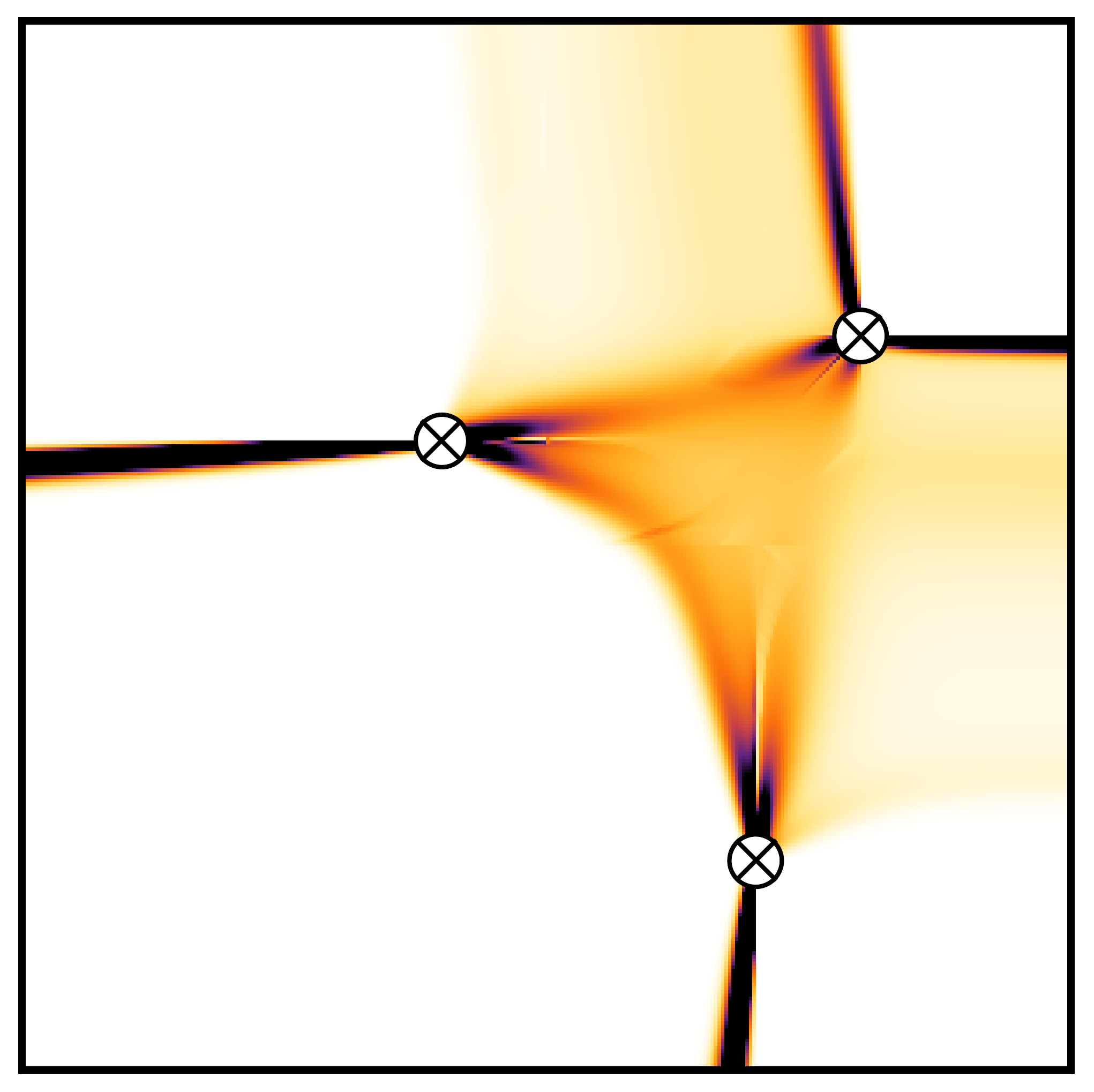}}\hspace{0.7cm}
	\subfloat[]{\includegraphics*[trim={0cm 0cm -0cm -0cm},clip,width=0.3\textwidth]{fig/3points_00.pdf}}
	\caption{The trimmed averaged approximate optimal thickness $\bar{b}_h$ for a three-point load $f = \sum_{i = 1}^3 P\,\delta_{x_i}$ and different mesh sizes: (a) $101 \times 101$, \ (b) $301 \times 301$, \ (c) $801 \times 801$ vertices.}
	\label{fig:3_point_loads_mesh}       
\end{figure}

\begin{figure}[h]
	\centering
	\subfloat[]{\includegraphics*[trim={0cm 0cm -0cm -0cm},clip,width=0.3\textwidth]{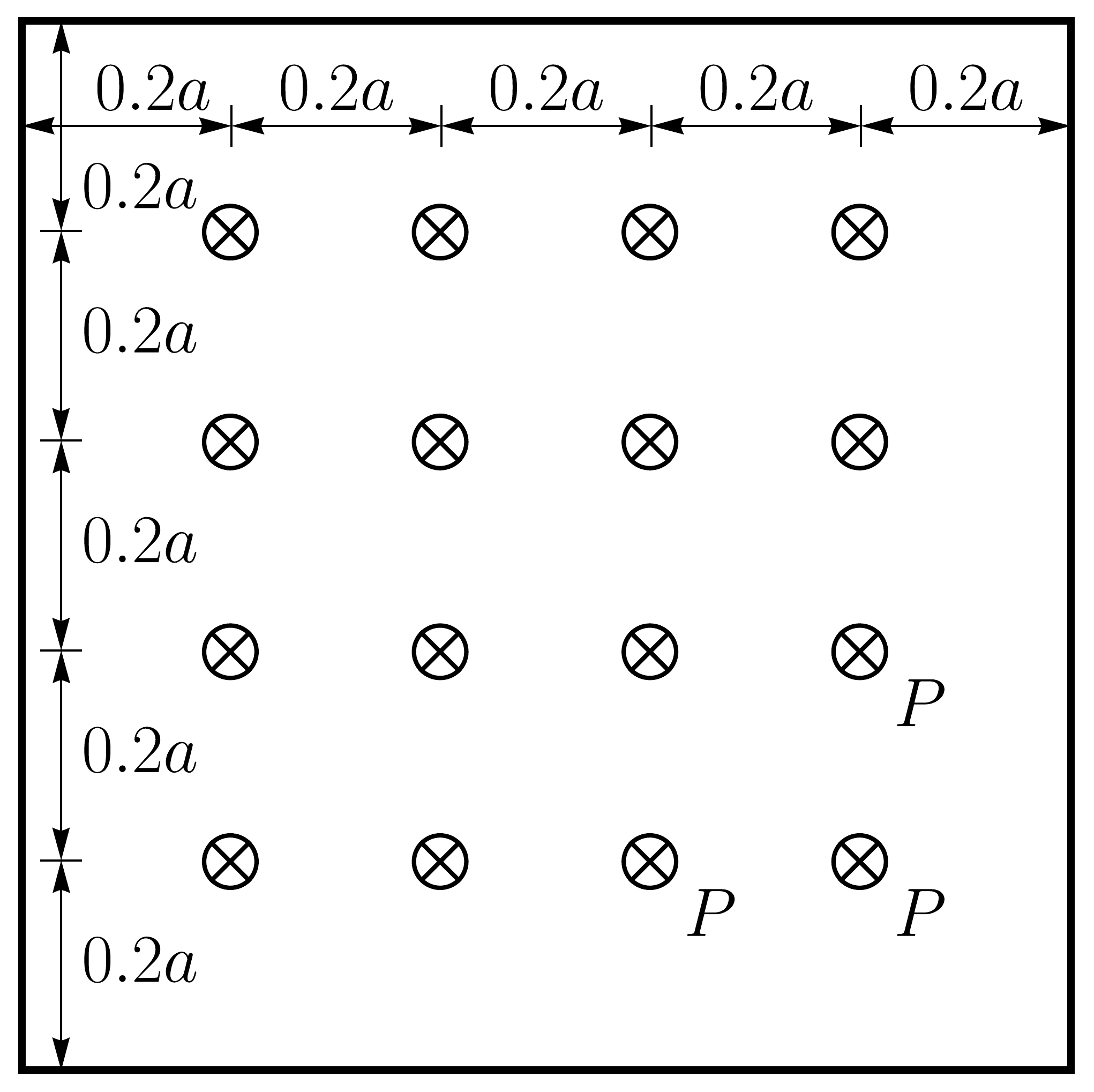}}\hspace{0.7cm}
	\subfloat[]{\includegraphics*[trim={0cm 0cm -0cm -0cm},clip,width=0.3\textwidth]{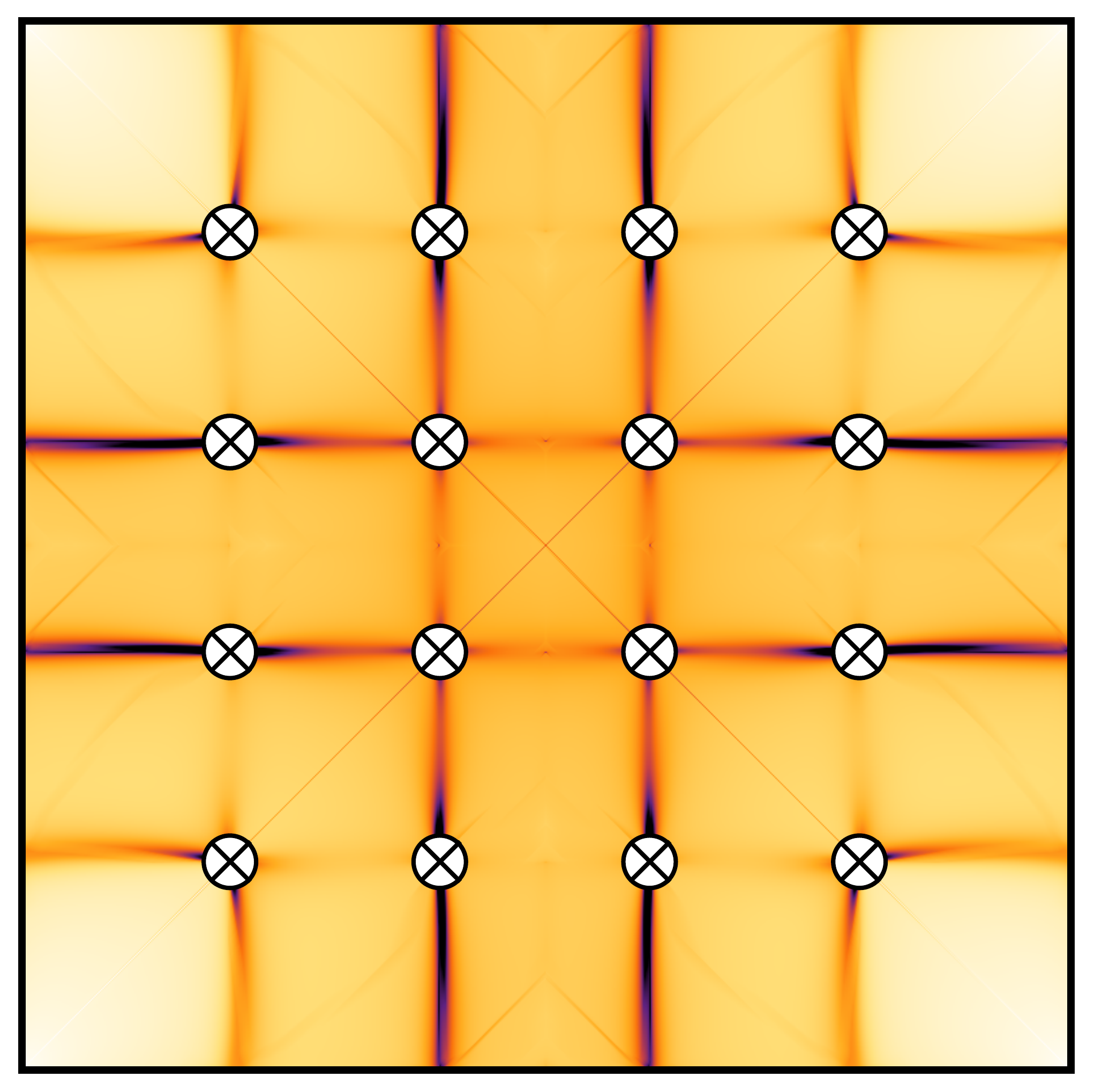}}\hspace{0.7cm}
	\subfloat[]{\includegraphics*[trim={0cm 0cm -0cm -0cm},clip,width=0.3\textwidth]{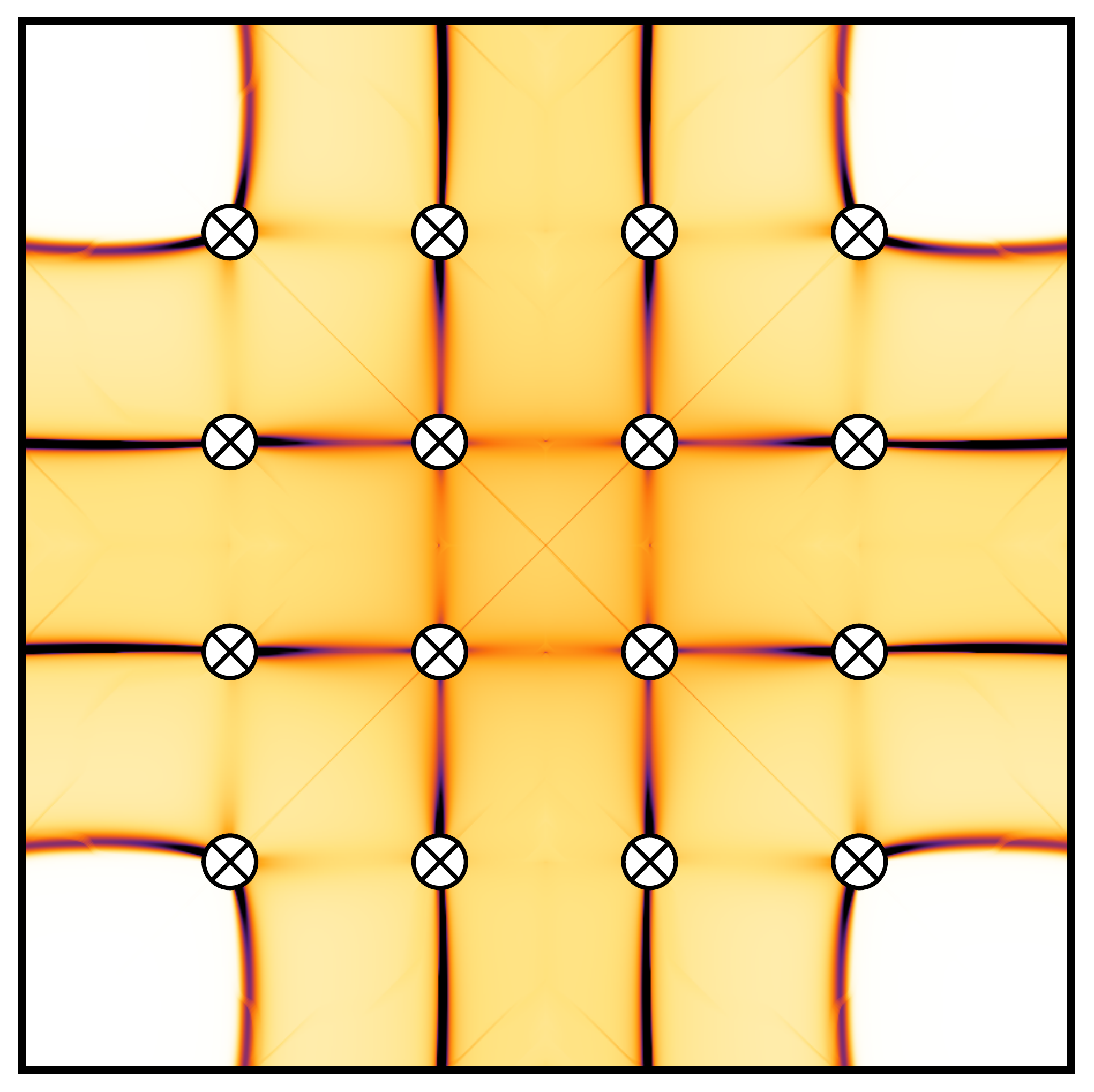}}\\
	\subfloat[]{\includegraphics*[trim={0cm 0cm -0cm -0cm},clip,width=0.3\textwidth]{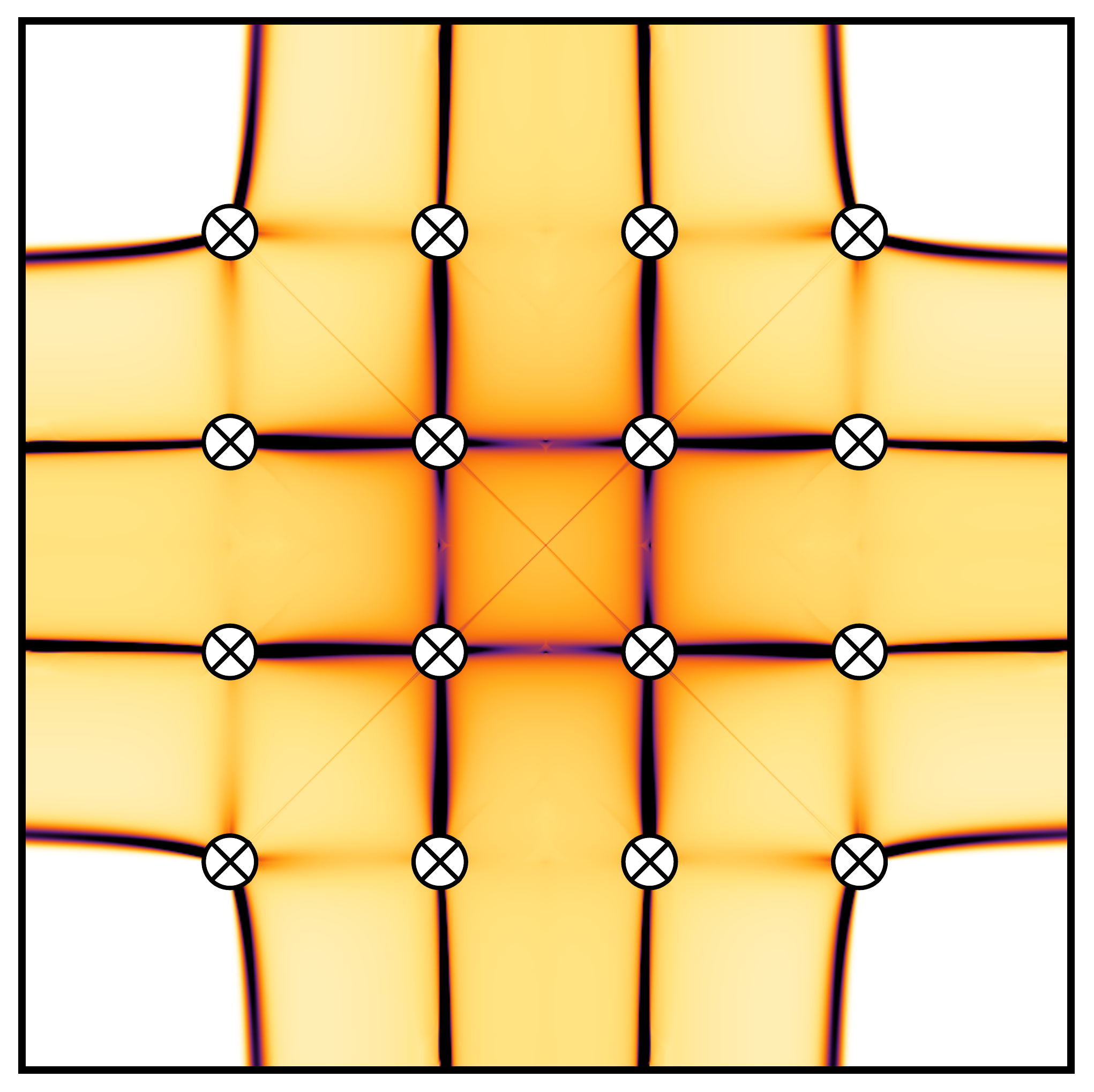}}\hspace{0.7cm}
	\subfloat[]{\includegraphics*[trim={0cm 0cm -0cm -0cm},clip,width=0.3\textwidth]{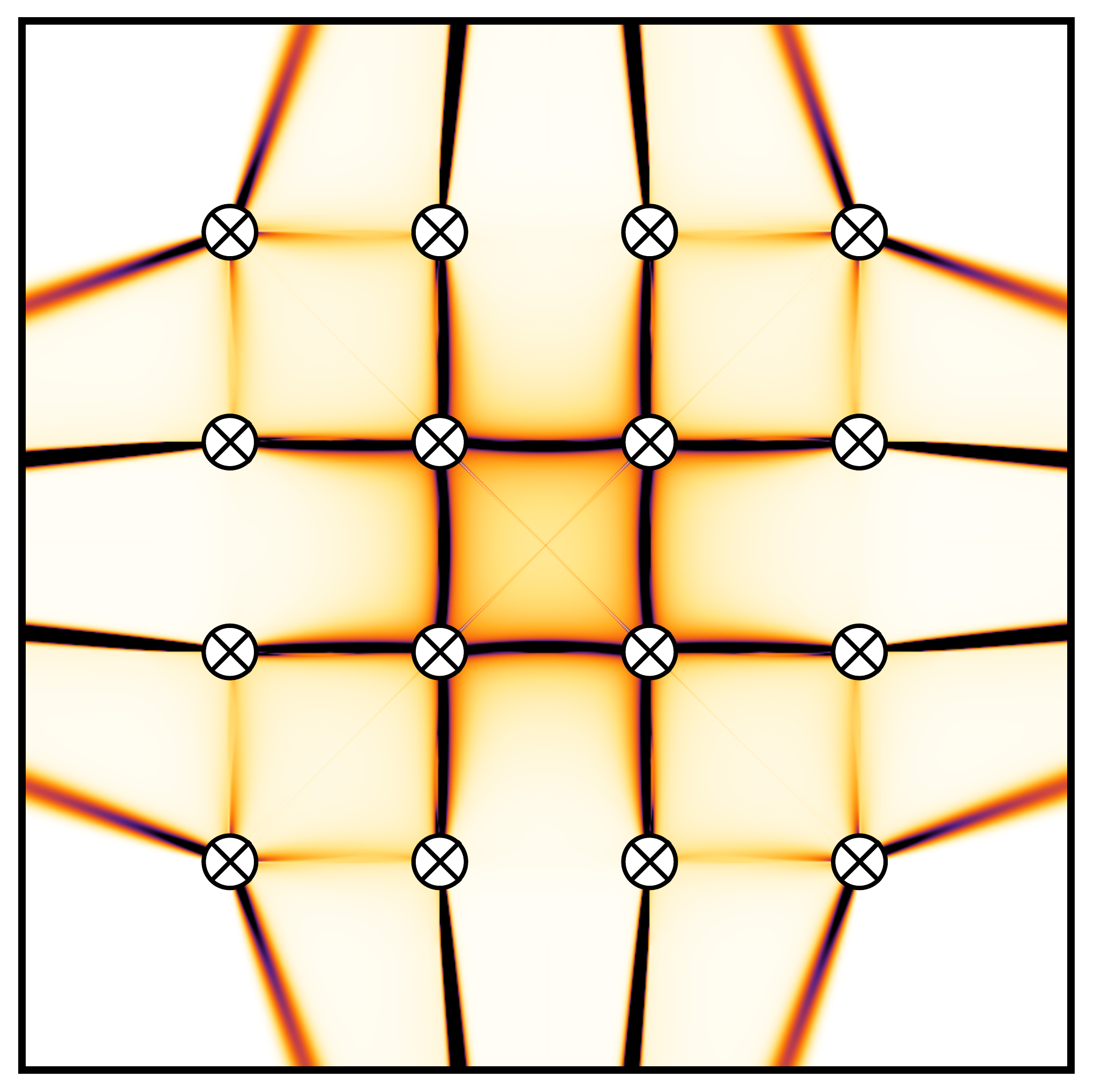}}\hspace{0.7cm}
	\subfloat[]{\includegraphics*[trim={0cm 0cm -0cm -0cm},clip,width=0.3\textwidth]{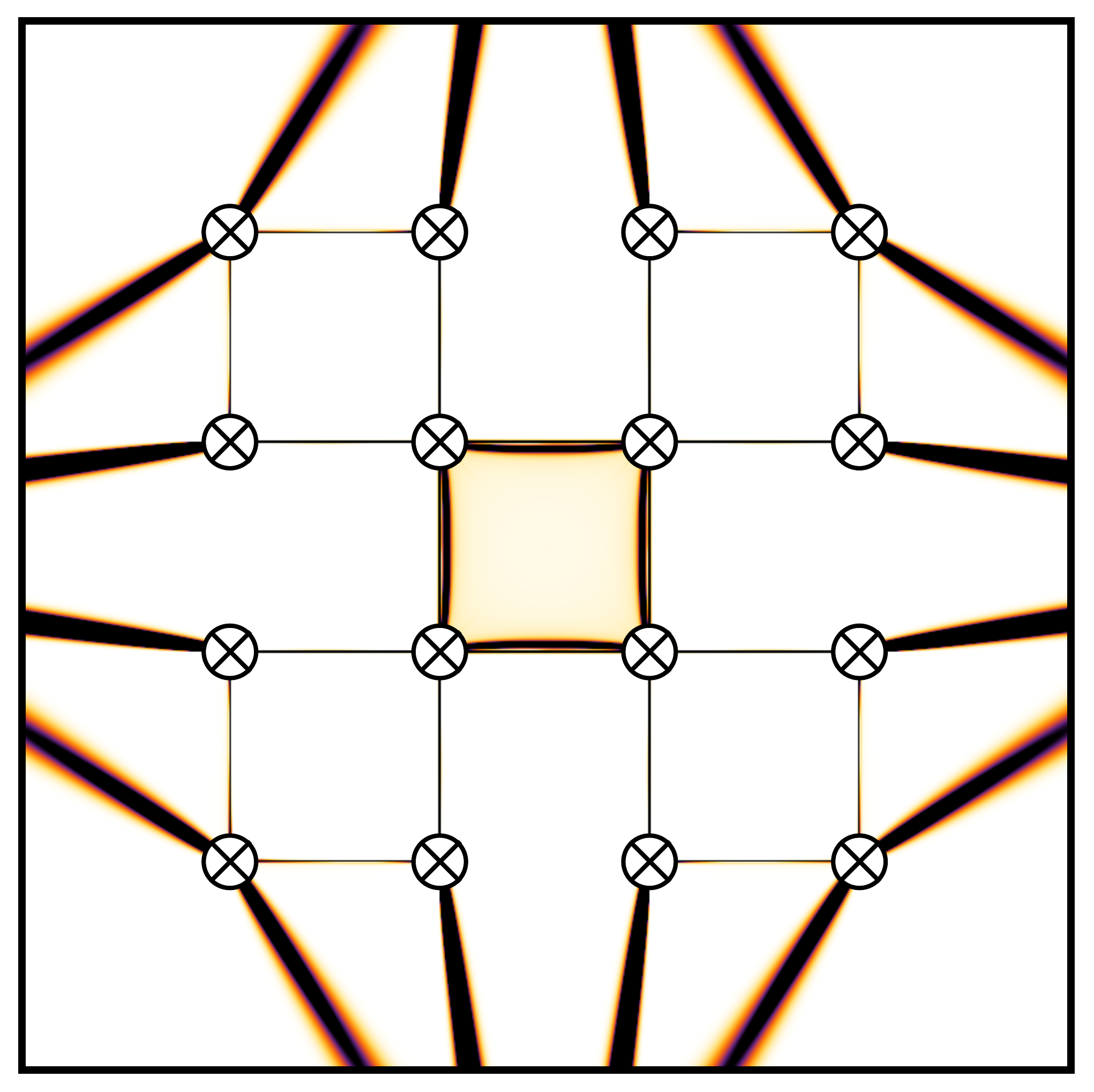}}
	\caption{(a) The optimal membrane problem for a symmetric multiple-point load $f = \sum_{i = 1}^{16} P\,\delta_{x_i}$. The trimmed averaged approximate optimal thickness $\bar{b}_h$ for Poisson ratios: (b) $\nu=0.5$, \ (c) $\nu=0.3$, \ (d) $\nu=0.0$, \ (e) $\nu = -0.6$, \ (f) $\nu = -1.0$.}
	\label{fig:multiple_point_loads}  
\end{figure}

	\begin{figure}[h]
	\centering
	\subfloat[]{\includegraphics*[trim={0cm 0cm -0cm -0cm},clip,width=0.3\textwidth]{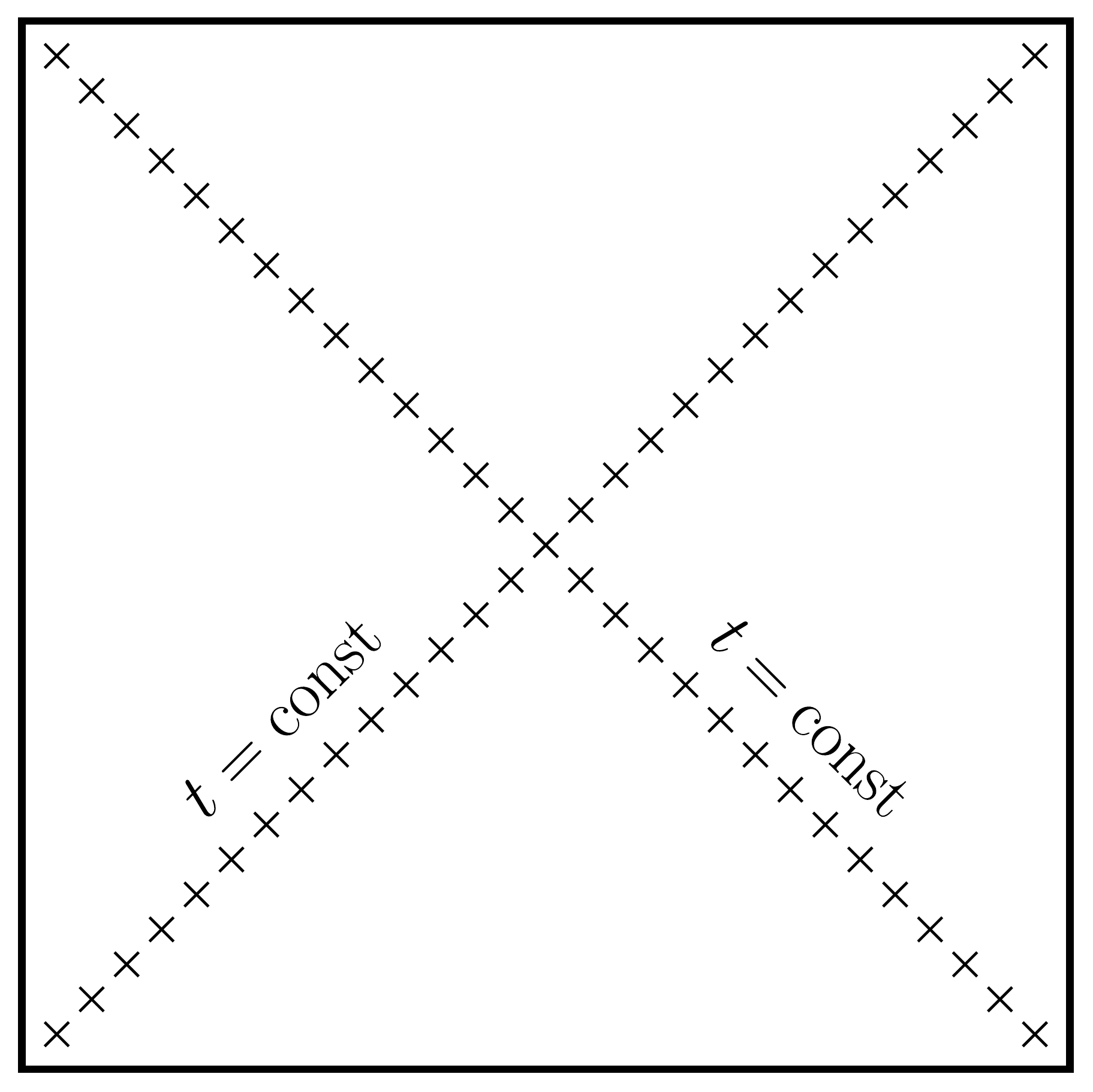}}\hspace{0.7cm}
	\subfloat[]{\includegraphics*[trim={0cm 0cm -0cm -0cm},clip,width=0.3\textwidth]{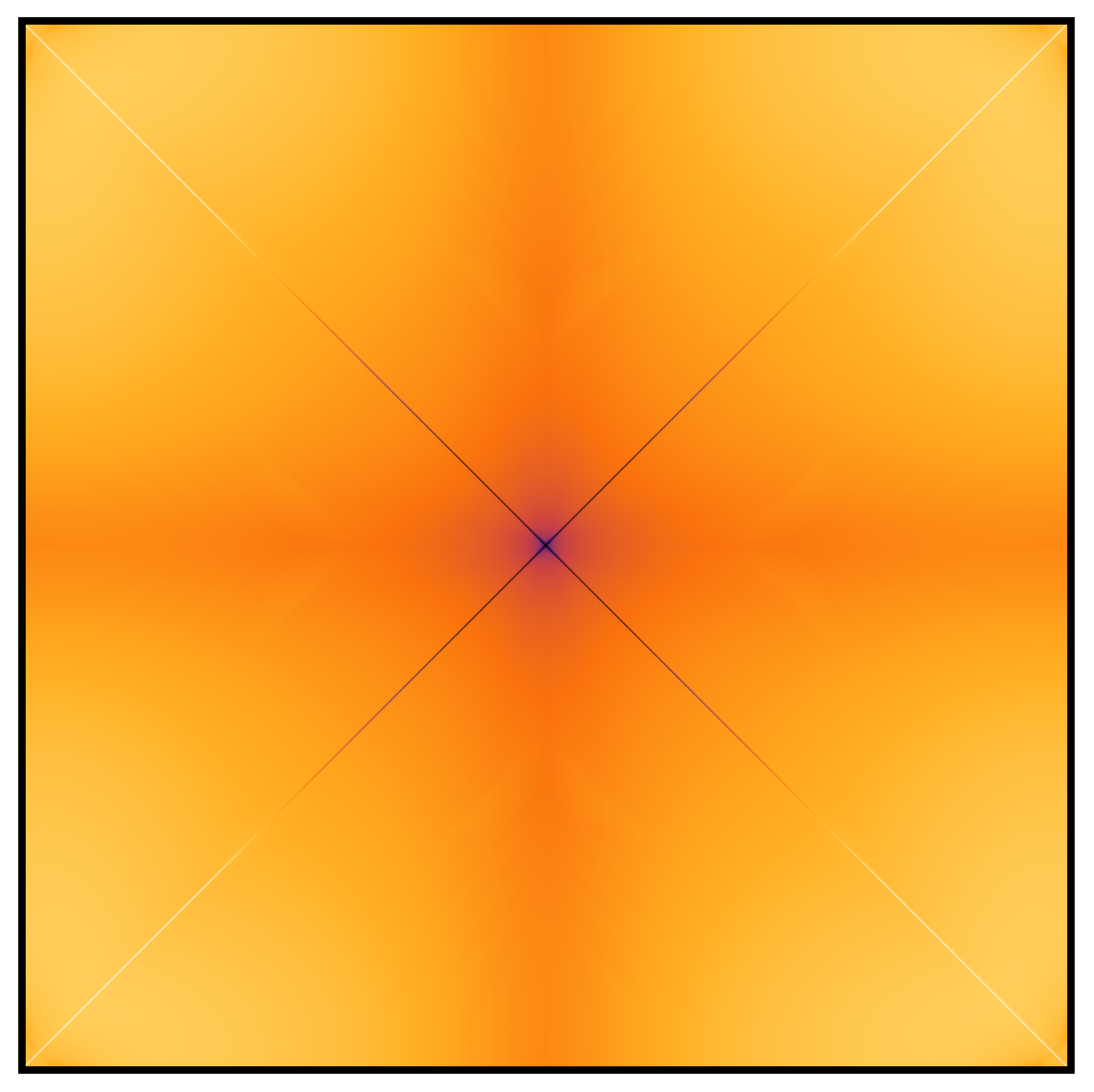}}\hspace{0.7cm}
	\subfloat[]{\includegraphics*[trim={0cm 0cm -0cm -0cm},clip,width=0.3\textwidth]{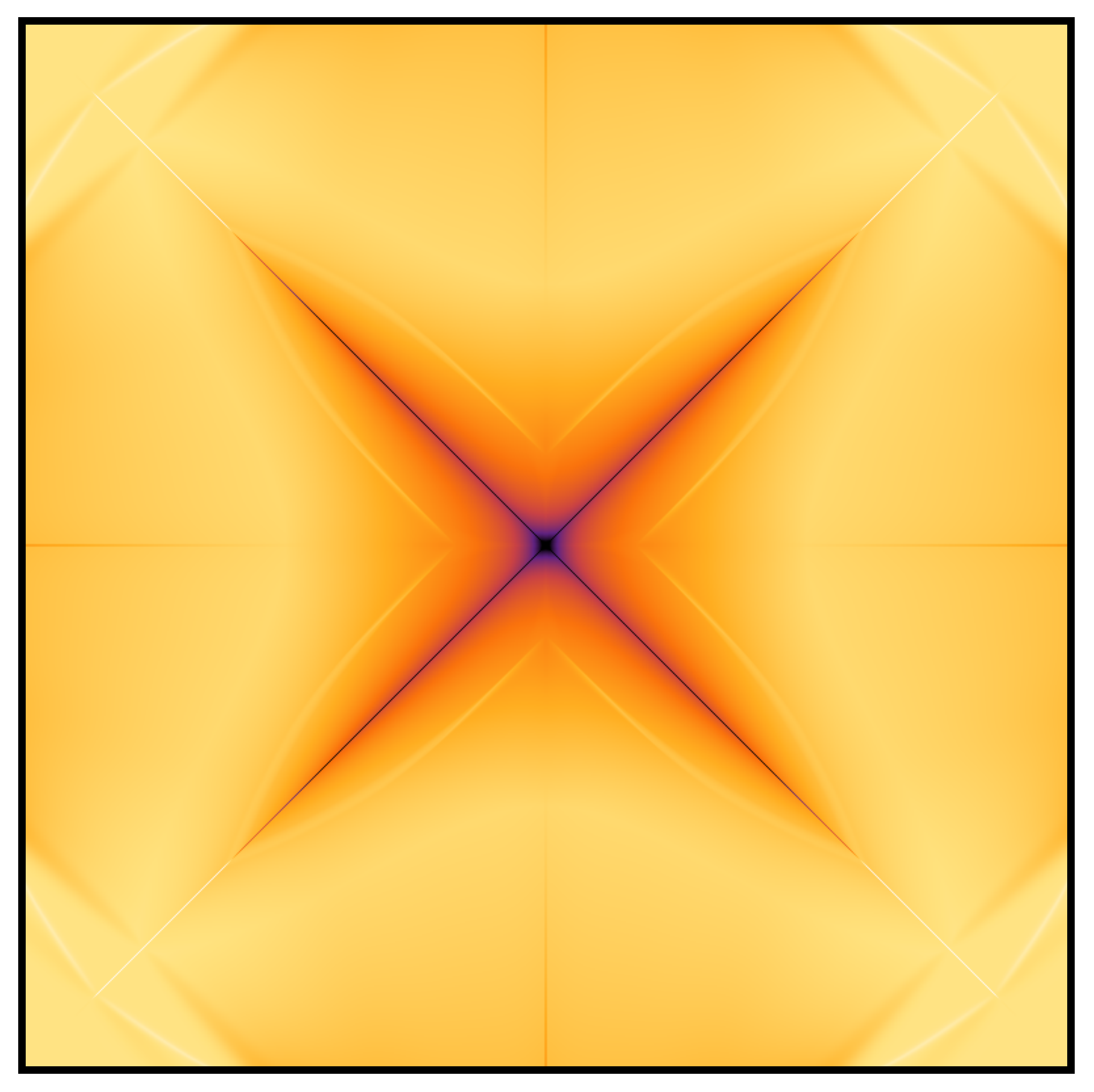}}\\
	\subfloat[]{\includegraphics*[trim={0cm 0cm -0cm -0cm},clip,width=0.3\textwidth]{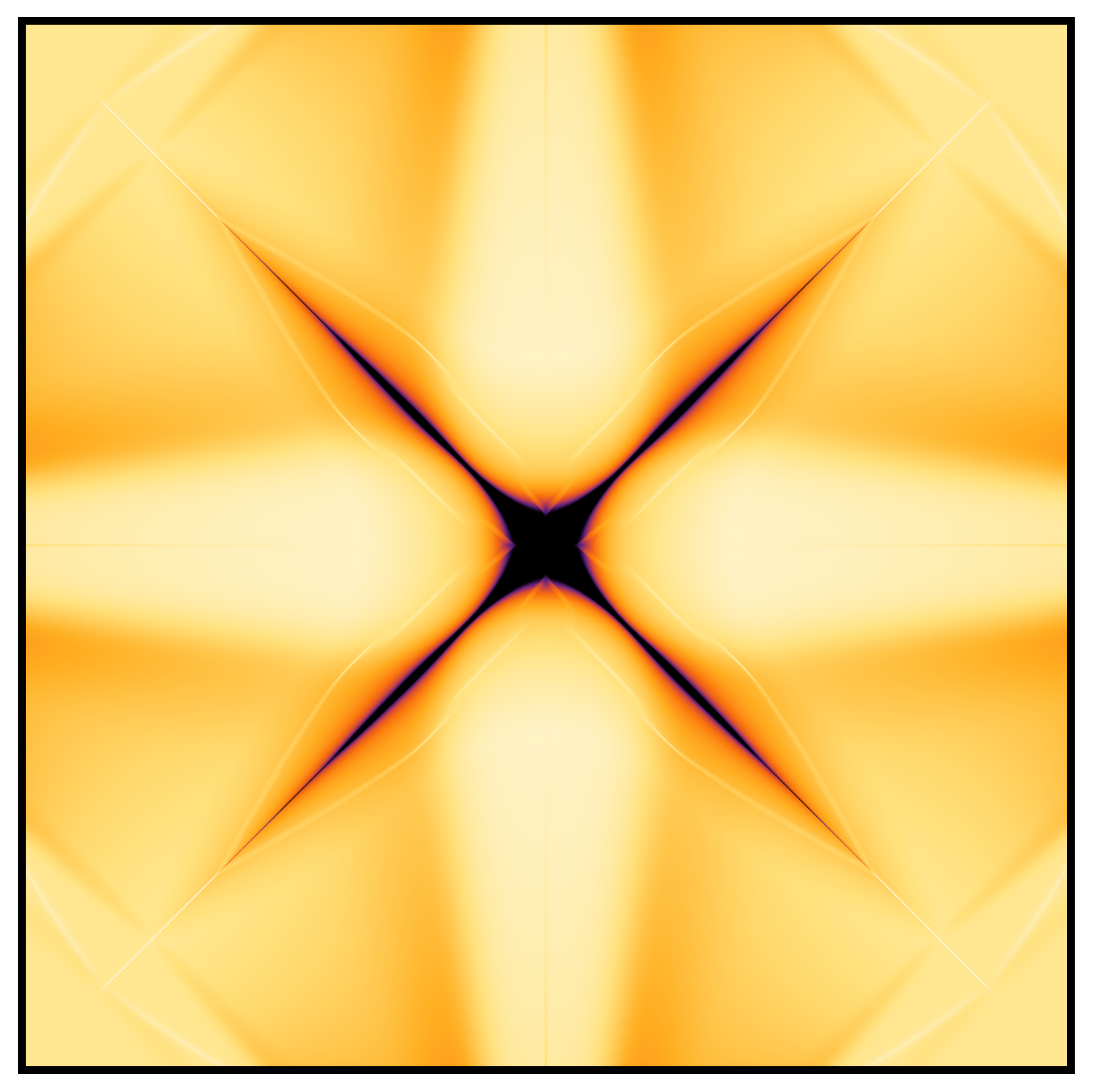}}\hspace{0.7cm}
	\subfloat[]{\includegraphics*[trim={0cm 0cm -0cm -0cm},clip,width=0.3\textwidth]{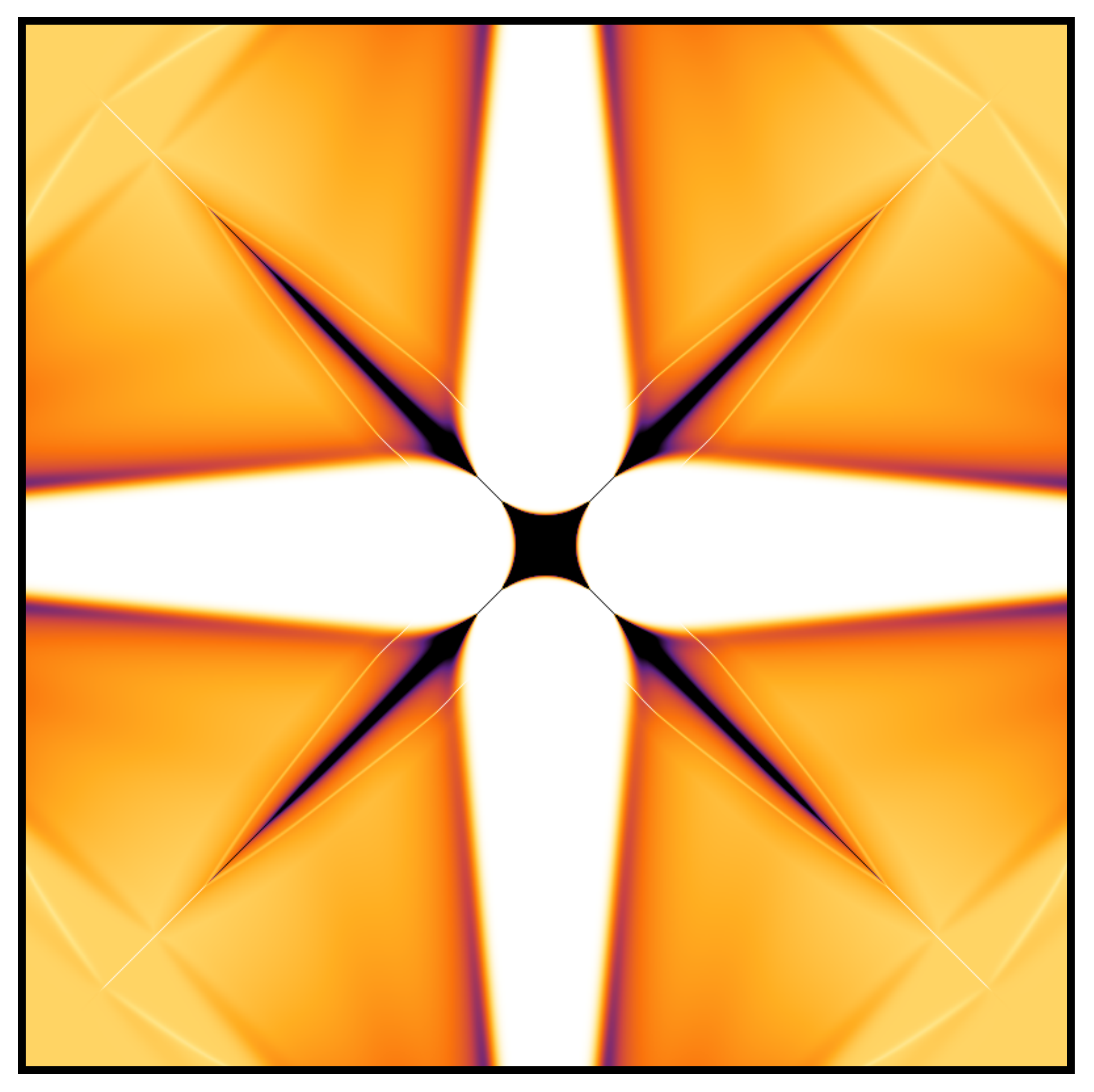}}\hspace{0.7cm}
	\subfloat[]{\includegraphics*[trim={0cm 0cm -0cm -0cm},clip,width=0.3\textwidth]{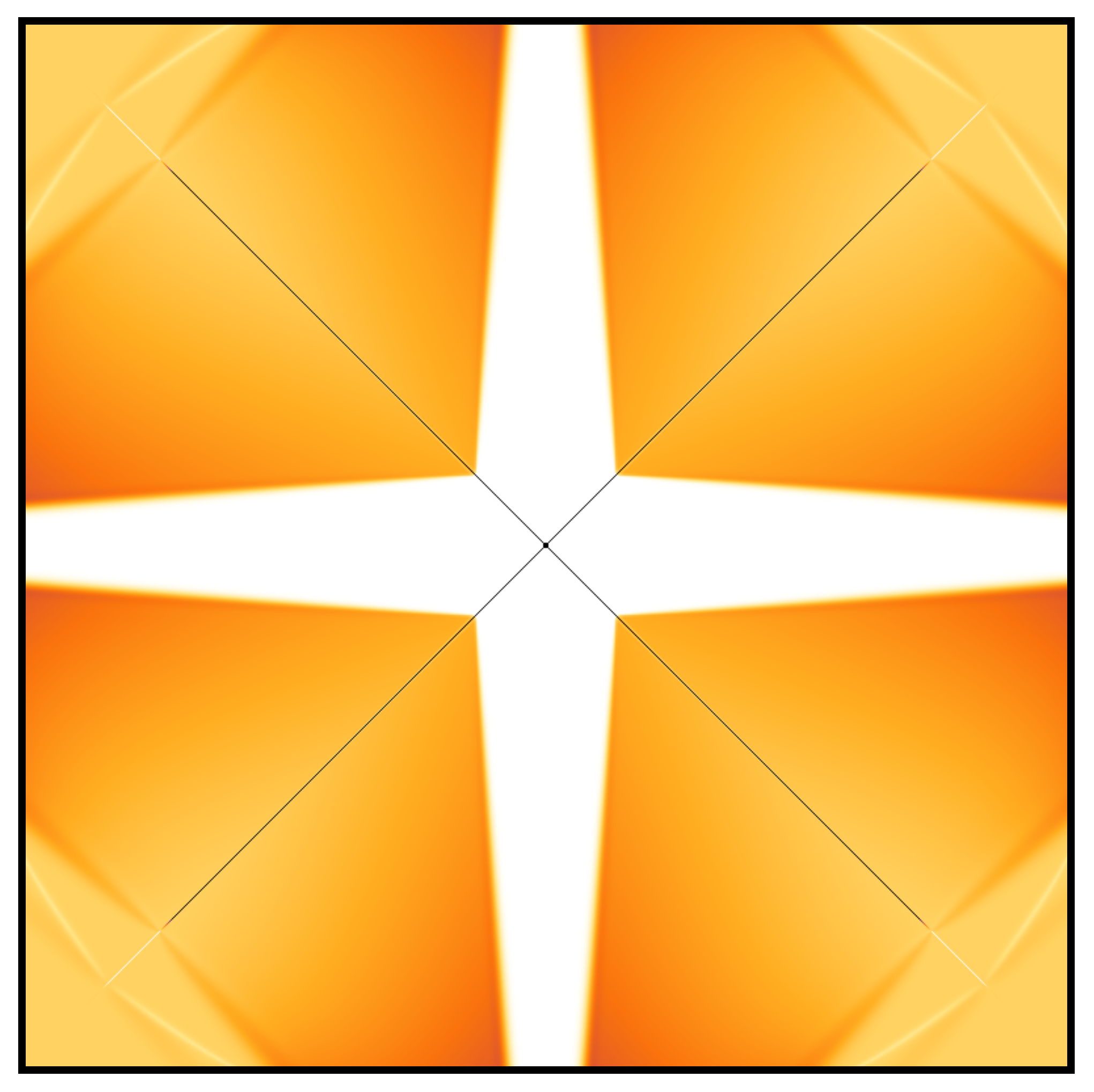}}
	\caption{(a) The optimal membrane problem for a diagonal load $f = t\,\Ha^1\mres[(0,0),(a,a)]+t\,\Ha^1\mres[(0,a),(a,0)]$. The trimmed averaged approximate optimal thickness $\bar{b}_h$ for Poisson ratios: (b) $\nu=0.0$, \ (c) $\nu=-0.85$, \ (d) $\nu=-0.92$, \ (e) $\nu = -0.95$, \ (f) $\nu = -1.0$.}
	\label{fig:diagonal_load}       
\end{figure}

\begin{remark}[\textbf{The checkerboard effect}]
	\label{rem:checkerboard}
	In almost all of simulations performed by the author the approximate optimal thickness $\check{b}_h$ found via the finite element method herein proposed suffers from a pathology. Patterns of inter-twining finite elements that admit zero and non-zero values $\check{b}_h$ may be discerned. For instance, in certain subregions of $\Omega$, zero values are found in the \textit{lower-triangle elements} (elements in the family $\mathcal{T}_1^h$) and non-zero in the \textit{upper-triangle elements} (family $\mathcal{T}_2^h$). Such occurrence is well recognized in optimal design when finite elements of low-order are used. In the case of regular meshes with quad elements this has come to be known as the \textit{checker-board phenomenon}, see \cite[Section 5.2.3]{allaire2002} or \cite{petersson1999} for more details. This issue, however, is easily remedied by averaging at the post-processing stage. In our case it turned out that averaging on each pair of lower- and upper-triangle elements does the trick. Namely, we shall compute \textit{the averaged thickness functions} $\tilde{b}_h \in L^\infty(\O;\R_+)$ as below:	
	\begin{equation*}
		\tilde{b}_h := \frac{1}{\vert E_i^h \cup E_{i+1}^h\vert}\int_{E_i^h \cup E_{i+1}^h} \hat{b}_h \,d\mathcal{L}^2 = \frac{n^2}{a^2}\, \frac{2 V_0}{Z_h}\, \big(\mbf{r}^0(i) + \mbf{r}^0(i+1)\big) \qquad \text{a.e. in }  E_i^h \cup E_{i+1}^h
	\end{equation*}
	where $i=2k-1$ for $k\in \{ 1,\ldots, n^2 \}$ (assuming that elements are numbered from left to right, row by row). Let us stress that such averaging does not spoil the convergence established in Corollary \ref{cor:mu_h}. More accurately, one can show that the sequence $\tilde{\mu}_h = \tilde{b}_h\,\mathcal{L}^2 \mres \O$ weakly-* converges to the very same optimal material distribution $\check{\mu}$ as $\check{\mu}_h = \check{b}_h\,\mathcal{L}^2 \mres \O$.
\end{remark}

\begin{remark}
	\label{rem:trim}
	In general the sequence $\norm{\tilde{b}_h}_{L^\infty(\O)}$ is unbounded. This will naturally be the case once the solution $\check{\mu}$, to which $\tilde{\mu}_h = \tilde{b}_h\,\mathcal{L}^2 \mres \O$ converges, is either not absolutely continuous, or $\check{\mu} = \check{b}\,\mathcal{L}^2 \mres \O$ for an unbounded $\check{b} \in L^1(\O)$. In general, displaying $\tilde{b}_h$ in the figures directly would thus lead to extreme contrast for fine meshes. To address this, in the figures we shall eventually display \textit{the trimmed averaged thickness function} as follows:
	\begin{equation}
		\label{eq:trim}
		\bar{b}_h(x) := \min\big\{\tilde{b}_h(x),b_0 \big\} \qquad \forall x\in \O,
	\end{equation}
	where $b_0$ is a positive constant. Unfortunately, the constant $b_0$ is chosen individually for each example by trial and error so that the best visual effect is achieved. One should bear that in mind when comparing solutions, especially for different loads $f$.
\end{remark}

\begin{example}[\textbf{Three point forces positioned asymmetrically}]
	\label{ex:3_point_loads}
	As the first example we consider a load consisting of three point forces positioned asymmetrically in the square $\Omega = (0,a)^2$, namely $f = \sum_{i = 1}^3 P\,\delta_{x_i}$ for $P>0$ and $x_1 = (\frac{7}{10}a,\frac{2}{10}a)$, $x_2 = (\frac{4}{10}a,\frac{6}{10}a)$, $x_3 = (\frac{8}{10}a,\frac{7}{10}a)$, see Fig. \ref{fig:3_point_loads}(a). In accordance with Section \ref{ssec:implementation} we choose the quadratic energy potential $j$ for Young modulus $E>0$ and for five different Poisson ratios listed in Table \ref{tab:poisson}. In Fig. \ref{fig:3_point_loads}(b)-(f) the heat-maps for the the function $\bar{b}_h$ (cf. Remark \ref{rem:trim}) found for the finite element mesh of $801 \times 801$ vertices are displayed; the darker the colour, the bigger the value of $\bar{b}_h$. In addition, Fig. \ref{fig:3_point_loads_mesh} shows how the solution changes with the mesh resolution. 
	
	Let us interpret the results, starting from the case $\nu = 0$, see Fig. \ref{fig:3_point_loads}(c). We can see two regions of relatively small and mildly varying thickness and one curved-triangle-like region of a bigger and less regular thickness. Then, large concentrations of thickness on four 1D subsets can be discerned. Based on this numerical approximation one may suspect that the exact optimal material distribution is of the form $\check{\mu} = b\,\mathcal{L}^2\mres\O +\sum_{i = 1}^4 s_i \Ha^1 \mres \mathrm{C}_i $ where: $b \in L^1(\O;\R_+)$, $\mathrm{C}_i$ are smooth curves, and $s_i$ are smooth functions. Two of those curves -- those which are not adjacent to the continuous regions -- are supposedly segments $C_1 = [x_1,y_1]$, $C_2 = [x_2,y_2]$ where $y_1,y_2 \in \bO$; in that case $s_1,s_2$ are constants. The function $b$ is unbounded in the neighbourhoods of forces' application points $x_1,x_2,x_3$. From the mechanical perspective, the singular part of $\check{\mu}$ models four strings, rather than 2D diffused film.
	
	Next, we study how the optimal solutions change with the Poisson ratio $\nu$. Fig. \ref{fig:3_point_loads} clearly indicates that $\nu$ affects the proportions of the diffused and the string part of the structure (absolutely continuous and singular part of $\check{\mu}$). The lower the Poisson ratio, the bigger the contribution of the string part. In particular, for $\nu=-1$ the diffused part appears to vanish completely, and the optimal structure becomes a system of eight strings.
\end{example}

\begin{example}[\textbf{Multiple point forces positioned symmetrically}]
	\label{ex:multiple_point_loads}
	We take a load consisting of point forces again, cf. Fig. \ref{fig:multiple_point_loads}. This time, there will be 16 of them ($f = \sum_{i = 1}^{16} P\,\delta_{x_i}$), and they will be positioned symmetrically in $\O = (0,a)^2$. Comments similar to those made in Example \ref{ex:3_point_loads} may be repeated here. Again, with the Poisson ratio decreasing, the optimal structure seems to rely more on the string elements rather than on the diffused membrane body. In case of $\nu = -1$ the structure is composed of strings almost exclusively, except for the small square in the center. However, a more thorough analysis points to a high non-uniqueness of the material distribution in this square. One can show that using only the four strings being the edges of the square (without the diffused part in the interior) is also optimal.
\end{example}

\begin{example}[\textbf{Diagonal load}]
	\label{ex:diagonal_load}
	For the next example we consider a 1D load that is uniformly distributed along the diagonals of the square $\O = (0,a)^2$, more accurately $f = t\,\Ha^1\mres[(0,0),(a,a)]+t\,\Ha^1\mres[(0,a),(a,0)]$ for $t>0$.
	
	In Fig. \ref{fig:diagonal_load}(b) the numerical solution for $\nu = 0$ is illustrated. The optimal thickness distribution is more or less uniform in the whole domain $\Omega$. The deviation of thickness on the two stripes of finite elements along the diagonals are is a finite element method artefact, and it was numerically verified to vanish with $h  \to 0$. In fact, for all positive Poisson ratios the solution barely changes.
	
	Considering the above, the remaining solutions in Fig. \ref{fig:diagonal_load} were generated for $\nu<0$ or even for $\nu$ close to $-1$. Already for $\nu=-0.85$, see Fig. \ref{fig:diagonal_load}(c) we can see a clear concentration of the material along central parts of the diagonals. This concentration becomes more crisp for $\nu=-0.92$ where also a certain disproportion in the thickness emerges -- the thickness seems to jump along straight lines. For $\nu = -0.95$ we see the regions of lower thickness turn into void. Finally, for $\nu=-1$ the material seems to perfectly concentrate along the diagonals -- two stripes, each one-element wide, admit thickness $\tilde{b}_h(x)$ about 100 times higher than in the diffused part of the membrane. Accordingly, we can suspect that $\check{\mu}$ admits a singular 1D part for $\nu=-1$. One should note that -- despite the visual resemblance -- this was not the case for $\nu=0$, where $\check{\mu}$ is most likely absolutely continuous. In fact, Fig. \ref{fig:diagonal_load}(b) displays $\tilde{b}_h$ ($\bar{b}_h$ for $b_0 = \norm{\tilde{b}_h}_{L^\infty}$), whereas in Fig. \ref{fig:pressure}(f) we see the heat map of $\bar{b}_h$ for $b_0 = 0.01  \norm{\tilde{b}_h}_{L^\infty}$.

\end{example}

\begin{example}[\textbf{Uniform pressure load}]
	\label{ex:pressure}
	For the last example we consider a uniform pressure load, i.e. $f = p\,\mathcal{L}^2\mres \O$, see Fig. \ref{fig:pressure}(a). Similarly as in Example \ref{ex:diagonal_load}, the optimal membranes for positive Poisson ratios admit a fairly uniform thickness distribution as in Fig. \ref{fig:pressure}(b). Once again, interesting shapes occur for negative $\nu$. By studying the solutions Fig. \ref{fig:pressure}(c)-(f) we observe that the round thickness concentration in the center gradually evolves in shape with $\nu$ decreasing.
	
	Finally, for $\nu =-1$ we discover a peculiar star-like shape, cf. Fig.  \ref{fig:pressure}(f). It is a fact that the thickness jump along the star's contour is around five-fold. In the very center of the shape we can discern another star-shaped concentration. The shape is similar to the external one, only smaller and rotated by $45\degree$. While Fig. \ref{fig:pressure}(f) shows the heat-map of $\bar{b}_h$ computed for $b_0 = 0.01\norm{\tilde{b}_h}_{L^\infty}$ (see \eqref{eq:trim}), in Fig. \ref{fig:fractals}(a) we observe the same result but for $b_0 = \norm{\tilde{b}_h}_{L^\infty}$; namely, it is the function $\tilde{b}_h$ whose heat-map we see. In the latter figure we get a better look at the thickness distribution in the smaller star shape, see the zoom-in window. A clear concentration can be discerned forming yet another star-like shape, again  smaller one and rotated by $45\degree$. These observations indicate a self-similarity that potentially emerges in the exact solution $\check{\mu}$.
	
	Such fractal-like numerical solutions may be discovered for different shapes of domain $\O$ when a uniform pressure is applied: $f = p \, \mathcal{L}^2 \mres \O$. In Fig. \ref{fig:fractals} a simulation for a triangular domain $\O = \mathrm{int \,co}\big( \big\{ (0,0),(0,a),(a,0) \big\} \big)$ can be found. This time around, the star-shape has three legs, each reaching towards one corner of the triangle $\O$. Inside the three-legged star there is another one, although the shape seems different; a higher resolution simulation would be helpful to examine the self-similarity effect more thoroughly.  
	
	Similar exotic shapes also occur in the optimal designs found for pressure loading restricted to subregions of $\O$. To give a flavour, in Fig. \ref{fig:fractals}(c) a simulation is shown for quadratic domain $\O = (0,a) \times (0,a)$ and for a pressure load restricted to the lower triangle, i.e. $f = p \, \mathcal{L}^2 \mres  \mathrm{co}\big( \big\{ (0,0),(0,a),(a,0) \big\} \big)$. 
	
	\begin{figure}[h]
		\centering
		\subfloat[]{\includegraphics*[trim={0cm 0cm -0cm -0cm},clip,width=0.3\textwidth]{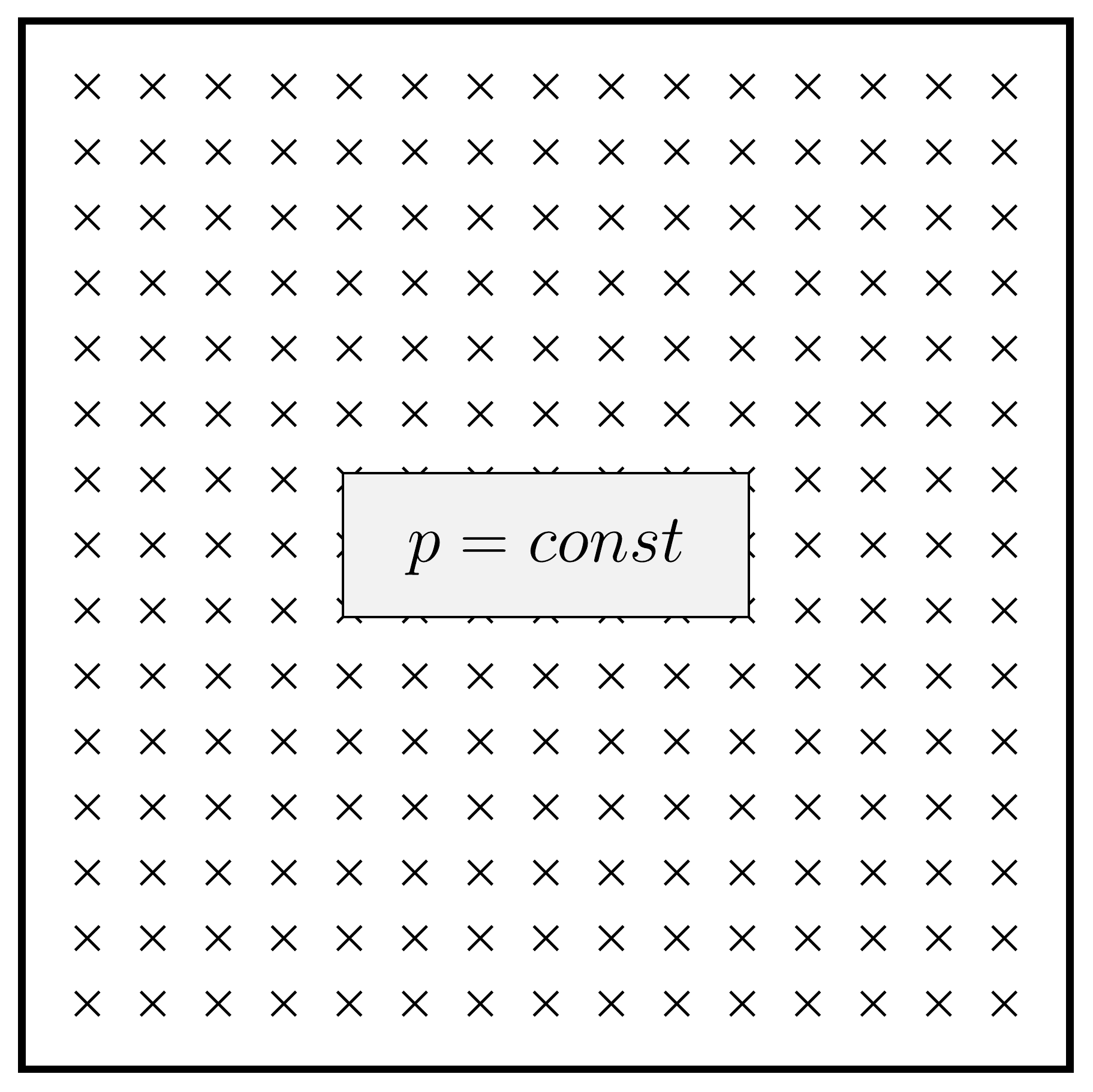}}\hspace{0.7cm}
		\subfloat[]{\includegraphics*[trim={0cm 0cm -0cm -0cm},clip,width=0.3\textwidth]{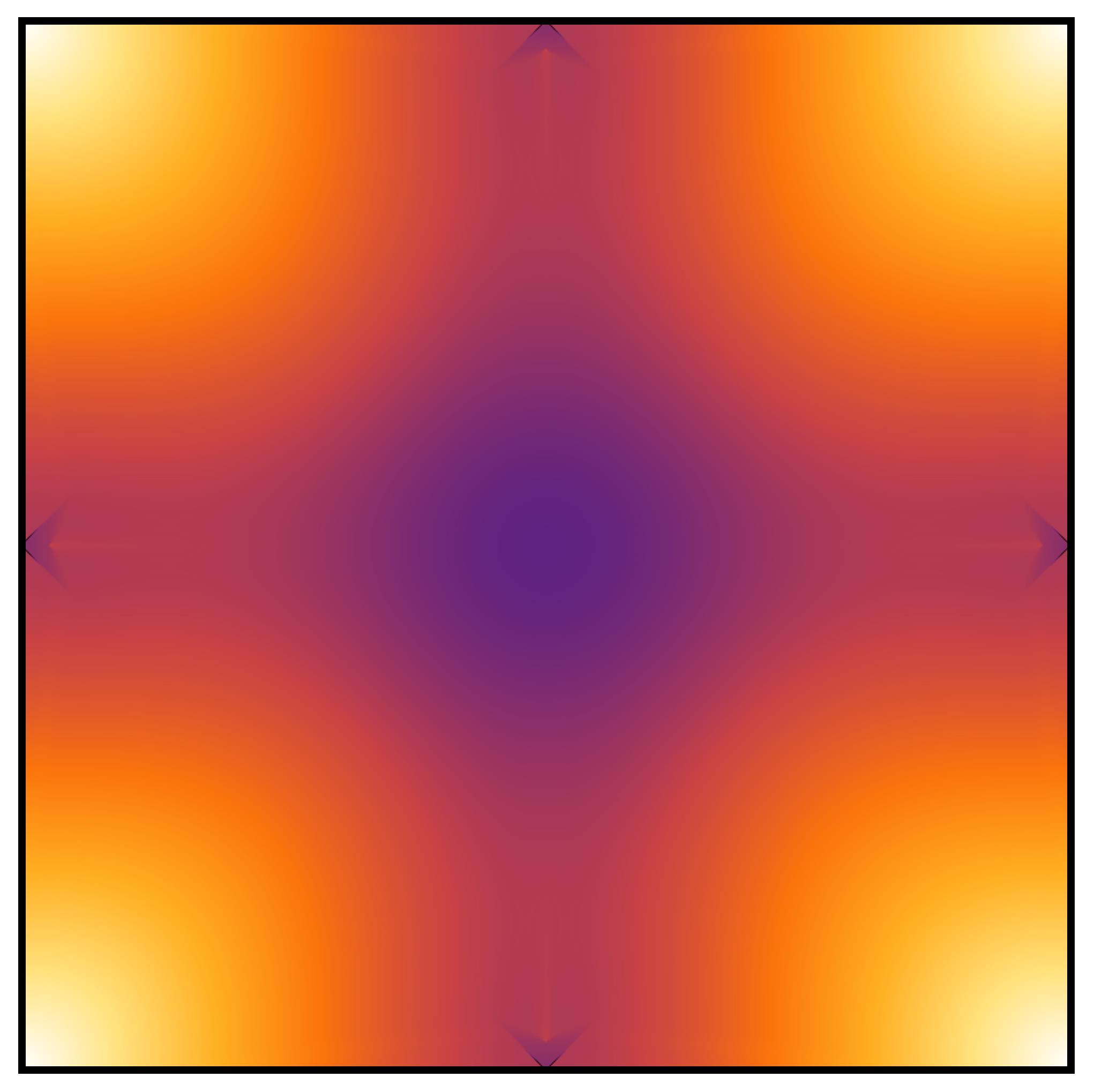}}\hspace{0.7cm}
		\subfloat[]{\includegraphics*[trim={0cm 0cm -0cm -0cm},clip,width=0.3\textwidth]{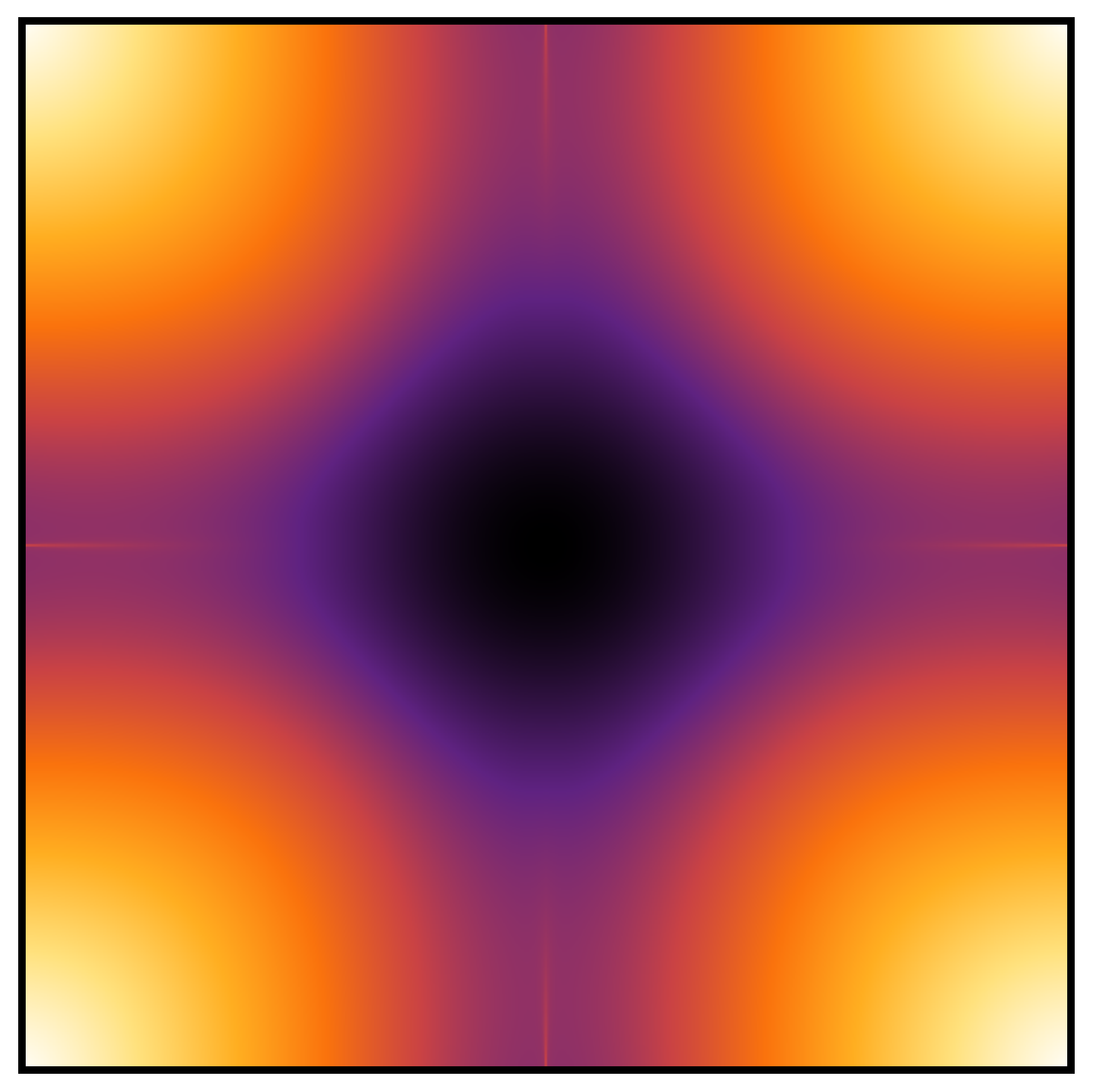}}\\
		\subfloat[]{\includegraphics*[trim={0cm 0cm -0cm -0cm},clip,width=0.3\textwidth]{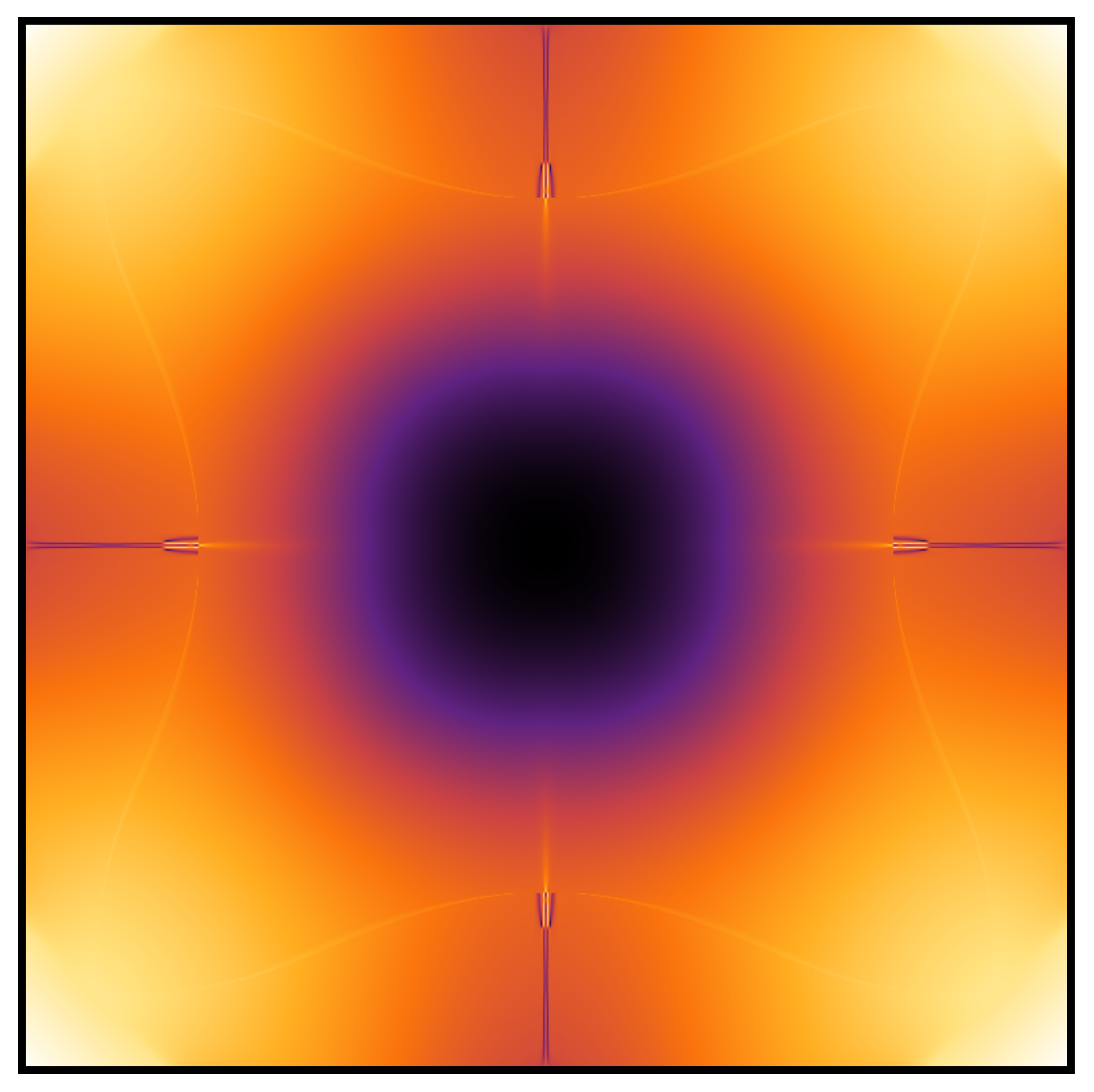}}\hspace{0.7cm}
		\subfloat[]{\includegraphics*[trim={0cm 0cm -0cm -0cm},clip,width=0.3\textwidth]{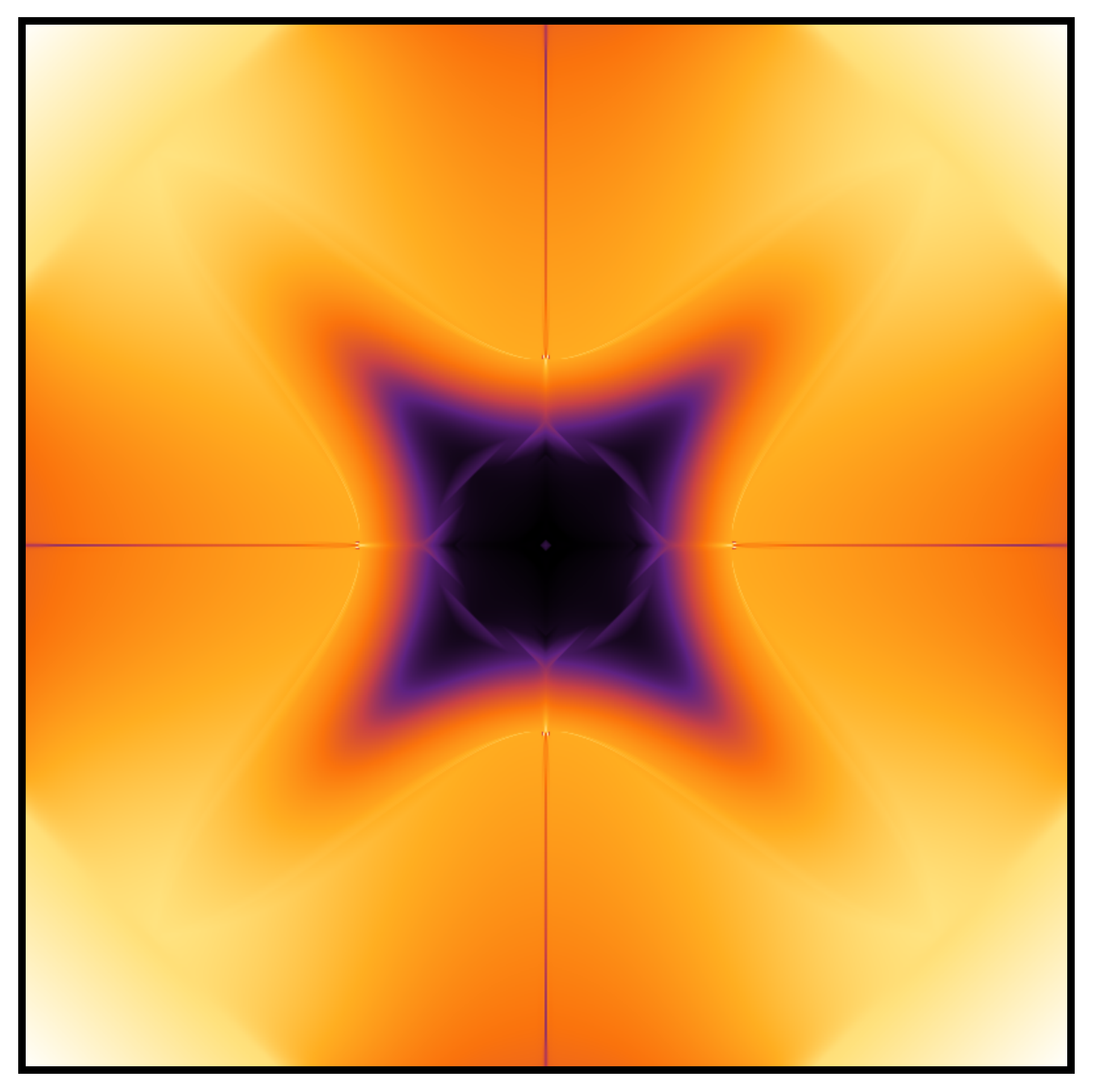}}\hspace{0.7cm}
		\subfloat[]{\includegraphics*[trim={0cm 0cm -0cm -0cm},clip,width=0.3\textwidth]{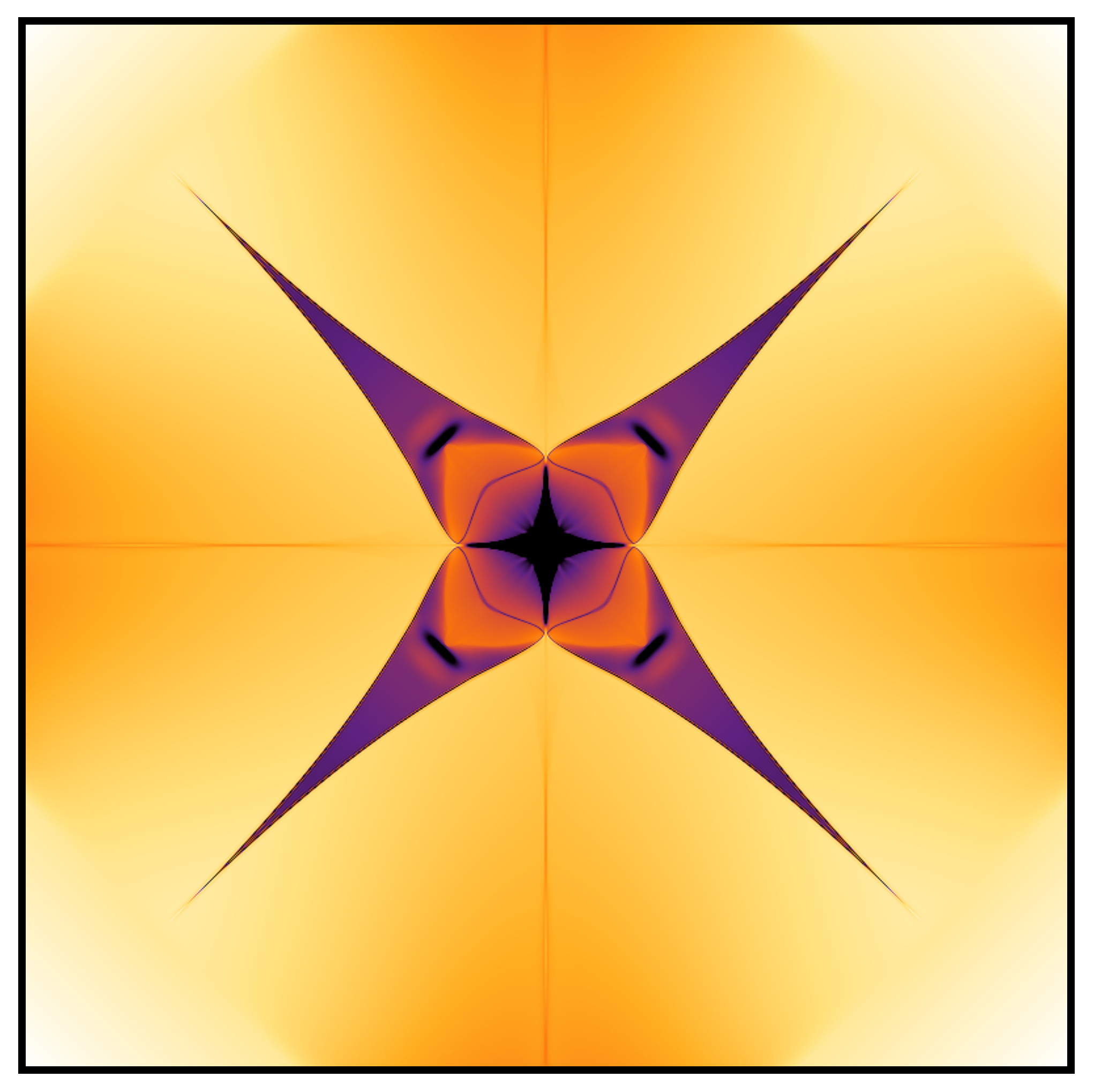}}
		\caption{(a) The optimal membrane problem for a pressure load $f = p\,\mathcal{L}^2\mres \O$. The trimmed averaged approximate optimal thickness $\bar{b}_h$ for Poisson ratios: (b) $\nu=0.3$, \ (c) $\nu=0.0$, \ (d) $\nu=-0.6$, \ (e) $\nu = -0.95$, \ (f) $\nu = -1.0$.}
		\label{fig:pressure}       
	\end{figure}
\end{example}

\begin{figure}[h]
	\centering
	\subfloat[]{\includegraphics*[trim={0cm 0cm -0cm -0cm},clip,width=0.3\textwidth]{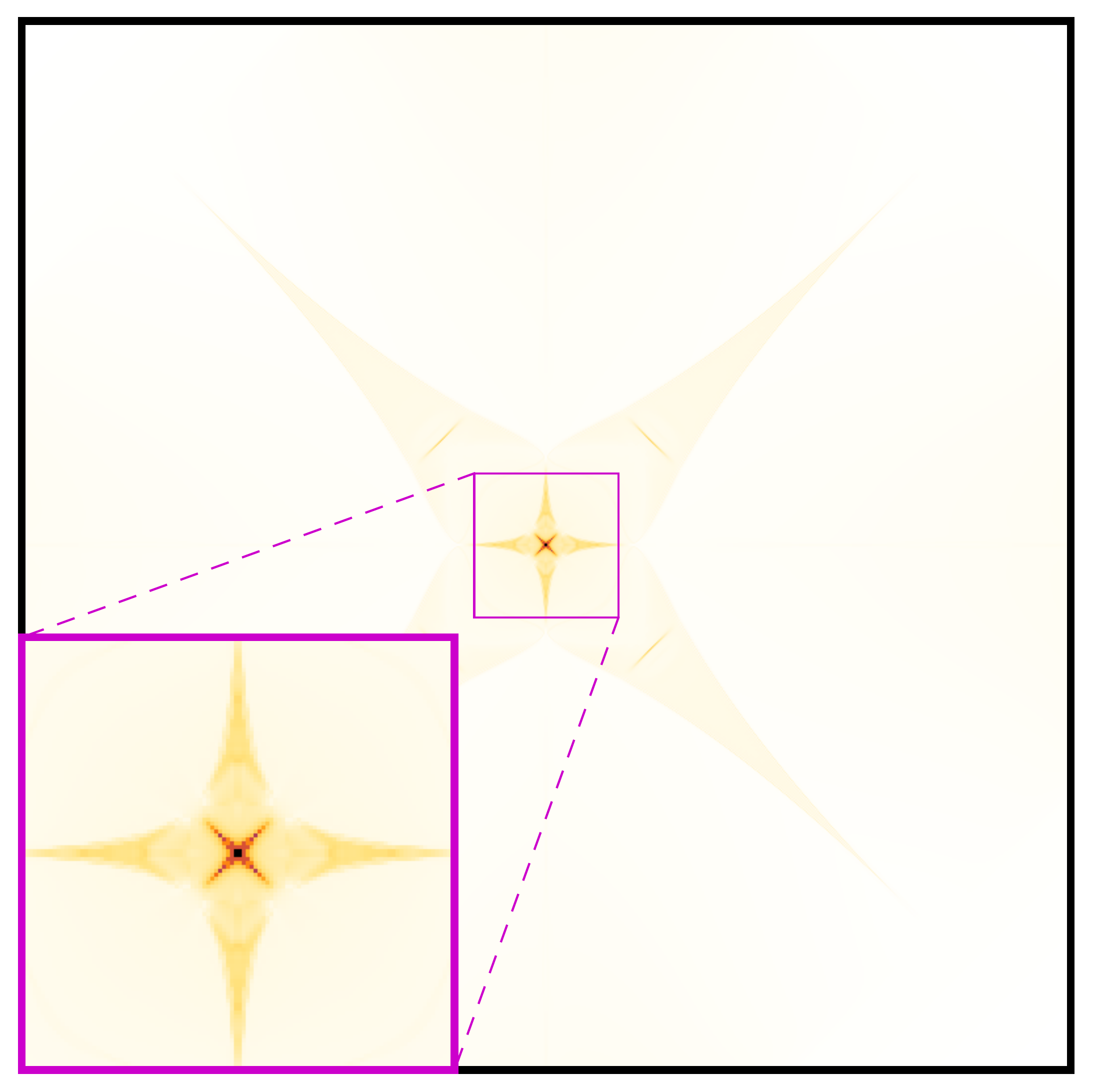}}\hspace{0.7cm}
	\subfloat[]{\includegraphics*[trim={0cm 0cm -0cm -0cm},clip,width=0.3\textwidth]{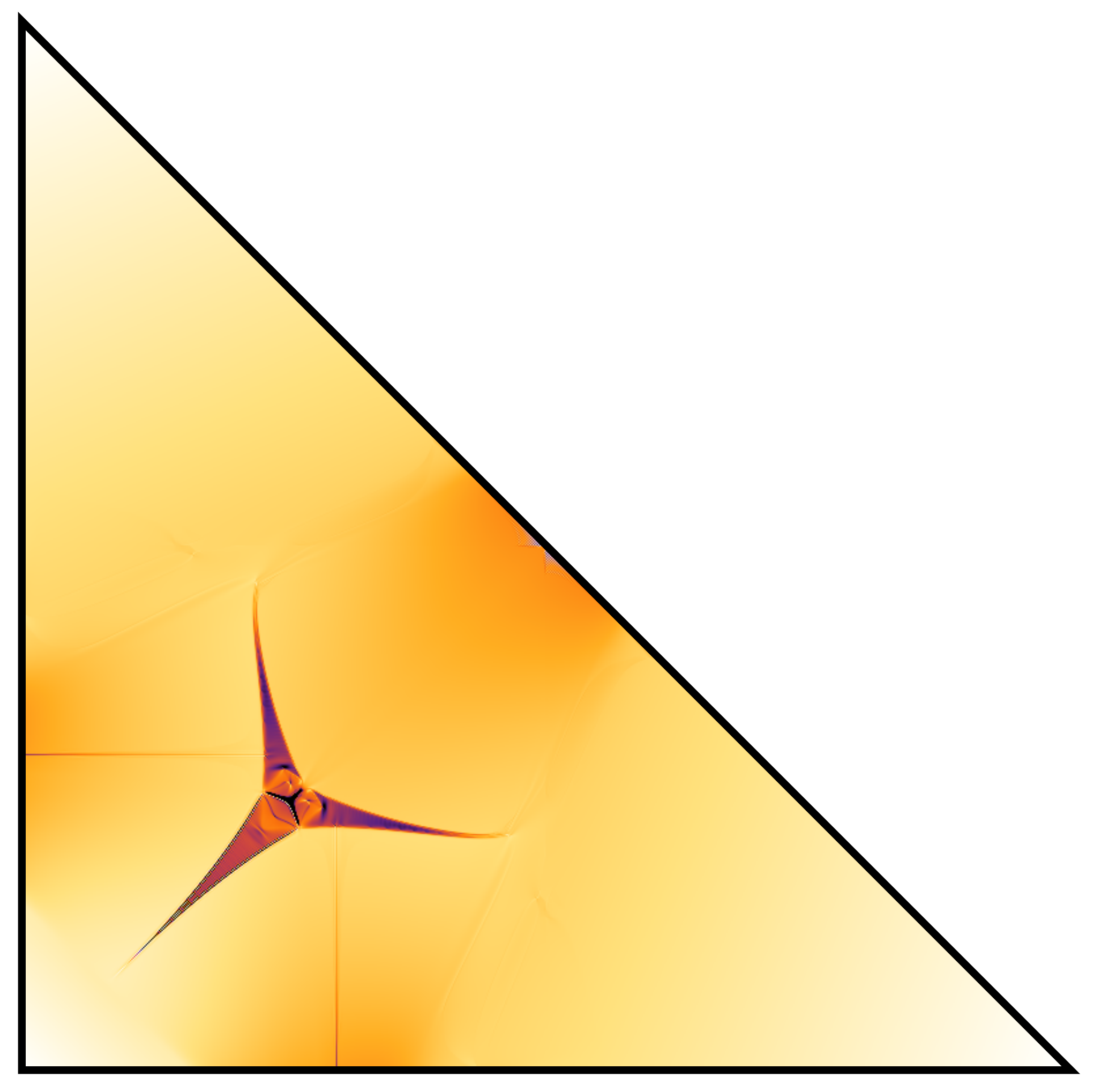}}\hspace{0.7cm}
	\subfloat[]{\includegraphics*[trim={0cm 0cm -0cm -0cm},clip,width=0.3\textwidth]{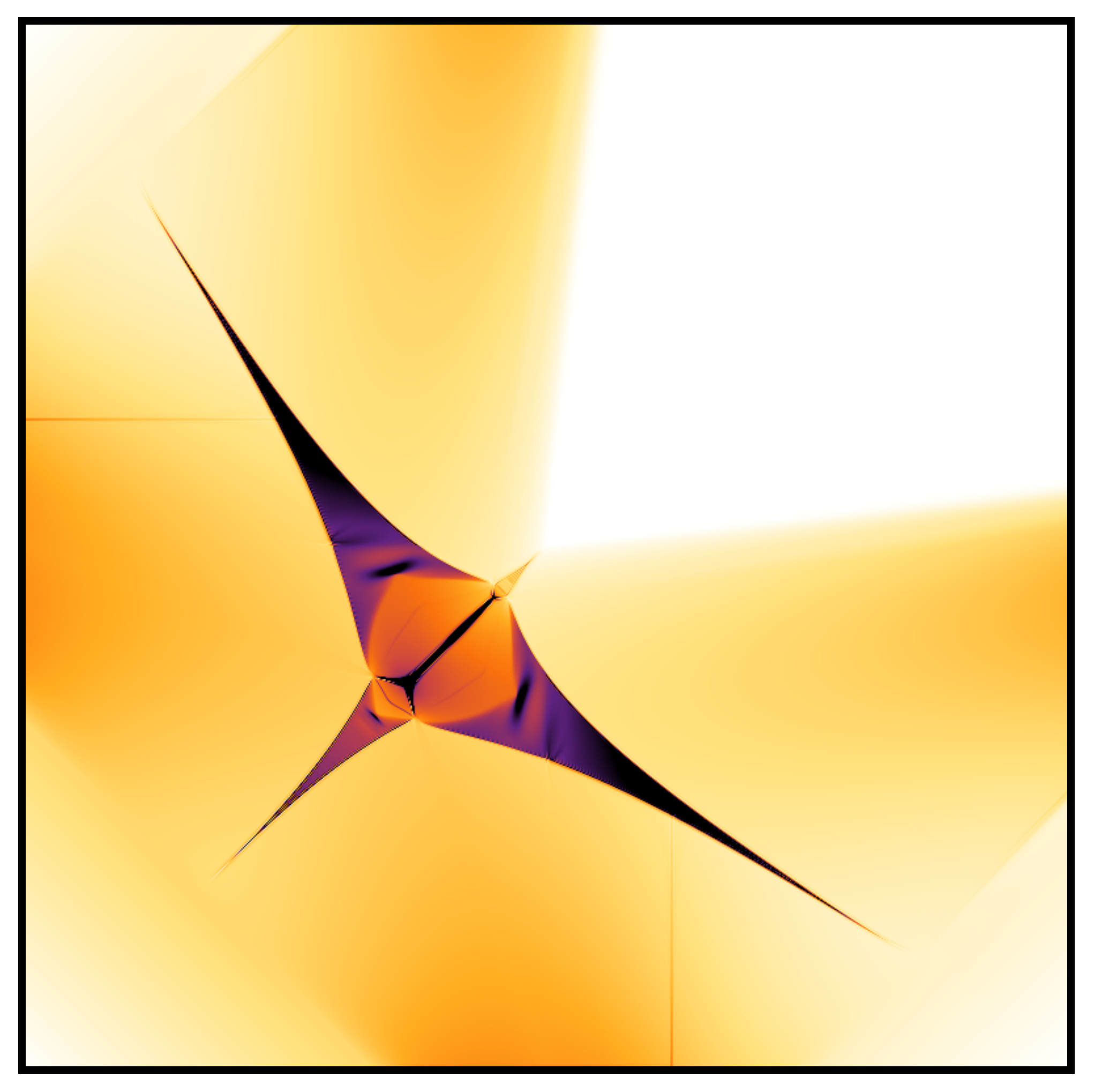}}
	\caption{(a)  The optimal membrane problem for a pressure load $f = p\,\mathcal{L}^2\mres \O$ revisited -- averaged approximate optimal thickness $\tilde{b}_h$ without trimming. The trimmed approximate optimal thickness $\bar{b}_h$ for: (b)  a triangular domain $\Omega$ and $f = p\,\mathcal{L}^2\mres \O$; (c) a square domain $\O$ and a pressure load restricted to a triangle $f = p \, \mathcal{L}^2 \mres  \mathrm{co}\big( \big\{ (0,0),(0,a),(a,0) \big\} \big)$.}
	\label{fig:fractals}       
\end{figure}

\section{The case of the Michell-like energy potential}
\label{sec:Michell}

Already in Remark \ref{rem:Michell} we have first noted the role of the Michell-like energy potential $j^\mathrm{M}$ and the related gauge $\rho^\mathrm{M}$, see \eqref{eq:nu=-1_Michell}. The relaxed potential $j_+^\mathrm{M}$ proved to be equal to the point-wise limit of $j_+^\nu$ for the Poisson ratio  $\nu$ approaching $-1$. Therefore, the solutions that were displayed in Figs \ref{fig:3_point_loads}(e),  \ref{fig:multiple_point_loads}(e), \ref{fig:diagonal_load}(e),  \ref{fig:pressure}(e), and Fig. \ref{fig:fractals} correspond to the pair of problems \ref{eq:PM}, \ref{eq:dPM} for $\rho = \rho^\mathrm{M}$. 

In the present section we will see that the pair \ref{eq:PM}, \ref{eq:dPM} for this particular choice of $\rho = \rho^\mathrm{M}$ is of significant interest and importance. Some extra mathematical implications will be exposed, as well as links and applications to different optimal design problems. It is worth recalling that the whole paper \cite{bolbotowski2022a} is devoted to the pair  \ref{eq:PM}, \ref{eq:dPM} for this special setting precisely.

\subsection{Reformulation as a free material design problem}
\label{ssec:FMD}

Thus far, the justification for solving the pair  \ref{eq:PM}, \ref{eq:dPM} for the Michell gauge $\rho = \rho^\mathrm{M}$ was based on the observation that this case corresponds to optimizing thickness of a membrane made of isotropic material with Poisson ration $\nu=-1$. Nonetheless, acquiring such an auxetic material in practice is difficult; e.g. a material of the chiral honeycomb microstructure  could be employed, see \cite{prall1997}. In what follows we will put forward another optimal design formulation for membranes that leads to the very same pair  \ref{eq:PM}, \ref{eq:dPM}  for $\rho = \rho^\mathrm{M}$. We will build upon a concept of the \textit{free material design}. Apart from controlling the material's density or the structure's thickness, this formulation allows to choose the Hooke operator $\mathscr{H}$  point by point. In the context of classical elasticity (rather than the membrane model) this design problem was originated in \cite{ringertz1993,bendsoe1994}. The theory of the free material design progressed in the series of works \cite{czarnecki2012,czarnecki2015,czarnecki2015,czubacki2015,czarnecki2017} where a connection to a pair of type  \ref{eq:PM}, \ref{eq:dPM} was found. The authors also generalized the formulation by allowing to choose $\mathscr{H}$ from a prescribed class $\mathscr{K}$ of Hooke operators: e.g. from the class of all isotropic operators. The mathematical foundation of the free material design problem, along with further generalizations, has been recently developed in \cite{bolbotowski2022b}.

By the \textit{set of Hooke operators} we will understand $\mathscr{L}_+(\Sdd)$ which is the closed convex cone of linear symmetric  positive semi-definite operators from $\Sdd$ to $\Sdd$. In turn, the \textit{Hooke fields $\lambda$} will be modelled by measures valued in this set, namely $\lambda \in \Mes\big(\Ob;\mathscr{L}_+(\Sdd)\big)$. Every such measure can be decomposed to $\lambda = \mathscr{H} \mu$, where $\mu \in \Mes_+(\Ob)$ and $\mathscr{H} \in  L^1_\mu\big(\Ob;\mathscr{L}_+(\Sdd)\big)$. That being said, we define the compliance of a membrane of  a Hooke field $\lambda$:
\begin{equation*}
	\widetilde\Comp(\lambda) :=- \inf \left\{  \int_\Ob \tilde{j}_+\bigl(\mathscr{H},\tfrac{1}{2} \, \nabla w \otimes \nabla w + e(u) \bigr) \, d\mu  -  \int_\Ob w\, df \ : \  (u,w) \in \D\bigl(\Omega;\R^2 \times \R\bigr) \right\}.
\end{equation*}
We can see that now the potential $ \tilde{j}_+: \mathscr{L}_+(\Sdd) \times \Sdd \to \R_+$ admits two arguments.  Conditions for the well posedeness of the free material design problem that one has to impose on such a potential can be found in \cite{bolbotowski2022b}. Here, we limit ourselves to taking  $ \tilde{j}_+(\mathscr{H},\xi) :=  \min_{\zeta \in \Sddp}  \tilde{j}(\mathscr{H},\xi +\zeta)$
for the initial potential corresponding to the linear Hooke's law:  $ \tilde{j}(\mathscr{H},\xi) = \frac{1}{2} \pairing{\mathscr{H}\xi,\xi}$. One can show that $ \tilde{j}_+$ is positively 1-homogeneous with respect to the first variable $\mathscr{H}$, and, as a result, the value of $\widetilde\Comp(\lambda)$ does not depend on the decomposition  $\lambda = \mathscr{H} \mu$, which justifies the above definition in the first place.

Let us next choose a \textit{class of admissible Hooke operators $\mathscr{K}$} being any closed convex cone contained in $\mathscr{L}_+(\Sdd)$. We are now in a position to formulate the free material design problem for an elastic membrane:
\begin{equation}
	\label{eq:FMD}\tag*{$(\mathrm{FMD}_\mathscr{K})$}
		\widetilde\Comp_\mathrm{min} =  \min \left \{	\widetilde\Comp(\lambda) \ : \ \lambda \in \Mes(\Ob;\mathscr{K}), \ \int_\Ob \tr\,\lambda \leq \Lambda_0 \right \}.
\end{equation}
In \cite{bolbotowski2022b}, in the context of classical elasticity, it was shown how the techniques pioneered in \cite{bouchitte2001} can be adapted to tackle the free material design problem. It suffices to say that the results hitherto developed for membranes -- concerning a design of positive measure $\mu$, which is a counterpart of what we have seen in \cite{bouchitte2001} -- can be generalized towards \ref{eq:FMD} in a manner very similar to \cite{bolbotowski2022b}. Accordingly, we shall state the result below without the proof.

In order to make a link to the Michell gauge $\rho^\mathrm{M}$ we focus on a particular class of admissible Hooke operators being the convex hull of \textit{uni-axial} operators:
\begin{equation*}
	\mathscr{K}_\mathrm{fib} =  \mathrm{co} \Big(\big \{ \mathscr{H} =a \, (\tau \otimes \tau) \otimes (\tau \otimes \tau) \, : \, a \geq 0, \ \tau \in S^1  \big\} \Big);
\end{equation*}
above $S^1$ is the unit sphere in $\Rd$.  Note that  $\tau \otimes \tau \in \Sdd$ is a matrix, whilst $(\tau \otimes \tau) \otimes (\tau \otimes \tau)$ is an operator. One may loosely think of $\mathscr{K}_\mathrm{fib}$ as of a set consisting of those Hooke operators which are achieved via \textit{fibrous microstructures}, i.e. by combining 1D bars aligned in several directions.

From \cite{bolbotowski2022b} we find that optimal choice of the Hooke operator $\mathscr{H}(x)$ can be done point-wisely based on the stress $\sigma(x)$. The optimality conditions are different for each set $\mathscr{K}$. For the set  $	\mathscr{K}_\mathrm{fib}$ such conditions may be found in \cite[Example 5.1]{bolbotowski2022b}. By exploiting those we eventually arrive at the following connection to the Michell setting of the problems  \ref{eq:PM}, \ref{eq:dPM}:
\begin{theorem}
		For a bounded domain $\Omega$  let us take a pair $(\hat\TAU,\hat\vartheta)$ that solves \ref{eq:dPM} for $\rho = \rho^\mathrm{M}$. Then, make the decomposition $\hat\TAU = \hat\sigma \hat\mu$ where $\hat\mu = (\rho^\mathrm{M})^0(\hat\TAU) = \tr\, \hat\TAU$ and, as a result, $\hat\sigma \in L^\infty_{\hat{\mu}}(\Ob;\Sddp)$ with $\tr\,\hat\sigma = 1$ $\ \hat\mu$-a.e.
		In addition, by $\hat{s}_\mathrm{I}, \hat{s}_\mathrm{II}:\Ob \to \R_+$ and $\hat\e_\mathrm{I},\hat\e_\mathrm{II}: \Ob \mapsto S^1$ denote the measurable functions point-wisely constituting the spectral decomposition $\hat\sigma(x) = \sum_{i  \in \{\mathrm{I},\mathrm{II}\}} \hat{s}_i(x) \, \hat\e_i(x) \otimes \hat\e_i(x)$.
		
		Then, the Hooke field
		\begin{equation*}
			\check{\lambda} = \check{\mathscr{H}} \check{\mu}, \quad \text{where} \quad \check{\mu} = \frac{2\Lambda_0}{Z}\, \hat{\mu}, \quad \check{\mathscr{H}}(x) = \sum_{i  \in \{\mathrm{I},\mathrm{II}\}} \hat{s}_i(x) \, \big(\hat\e_i(x) \otimes \hat\e_i(x)\big) \otimes  \big(\hat\e_i(x) \otimes \hat\e_i(x)\big) \ \ \text{for $\mu$-a.e. $x$,}
		\end{equation*}
		solves the Free Material Problem \ref{eq:FMD} for $\mathscr{K} = \mathscr{K}_\mathrm{fib}$.
\end{theorem}

\noindent We can readily see that the pair of problems \ref{eq:PM}, \ref{eq:dPM} with the Michell gauge $\rho = \rho^\mathrm{M}$ arises not only for designing an isotropic membrane with $\nu=-1$, but also corresponds to optimizing the material anisotropy in the class of fibrous microstructures. 

\subsection{The two-point condition and approximation via a system of strings}
\label{ssec:two_point}

In Examples \ref{ex:3_point_loads} and \ref{ex:multiple_point_loads}, where point-loads $f = \sum_{i=1}^N P\, \delta_{x_i}$ were considered, the numerical simulations for $\nu =-1$, that is, for the Michell gauge $\rho = \rho^\mathrm{M}$, have indicated that in this case optimal solutions $\check{\mu}$ are absolutely continuous with respect to the Hausdorff measure $\Ha^1$ restricted to a graph, cf. Fig. \ref{fig:3_point_loads}(e) and Fig. \ref{fig:multiple_point_loads}(e). In mechanical terms, the optimal membrane degenerates to a system of strings. We shall now lay out the extra mathematical structure of the problems \ref{eq:PM}, \ref{eq:dPM} for the particular choice $\rho = \rho^\mathrm{M}$ that explains this phenomenon. All the results to follow in this subsections are directly due to \cite{bolbotowski2022a}, and therein the reader is referred for the details and proofs.

Throughout the rest of this and next subsection we shall assume that the bounded domain $\O$ is convex. We start by analyzing the constraint in the problem \ref{eq:PM} for  $\rho = \rho^\mathrm{M}$. Thanks to the formula for $\rho_+^\mathrm{M}$ in \eqref{eq:est_rhoM}, one can easily show that this constraint is equivalent to enforcing  $\frac{1}{2}  \, \nabla w \otimes \nabla w + e(u)  \preceq \mathrm{I}_2 $ in  $\Ob$, where $\mathrm{I}_2$ is the $2 \times 2$ identity matrix, and the inequality is meant in the sense of the induced quadratic forms. Such constraint can be further rewritten as a two-point inequality (of somewhat similar form to the one exposed for the Michell problem in \cite{bouchitte2008}), so that eventually we have the equivalence below:
\begin{equation}
	\label{eq:two-point_ineq}
	\rho_+^\mathrm{M}\Big(\frac{1}{2}  \nabla w \otimes \nabla w + e(u) \Big)  \leq 1 \ \ \text{in $\Ob$}  \quad \ \Leftrightarrow \quad  \frac{1}{2} \big(w(x)-w(y) \big)^2 + \pairing{u(x)-u(y),x-y} \leq \abs{x-y}^2 \quad \forall\,(x,y) \in \Ob \times \Ob
\end{equation}
Using the two-point inequality in \ref{eq:PM} as the constraint (rather than the differential form) paves the way to a dual to \ref{eq:PM} that is different than \ref{eq:dPM}:
\begin{equation}
	\label{eq:dPMalt}\tag*{$(\mathscr{P}^*)$}
	\inf\left\{ \int_{\Ob \times \Ob} \abs{x-y} \Big( 1 + \tfrac{1}{2} \big(\tfrac{d\pi}{d\Pi}\big)^2 \Big) \Pi(dxdy) \, : \,  (\Pi,\pi) \in \Xi(\O,f)  \right\} 
\end{equation}
Quite similarly to the Monge-Kantorovich optimal transport problem \cite{villani2003}, above we are looking for coupling measures, one positive and one signed. The set  $\Xi(\O,f)$ entails the following constraints:
\begin{equation*}
	 \Xi(\O,f)  = \left\{ (\Pi,\pi) \in \Mes(\Ob \times \Ob;\R_+ \times \R) \, : \, 
	 \def\arraystretch{1.3} 
	  \begin{array}{ll} 
	 	\int_{\Ob \times \Ob} \big\langle \phi(x) - \phi(y) ,\tfrac{x-y}{\abs{x-y}} \big\rangle \Pi(dxdy)  = 0   \quad & \forall\, \phi \in \D(\O;\Rd), \\
	 	\int_{\Ob \times \Ob} \big( \varphi(x) - \varphi(y) \big) \pi(dxdy) = \int_\Ob \varphi \,df \quad & \forall\, \varphi \in \D(\O;\R) 
	 	\end{array}
	  \right\}.
\end{equation*}
Explaining the relation between such pairs $(\Pi,\pi)$ and the pairs $(\TAU,\vartheta)$ sought in \ref{eq:dPM} is in order. For each pair of distinct points $x,y \in \Ob$ we define the 1D matrix- and vector-valued measures, respectively, $(\TAU^{x,y},\vartheta^{x,y}) \in \Mes(\Ob;\Sddp \times \Rd)$:
\begin{equation*}
	 \TAU^{x,y} := \tau^{x,y} \otimes \tau^{x,y} \ \Ha^1 \mres [x,y], \qquad \vartheta^{x,y} := \tau^{x,y} \, \Ha^1 \mres [x,y], \qquad	\tau^{x,y} := \frac{x-y}{\abs{x-y}}.
\end{equation*}
The measure $\TAU^{x,y}$ models a string $[x,y]$ undergoing a unit uni-axial tensile force, whilst  $\vartheta^{x,y}$ corresponds to a unit transverse force in such string. With respect to any coupling measures $(\Pi,\pi) \in \Mes(\Ob \times \Ob;\R_+ \times \R)$ we integrate the measures above, thus defining $(\TAU_\Pi,\vartheta_\pi) \in \Mes(\Ob;\Sddp \times \Rd)$ as follows:
\begin{equation}
	\label{eq:integration_formulas}
	\TAU_\Pi(B) := \int_{\Ob \times \Ob} \TAU^{x,y}(B) \,\Pi(dxdy), \quad \vartheta_\pi(B) := \int_{\Ob \times \Ob} \vartheta^{x,y}(B) \,\pi(dxdy) \quad \text{for every Borel set $B\subset \Rd$}.
\end{equation}
It is straightforward to show that for any pair $(\Pi,\pi) \in \Mes(\Ob \times \Ob;\R_+ \times \R)$  the equivalence below holds true:
\begin{equation*}
	(\Pi,\pi) \in  \Xi(\O,f) \qquad \Leftrightarrow \qquad -\DIV\, \TAU_\Pi = 0, \ \ -\dive \, \vartheta_\pi = f  \quad \text{in $\O$}.
\end{equation*}
We can deduce that a pair $(\Pi,\pi)$ represents a \textit{generalized string system}, in which the axial tensile forces are modelled by $\Pi$, whilst $\pi$ describes the transverse forces. Generalized, since the support of $(\Pi,\pi)$ might be infinite, which then models a membrane of a fibrous-like material.

By using a rather subtle duality argument one can show that the duality gap between \ref{eq:PM} and \ref{eq:dPMalt} vanishes, see \cite[Theorem 3.18]{bolbotowski2022a}: 
\begin{theorem}
	\label{thm:strings}
	For any bounded convex domain these inequalities hold true:
	\begin{equation*}
		Z = \max \overline{\mathcal{P}} = \min \mathcal{P}^* = \inf \mathscr{P}^*.
	\end{equation*}
	As a consequence, any minimizing sequence $(\Pi_h,\pi_h)$ for the problem \ref{eq:dPMalt} furnishes a minimizing sequence of generalized string systems $(\TAU_{\Pi_h},\vartheta_{\pi_h})$ for the problem \ref{eq:dPM}.
\end{theorem}
\begin{figure}[h]
	\centering
	{\includegraphics*[trim={0cm 0cm -0cm -0cm},clip,width=0.22\textwidth]{fig/3points_-10.pdf}}\hspace{0.4cm}
	{\includegraphics*[trim={0cm 0cm -0cm -0cm},clip,width=0.22\textwidth]{fig/multi_-10.pdf}}\hspace{0.4cm}
	{\includegraphics*[trim={0cm 0cm -0cm -0cm},clip,width=0.22\textwidth]{fig/diag_-10.pdf}}\hspace{0.4cm}
	{\includegraphics*[trim={0cm 0cm -0cm -0cm},clip,width=0.22\textwidth]{fig/pressure_-10.pdf}}\\
	\subfloat[]{\includegraphics*[trim={0cm 0cm -0cm -0cm},clip,width=0.22\textwidth]{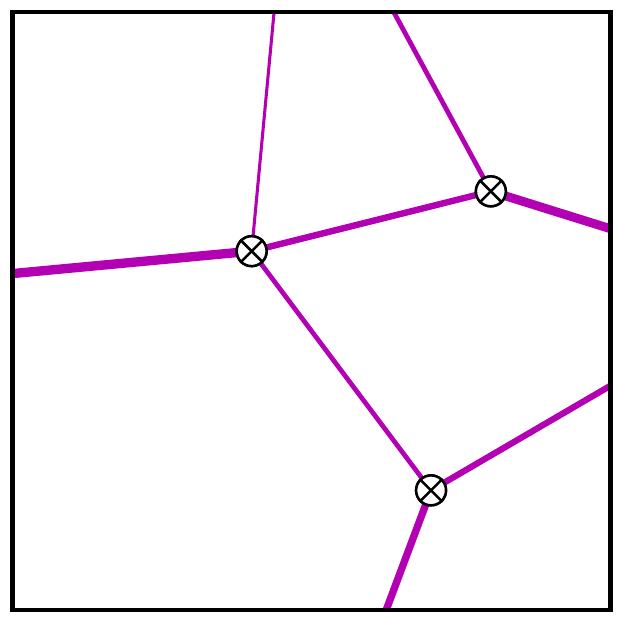}}\hspace{0.4cm}
	\subfloat[]{\includegraphics*[trim={0cm 0cm -0cm -0cm},clip,width=0.22\textwidth]{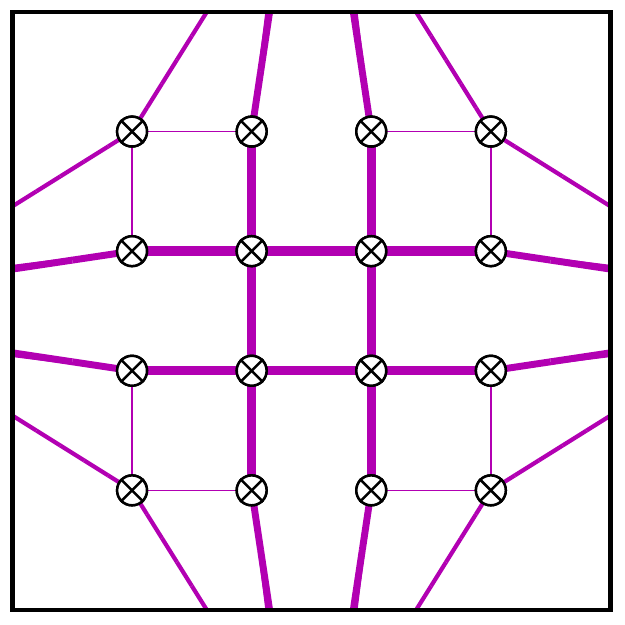}}\hspace{0.4cm}
	\subfloat[]{\includegraphics*[trim={0cm 0cm -0cm -0cm},clip,width=0.22\textwidth]{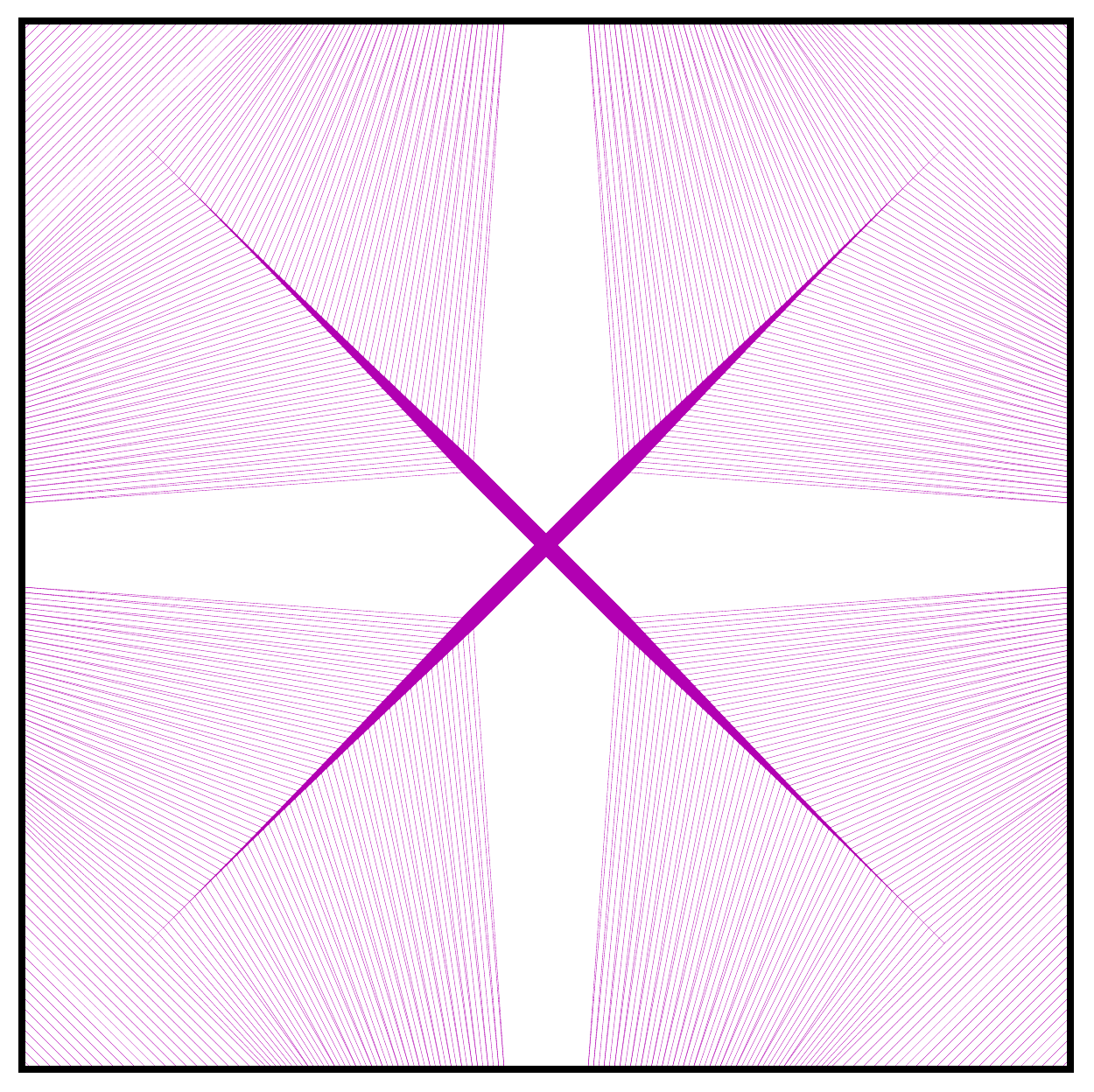}}\hspace{0.4cm}
	\subfloat[]{\includegraphics*[trim={0cm 0cm -0cm -0cm},clip,width=0.22\textwidth]{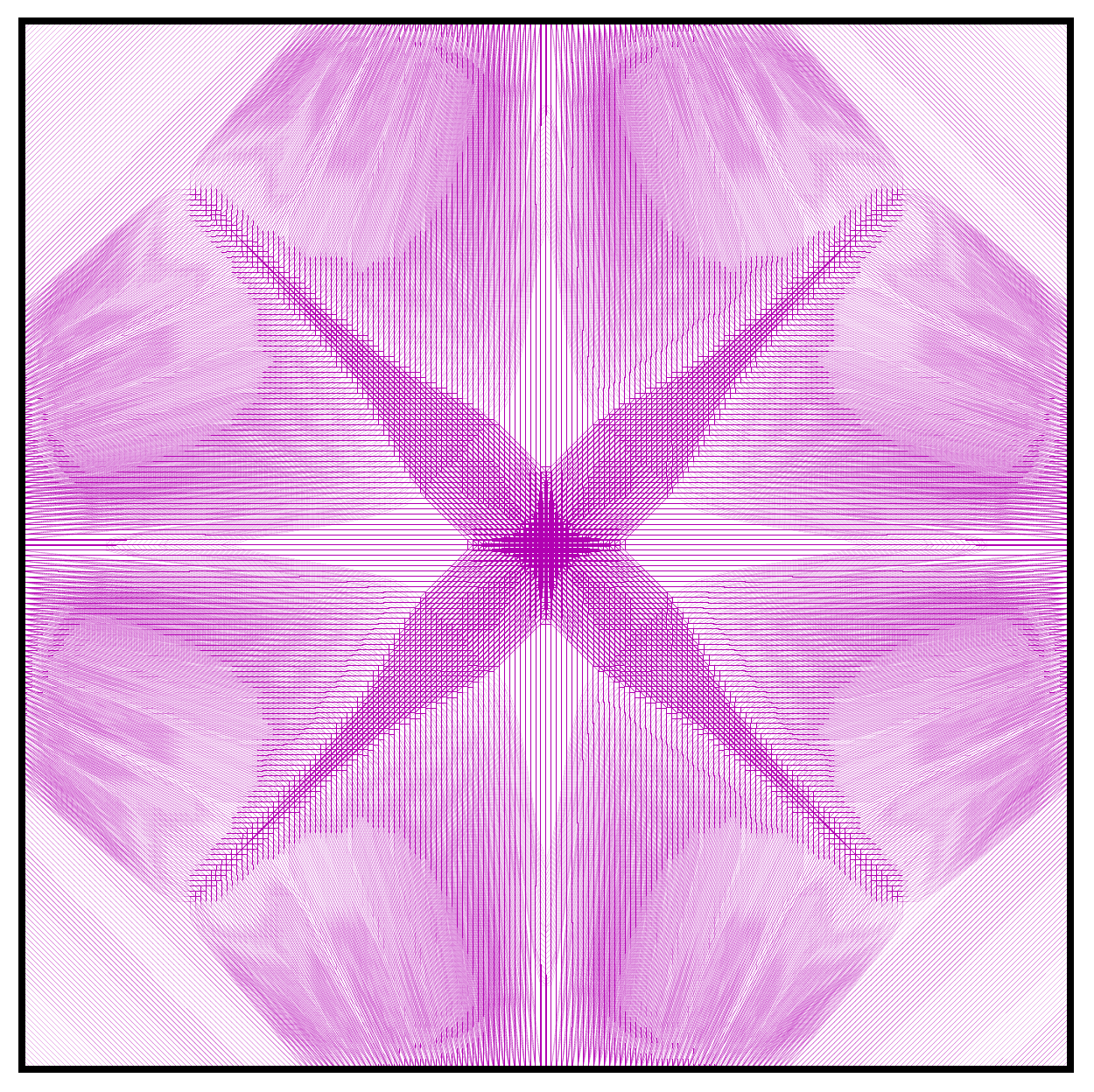}}
	\caption{The optimal membrane problem for $\nu=-1$ and various loads from Examples \ref{ex:3_point_loads}--\ref{ex:pressure} revisited. Top row: the trimmed averaged approximate optimal thickness $\bar{b}_h$ obtained via the finite element method; bottom row: approximate optimal string systems, the positive measures $\check{\mu}_h = \frac{2V_0}{\mathcal{Z}_h} \tr\, \TAU_{\hat\Pi_h}$ .}
	\label{fig:string_systems}       
\end{figure}
Comparing to the Monge-Kantorovich optimal transport, the optimal string problem \ref{eq:dPMalt} suffers from a clear drawback: in general its solutions does not exist, cf. \cite[Remark 4.6]{bolbotowski2022a} for a counter-example. Notwithstanding this, the alternative formulation \ref{eq:dPMalt} unlocks a competitive numerical scheme that is to deliver the sequence $(\Pi_h,\pi_h)$ mentioned in the second part of Theorem \ref{thm:strings}. The premise is simple enough as it also lies at the core of the discrete numerical methods for optimal transport, cf. \cite{merigot2021}. We propose an approximate optimal string system problem $(\mathscr{P}_h^*)$ which differs from \ref{eq:dPMalt} by an additional constraint on the support $\mathrm{sp} (\Pi,\pi) \subset X_h \times X_h$  where $X_h \subset \Ob$ is a finite set of nodes. In case of a quadratic domain $\O$, it is natural to choose a regularly spaced set $X_h$ with distances between adjacent nodes not exceeding $h$.  Then, the approximate problem $(\mathscr{P}_h^*)$  is finite dimensional and can be rewritten as a conic quadratic program. Just this numerical method was employed in \cite{bolbotowski2022a}, but more details on its implementation may be found in \cite{bolbotowski2022c}. In neither paper, however, the convergence of the numerical method has been examined.

In Fig. \ref{fig:string_systems} we revisit the optimal membrane problems examined  for the four different loads from Examples \ref{ex:3_point_loads}--\ref{ex:pressure}. We consider only the case of Poisson ratio $\nu =-1$ for which the design problem is equivalent to the pair \ref{eq:PM}, \ref{eq:dPM} for $\rho = \rho^\mathrm{M}$. The bottom row illustrates solutions of the problem \ref{eq:dPMalt} for a grid $X_h$ counting $201 \times 201$ nodes, which induces $816\,000\,000$ potential string connections. More accurately, the figures represent the measures $\check{\mu}_h = \frac{2V_0}{\mathcal{Z}_h} \tr\, \TAU_{\hat\Pi_h}$ where $\mathcal{Z}_h = \min \mathscr{P}^*_h$ and $\hat\Pi_h$ solves  \ref{eq:dPMalt} together with $\hat\pi_h$. Since the support of $\hat\Pi_h$ is finite, by the formulas \eqref{eq:integration_formulas} it is clear that $\check{\mu}_h$ has a density with respect to $\Ha^1$ measure restricted to the graph spanned by $X_h$. In the bottom row of Fig. \ref{fig:string_systems} the lines' thicknesses are proportional to this density exactly.

Figs \ref{fig:string_systems}(a,b) approve what we have suspected based on the finite element approximations: in the two cases where the load $f$ is finitely supported, there appears to exist an optimal solution being a finite system of strings, that is when $\rho = \rho^\mathrm{M}$. Accordingly, it is fair to deem the numerical scheme based on  $(\mathscr{P}_h^*)$ superior to the finite element method \ref{eq:dPMh} for this type of loads. When a 1D load as in  Fig. \ref{fig:string_systems}(c) both the approximations seem to deliver quality results. For the pressure load in Fig. \ref{fig:string_systems}(d), however, it is quite clear that the finite element method is more suited than the approximation through a system of strings.

The two results in Figs \ref{fig:string_systems}(a,b) beg a natural question: is it true that for loads that are supported on a finite set of points there always exists a finite string system solving \ref{eq:dPMalt}? The following observation provides a clue: in all the simulations performed by the present author the suboptimal string system $(\Pi_h,\pi_h)$ consists of strings that connect two types of points in $X_h$: either the points of loads application, namely in $\mathrm{sp}\,f$, or the points lying on the boundary $\bO$. The following conjecture can thus be made:
\begin{conjecture}
	\label{conj:finite_string_system}
	For a bounded convex domain $\O$ assume that the load  $f \in \Mes(\Ob;\R)$ is finitely supported. Then, solution of the optimal string system problem \ref{eq:dPMalt} exists. Moreover, there is at least one  solution $(\hat\Pi,\hat{\pi})$ satisfying:
	\begin{equation*}
		\mathrm{sp}(\hat\Pi,\hat\pi) \subset (\bO \cup \mathrm{sp} \,f) \times  (\bO \cup \mathrm{sp} \,f).
	\end{equation*}
\end{conjecture}

\noindent From \cite[Proposition 5.25]{bolbotowski2022a} we find that the above conjecture can be posed in a geometric form  that entails extending a monotone map defined on $\bO \cup \mathrm{sp} \,f$ to a maximal monotone map on whole $\Rd$, see the next subsection for more details.

\subsection{A link to finding maximal Wasserstein distance}
\label{ssec:maximal_MK}

As already mentioned  in the previous subsection, the problem \ref{eq:dPMalt} shows some resemblance to the Monge-Kantorovich problem, but it is not quite it. Nonetheless, a link between the pair \ref{eq:PM}, \ref{eq:dPM} for $\rho = \rho^\mathrm{M}$ and the classical optimal transport has been exposed in \cite{bolbotowski2022a} after all, and it involves maximizing the transport cost.

For simplicity we shall assume that in this subsection the load is a probability $f \in \mathrm{P}(\O)$, namely $f \geq 0$ and  $f(\O) = 1$. 
Let us consider a pair $(u,w) \in \D(\O;\Rd \times \R)$ that is feasible in the problem \ref{eq:PM} assuming that $\rho = \rho^\mathrm{M}$. We define the vector valued map $v: \Rd \to \Rd$ as below
\begin{equation*}
	v = \mathrm{id} -u
\end{equation*}
where $\mathrm{id}$ is the identity mapping on $\Rd$. Clearly we have $v(x) = x$ for any $x \in \Rd \setminus \O$. Since the pair $(u,w)$ is admissible in \ref{eq:PM}, we exploit the two-point inequality in \eqref{eq:two-point_ineq} to deduce that
\begin{equation}
	\label{eq:mono_prop}
	\pairing{v(x)-v(y),x-y} = \abs{x-y}^2 - \pairing{u(x)-u(y),x-y} \geq \tfrac{1}{2}\, \abs{w(x)-w(y)}^2 \geq 0 \qquad \forall\,(x,y) \in \Rd \times \Rd.
\end{equation}
We thus discover that $v$ is a (maximal) monotone map on $\Rd$, see \cite{alberti1999}. We can decide to change variables in \ref{eq:PM} from $(u,w)$ to $(v,w) \in C^\infty(\Rd;\Rd) \times \D(\O;\R)$, and, after imposing that $v$ is monotone, the two-point constraint can be written as $w(x)-w(y) \leq \sqrt{2 \pairing{v(x)-v(y),x-y}}$.  The issue now lies in establishing the form of this inequality constraint in the relaxed problem \ref{eq:relPM} where $v = \mathrm{id}-u \in BV_{\mathrm{loc}}(\Rd;\Rd)$ nay no longer be continuous, cf. Theorem  \ref{prop:char_Kro}. In \cite{bolbotowski2022a} this was resolved by exploiting the paper \cite{alberti1999} on maximal monotone maps, which, a priori, are multifunctions on $\R^d$. We define a convex subset of all such multifunctions:
\begin{equation*}
	\mbf{M}_\O  :=\Big\{ \mbf{v}: \Rd \to 2^{\Rd} \ :\ \mbf{v} \ \text{maximal monotone},\ 
	\mbf{v}(x)= \{x\} \quad \text{ $ \forall\, x \in \Rd\setminus \Ob$}\ \Big\}.
\end{equation*}
In \cite[Lemma 5.3]{bolbotowski2022a} we find that due to the  boundedness of $\O$ the set $\mbf{M}_\O$ is compact when endowed with the topology induced by the Kuratowski's convergence of the multifunctions' graphs. For each pair $(x,y) \in \Rd \times  \Rd$ we define 
\begin{equation}
	\label{def:cost}
	c_{\mbf{v}}(x,y) :=  \inf_{N \geq 2, \ \xi_i \in\Rd} \left\{ \sum_{i=1}^{N-1} \sqrt{2\pairing{\xi'_i - \xi'_{i+1},\xi_i-\xi_{i+1}}} \ :\  \xi_1=x, \ \xi_N=y, \ \  \xi'_i \in \mbf{v}(\xi_i)\right\}.
\end{equation}
One can prove that $c_\mbf{v}$ is a pseudo distance, i.e. it is non-negative, symmetric, and subadditive (it satisfies the triangle inequality). However, it may happen that $c_\mbf{v}(x,y) =0$ for $x \neq y$. Whenever $\mbf{v}(x) = \{v(x)\}$ for $v \in C^1(\Rd;\Rd)$, $c_\mbf{v}$ is the Riemannian distance for the metric tensor  $2\, e(v)$. By \cite[Theorem 5.10]{bolbotowski2022a} we arrive at the \textit{geometric variant} of the relaxed problem \ref{eq:relPM}:
\begin{equation}\tag*{$(\ov{\mathcal{P}}_{\rm geo})$}
	\label{eq:Pgeo}
	Z = \max_{ \mbf{v}\in  \mbf{M}_\O} \Bigg\{ \max_{ w\in   C_0(\Omega)} \ \left\{  \int_\O w\, df \, :\,  w(x) - w(y)  \le  c_{\mbf{v}}(x,y) \quad \forall\, (x,y) \in \Ob \times \Ob\right\} \Bigg\}
\end{equation}
where above $C_0(\O)$ is the linear space of continuous functions vanishing at $\bO$.  Let us stress that above the maximum is achieved, and the value function equals $Z = \max \overline{\mathcal{P} }= \min \mathcal{P}^*$. 

By the subadditivity of $c_\mbf{v}$, to the inner supremum in \ref{eq:Pgeo} we can apply the Rubinstein-Kantorovich duality (see \cite[Theorem 1.14]{villani2003}), and effectively we recover the Monge-Kantorovich optimal transport problem for the (pseudo) distance cost induced by a maximal monotone map $\mbf{v}$: 
\begin{equation}\tag*{$(\mathrm{MK}_\mbf{v})$}
	W_{c_{\mbf{v}}}(f,\bO) := \min \left\{\int_{\Ob\times\Ob} c_{\mbf{v}}(x,y) \,\gamma(dxdy)\ :\ \gamma \in \mathrm{P}(\O \times \bO), \ \ \pi_x \# \gamma = f  \right\}
\end{equation}
where $\#$ stands for the push-forward of the measure, i.e. $\pi_x \# \gamma(B) = \gamma\big(\pi_x^{-1}(B)\big) = \gamma(B \times \R^2)$ for every Borel set $B\subset \Rd$. The measure $\pi_x \# \gamma$ is thus  the first marginal of $\gamma$.
The lack of a condition on the second marginal of $\gamma$ is essentially due to the Dirichlet condition on $w$ in \ref{eq:Pgeo}. As a result, the value $W_{c_{\mbf{v}}}(f,\bO)$ earns the name of the \textit{Wasserstein distance between the source $f$ and the boundary $\bO$}. Let us note that the maximization with respect to $\mbf{v} \in \mbf{M}_\O$ remains. Ultimately, we find the following result:
\begin{theorem}
	Assume a bounded convex domain $\Omega$ and a probability $f \in \mathrm{P}(\O)$. Then,  we  have the equality
	\begin{equation}
		\label{maximalMK}
		Z  = \max\, \Big\{  W_{c_{\mbf{v}}}(f,\bO) \ : \  \mbf{v}\in \mbf{M}_\O \Big\}
	\end{equation}
	where $Z = \max \overline{\mathcal{P} }= \min \mathcal{P}^*$. In particular, the solution, i.e. the map $\hat{\mbf{v}}$ inducing "the worst Wasserstein distance" always exists.
	Moreover, for each such map $\hat{\mbf{v}}$ the function $\hat{w} \in C_0(\O)$ satisfying
	\begin{equation*}
		\hat{w}(x) = \mathrm{dist}_{c_{\hat{\mbf{v}}}}(x,\bO) := \inf_{y \in \bO} c_{\hat{\mbf{v}}}(x,y)
	\end{equation*}
	solves \ref{eq:relPM} together with some $\hat{u} \in BV(\Rd;\Rd)$.
\end{theorem}

The work \cite{bolbotowski2022a} leaves some important open questions regarding the link between the optimal transport problem $(\mathrm{MK}_{\hat{\mbf{v}}})$ for "the worst map" $\hat{\mbf{v}}$ and the optimal membrane problem. The main challenge lies in recasting measures $(\hat{\TAU},\hat{\vartheta})$ solving \ref{eq:dPM} upon an optimal transportation plan $\hat{\gamma}$ for $(\mathrm{MK}_{\hat{\mbf{v}}})$ (an analogous problem is solved for the optimal design of heat conductors, cf. \cite{bouchitte2001}). One of the issues concerns the regularity of geodesics  in the metric space $(\Rd, c_{\hat{\mbf{v}}})$, especially those connecting points $x \in \mathrm{sp}\,f$ and points $y \in \bO$ that minimize $c_{\hat{\mbf{v}}}(x,y)$ (such geodesics always exist, see \cite[Corollary 5.9]{bolbotowski2022a}). One can prove that if Conjecture \ref{conj:finite_string_system} turns out to be true, then for finitely supported source $f$ such geodesics are polygonal chains with their vertices lying in $\mathrm{sp}\, f \cup \bO$. For instance, in examples illustrated in Figs \ref{fig:string_systems}(a,b,c) (bottom row) such geodesics will be contained in the graphs composed of the strings (presumably, since the results are numerical simulations). For more details the reader is referred to \cite{bolbotowski2022a}.

\subsection{Towards optimal 3D vaults}
\label{ssec:vault}

This last subsection oughts to emphasize the practical importance of the pair of problems \ref{eq:PM}, \ref{eq:dPM} for $\rho = \rho^\mathrm{M}$. Although these problems are posed on a two-dimensional (horizontal) domain $\O$, in the paper \cite{bolbotowski2022c} by the present author it is proved that their solutions can be directly exploited to erect an optimal 3D \textit{vaulted structure} above $\O$. Vaulted structures, or \textit{vaults}, are three-dimensional bodies that concentrate on a single surface. Often made of stone or a different brittle material such as brick or glass, they are assumed to withstand compressive stresses only. Classically, their role is to transfer the vertical gravitational load to the supports, e.g. to the planar curve $\bO$. If we are to design a vault, then its surface is to be found. We shall assume that the vertical position of the load will change along with this surface.

As the starting point of theoretical research on optimal design of vaults we can regard the paper \cite{rozvany1979}. Therein, however, the authors worked with a quite stringent constraint. They assumed that the structure is an \textit{arch-grid}, i.e. it decomposes to planar arches going between two points on $\bO$. In the sequel \cite{rozvany1982} the authors clearly indicated a will to let go of this assumption, but no significant progress was made apart from a conjecture for a one point load. Recently, the optimal arch-grid problem was revisited in \cite{czubacki2020,dzierzanowski2021}. Then, in \cite{bolbotowski2022c} it was shown not only how to construct optimal vaults without additional constraints, but also that such vaulted structures are in fact optimal amongst all 3D structures  under the assumption that the load is \textit{vertically transmissible}. However, for those results to hold true we must assume the Michell-like energy potential.  Below we lay out the main idea behind \cite{bolbotowski2022c}.

For a bounded planar domain $\Omega \subset \R^2$, that is to be treated as horizontal, let us define a 3D cylinder with $\Omega$ as its base: $\Oh := \O \times (0,l_0)$ with $l_0>0$ as the height parameter. In the sequel we shall use the underline symbol $\ub{\,\cdot \ }$ to stress that the object is 3D. The set $\Oh$ will constitute the design domain for the optimal vault problem, whilst the planar curve  $\ub{\Upsilon}_0 := \bO \times \{0\}$ will serve as the supporting curve. Exactly as in the case of membranes, the loading data is 2D, namely $f \in \Mes(\Ob;\R_-)$. The non-positivity $f \leq 0$ is to enforce that the load is gravitational, hence pointing downwards. The target load will be 3D, and it is to be transmitted vertically from $f$. More precisely, we shall choose the vector-valued measure $\ub{\mathcal{F}} \in \Mes(\Obh;\R^3)$ in the class:
\begin{equation*}
	\mathscr{T}(\Oh,f) := \Big\{ \ub{\mathcal{F}} = \e_3\,\ub{f} \, : \,  \ub{f}\in\Mes\bigl(\Obh;\R_- \bigr),  \ \  \pi_{\Rd} \, \#\, \ub{f} = f \Big\}.
\end{equation*}
Above $\e_3$ stands for the unit vector orthogonal to the base plane $\R^2$ that contains $\O$ (the vector points upwards); the map $\pi_{\Rd}:\R^3 \to \Rd$ is the orthogonal projection onto $\Rd$.  One can easily check that $\mathscr{T}(\Oh,f)$ is  a weakly-* compact subset of $\Mes(\Obh;\R^3)$.

\begin{figure}[h]
	\centering
	\subfloat[]{\includegraphics*[trim={0cm 0.8cm 1cm 3cm},clip,width=0.32\textwidth]{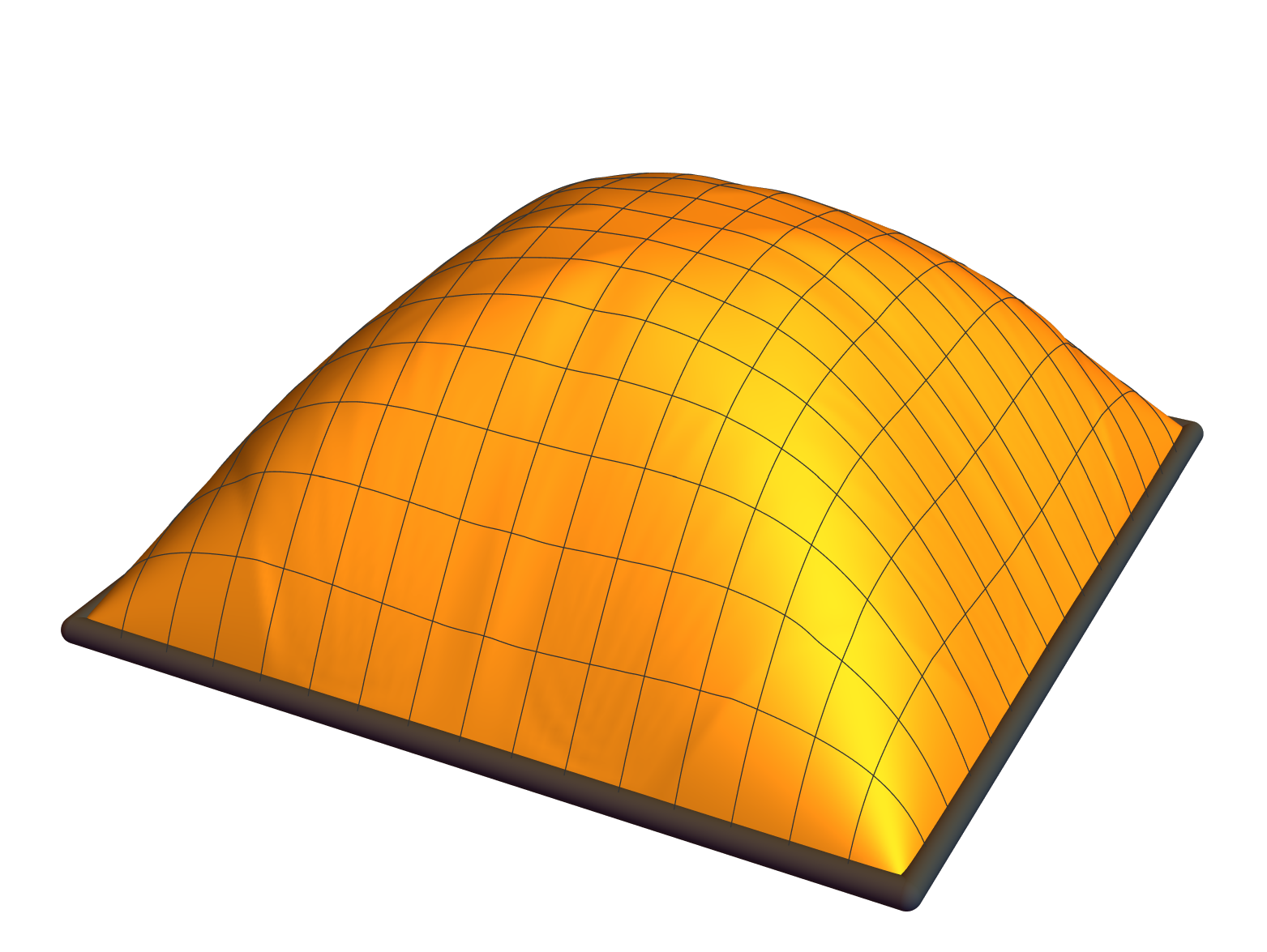}}\hspace{0.cm}
	\subfloat[]{\includegraphics*[trim={0cm 0.8cm 1cm 3cm},clip,width=0.32\textwidth]{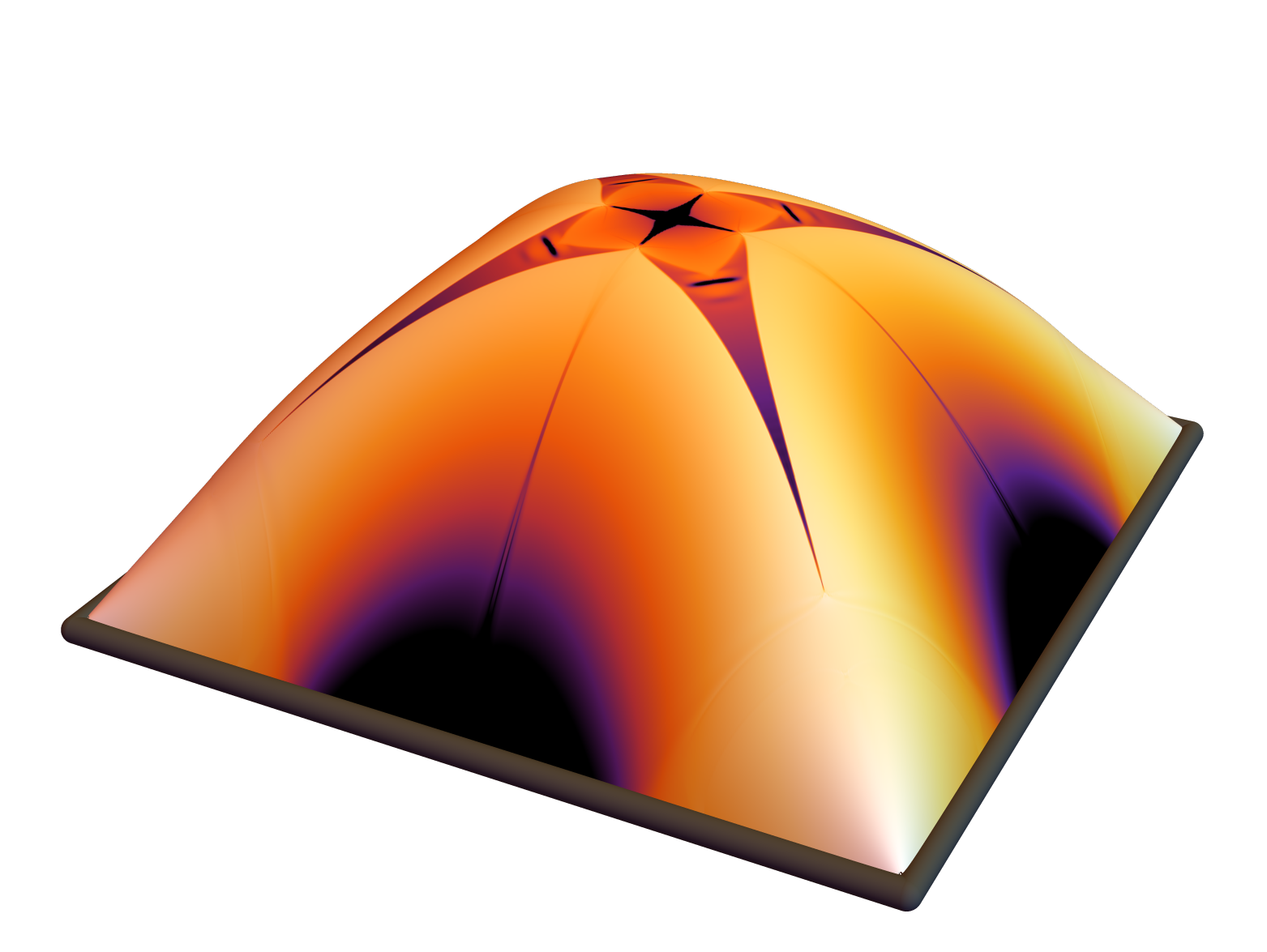}}\hspace{0.cm}
	\subfloat[]{\includegraphics*[trim={0cm 0.8cm 1cm 3cm},clip,width=0.32\textwidth]{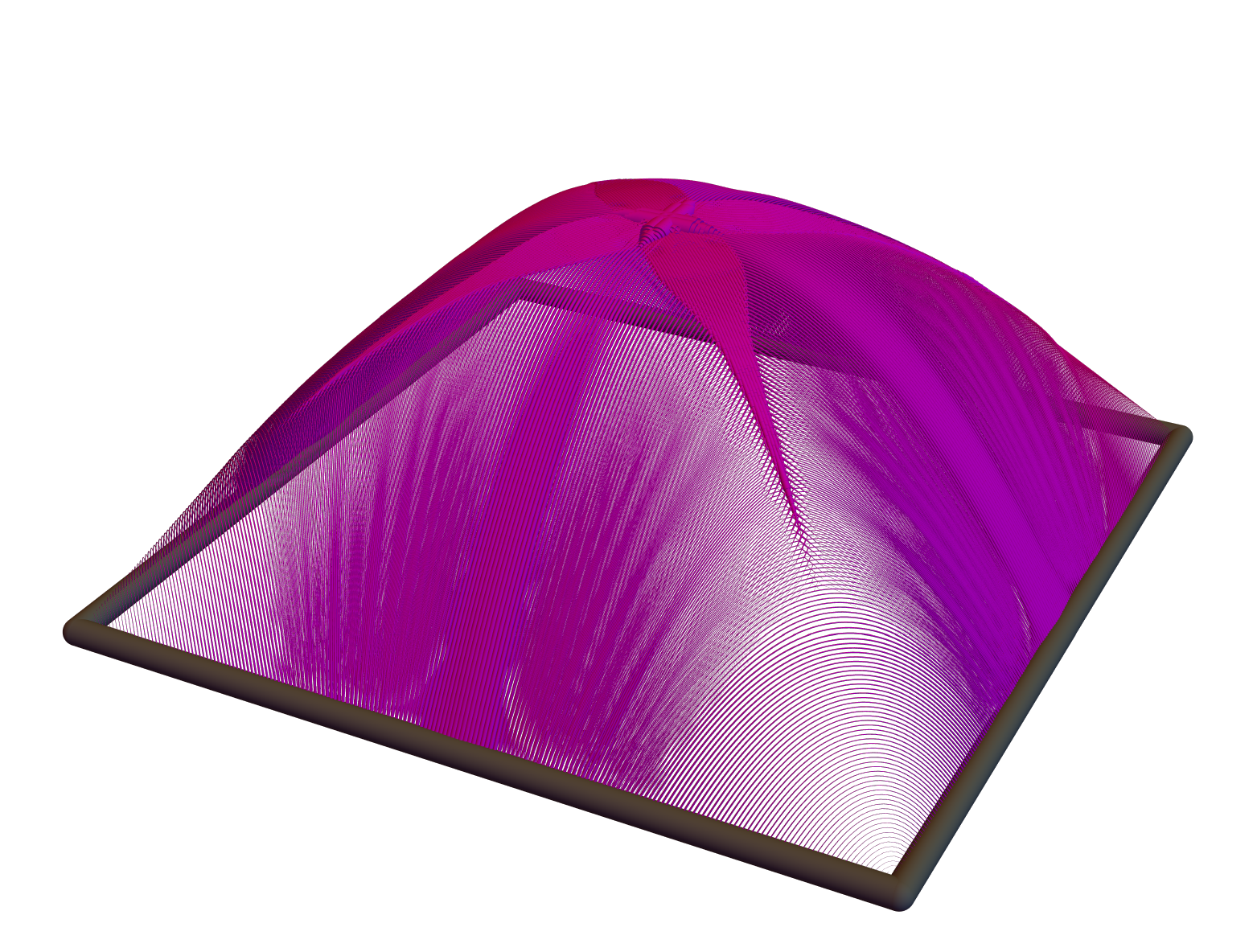}}
	\caption{Numerical approximations of an optimal vault. (a) An approximate elevation surface -- the graph of $z_h = -w_h/\sqrt{2}$; (b) an approximation by a shell of varying thickness $\check{\ub{\beta}}\,\!_h$; (c) an approximation by a 3D truss -- a grid-shell.}
	\label{fig:vaults}       
\end{figure}

The 3D structure occupying the domain $\Oh$ will be represented by a measure $\ub{\mu} \in \Mes_+(\Obh)$ modelling the material's distribution. As the compliance of such a structure $\ub{\mu}$, subject to a load $\ub{\mathcal{F}} \in \Mes(\Obh;\R^3)$ and with $\ub{\Upsilon}_0$ as the supporting curve, we will understand
\begin{equation}
	\label{eq:vault_comp}
	\mathscr{C}(\ub{\mathcal{F}},\ub{\mu}) := - \inf \bigg\{ \int_\Obh\, \ub{j}^\mathrm{M}_-\bigl( \ub{e}(\ub{u}) \bigr) \, d\ub{\mu}  - \int_\Obh \pairing{\ub{u},\ub{\mathcal{F}}}  \ : \  \ub{u} \in \D\bigl(\R^3 \setminus \ub{\Upsilon}_0;\R^3\bigr) \bigg\}.
\end{equation}
Above the 3D elastic energy potential $\ub{j}^\mathrm{M}_-: \mathcal{S}^{3 \times 3} \to \R_+$ comes from a 3D Michell-like potential $\ub{j}^\mathrm{M}(\ub{\xi}) = \frac{1}{2}  (\max_{i \in \{\mathrm{I},\mathrm{II},\mathrm{III}\}} \{ |\lambda_i(\ub{\xi})| \} )^2$. The modification $\ub{j}^\mathrm{M}_-$ is to rule out tensile stresses rather than the compressive ones (as it was the case for $j^\mathrm{M}_+$). As a result, the formula reads $\ub{j}^\mathrm{M}_-(\ub{\xi}) :=  \min_{\ub{\zeta} \in \mathcal{S}_-^{3 \times 3}} \ub{j}^\mathrm{M}(\ub{\xi} +\ub{\zeta})$, where $\mathcal{S}_-^{3 \times 3}$ is the cone of negative semi-definite matrices. We now propose a formulation of the \textit{relaxed Optimal Vault problem}: 
\begin{equation*}\tag*{$(\mathrm{OV})$}
	\label{eq:OV}
	\min_{\substack{\ub{\mathcal{F}} \in \Mes(\Obh;\R^3) \\ \ub{\mu} \in \Mes_+(\Obh)}} \Big\{	\mathscr{C}(\ub{\mathcal{F}},\ub{\mu})  \, : \, \ub{\mathcal{F}} \in \mathscr{T}(\Oh,f), \ \ \ub{\mu}(\Obh) \leq V_0  \Big\}
\end{equation*}
Let us stress that the definition of compliance in \eqref{eq:vault_comp} is more classical than the one defined for membranes in \ref{eq:nlcomp} in the following sense: here the strain field $\ub{\xi}: \Obh \to  \mathcal{S}^{3 \times 3} $ entering the energy potential is linear with respect to the displacement $\ub{u}$, that is $\ub{\xi} = \ub{e}(\ub{u}) = \frac{1}{2}(\ub\nabla \ub{u} + \ub\nabla^\top\! \ub{u})$. We say that the model of the optimized body is \textit{geometrically linear} (but not fully linear as the constitutive law induced by the potential $\ub{j}^\mathrm{M}_-$ is nonlinear).

With the comment above in mind, we can now observe that the problem \ref{eq:OV} is similar to the classical 3D Michell problem \cite{bouchitte2008,lewinski2019a} up to two aspects: (i) only compression is allowed, which is reflected by the potential $\ub{j}^\mathrm{M}_-$; (ii) the load $\ub{\mathcal{F}}$ is a design variable. A priori, there seems to be no reason why the solution should represent a vault, i.e. that $\ub{\mu}$ and $\ub{\mathcal{F}}$ are lower dimensional measures concentrating on a graph of some \textit{elevation function} $z\in C_0(\O;[0,l_0])$. This explains the term "\textit{relaxed} Optimal Vault problem": we allow the design to be any 3D structure, but we hope that vault-like shape will emerge by optimality. This would not be the case if it was not for the possibility to vertically transmit the load $\ub{\mathcal{F}}$. Moreover, numerical experiments conducted by the present author disapprove of this conjecture for potentials $\ub{j}$ other than $\ub{j}^\mathrm{M}$. Incapability of withstanding tensile stresses is also key here. To sum up, the work \cite{bolbotowski2022c} exploits the very particular  form of \ref{eq:OV} to show that vaults are indeed optimal.

The idea behind the construction to follow can be traced back to the work of Antoni Gaud\'{i} \cite{bergos1999}. We first take a planar membrane which we subject to a downward load $f \leq 0$. The membrane then deflects by $w: \Omega \to \R_-$. Next, we \textit{flip} the surface being the graph of $w$ and rescale it to our needs by a factor  $\alpha >0$. We arrive at an elevation function $z = - \alpha \,w$. The structure concentrating on its graph is capable of withstanding the load $f$ (transmitted vertically to this graph) via pure compressive stresses. This was Gaud\'{i}'s ingenious method  for shaping stone vaults. Roughly speaking, the following result shows that if the membrane is optimally designed and $\alpha$ is strategically chosen, then the obtained vault will be optimal as well:
\begin{theorem}
	\label{thm:vault}
	For a bounded domain $\Omega$ that is star-shaped with respect to a ball let us take solutions $(u,w)$ and $(\TAU,\vartheta)$ that solve problems \ref{eq:relPM} and, respectively, \ref{eq:dPM} for $\rho = \rho^\mathrm{M}$. Assume, in addition, that $u$, $w$ are Lipschitz continuous and $(u,w) = (0,0)$ on $\bO$. Let us define the elevation function $z \in \mathrm{Lip}(\Ob;\R)$ as follows:
	\begin{equation*}
		z := - \frac{1}{\sqrt{2}} \, w.
	\end{equation*}
	Then, the following 3D load and 3D material distribution below
	\begin{equation}
		\label{eq:unprojection}
		\check{\ub{\mathcal{F}}} = \ub{z} \, \#\, (\e_3\,f) ,    \quad  \check{\ub{\mu}} :=  \ub{z}\, \#  \Big( \tfrac{V_0}{Z} \vro^0(\TAU,\vartheta) \Big), \qquad \text{where} \quad \ub{z}(x) = \big(x,z(x)\big) \in \R^3 \quad \forall\, x \in \Rd,
	\end{equation}
	solves the optimal vault problem \ref{eq:OV} provided that $\Oh = \O \times (0,l_0)$ for $l_0 \geq \frac{1}{2}\,\mathrm{diam}(\O)$.
\end{theorem}

Apart from the theoretical gains, Theorem \ref{thm:vault} delivers efficient numerical schemes for finding optimal vaults. The numerical methods known from the literature, e.g. \cite{jiang2018}, require discretizing  the 3D design space $\Oh$, being, of course, computationally costly. Instead, we can now focus on numerical approximation of solutions of the 2D pair \ref{eq:relPM}, \ref{eq:dPM} -- either by the finite element method developed here or via system of strings like in \cite{bolbotowski2022a} -- whilst the passage to a 3D structure is analytical.

Let us discuss this algorithm on an example of uniform load: $f = - p \, \mathcal{L}^2 \mres \O$ for a square domain $\O = (0,a) \times (0,a)$. In Fig. \ref{fig:string_systems}(d) we have seen the simulations for such data acquired by both methods. In the sequel let us focus on the finite element approximation, for the other is explained in \cite{bolbotowski2022c}. One must remember, that approximation $(u_h,w_h)$ is generated automatically when running the MATLAB code. Accordingly, in Fig. \ref{fig:vaults}(a) we can see the graph $G_{z_h}$ of the function $z_h = - \frac{1}{\sqrt{2}}\, w_h$ (note that $z_h \geq 0$ since $w_h \leq 0$ due to the negative load $f$). According to Theorem \ref{thm:vault}, this graph approximates the optimal elevation surface on which the optimal vault is to concentrate. For the next step we must compute the approximation of the positive 2D measure: $ \tfrac{V_0}{Z_h} \vro^0(\TAU_h,\vartheta_h)$. Naturally, it will be an absolutely continuous measure with density $\check{\beta}_h = \tfrac{V_0}{Z_h} \, \vro^0(\sigma_h,q_h)$ which is different than the approximate optimal thickness $\check{b}_h = \tfrac{2\,V_0}{Z_h} \, (\rho^M)^0(\sigma_h) = \tfrac{2\,V_0}{Z_h} \, \tr\,\sigma_h$ whose heat-map could be observed at the top of Fig. \ref{fig:string_systems}(d). In fact, one can show that $\check{\beta}_h = \tfrac{V_0}{Z_h} \big(\tr\,\sigma_h +  \pairing{\nabla z_h \otimes \nabla z_h,\sigma_h} \big)$. For absolutely continuous measures the push forwards in \eqref{eq:unprojection} can be computed explicitly by the area formula, see \cite{ambrosio2000}. For instance,
\begin{equation*}
	\check{\ub{\mu}}\,\!_h = \check{\ub{\beta}}\,\!_h  \, \Ha^2 \mres G_{z_h},  \qquad  \text{where} \quad \check{\ub{\beta}}\,\!_h (\ub{x}) = \frac{\check{\beta}_h\big( \ub{z}_h^{-1}(\ub{x}) \big) }{J_{\nabla \ub{z}_h}\big( \ub{z}_h^{-1}(\ub{x}) \big)} \quad \text{for  \ \ $\Ha^2$-a.e. $\ub{x} \in G_{z_h}$}.
\end{equation*}
Above $\ub{z}_h= (\mathrm{id},z_h): \Ob \to \R^3$, the set  $G_{z_h} = \ub{z}_h(\O)$ is the Lipschitz graph of $z_h$ illustrated in Fig. \ref{fig:vaults}(a), and $J_{\nabla \ub{z}_h}$ is the Jacobian of $\ub{z}_h$. One can compute that $J_{\nabla \ub{z}_h}(x) = \sqrt{1+ \vert\nabla z_h(x)\vert^2}$ for a.e. $x \in \O$.

Ultimately, in Fig. \ref{fig:vaults}(b) we can see the heat-map of $\check{\ub{\beta}}\,\!_h $  plotted on the surface $G_{z_h}$ exactly. Thus, by using the finite element method for the pair \ref{eq:relPM}, \ref{eq:dPM} we arrived at an approximation of an optimal vault being a \textit{shell} of varying thickness $\check{\ub{\beta}}\,\!_h $. In turn, when departing from a numerical solution by a system of string visible at the bottom of Fig. \ref{fig:string_systems}(d), we end up with a 3D truss approximation of an optimal vault illustrated in Fig. \ref{fig:vaults}(c), see \cite{bolbotowski2022c} for more details on this construction. In the literature, such a "dense" truss which lies on a single surface is often referred to as a \textit{grid-shell}. 

\bigskip
\noindent\textbf{Acknowledgements}

This paper was partially prepared within the author's Ph.D. thesis entitled \textit{Elastic bodies and structures of the optimum form, material distribution, and anisotropy}. To his supervisors, Professor Tomasz Lewi\'{n}ski and Professor Piotr Rybka, the author gives many  thanks for their invaluable guidance.

The author also wishes to express his gratitude to Professor Guy Bouchitt\'{e} for fruitful discussions on the topic of this work that took place during a  three-week visit at Laboratoire IMATH, Universit\'{e} de Toulon in February of 2020.

Finally, the author would like to thank the National Science Centre (Poland) for the financial support and acknowledge the Research Grant no 2019/33/B/ST8/00325 entitled \textit{Merging the optimum design problems of structural topology and of the optimal choice of material characteristics. The theoretical foundations and numerical methods}.

\appendix
\renewcommand{\thesection}{\Alph{section}}
\setcounter{section}{0}

\section*{Appendix}

\section{Convex duality}
\label{ap:convex}

For a locally convex linear topological space $X$ and its dual $X^*$ by $\pairing{\argu,\argu}$ denote the duality pairing. For any extended-real function $g:X \to \Rb = [-\infty,\infty] $ we define its convex conjugate (or the Legandre-Fenchel transform) $g^*:X^* \to \Rb$ and the convex biconjugate $g^{**}:X \to \Rb$ as below:
\begin{equation*}
g^*(x^*) := \sup_{x \in X} \Big\{ \pairing{x,x^*} - g(x) \Big\}, \qquad g^{**}(x) := \sup_{x^* \in X^*} \Big\{ \pairing{x,x^*} - g^*(x^*) \Big\}.
\end{equation*}

For any $x \in X$  such that $g(x)$ is finite, by the subdifferential  $\partial g (x)$  we understand the set of those $x^* \in X^*$ that satisfy $g(z)  \geq g(x) + \pairing{z-x, x^*}$ for all $z \in X$. It is straightforward to show the following equivalence:
\begin{equation}
	\label{eq:subdiff}
	x^* \in \partial g(x)\qquad \Leftrightarrow \qquad  \pairing{x,x^*} = g(x)+g^*(x^*).
\end{equation}

Introducing another locally convex linear topological space $Y$ together with its dual $Y^*$,  we consider a linear continuous operator $\Lambda:X \to Y$, while the \textit{adjoint} $\Lambda^*:Y^* \to X^*$ is the uniquely defined linear continuous operator that for each pair $x \in X$, $y^* \in Y^*$ satisfies $\pairing{x,\Lambda^* y^*} = \pairing{\Lambda x, y^*}$. 
The classical tool for examining duality between two variational problems can be readily stated:
\begin{theorem}[\textbf{Classical duality}]
	\label{thm:duality_classical}
	For $X,Y$ and $\Lambda$ as above let $\Phi: X \to \Rb$, $\Psi:Y \to \Rb$ be convex lower semi-continuous functionals, while by $\Phi^*:X^* \to \Rb$, $\Psi^*:Y^* \to \Rb$ denote their convex conjugates.
	
	If there exists $x_0 \in X$ such that $\Phi(x_0) < + \infty$ and $\Psi$ is continuous at $\Lambda x_0$,  then
	\begin{equation}
	\label{eq:sup=inf}
	\mathscr{Z} := \sup _{x \in X} \Big \{ - \Phi (x) - \Psi (\Lambda x)   \Big \} = \inf
	_{y^* \in Y ^*} \Big \{ \Phi ^* (- \Lambda ^* y^*)+\Psi ^* (y^*) \Big \},
	\end{equation}
	and the minimum on the right hand side is achieved provided that $\mathscr{Z} < \infty$.
\end{theorem}
\begin{proof}
	See Theorem 4.1, equation (4.18) and condition (4.21) in Chapter III of \cite{ekeland1999}.
\end{proof}

\section{Algebraic formulation for the finite element approximation of the pair $\Prob$, $\dProb$}
\label{app:algebraic}

Assuming that $\sigma \in \Sdd$  is (strictly) positive definite, the result to be stated below is a particular case of what in the literature is known as the \textit{Schur complement lemma}, see \cite[Lemma 4.2.1]{ben2001}. Below we give its independent proof under a weaker assumption that $\sigma \in \Sddp$, namely that $\sigma$ is positive semi-definite.

\begin{proposition}[\textbf{Schur complement}]
	\label{prop:pos_semdef_block_mat}
	For a matrix $\sig \in \Sdd$, a vector $q \in \Rd $, and a number $c \in \R$ the following conditions are equivalent:
	\begin{enumerate}[label={(\roman*)}]
		\item the $3 \times 3$ matrix below is positive semi-definite:
		\begin{equation}
			\label{eq:3x3_mat}
			 \begin{bmatrix}
			\,\sig  &  \tfrac{1}{\sqrt{2}}\,q\,\\
			\tfrac{1}{\sqrt{2}}\,q^\top\, & c\,
			\end{bmatrix};
		\end{equation}
		\item there hold:
		\begin{equation*}
			\begin{cases}
				\sig \in \Sddp,\\
				q \in \IM \, \sig,\\
				c \geq \jsq,
			\end{cases}
		\end{equation*}
		where the definition of $\jsq$ for an arbitrary semi-definite matrix $\sigma \in \Sddp$ is given in \eqref{eq:jsq}.
	\end{enumerate}
\end{proposition}
\begin{proof}
	By definition, the statement (i) is true if and only if the quadratic form generated by the matrix \eqref{eq:3x3_mat} is non-negative for any vector $\eta \in \R^3$. Let us consider vectors  $\eta = [\,\frac{1}{\sqrt{2}}\theta^\top \ -\!t\,]^\top$ where $\theta \in \R^2$, $t \in \R$. Then (i) can be rewritten as:
	\begin{align}
		0 \leq \inf_{\theta\in \R^2, \ t\in \R} \Big\{ \tfrac{1}{2} \pairing{\sig \theta,\theta} - \pairing{t q,\theta} + c\, t^2 \Big\} &=  \inf_{ t\in \R}
		\begin{cases} - \tfrac{1}{2} \pairing{\sig^{-1}(tq),(tq)} + c\, t^2  & \text{if $\sig\in \Sddp$ \ and \  $\q \in  \IM(\sig)$,   }\\
		\nonumber
		-\infty & \text{otherwise}
		\end{cases}\\
		&=
		\begin{cases} 0  & \text{if $\sig\in \Sddp$, \ and \  $\q \in  \IM(\sig)$, \ and \ $c\geq \jsq$,   }\\
		\nonumber
		-\infty & \text{otherwise.}
		\end{cases}
	\end{align}
	Clearly, the inequality above holds true if and only if (ii) is satisfied.
\end{proof}

\begin{proof}[Proof of Theorem \ref{thm:duality_for_FEM}]
	In the proof we will draw upon \cite[Section 2.5]{ben2001} where a summary on conic duality may be found. If in the mutually dual conic problems $(\mathrm{Pr})$ and $(\mathrm{Dl})$ therein we make the following substitutions
	\begin{equation*}
	\mbf{x} = \!\begin{bmatrix}
	\mbf{r}^0 \\ \bm{\tau}^{11}\\ \bm{\tau}^{22} \\ \bm{\tau}^{33} \\ \bm{\tau}^{12}\\ \bm{\tau}^{13} \\ \bm{\tau}^{23}  
	\end{bmatrix}\!,
	\quad
	\mbf{c} = \! \begin{bmatrix}
	\mbf{1} \\ \mbf{0} \\ \mbf{0} \\ \mbf{1} \\  \mbf{0} \\ \mbf{0} \\  \mbf{0}
	\end{bmatrix}\!,
	\quad
	\mbf{P} = \begin{bmatrix}
	\mbf{0} & \mbf{0} & \mbf{0} \\
	\mbf{B}^{11} & \mbf{0} & \mbf{0} \\
	\mbf{0} & \mbf{B}^{22} & \mbf{0} \\
	\mbf{0} & \mbf{0} & \mbf{0} \\
	\mbf{B}_1^{12} & \mbf{B}_2^{12} & \mbf{0}\\
	\mbf{0} & \mbf{0} &  \mbf{D}^1 \\
	\mbf{0} & \mbf{0} &  \mbf{D}^2
	\end{bmatrix}^\top \! \!\!\!\!,
	\ \
	\mbf{p} = \begin{bmatrix}
	\mbf{0} \\ \mbf{0} \\ \mbf{f}
	\end{bmatrix}\!,
	\quad
	\mbf{v} = \!\begin{bmatrix}
	\mbf{u}_1\\ \mbf{u}_2  \\ \mbf{w}
	\end{bmatrix}\!,
	\quad
	\bm{\eta}^1_i = \!\begin{bmatrix}
	\mbf{r}(i) \\ \tilde{\bm{\epsilon}}^{11}(i)\\ \tilde{\bm{\epsilon}}^{22}(i) \\ \tilde{\bm{\epsilon}}^{12}(i)
	\end{bmatrix}\!,
	\quad
	\bm{\eta}^2_i = \!\begin{bmatrix}
	\bm{\zeta}^{11}(i)\\ \bm{\zeta}^{22}(i) \\ \bm{\zeta}^{12}(i) \\ \bm{\zeta}^{33}(i)\\ \bm{\zeta}^{13}(i) \\ \bm{\zeta}^{23}(i) 
	\end{bmatrix}\!;
	\end{equation*}
	then we can recognize that $(\mbf{P}_h) = (\mathrm{Dl})$ and $(\mbf{P}^*_h) = (\mathrm{Pr})$. To complete the definition of $(\mathrm{Pr})$, $(\mathrm{Dl})$ along the lines of  \cite[Section 2.5]{ben2001} we put $\mbf{b}=\mbf{0}$ and, in addition, for each $i$ we define the matrices $\mbf{A}^1_i \in \R^{4 \times 7n}$ and $\mbf{A}^2_i \in \R^{6 \times 7n}$ that consist of zeros or ones and allocate the variables into two groups of conic constraints with the cones $\mathrm{K}_{1} = \mathrm{K}^4_\dro$ and $\mathrm{K}_{2} = \mathrm{K}^6_+$, respectively. To recast the problem $(\mbf{P}_h)$ from $(\mathrm{Dl})$ we exploit the duality between the cones $\mathrm{K}^4_\rho$ and $\mathrm{K}^4_\dro$ (see Proposition \ref{prop:cone_rho_dual}), and, in addition, we must change variables: $\bm{\epsilon}^{11} = -\tilde{\bm{\epsilon}}^{11}$, $\bm{\epsilon}^{22} = -\tilde{\bm{\epsilon}}^{22}$, $\bm{\epsilon}^{12} = -\tilde{\bm{\epsilon}}^{12}$.
	
	According to \cite[Section 2.5]{ben2001}, in order to prove that the duality gap vanishes and moreover that problem $(\mbf{P}^*_h)$ admits a solution it is enough to show that $(\mbf{P}_h)$ is \textit{strictly feasible}. Since the rows of matrix $\mbf{P}$ are linearly independent (which is due to the elimination of "rigid motions"), this matter amounts to pointing to feasible variables for $(\mbf{P}_h)$ such that $\bm{\eta}^1_i$ and $\bm{\eta}^2_i$ lie in the interior of the cones $\mathrm{K}^4_\dro$ and, respectively, $\mathrm{K}^6_+$. This is achieved if, for sufficiently small $\delta>0$, we choose: $(\mbf{u}_1, \mbf{u}_2  , \mbf{w})= (\mbf{0},\mbf{0},\mbf{0})$, \ \ $\mbf{r}(i) =1, \ \tilde{\bm{\epsilon}}^{11}(i) = \tilde{\bm{\epsilon}}^{22}(i) = -\delta,\ \tilde{\bm{\epsilon}}^{12}(i) =0$, and $\bm{\zeta}^{11}(i) = \bm{\zeta}^{22}(i) = \delta, \ \bm{\zeta}^{33}(i) =1 ,  \ \bm{\zeta}^{12}(i)=\bm{\zeta}^{13}(i)= \bm{\zeta}^{23}(i) =0$. Existence of solution of $(\mbf{P}_h)$ itself follows directly from the existence of solution for \ref{eq:PMh} as the two problems are equivalent. 
	
	Based on \cite[Section 2.5]{ben2001}, the optimality conditions for feasible variables in a pair of conic problems may be written in a form of complementary slackness conditions that for $(\mathrm{Pr})$,\,$(\mathrm{Dl})$ read as follows: $\pairing{\,\mbf{A}^1_i \mbf{x} \, , \, \bm{\eta}^1_i\, } =0$ and $\pairing{\,\mbf{A}^2_i \mbf{x} \, , \, \bm{\eta}^2_i\, } =0$ for each $i$. Let us, in addition, put: $\epsilon_i = \chi_2\big(\bm{\epsilon}^{11}(i), \bm{\epsilon}^{22}(i), \bm{\epsilon}^{12}(i) \big)$, \ $\zeta_i = \chi_3\bigl(\bm{\zeta}^{11}(i), \bm{\zeta}^{22}(i), \bm{\zeta}^{33}(i), \bm{\zeta}^{12}(i), \bm{\zeta}^{13}(i), \bm{\zeta}^{23}(i)\bigr)$, and $\tau_i = \chi_3\bigl(\bm{\tau}^{11}(i), \bm{\tau}^{22}(i), \bm{\tau}^{33}(i), \bm{\tau}^{12}(i), \bm{\tau}^{13}(i), \bm{\tau}^{23}(i)\bigr)$. We observe that the quadruple $\xi_i,\theta_i,\zeta_i,\epsilon_i$ satisfies the conditions \eqref{eq:quadruple_instead_of_pair}. From the first complementary slackness condition we find that for every $i$
	\begin{align}
	\nonumber
	0=&\,\Big\langle \bigl(\mbf{r}^0(i),\bm{\tau}^{11}(i), \bm{\tau}^{22}(i),\, \bm{\tau}^{12}(i)\bigr) \, , \, \bigl(\mbf{r}(i), \tilde{\bm{\epsilon}}^{11}(i), \tilde{\bm{\epsilon}}^{22}(i), \tilde{\bm{\epsilon}}^{12}(i) \bigr) \Big\rangle_{\R^4}\\
	\nonumber
	=&\,\Big\langle\bigl(\mbf{r}^0(i),\bm{\tau}^{11}(i), \bm{\tau}^{22}(i),\, \bm{\tau}^{12}(i)\bigr) \, , \,  \bigl(1, -{\bm{\epsilon}}^{11}(i), -{\bm{\epsilon}}^{22}(i), -{\bm{\epsilon}}^{12}(i) \bigr) \Big\rangle_{\R^4}
	= \,\mbf{r}^0(i) - \pairing{\epsilon_i,\sigma_i}_{\mathcal{S}^{2\times 2}}.
	\end{align} 
	From \eqref{eq:quadruple_instead_of_pair} we find that $\rho(\epsilon_i) \leq 1$,  whilst from the conic constraints we directly have $\dro(\sig_i) \leq \mbf{r}^0(i)$. As a result the first complementary slackness condition is equivalent to $\pairing{\epsilon_i,\sigma_i}_\Sdd =  \mbf{r}^0(i) = \dro(\sig_i)$. The second one may be rewritten as
	\begin{align*}
	\nonumber
	0 =& \, \Big\langle \bigl(\bm{\tau}^{11}(i), \bm{\tau}^{22}(i), \bm{\tau}^{33}(i), \bm{\tau}^{12}(i), \bm{\tau}^{13}(i), \bm{\tau}^{23}(i)\bigr) \, , \, \bigl(\bm{\zeta}^{11}(i), \bm{\zeta}^{22}(i), \bm{\zeta}^{33}(i), \bm{\zeta}^{12}(i), \bm{\zeta}^{13}(i), \bm{\zeta}^{23}(i)\bigr) \bigr) \Big\rangle_{\R^6}\\
	=&\, \pairing{\zeta_i,\tau_i}_{\mathcal{S}^{3\times 3}} = \pairing{\begin{bmatrix}
		\,\epsilon_i  &  0_2\,\\
		\,0_2^\top\, & 1\,
		\end{bmatrix}-\begin{bmatrix}
		\,\xi_i  &  \tfrac{1}{\sqrt{2}}\,\theta_i\,\\
		\tfrac{1}{\sqrt{2}}\,\theta_i^\top\, & 0\,
		\end{bmatrix},
		\begin{bmatrix}
		\,\sig_i  &  \tfrac{1}{\sqrt{2}}\,q_i\,\\
		\tfrac{1}{\sqrt{2}}\,q_i^\top\, & \bm{\tau}^{33}(i)\,
		\end{bmatrix}
	}_{\mathcal{S}^{3\times 3}}\\
	=&\, \pairing{\epsilon_i,\sigma_i}_{\mathcal{S}^{2\times 2}} + \bm{\tau}^{33}(i) - \Big(\pairing{\xi_i,\sigma_i}_{\mathcal{S}^{2\times 2}} + \pairing{\theta_i,q_i}_{\Rd} \Big) = \dro(\sig_i) + \bm{\tau}^{33}(i) - \Big(\pairing{\xi_i,\sigma_i}_{\mathcal{S}^{2\times 2}} + \pairing{\theta_i,q_i}_{\Rd} \Big).
	\end{align*}
	Since $\tau_i$ is positive semi-definite, from Proposition \ref{prop:pos_semdef_block_mat} we infer that $q_i \in \IM\,\sig_i$ and that $\bm{\tau}^{33}(i) \geq  \frac{1}{2} \,\pairing{\sigma_i^{-1} q_i,q_i}$, which in turn implies that $\dro(\sig_i) + \bm{\tau}^{33}(i) \geq \vro^0(\sig_i,q_i)$. On the other hand, owing to Lemma \ref{lem:quadruple_instead_of_pair}, $\vro(\xi_i,\theta_i) \leq 1$. Combining these facts with Proposition \ref{prop:varrho_polar} we can readily infer the conditions (i),\,(ii),\,(iii) in the "moreover part" of the theorem.
\end{proof}

\section{$L^\infty$-error estimates for finite element interpolations}

The finite element approximation proposed in this work involves a 2D triangular mesh for which the approximations $(u_h,w_h):\O \to \Rd \times \R$ are assumed to be continuous and element-wise affine.  Below a classical interpolation result \cite{ciarlet2002} is slightly adjusted to meet the needs of the present paper: 

\begin{lemma}
	\label{lem:interpolation}
	For a polygonal bounded domain $\Omega$ let us fix a pair of smooth functions $(\tilde{u},\tilde{w}) \in \D(\Omega;\Rd\times \R)$. Then there exist constants $C_1,C_2,C_3 >0$, depending on $\alpha_0$ and $(\tilde{u},\tilde{w})$ only, such that for any $\alpha_0$-regular triangulation $\mathcal{T}^h$ of $\Omega$ there hold the estimates
	\begin{align}
	\label{eq:o2u}
	\norm{u_h-\tilde{u}}_{L^\infty(\O)} & \leq C_1 h^2,\\
	\label{eq:o2w}
	\norm{w_h-\tilde{w}}_{L^\infty(\O)} & \leq C_2 h^2,\\
	\label{eq:o1}
	\norm{\vro\big(e(u_h)-e(\tilde{u}),\nabla w_h-\nabla \tilde{w}\big)}_{L^\infty(\O)} &\leq C_3 h,
	\end{align}
	where $(u_h,w_h) = P_h (\tilde{u},\tilde{w}) $ according to \eqref{eq:projection}.
\end{lemma}
\begin{proof}
	By a standard interpolation result (see \cite[Theorem 3.1.6]{ciarlet2002}) the estimates \eqref{eq:o2u}, \eqref{eq:o2w} hold true together with
	\begin{align}
		\label{eq:gradu_est}
		\norm{\nabla u_h-\nabla \tilde{u}}_{L^\infty(\O)} & \leq C_4 h,\\
		\label{eq:gradw_est}
		\norm{\nabla w_h-\nabla \tilde{w}}_{L^\infty(\O)} & \leq C_5 h,
	\end{align}
	where $C_1,C_2,C_4$, $C_5$ depend on the regularity parameter $\alpha_0$ and on $\norm{D^2 u}_{L^\infty(\O)}, \norm{D^2 w}_{L^\infty(\O)}$. Let us agree that the $L^\infty$-norm in the estimate \eqref{eq:gradu_est} is meant with respect to the operator norm $\abs{\argu}_\mathrm{op}$ on $\R^{2\times 2}$ (which affects the constant $C_4$ only).
	Next, we observe that $(\xi,\theta)\mapsto \abs{\argu}_{\mathrm{op}}+ \abs{\argu}$ is a norm on $\Sdd \times \Rd$. Since $\vro:\Sdd \times \Rd \to \R_+$ is a positively 1-homogeneous convex function that is finite-valued, it is well established that there exists a constant $C_6$ such that $\vro(\xi,\theta) \leq C_6\big(\abs{\xi}_{\mathrm{op}}+ \abs{\theta}\big)$. Ultimately, we arrive at \eqref{eq:o1} for $C_3 = C_6({C}_4 + C_5)$ by combining \eqref{eq:gradu_est} and \eqref{eq:gradw_est}. To that aim we additionally remark that the operator norm of a square matrix is always bigger than the operator norm of its symmetric part, and, as a result, inequality \eqref{eq:gradu_est} remains true when $\nabla u_h - \nabla \tilde{u}$ is replaced by $ e(u_h) -  e(\tilde{u})$.
\end{proof}

\setlength{\bibsep}{2pt}
\bibliographystyle{spbasic}      


\end{document}